\documentclass[12pt,a4paper]{report}

\usepackage{amsmath,amsthm,amsfonts,amssymb,array,amscd}
\usepackage{amsmath,geometry,amssymb,amsfonts,amsthm,graphicx,enumerate,latexsym,tabularx,amscd}
\usepackage{graphicx,color}
\usepackage[all,cmtip]{xy}

\usepackage{refcheck} 
\norefnames
\nocitenames
\usepackage{amsthm}
 \newtheorem{thm}{Theorem}[section]
 \newtheorem{prop}[thm]{Proposition}
 \newtheorem{lem}[thm]{Lemma}
 \newtheorem{cor}[thm]{Corollary}
\theoremstyle{definition}
 \newtheorem{exm}[thm]{Example}
 \newtheorem{dfn}[thm]{Definition}
\theoremstyle{remark}
 \newtheorem{rem}[thm]{Remark}
 \numberwithin{equation}{section}
\theoremstyle{definition}
\theoremstyle{remark}
 \numberwithin{equation}{section}

\renewcommand{\le}{\leqslant}\renewcommand{\leq}{\leqslant}
\renewcommand{\ge}{\geqslant}\renewcommand{\geq}{\geqslant}
\renewcommand{\setminus}{\smallsetminus}
\usepackage{tikz}
\usepackage[pdftex]{hyperref}
\usetikzlibrary{matrix,arrows}

\theoremstyle{definition}

\newcommand{\bbC}{\mathbb{C}}
\newcommand{\bbF}{\mathbb{F}}
\newcommand{\bbN}{\mathbb{N}}

\newcommand{\bbQ}{\mathbb{Q}}
\newcommand{\bbR}{\mathbb{R}}

\newcommand{\bbZ}{\mathbb{Z}}   

\newcommand{\cN}{\mathcal{N}}

\renewcommand{\and}{\quad \mbox{and} \quad}  
\renewcommand{\le}{\leqslant}\renewcommand{\leq}{\leqslant}
\renewcommand{\ge}{\geqslant}\renewcommand{\geq}{\geqslant}
\renewcommand{\setminus}{\smallsetminus}
\begin{document}

\pagenumbering{roman}

\thispagestyle{empty}

\begin{center}

\large{\textbf{LOCAL CONSTANTS FOR GALOIS REPRESENTATIONS - SOME EXPLICIT RESULTS}}\\

\vspace{.6in}
A thesis submitted during 2015 to the University of Hyderabad in partial fulfillment of the award of a 
\textbf{Doctor of Philosophy degree} in {\bf Mathematics}\\
by\\
\textbf{SAZZAD ALI BISWAS}\\
\vspace{.5in}

\vspace{.6in}
School of Mathematics and Statistics\\
University of Hyderabad\\
Hyderabad - 500046, India\\
August $2015$

\end{center}







\newpage

\thispagestyle{empty}
\begin{center}
 \large{\textbf{DECLARATION}}
\end{center}

I, \textbf{SAZZAD ALI BISWAS} hereby declare that this thesis entitled
``\textbf{LOCAL CONSTANTS FOR GALOIS REPRESENTATIONS - SOME EXPLICIT RESULTS}'' submitted by me under the guidance and 
supervision of Professor Rajat Tandon is a bona fide research work which is also free from plagiarism. 
I also declare that it has not been 
submitted previously in part or in full to this University or any other University or Institution for the award of any degree or 
diploma. I hereby agree that my thesis can be deposited in {\bf Shodganga-INFLIBNET}.\\
{\bf A report on plagiarism statistics from the University Librarian is enclosed}\\

\vspace{1in}

Candidate\\

\hspace{13cm}Date:\\
 \textbf{SAZZAD ALI BISWAS}\\
 Regd. No: 08MMPP04\\
 \begin{center}
  \hspace{3cm}//\Large{Countersigned}//
 \end{center}
 \vspace{1cm}
 Signature of the Supervisor:

\newpage

\thispagestyle{empty}
\begin{center}
 \large{\textbf{CERTIFICATE}}
\end{center}
\vspace{.5in}
School of Mathematics and Statistics\\
University of Hyderabad\\
Hyderabad - 500046, India

\vspace{.3in}
This is to certify that the thesis entitled ``\textbf{LOCAL CONSTANTS FOR GALOIS REPRESENTATIONS - SOME EXPLICIT RESULTS}''
submitted by Mr. {\bf SAZZAD ALI BISWAS} bearing Regd. No: {\bf 08MMPP04} in partial fulfillment of the requirements for the award of 
Doctor of Philosophy in {\bf MATHEMATICS} is a bona fide work carried out by him under my supervision and guidance which 
is a plagiarism free thesis.\\

{\bf The thesis has not been submitted previously in part or in full to this or any other University or Institution for 
the award of any degree or diploma.}\\
\vspace{.5in}

\begin{center}
\begin{minipage}{4.5in}
Supervisor \hfill$\underset{\hbox{(Prof. Rajat Tandon)}}{\rule{3in}{1pt}}$\\[15pt]
\bigskip
\vfill
Dean of the School\hfill$\underset{\hbox{(Prof. B. Sri Padmavati)}}{\rule{3in}{1pt}}$\\[15pt]

\end{minipage}\\

\end{center}
 
\newpage

\chapter*{Dedication}
\vspace*{\fill}
\begin{center}
 \Large{To my parents\\
  Chhakina Bibi, Mainuddin Biswas}
\end{center}
\vspace*{\fill}

\newpage

\chapter*{Acknowledgments}
\addcontentsline{toc}{chapter}{Acknowledgments}

I would like to thank my advisor Prof. Rajat Tandon for teaching me so many things and for his patience, unending encouragements
 and helps. I extend my gratitude to him for allowing me to visit Berlin during my Ph.D.
 
 I would also like to thank Prof. E.-W. Zink, Humboldt University, Berlin, for suggesting the problems in the thesis 
 and his constant 
 guidance, valuable suggestions and comments. Without his helps, this thesis work would not have been possible. 
 Many ways I am indebted to him. I also express my sincere 
 gratitude to Prof. Elmar Grosse-Kl\"{o}nne, Humboldt University, for providing me a very good mathematical environment during my stay 
 in Berlin.
 
 In the last year of my Ph.D program, I had very useful discussion with Prof. Dipendra Prasad, TIFR, India. He is always very kind to 
 me. I express my deep gratitude to him for his constant support. I also express my gratitude to Prof. Robert Boltje, Department 
 of Mathematics, University of California, Santa Cruz, USA for sending me hard copy of his Ph.D thesis and paper which were 
 very much useful for my works.
 
 I would also like to thank Prof. T. Amaranath (former dean of School of Mathematics and Statistics) for his all kind of
 help and encouragement. Many thanks to my doctoral committee members- Prof. V. Kannan and Prof. V. Suresh (currently 
 Professor at Emory University, USA) for their advices, discussions and valuable comments.
 
 I express my very deep gratitude to Prof. Kumaresan  and all the members of School of Mathematics and Statistics of University of 
 Hyderabad.

 I am grateful to CSIR, Delhi and IMU-Berlin Einstein Foundation, Berlin for 
 their financial support during my Ph.D.

 I would also like to thank all my colleagues, specially to Deborjoti for his help and very good friendship which made my life 
 enjoyable in Hyderabad. I also thank to Niraj, Srikant, Sampath, Rakhee, Ambika, Chiru, Sharan and Bharbhava for their love and 
 beautiful company and making my stay very easy in Hyderabad.
 
I am grateful to all my teachers during my school days. At school, specially I thank my teachers Mr. Nurul Islam 
(General Secretary, Al-Ameen Mission, Howrah), Dr. Musarraf Hossain (Chemistry), Dr. Mrinal Kanti Dwari (Physics),
Mr. Sk Hasibul Alam (Maths) and Mr. Amir 
Hossain Mandal (Maths) for their love, encouragement and very good teaching. At my undergraduate and post graduate level 
(Jadavpur University, Calcutta, India),
I would like to thank Dr. Sujit Kumar Sardar, Dr. Kallol Paul, Dr. Shamik Ghosh and Dr. Indranath Sengupta for their beautiful
teaching.

I would also like to thank all my friends, Sudip, Prabir, Somnath, Angshuman, Atanu, Tapas, Subhodip, Tilak, Prabin, 
Khan, Nasarul, Hira, Biswa, Ekta, for their
very good company throughout my life.

Finally, I want to thank my family, specially my elder brother, Manajat Ali Biswas, for their constant unconditional support, 
encouragement and freedom to achieve any goal to which I desire.

 \newpage

 \chapter*{Abstract}
\addcontentsline{toc}{chapter}{Abstract}

We can associate local constant to every  continuous finite dimensional
complex representation of the absolute Galois group $G_F$ of a non-archimedean local field 
$F/\bbQ_p$ by Deligne and Langlands. For linear characters of $F^\times$, 
Tate \cite{JT1} gives an explicit formula for the abelian local constants, 
but in general, there is no such explicit formula of local constant for any arbitrary representation of a local Galois group. 
To give explicit formula of local constant of a representation, we need to compute $\lambda$-functions explicitly. 
In this thesis we compute $\lambda_{K/F}$ explicitly, where $K/F$ is a finite degree Galois extension of 
a non-archimedean local field $F$, except when $K/F$ is a wildly ramified quadratic extension with $F\ne\bbQ_2$.

Then by using this $\lambda$-function computation, in general, we give an invariant formula of local constant of finite dimensional 
Heisenberg representations 
of the absolute Galois group $G_F$ of a non-archimedean local field $F$. But for explicit invariant formula of local constant
for a Heisenberg representation, we should have information about the dimension of a Heisenberg representation and the arithmetic
description of the determinant of a Heisenberg representation. In this thesis,
we give explicit arithmetic description of the determinant of 
Heisenberg representation.

We also construct all Heisenberg representations of dimensions prime to $p$, and study their various properties.
By using $\lambda$-function computation and arithmetic description of the determinant of Heisenberg representations, we give an 
invariant formula of local constant for a Heisenberg representation of dimension prime to $p$.\\

\textbf{Keywords:} Local Fields, Extendible functions, Local constants, Lambda functions, Classical Gauss sums, 
Heisenberg representations, Transfer map, Determinant, Artin conductors, Swan conductors.

\newpage

\chapter*{Notation}

\addcontentsline{toc}{chapter}{Notation}

By $\mathbb{N}$, $ \mathbb{Z}$, $\mathbb{Q}$, $\mathbb{R}$ and $\mathbb{C}$ we denote the natural numbers,
integers, rational numbers, real numbers and complex numbers respectively.

For two integers $m, n$, we denote $gcd(m,n)$ and $\rm{lcm}(m,n)$ as the greatest common divisor
and least common multiple of $m,n$ respectively.

For two sets $A$ and $B$, $A\subset B$ means $A$ is contained in $B$. The number of elements in a finite set $A$ will be 
denoted by $|A|$. 

For any homomorphism (group or ring) $\varphi$, we denote $\mathrm{Ker}(\varphi)=$ kernel of $\varphi$ and $\mathrm{Im}(\varphi)=$
image of the homomorphism $\varphi$.

Representation of a group means complex representation, otherwise it will be stated.
For a group $G$ we denote by $\widehat{G}$ the group of linear characters of $G$, by
$\mathrm{Irr}(G)$ the set of irreducible
representations of $G$, by $\rm{PI}(G)$ the set of all projective irreducible representation of $G$,
by $Z(G)$ the center of $G$ and by $[G,G]$ the commutator subgroup of $G$. 

For a group $G$, $H\leq G$, $H<G$, $H\vartriangleleft G$ denotes that $H$ is a subgroup, a proper subgroup,
 a normal subgroup of $G$ and for index of $H$ in $G$, we write $[G:H]$.
 
 For a representation $\rho$ of $G$, if $H\leq G$ we denote by $\rho|_H$ the restriction of $\rho$ to $H$. 
 For $G$-set we denote by $M^{G}$ the set of $G$-invariant
 elements of $M$. For representation $\rho$ of a group, we define $\mathrm{dim}\,\rho=$ dimension of $\rho$. 
 
 Throughout this thesis, $F$ will denote a non-archimedean local field, that is a finite extension of $\bbQ_p$ for some 
 prime $p$. We write $O_F$ for the ring of integers in $F$, $P_F$ for the prime ideal in $O_F$ and $\nu_F$ for the 
 valuation of $F$. The residue field $O_F/P_F$ is denoted by $k_{q_F}$, and its number of elements is denoted by 
 $q_F$. We write $U_F$ for $O_{F}^{\times}$, i.e., the group of units of $F$ and $\overline{F}$ for an algebraic closure of $F$. 
 
 The degree of finite fields extension
$K/F$ is denoted by $[K:F]$; if the extension is Galois, we write $\mathrm{Gal}(K/F)$ for the Galois group of $K/F$.
We denote $\det(\rho)$ for the determinant of the representation $\rho$.
We also denote $\Delta_{K/F}$ for the determinant of the representation $\rm{Ind}_{K/F}(1)$.
$N_{K/F}$ denotes the norm map from $K^\times$ to $F^\times$ and $\rm{Tr}_{K/F}$ for the trace map from $K$ to $F$. 

We denote $\mu_{p^\infty}$ as the group of roots of unity of $p$-power order, and $T_{G/H}$ as the transfer map from 
$G\to H/[H,H]$.

\newpage

\tableofcontents


\newpage

\pagenumbering{arabic}  

\chapter{\textbf{Introduction}}

Let $F$ be a non-archimedean local field 
(i.e., finite extension of the $p$-adic field $\mathbb{Q}_p$, for some prime $p$). Let $\overline{F}$ be an algebraic closure
of $F$, and $G_F:=\rm{Gal}(\overline{F}/F)$ be the absolute Galois group of $F$. Let
\begin{center}
 $\rho:G_F\to \mathrm{Aut}_{\mathbb{C}}(V)$ 
\end{center}
be a finite dimensional continuous complex representation of the Galois group $G_F$. For this
$\rho$, we can associate a constant 
$W(\rho)$ with absolute value $1$ by Langlands (cf. \cite{RL}) and Deligne (cf. \cite{D1}).
This constant is called the \textbf{local constant} (also known as local epsilon factor) of the representation $\rho$. 
Langlands also proves that 
these local constants are weakly extendible functions (cf. \cite{JT1}, p. 105, Theorem 1).

The existence of this local constant is proved by Tate for one-dimensional representation in  
\cite{JT3} and 
the general proof of the existence of the local constants is proved by Langlands (see \cite{RL}). 
In 1972 Deligne also gave a proof  using global methods in \cite{D1}.
But in Deligne's terminology this local constant $W(\rho)$ is 
 $\epsilon_{D}(\rho,\psi_F,\mathrm{dx},1/2)$, where $\mathrm{dx}$ is the Haar
 measure on $F^{+}$ (locally compact abelian group) which is self-dual with respect to the canonical additive character $\psi_F$ of $F$. 
Tate in his article \cite{JT2} denotes this Langlands convention of local constants as $\epsilon_{L}(\rho,\psi)$. 
According to Tate (cf. \cite{JT2}, p. 17),
 the Langlands factor $\epsilon_{L}(\rho,\psi)$ is 
 $\epsilon_{L}(\rho,\psi)=\epsilon_{D}(\rho\omega_{\frac{1}{2}},\psi,\mathrm{dx_{\psi}})$, where $\omega$ denotes the normalized 
 absolute value of $F$, i.e., $\omega_{\frac{1}{2}}(x)=|x|_{F}^{\frac{1}{2}}=q_{F}^{-\frac{1}{2}\nu_{F}(x)}$ 
 (in Section 2.3 we discuss them in more details)
 which we may consider as a character of $F^\times$, and where
 $\mathrm{dx_{\psi}}$ is the self-dual Haar measure corresponding to the additive character $\psi$ and 
 $q_F$ is the cardinality of the residue field of $F$. According to Tate (cf. \cite{JT1}, p. 105) 
 the relation among three conventions of the local constants is:
 \begin{equation}
  W(\rho)=\epsilon_{L}(\rho,\psi_F)=\epsilon_{D}(\rho\omega_{\frac{1}{2}},\psi_F,\mathrm{dx_{\psi_F}}).
 \end{equation}

In Chapter 2 we recall almost all necessary ingredients for this thesis, and they 
will be used in our next chapters. First we mention the basic properties of local fields and their extensions.
We study the Lemma \ref{Lemma 2.21} of J-P. Serre and it will be used in Chapter 3. This result has a great role for the 
computation of $\lambda$-functions.
The central keywords  of this thesis are the local constants and Heisenberg representations. Therefore in Chapter 2 we spend little
more time on them.
All properties of local constants are well-studied by various people:
Langlands \cite{RL}, Tate \cite{JT1}, Deligne \cite{D1}, Bushnell-Henniart \cite{BH}. We show the connection
between all conventions of local constants. We also know that there is a connection between the classical Gauss sums and the 
local constants of representations of a local Galois group. When the conductor of a multiplicative character of a local field is one,
then by proper choice of additive character we see that the computation 
of local constants comes down to the computation of the classical Gauss sums. Therefore for our lambda-functions computation we 
need the Theorem 
\ref{Theorem 3.5} (regarding classical Gauss sum of quadratic character).
In Section 2.6 we discuss the group theoretical structure of Heisenberg representations, and for this we follow the articles 
\cite{Z2}, \cite{Z3}, \cite{Z5}.

In Chapter 3 we compute the lambda-functions for finite Galois extensions explicitly, except wildly ramified quadratic extensions.  
Let $K/F$ be a finite Galois extension of a non-archimedean local field $F/\bbQ_p$ and $G=\rm{Gal}(K/F)$. 
The lambda function 
for the extension $K/F$ is $\lambda_{K/F}(\psi_F):=W(\mathrm{Ind}_{K/F} 1,\psi_F)$, where $W$ denotes the 
local constant  and $\psi_F$ is the canonical additive character of $F$. We also can define $\lambda$-function
via Deligne's constant $c(\rho):=\frac{W(\rho)}{W(\det(\rho))}$, where $\rho$ is a representation of $G$ and 
\begin{center}
 $\lambda_{1}^{G}=W(\rho)=c(\rho)\cdot W(\det(\rho))=c(\rho)\cdot W(\Delta_{1}^{G})$,
\end{center}
where $\rho=\rm{Ind}_{1}^{G}(1)$ and $\Delta_{1}^{G}:=\det(\rm{Ind}_{1}^{G}(1))$.

Firstly, in Section 3.2 we compute the $\lambda$-function for odd degree Galois extension by using some functoriality 
properties of $\lambda$-functions and the following lemmas (cf. Lemma \ref{Lemma 4.1}, Lemma \ref{Lemma 4.2});
\begin{lem}\label{Lemma 4.1}
Let $L/F$ be a finite Galois extension and $F\subset K\subset L$ with $H=\rm{Gal}(L/K)$ and  $G=\rm{Gal}(L/F)$.
If $H\leq G$ is a normal subgroup and if $[G:H]$ is odd, then $\Delta_{K/F}\equiv 1$ and 
$\lambda_{K/F}^{2}=1.$
\end{lem}

\begin{lem}\label{Lemma 4.2}
The $G$ and $H$ are as in the previous lemma.
\begin{enumerate}
 \item If $H\leq G$ is a normal subgroup of odd index $[G:H]$, then $\lambda_{H}^{G}=1$.
 \item If there exists a normal subgroup $N$ of $G$ such that $N\leq H\leq G$ and $[G:N]$ odd, then $\lambda_{H}^{G}=1$.
\end{enumerate}
\end{lem}
 we obtain the following result (cf. Theorem \ref{General Theorem for odd case})
 \begin{thm}
Let $F$ be a non-archimedean local field and $E/F$ be an odd degree Galois extension. If 
$L\supset K\supset F $ be any finite extension inside $E$, then $\lambda_{L/K}=1$. 
\end{thm}

And in Section 3.3 we compute $\lambda_{1}^{G}$, where $G$ is a local Galois group for a finite Galois extension.
By using Bruno Kahn's results (cf. \cite{BK}, Theorem 1) and Theorem \ref{Theorem 4.1} (due to Deligne)
we obtain the following result (cf. Theorem \ref{Theorem 4.3}).
\begin{thm}\label{Theorem 4.3}
 Let $G$ be a finite local Galois group of a non-archimedean local field $F$. Let $S$ be a Sylow 2-subgroup of $G$. 
 Denote $c_{1}^{G}=c(\rm{Ind}_{1}^{G}(1))$.
 \begin{enumerate}
 \item If $S=\{1\}$, then we have $\lambda_{1}^{G}=1$. 
  \item If the Sylow 2-subgroup $S\subset G$ is nontrivial cyclic (\textbf{exceptional case}), then
  \begin{equation*}
   \lambda_{1}^{G}=\begin{cases}
                    W(\alpha) & \text{if $|S|=2^n\ge 8$}\\
                    c_{1}^{G}\cdot W(\alpha) & \text{if $|S|\le 4$},
                   \end{cases}
\end{equation*}
where $\alpha$ is a uniquely determined quadratic 
  character of $G$.
  \item If $S$  is metacyclic but not cyclic ({\bf invariant case}), then 
  \begin{equation*}
   \lambda_{1}^{G}=\begin{cases}
                    \lambda_{1}^{V} & \text{if $G$ contains Klein's $4$ group $V$}\\
                    1 &  \text{if $G$ does not contain Klein's $4$ group $V$}.
                   \end{cases}
\end{equation*}
  \item If $S$ is nontrivial and not metacyclic, then $\lambda_{1}^{G}=1$.
 \end{enumerate}
\end{thm}
From the above theorem we observe that $\lambda_{1}^{G}=1$, except the {\bf exceptional case}  and 
the {\bf invariant case} when $G$ contains Klein's $4$-group.
Moreover, here $\alpha$ is the uniquely determined quadratic character of $G$, then $W(\alpha)=\lambda_{F_2/F}$, where $F_2/F$ is the 
quadratic extension corresponding to $\alpha$ of $G=\rm{Gal}(K/F)$. In fact, in the invariant case we need to compute 
$\lambda_{1}^{V}$, where $V$ is Klein's $4$-group. If $p\ne 2$ then $V$ corresponds to a tame extension and there we have  
the explicit computation of $\lambda_{1}^{V}$ in Lemma \ref{Lemma 4.6}. 

\begin{lem}\label{Lemma 4.6}
Let $F/\bbQ_p$ be a local field with $p\ne 2$. Let $K/F$ be the uniquely determined extension with $V=\rm{Gal}(K/F)$, Klein's $4$-group. 
Then \\
 $\lambda_{1}^{V}=\lambda_{K/F}=-1$ if $-1\in F^\times$ is a square, i.e., $q_F\equiv 1\pmod{4}$, and\\
 $\lambda_{1}^{V}=\lambda_{K/F}=1$ if $-1\in F^\times$ is not a square, i.e., if $q_F\equiv 3\pmod{4}$,\\
 where $q_F$ is the cardinality of the residue field of $F$.
\end{lem}

In \cite{D2} Deligne computes the Deligne's constant $c(\rho)$, where $\rho$ is a finite dimensional orthogonal representation of the 
absolute Galois group $G_F$ of a non-archimedean local field $F$ via second Stiefel-Whitney class of $\rho$. 
The computations of these Deligne's constants are important to give explicit formulas of lambda functions.
But the second Stiefel-Whitney classes $s_2(\rho)$ are not the same as Deligne's
constants, therefore full information of $s_2(\rho)$ will not give complete information about Deligne's constants.
Therefore to complete the 
explicit computations of
$\lambda_{1}^{G}$ we need to use the definition 
$$\lambda_{1}^{G}=\lambda_{K/F}(\psi)=W(\rm{Ind}_{K/F} 1, \psi),$$
where $\psi$ is a nontrivial additive character of $F$. 
When we take the canonical additive character $\psi_F$, we simply write $\lambda_{K/F}=W(\rm{Ind}_{K/F} 1,\psi_F)$,
instead of $\lambda_{K/F}(\psi_F)$.

When $p\ne 2$, in Theorem \ref{Theorem 4.3}, 
we notice that to complete the whole computation we need to compute $\lambda_{K/F}$, where $K/F$ is a quadratic tame extension.
In general, in Section 3.4 we study the explicit computation for even degree local Galois extension. 

 Moreover, by using the properties of $\lambda$-function and Lemma \ref{Lemma 3.4} we give general formula of $\lambda_{K/F}$, 
where $K/F$ is an even degree {\bf unramified extension} and the result is (cf. Theorem \ref{Theorem 3.6}):
\begin{center}
 $\lambda_{K/F}(\psi_F)=(-1)^{n(\psi_F)}$.
\end{center}
When $K/F$ is an even degree Galois extension with odd ramification index we have the following result
(cf. Theorem \ref{Theorem 3.8}).
\begin{thm}
 Let $K$ be an even degree Galois extension of $F$ with odd ramification index. Let $\psi$ be a nontrivial
 additive character of $F$. Then 
 \begin{equation*}
  \lambda_{K/F}(\psi)=(-1)^{n(\psi)}.
 \end{equation*}
\end{thm}
When $K/F$ is a tamely ramified quadratic extension, by using classical Gauss sum we have an explicit formula 
(cf. Theorem \ref{Theorem 3.21}) for $\lambda_{K/F}$.
\begin{thm}\label{Theorem 3.21}
 Let $K$ be a tamely ramified quadratic extension of $F/\bbQ_p$ with $q_F=p^s$. Let $\psi_F$ be the canonical additive character of $F$.
 Let $c\in F^\times$ with $-1=\nu_F(c)+d_{F/\bbQ_p}$, and $c'=\frac{c}{\rm{Tr}_{F/F_0}(pc)}$, where $F_0/\bbQ_p$ is the maximal unramified
 extension in $F/\bbQ_p$. Let $\psi_{-1}$ be an additive character with conductor $-1$, of the form $\psi_{-1}=c'\cdot\psi_F$.
 Then 
 \begin{equation*}
  \lambda_{K/F}(\psi_F)=\Delta_{K/F}(c')\cdot\lambda_{K/F}(\psi_{-1}),
 \end{equation*}
where 
 \begin{equation*}
 \lambda_{K/F}(\psi_{-1})=\begin{cases}
                                               (-1)^{s-1} & \text{if $p\equiv 1\pmod{4}$}\\
                                                  (-1)^{s-1}i^{s} & \text{if $p\equiv 3\pmod{4}$}.
                                            \end{cases}
\end{equation*}
If we take $c=\pi_{F}^{-1-d_{F/\bbQ_p}}$, where $\pi_F$ is a norm for $K/F$, then 
\begin{equation}
 \Delta_{K/F}(c')=\begin{cases}
                   1 & \text{if $\overline{\rm{Tr}_{F/F_0}(pc)}\in k_{F_0}^{\times}=k_{F}^{\times}$ is a square},\\
                   -1 & \text{if $\overline{\rm{Tr}_{F/F_0}(pc)}\in k_{F_0}^{\times}=k_{F}^{\times}$ is not a square}.
                  \end{cases}
\end{equation}
Here "overline" stands for modulo $P_{F_0}$.
\end{thm}
By the properties of $\lambda$-function and the above theorem, we can give complete computation of $\lambda_{K/F}$, where 
$K/F$ is a tamely ramified even degree Galois extension. But the computation of $\lambda_{K/F}$, where $K/F$ is a wildly
ramified quadratic extension, seems subtle. In Subsection 3.4.2 we give few computation in the wild case, and they are:

\begin{lem}
 Let $K$ be the finite abelian extension of $\bbQ_2$ for which $N_{K/\bbQ_2}(K^\times)={\bbQ_{2}^{\times}}^2$. 
 Then $\lambda_{K/\bbQ_2}=1$.
\end{lem}
More generally, we show (cf. Theorem \ref{Theorem 3.26}) that 
when $F/\bbQ_2$, and $K$ is the abelian extension of $F$ for which $N_{K/F}(K^\times)={F^\times}^2$, then 
$\lambda_{K/F}=1$.
In Example \ref{Example wild} we compute
$\lambda_{F/\bbQ_2}$, where $F$ is a quadratic extension of $\bbQ_2$.

In Chapter 4 we give an invariant formula (group theoretical) of the determinant of a Heisenberg representation.
From Gallagher's theorem 
\ref{Theorem Gall} we know that to compute the determinant of an induced representation we need to compute 
the transfer map. Since our Heisenberg representations of degree $\ge 2$ are induced representation, we first need to 
give formula for transfer maps which we give in Lemmas \ref{Lemma 2.2} and \ref{Lemma 2.10}. By using them in 
Proposition \ref{Proposition 2.13} we give an invariant formula of the determinant of Heisenberg representation. 


Let $G$ be a finite group and $\rho$ be a Heisenberg representation of $G$. Let $Z$ be the scalar group for $\rho$ and 
$H$ be a maximal isotropic subgroup of $G$ for $\rho$. Let $\chi_H$ be a linear character of $H$ which is an extension of 
the central character $\chi_Z$ of $\rho$. Then  we know that 
 $\rho=\mathrm{Ind}_{H}^{G}\chi_H$. The main aim of Chapter 4 is to give an invariant formula for:
\begin{center}
 $\mathrm{det}(\rho)=\mathrm{det}(\mathrm{Ind}_{H}^{G}\chi_H)$.
\end{center}
In other words, we will show that this formula is independent of the choice of the maximal isotropic subgroup $H$ because 
the maximal isotropic subgroups are {\bf not} unique.
From Gallagher's result (cf. Theorem \ref{Theorem Gall}) we know that 
\begin{center}
 $\det(\rm{Ind}_{H}^{G}\chi_H)(g)=\Delta_{H}^{G}(g)\cdot\chi_H(T_{G/H}(g))$ for all $g\in G$.
\end{center}
Therefore for explicit computation of determinant of Heisenberg representation $\rho$, we first need to compute the transfer map 
$T_{G/H}$. In Lemma \ref{Lemma 2.2} we compute $T_{G/H}$, when $H$ is an abelian normal subgroup of a finite group $G$ 
(of odd index in $G$)
with $[G,[G,G]]=\{1\}$, and $G/H$ is an abelian quotient group.
\begin{lem}\label{Lemma 2.2}
 Assume that $G$ is a finite group and $H$ a normal subgroup of $G$ such that 
 \begin{enumerate}
  \item $H$ is abelian,
  \item $G/H$ is abelian of odd order $d$,
  \item $[G,[G,G]]=\{1\}$. 
 \end{enumerate}
Then we have $T_{G/H}(g)=g^{d}$ for all $g\in G$.\\
As a consequence one has $[G,G]^d=\{1\}$, in other words, $G^d$ is contained in the center of $G$.
\end{lem}
More generally, combining this above Lemma  and the elementary divisor theorem, we obtain the following 
result (cf. Lemma \ref{Lemma 2.10}).

\begin{lem}\label{Lemma 2.10}
 Assume that $G$ is a finite group and $H$ a normal subgroup of $G$ such that 
 \begin{enumerate}
  \item $H$ is abelian
  \item $G/H$ is abelian of order $d$, such that (according to the elementary divisor theorem):
  \begin{center}
   $G/H\cong\mathbb{Z}/m_1\times\cdots\times\mathbb{Z}/m_s$
  \end{center}
where $m_1|\cdots|m_s$ and $\prod_{i}m_i=d$. Moreover, we fix elements $t_1,t_2,\cdots,t_s\in G$ such that $t_iH\in G/H$ generates the 
cyclic factor $\cong\mathbb{Z}/m_i$, hence $t_{i}^{m_i}\in H$.
\item $[G,[G,G]]=\{1\}$. In particular, $[G,G]$ is in the center $Z(G)$ of $G$.
 \end{enumerate}
Then each $g\in G$ has a unique decomposition 
\begin{enumerate}
 \item[(i)] 
 \begin{align*}
  g=t_{1}^{a_1}\cdots t_{s}^{a_s}\cdot h, \hspace{.5cm} T_{G/H}(g)=\prod_{i}^{s}T_{G/H}(t_i)^{a_i}\cdot T_{G/H}(h),
 \end{align*}
where $0\leq a_i\leq m_i-1$, $h\in H$ and
\item[(ii)] 
\begin{align*}
 T_{G/H}(t_i)=t_{i}^{d}\cdot[t_{i}^{m_i},\alpha_i], \quad\quad T_{G/H}(h)=h^{d}\cdot[h,\alpha],
\end{align*}
where $\alpha_i\in G/H$ is the product over all elements from $C_i\subset G/H$, the subgroup which is complementary to the cyclic subgroup
$<t_i>$ mod $H$, and where $\alpha\in G/H$ is product over all elements from $G/H$.\\
Here we mean $[t_{i}^{m_i},\alpha_i]:=[t_{i}^{m_i},\widehat{\alpha_i}]$, $[h,\alpha]:=[h,\widehat{\alpha}]$ for any representatives 
$\widehat{\alpha_i},\widehat{\alpha}\in G$. The commutators
are independent of the choice of the representatives and are always elements of order $\leq 2$ because 
$\widehat{\alpha_i}^{2},\widehat{\alpha}^{2}\in H$, and $H$ is abelian. \\
As a consequence of $(i)$ and $(ii)$ we can always obtain
\item[(iii)]
\begin{align*}
 T_{G/H}(g)=g^{d}\cdot\varphi_{G/H}(g),
\end{align*}
where $\varphi_{G/H}(g)\in Z(G)$ is an element of order $\leq 2$.
\end{enumerate}
As a consequence of the second equality in $(ii)$ combined with $[G,G]\subseteq H\cap\mathrm{Ker}(T_{G/H})$, one has $[G,G]^d=\{1\}$,
in other words, $G^d$ is contained in the center $Z(G)$ of $G$. 
\end{lem}

Now let $\rho=(Z,\chi_\rho)$ be a Heisenberg representation of $G$. Then from the definition of Heisenberg representation 
we have 
$$[[G,G], G]\subseteq \rm{Ker}(\rho).$$
Now let $\overline{G}:=G/\rm{Ker}(\rho)$. Then we obtain
$$[\overline{G},\overline{G}]=[G/\rm{Ker}(\rho),G/\rm{Ker}(\rho)]=[G,G]\cdot\rm{Ker}(\rho)/\rm{Ker}(\rho)=[G,G]/[G,G]\cap\rm{Ker}(\rho).$$
Since $[[G,G],G]\subseteq\rm{Ker}(\rho)$, then  $[x,g]\in\rm{Ker}(\rho)$ for all $x\in [G,G]$ and $g\in G$.
Hence we obtain
\begin{center}
 $[[\overline{G},\overline{G}],\overline{G}]=[[G,G]/[G,G]\cap \rm{Ker}(\rho), G/\rm{Ker}(\rho)]\subseteq\rm{Ker}(\rho)$,
\end{center}
This shows that $\overline{G}$ is a two-step nilpotent group.
Hence for computing determinant of a Heisenberg representation
of a finite group \textbf{modulo $\rm{Ker}(\rho)$} we can use the Lemmas \ref{Lemma 2.2} and \ref{Lemma 2.10} and 
 we obtain the following
result (cf. Proposition \ref{Proposition 2.13}).

\begin{prop}\label{Proposition 2.13}
  Let $\rho=(Z,\chi_\rho)$ be a Heisenberg representation of $G$, of dimension $d$, and put $X_\rho(g_1,g_2):=\chi_\rho\circ [g_1,g_2]$.
 Then we obtain
 \begin{equation}\label{eqn 2.16}
  (\mathrm{det}(\rho))(g)=\varepsilon(g)\cdot\chi_\rho(g^d),
 \end{equation}
where $\varepsilon$ is a function on $G$ with the following properties:
\begin{enumerate}
 \item $\varepsilon$ has values in $\{\pm 1\}$.
 \item $\varepsilon(gx)=\varepsilon(g)$ for all $x\in G^2\cdot Z$, hence $\varepsilon$ is a function on the factor group
 $G/G^2\cdot Z$, and in particular, $\varepsilon\equiv 1$ if $[G:Z]=d^2$ is odd.
 \item If $d$ is even, then the function $\varepsilon$ need not be a homomorphism but:
 \begin{center}
  $\frac{\varepsilon(g_1)\varepsilon(g_2)}{\varepsilon(g_1g_2)}=X_\rho(g_1,g_2)^{\frac{d(d-1)}{2}}=X_\rho(g_1,g_2)^{\frac{d}{2}}$.
 \end{center}
 Furthermore,
 \begin{enumerate}
  \item \textbf{When $\mathrm{rk}_2(G/Z)\ge 4$:} $\varepsilon$ is a homomorphism, and exactly $\varepsilon\equiv 1$.
  \item \textbf{When $\mathrm{rk}_2(G/Z)=2$:}  $\varepsilon$ is not a homomorphism and $\varepsilon$ is a function
  on $G/G^2Z$ such that
  \begin{center}
   $(\det\rho)(g)=\varepsilon(g)\cdot\chi_\rho(g^d)=\begin{cases}
                                                     \chi_\rho(g^d) & \text{for $g\in G^2Z$}\\
                                                     -\chi_\rho(g^d) & \text{for $g\notin G^2Z$.}
                                                    \end{cases}
$
  \end{center}

 \end{enumerate}

\end{enumerate}

\end{prop}

In Chapter 5, our  
main aim is
to compute local constants for Heisenberg representations of a local Galois group. We know that Heisenberg representations
are {\bf monomial} (that is, induced from linear character of a finite-index subgroup),
and local constants are weakly extendible functions.
Therefore to compute local constants for Heisenberg representations we need to compute lambda-functions. In Chapter 3 we 
compute lambda-functions explicitly for finite Galois extensions (except wildly ramified quadratic extension), 
and in Chapter 5 we use them.

In Section 5.1, we study the arithmetic descriptions of Heisenberg representations. We define U-isotropic
Heisenberg representations and prove the following lemma (cf. \ref{Lemma U-isotropic}).

 \begin{lem}\label{Lemma U-isotropic}
 Fix a uniformizer $\pi_F$ and write $U:=U_F$. Then we obtain an isomorphism 
 $$\widehat{U}\cong \widehat{FF^\times/U\wedge U}, \quad \eta\mapsto X_\eta,\quad \eta_X\leftarrow X$$
 between characters of $U$ and $U$-isotropic alternating characters as follows:
 \begin{equation}\label{eqn 5.1.25}
  X_\eta(\pi_F^a\varepsilon_1,\pi_F^b\varepsilon_2):=\eta(\varepsilon_1)^b\cdot\eta(\varepsilon_2)^{-a},\quad
  \eta_X(\varepsilon):=X(\varepsilon,\pi_F),
 \end{equation}
 where $a,b\in\bbZ$, $\varepsilon,\varepsilon_1,\varepsilon_2\in U$, and $\eta:U\to\bbC^\times$.
 Then 
 $$\rm{Rad}(X_\eta)=<\pi_F^{\#\eta}>\times\rm{Ker}(\eta)=<(\pi_F\varepsilon)^{\#\eta}>\times\rm{Ker}(\eta),$$
 does not depend on the choice of $\pi_F$, where  $\#\eta$ is the order of the character $\eta$, hence 
 $$F^\times/\rm{Rad}(X_\eta)\cong <\pi_F>/<\pi_F^{\#\eta}>\times U/\rm{Ker}(\eta)\cong \bbZ_{\#\eta}\times\bbZ_{\#\eta}.$$
 Therefore all Heisenberg representations of type $\rho=\rho(X_\eta,\chi)$ have dimension $\rm{dim}(\rho)=\#\eta$.
\end{lem}

From the above lemma we can construct all Heisenberg representations of dimensions prime to $p$ 
(cf. Corollary \ref{Corollary U-isotropic}). And we also have the following explicit lemma (cf. Lemma \ref{Explicit Lemma}).

 \begin{lem}[{\bf Explicit Lemma}]\label{Explicit Lemma}
 Let $\rho=\rho(X_\eta,\chi_K)$ be a U-isotropic Heisenberg representation of the absolute Galois group $G_F$ of a local field 
 $F/\bbQ_p$. Let $K=K_\eta$ and let $E/F$ be the maximal unramified subextension in $K/F$. Then: 
 \begin{enumerate}
  \item The norm map induces an isomorphism:
  $$N_{K/E}:K_F^\times/I_FK^\times\stackrel{\sim}{\to}I_FE^\times/I_F\cN_{K/E}.$$
  \item Let $c_{K/F}:F^\times/\rm{Rad}(X_\eta)\wedge F^\times/\rm{Rad}(X_\eta)\cong K_F^\times/I_FK^\times$ be the isomorphism
  which is induced by the commutator in the relative Weil-group $W_{K/F}$. Then for units $\varepsilon\in U_F$ we 
  explicitly have:
  $$c_{K/F}(\varepsilon\wedge\pi_F)=N_{K/E}^{-1}(N_{E/F}^{-1}(\varepsilon)^{1-\varphi_{E/F}}),$$
  where $\varphi_{E/F}$ is the Frobenius automorphism for $E/F$ and where $N^{-1}$ means to take a preimage of the norm map.
  \item The restriction $\chi_K|_{K_F^\times}$ is characterized by:
  $$\chi_K\circ c_{K/F}(\varepsilon\wedge\pi_F)=X_\eta(\varepsilon,\pi_F)=\eta(\varepsilon),$$
  for all $\varepsilon\in U_F$, where $c_{K/F}(\varepsilon\wedge\pi_F)$ is explicitly given via (2).
 \end{enumerate}

\end{lem}

After these, in Subsection 5.1.2, we study the Artin conductors, Swan conductors of Heisenberg representations, and these 
results are important for giving explicit invariant formulas of local constants for Heisenberg representations of dimensions prime 
to $p$.

In the following theorem (cf. Theorem \ref{Theorem invariant odd}) we give an invariant formula of $W(\rho)$
for the Heisenberg representation $\rho$.

\begin{thm}\label{Theorem invariant odd}
 Let $\rho=\rho(X,\chi_K)$ be a Heisenberg representation of the absolute Galois group $G_F$ of a local field $F/\bbQ_p$
 of dimension $d$. Let $\psi_F$ be the canonical additive character of $F$ and $\psi_K:=\psi_F\circ\rm{Tr}_{K/F}$.
 Denote $\mu_{p^\infty}$ as the group of roots of unity of $p$-power order and $\mu_{n}$ as the group of 
 $n$-th roots of unity, where $n>1$ is an integer.
 \begin{enumerate}
  \item When the dimension $d$ is odd, we have 
   \begin{center}
  $W(\rho)\equiv W(\chi_\rho)'\mod{\mu_{d}}$,
 \end{center}
where $W(\chi_\rho)'$ is any $d$-th root of 
$W(\chi_K,\psi_K)$.
\item When the dimension $d$ is even, we have 
 \begin{center}
  $W(\rho)\equiv W(\chi_\rho)'\mod{\mu_{d'}}$,
 \end{center}
 where $d'=\rm{lcm}(4,d)$.
 \end{enumerate}
 
\end{thm}

For giving more explicit invariant formula of local constants for the Heisenberg representations, we need the following result
(cf. Proposition \ref{Proposition arithmetic form of determinant}).

\begin{prop}\label{Proposition arithmetic form of determinant}
 Let $\rho=\rho(Z,\chi_\rho)=\rho(G_K,\chi_K)$ be a Heisenberg representation of the absolute Galois group $G_F$.
 Let $E$ be a base field of a maximal isotropic for $\rho$. Then $F^\times\subseteq\cN_{K/E}$, and 
 \begin{equation}\label{eqn 5.1.12}
  \det(\rho)(x)=\Delta_{E/F}(x)\cdot\chi_K\circ N_{K/E}^{-1}(x)\quad \text{for all $x\in F^\times$},
 \end{equation}
where, for all $x\in F^\times$,
\begin{equation}\label{eqn 5.1.13}
 \Delta_{E/F}(x)=\begin{cases}
                  1 & \text{when $\rm{rk}_2(\rm{Gal}(E/F))\ne 1$}\\
                  \omega_{E'/F}(x) & \text{when $\rm{rk}_2(\rm{Gal}(E/F))= 1$},
                 \end{cases}
\end{equation}
where $E'/F$ is a uniquely determined quadratic subextension in $E/F$, and $\omega_{E'/F}$ is the character of $F^\times$ which 
corresponds to $E'/F$ by class field theory.
\end{prop}

By using the above proposition we have the following invariant 
formula (cf. Theorem \ref{invariant formula for minimal conductor representation})
of local constant for a minimal conductor U-isotopic Heisenberg representation of dimension prime to $p$.

 \begin{thm}\label{invariant formula for minimal conductor representation}
  Let $\rho=\rho(X,\chi_K)$ be a minimal conductor Heisenberg representation of the absolute Galois group $G_F$ of a non-archimedean
  local field $F/\bbQ_p$ of dimension $m$ with $gcd(m,p)=1$. Let $\psi$ be a nontrivial additive character of $F$. Then 
  \begin{equation}
   W(\rho,\psi)=R(\psi,c)\cdot L(\psi,c),
  \end{equation}
where 
$$R(\psi,c):=\lambda_{E/F}(\psi)\Delta_{E/F}(c),$$
is a fourth root of unity that depends on $c\in F^\times$ with $\nu_F(c)=1+n(\psi)$ but not on the totally ramified cyclic subextension
$E/F$ in $K/F$, and 
$$L(\psi,c):=\det(\rho)(c)q_F^{-\frac{1}{2}}\sum_{x\in k_F^\times}(\chi_K\circ N_{E_1/F}^{-1})^{-1}(x)\cdot (c^{-1}\psi)(mx),$$
where $E_1/F$ is the unramified extension of $F$ of degree $m$.
 \end{thm}

And when $\rho=\rho(X,\chi_K)$ is not minimal conductor, we have the following theorem 
(cf. Theorem \ref{Theorem using Deligne-Henniart}). 

 \begin{thm}\label{Theorem using Deligne-Henniart}
  Let $\rho=\rho(X_\eta,\chi_K)=\rho_0\otimes\widetilde{\chi_F}$ be a Heisenberg representation of the absolute Galois group
  $G_F$ of a non-archimedean local field $F/\bbQ_p$ of dimension $m$ with $gcd(m,p)=1$, where 
  $\rho_0=\rho_0(X_\eta,\chi_0)$ is a minimal conductor Heisenberg representation of $G_F$
  and $\widetilde{\chi_F}:W_F\to\bbC^\times$
  corresponds to $\chi_F:F^\times\to \bbC^\times$ by class field theory.
  If $a(\chi_F)\ge 2$, then we have 
 \begin{equation}\label{eqn 5.4.9}
  W(\rho,\psi)=W(\rho_0\otimes\widetilde{\chi_F},\psi)=W(\chi_F,\psi)^m\cdot\det(\rho_0)(c),
 \end{equation}
where $\psi$ is a nontrivial additive character of $F$, and $c:=c(\chi_F,\psi)\in F^\times$, satisfies 
\begin{center}
 $\chi_F(1+x)=\psi(c^{-1}x)$ for all $x\in P_{F}^{a(\chi_F)-[\frac{a(\chi_F)}{2}]}$.
\end{center}
\end{thm}

In Section 5.3 we use the Tate's {\bf root-of-unity criterion} (cf. \cite{JT1}, p. 112, Corollary 4)
in our local constant computation.
Let $\rho$ be a Heisenberg representation of the absolute Galois group $G_F$.
In the following proposition (cf. Proposition \ref{Proposition 4.12}) we show 
that if $W(\rho)$ is not a root of unity, then $\rm{dim}(\rho)|(q_F-1).$
\begin{prop}\label{Proposition 4.12}
 Let $F/\bbQ_p$ be a local field and let $q_F=p^s$ be the order of its finite residue field. If
$\rho=(Z_\rho,\chi_\rho)=\rho(X_\rho,\chi_K)$ is a Heisenberg representation of the absolute 
Galois group $G_F$ such that $W(\rho)$ is not a root of unity, 
then $dim(\rho)|(q_F-1)$.
\end{prop}

\chapter{\textbf{Preliminaries}}

In this chapter, we will recall some background of local constants, classical Gauss sums, Heisenberg representations that
will be relevant to the next chapters. We also state some known facts which we will use in our next chapters.
 In this chapter for local constants we follow 
  \cite{RL}, \cite{JT1}, \cite{JT2} and for extendible functions \cite{HK}.
  For classical Gauss sums and Heisenberg representations
  we refer \cite{LN}, \cite{BRK} and \cite{Z2}, \cite{Z3} respectively.

\section{\textbf{Local Fields and their finite extensions}}

Let $F$  be a non-archimedean local field, i.e., a finite extension of the field $\mathbb{Q}_p$ (field of $p$-adic numbers),
where $p$ is a prime.
Let $K/F$ be a finite extension of the field $F$. Let $e_{K/F}$ be the ramification index for the extension $K/F$ and $f_{K/F}$ be 
the residue degree of the extension $K/F$. The extension $K/F$ is called \textbf{unramified} 
if $e_{K/F}=1$; equivalently $f_{K/F}=[K:F]$. The extension $K/F$ is \textbf{totally ramified} if 
$e_{K/F}=[K:F]$; equivalently $f_{K/F}=1$. 
Let
$q_F$ be the cardinality of the residue field $k_F$ of $F$. If $\mathrm{gcd}(p,[K:F])=1$, then the extension 
$K/F$ is called \textbf{tamely ramified}, otherwise \textbf{wildly ramified}. The extension $K/F$ is \textbf{totally tamely ramified}
if it is both totally ramified and tamely ramified.

For a tower of {\bf local} fields $K/L/F$, we have (cf. \cite{FV}, p. 39, Lemma 2.1) 
\begin{equation}
 e_{K/F}(\nu_K)=e_{K/L}(\nu_K)\cdot e_{L/F}(\nu_L),
\end{equation}
where $\nu_K$ is a valuation on $K$ and $\nu_L$ is the induced 
valuation on $L$, i.e., $\nu_L=\nu_K|_{L}$. For 
the tower of fields $K/L/F$ we simply write $e_{K/F}=e_{K/L}\cdot e_{L/F}$.
Let $O_F$ be the 
ring of integers in the local field $F$ and $P_F=\pi_F O_F$ is the unique prime ideal in $O_F$ 
and $\pi_F$ is a uniformizer, i.e., an element in $P_F$ whose valuation is one, i.e.,
 $\nu_F(\pi_F)=1$.
Let $U_F=O_F-P_F$ be the group of units in $O_F$.
Let $P_{F}^{i}=\{x\in F:\nu_F(x)\geq i\}$ and for $i\geq 0$ define $U_{F}^i=1+P_{F}^{i}$
(with proviso $U_{F}^{0}=U_F=O_{F}^{\times}$).
We also consider that $a(\chi)$ is the conductor of 
 nontrivial character $\chi: F^\times\to \mathbb{C}^\times$, i.e., $a(\chi)$ is the smallest integer $m\geq 0$ such 
 that $\chi$ is trivial
 on $U_{F}^{m}$. We say $\chi$ is unramified if the conductor of $\chi$ is zero and otherwise ramified.
Throughout the thesis when $K/F$
is unramified we choose uniformizers $\pi_K=\pi_F$. And when $K/F$ is totally ramified (both tame and wild) we choose
uniformizers $\pi_F=N_{K/F}(\pi_K)$, where $N_{K/F}$ is the norm map from $K^\times$ to $F^\times$.

\begin{dfn}[\textbf{Different and Discriminant}] 
 Let $K/F$ be a finite separable extension of non-archimedean local field $F$. We define the \textbf{inverse different (or codifferent)}
 $\mathcal{D}_{K/F}^{-1}$ of $K$ over $F$ to be $\pi_{K}^{-d_{K/F}}O_K$, where $d_{K/F}$ is the largest integer (this is the 
 exponent of the different $\mathcal{D}_{K/F}$) such that 
 \begin{center}
  $\mathrm{Tr}_{K/F}(\pi_{K}^{-d_{K/F}}O_K)\subseteq O_F$,
 \end{center}
 where $\rm{Tr}_{K/F}$ is the trace map from $K$ to $F$.
Then the \textbf{different} is defined by:
\begin{center}
 $\mathcal{D}_{K/F}=\pi_{K}^{d_{K/F}}O_K$
\end{center}
and the \textbf{discriminant} $D_{K/F}$ is 
\begin{center}
 $D_{K/F}=N_{K/F}(\pi_{K}^{d_{K/F}})O_F$.
\end{center}
 Thus it is easy to see that $D_{K/F}$ is an \textbf{ideal of $O_F$}.

We know that if $K/F$ is unramified, then $D_{K/F}$ is \textbf{a unit in $O_F$}. If $K/F$ is 
tamely ramified, then 
\begin{equation}\label{eqn 2.2}
 \nu_K(\mathcal{D}_{K/F})=d_{K/F}=e_{K/F} - 1.
\end{equation}
\end{dfn}
(see \cite{JPS}, Chapter III, for details about different and discriminant of the extension $K/F$.)
We need to mention a very important result of J-P. Serre for our purposes.

\begin{lem}[\cite{JPS}, p. 50, Proposition 7]\label{Lemma 2.21}
Let $K/F$ be a finite separable extension of the field $F$. Let $I_F$ (resp. $I_K$) be a fractional ideal of $F$ (resp. $K$)
relative to $O_F$ (resp. $O_K$). Then the following two properties are equivalent:
\begin{enumerate}
 \item $\mathrm{Tr}_{K/F}(I_K)\subset I_F$.
 \item $I_K\subset I_F.\mathcal{D}_{K/F}^{-1}$.
\end{enumerate}

 \end{lem}

\begin{dfn}[\textbf{Canonical additive character}] 

We define the non trivial additive character of $F$, $\psi_F:F\to\mathbb{C}^\times$ as the composition of the following 
four maps:
\begin{center}
 $F\xrightarrow{\mathrm{Tr}_{F/\mathbb{Q}_p}}\mathbb{Q}_p\xrightarrow{\alpha}\mathbb{Q}_p/\mathbb{Z}_p
 \xrightarrow{\beta}\mathbb{Q}/\mathbb{Z}\xrightarrow{\gamma}\mathbb{C}^\times$,
\end{center}
where
\begin{enumerate}
 \item $\mathrm{Tr}_{F/\mathbb{Q}_p}$ is the trace from $F$ to $\mathbb{Q}_p$,
 \item $\alpha$ is the canonical surjection map,
 \item $\beta$ is the canonical injection which maps $\mathbb{Q}_p/\mathbb{Z}_p$ onto the $p$-component of the 
 divisible group $\mathbb{Q}/\mathbb{Z}$ and 
 \item $\gamma$ is the exponential map $x\mapsto e^{2\pi i x}$, where $i=\sqrt{-1}$.
\end{enumerate}
For every $x\in\mathbb{Q}_p$, there is a rational $r$, uniquely determined modulo $1$, such that $x-r\in\mathbb{Z}_p$.
Then $\psi_{\bbQ_p}(x)=\psi_{\bbQ_p}(r)=e^{2\pi i r}$.
The nontrivial additive character  $\psi_F=\psi_{\bbQ_p}\circ \rm{Tr}_{F/\bbQ_p}$ of $F$ 
is called the \textbf{canonical additive character} (cf. \cite{JT1}, p. 92).
\end{dfn}
The {\bf conductor} of any nontrivial additive character $\psi$ of the field $F$ is an integer $n(\psi)$ if $\psi$ is trivial
on $P_{F}^{-n(\psi)}$, but nontrivial on $P_{F}^{-n(\psi)-1}$. So, from Lemma \ref{Lemma 2.21} we can show (cf. Lemma \ref{Lemma 3.4})
that 
\begin{center}
 $n(\psi_F)=n(\psi_{\bbQ_p}\circ\rm{Tr}_{F/\bbQ_p})= \nu_F(\mathcal{D}_{F/\bbQ_p})$,
\end{center}
because $d_{\bbQ_p/\bbQ_p}=0$, and hence $n(\psi_{\bbQ_p})=0$.

\section{\textbf{Extendible functions}}

Let $G$ be any finite group. We denote $R(G)$ the set of all pairs $(H,\rho)$, where $H$ is a subgroup of $G$ and $\rho$ is a
virtual representation of $H$
. The group $G$ acts on $R(G)$ by means of
\begin{center}
$(H,\rho)^g=(H^g,\rho^g)$,     $g\in G$,\\
$\rho^g(x)=\rho(gxg^{-1})$,   $x\in H^g:=g^{-1}Hg$
\end{center}
Furthermore we denote by $\widehat{H}$ the set of all one dimensional representations of $H$ and 
by $R_1(G)$ the subset of $R(G)$ of pairs $(H,\chi)$ with 
$\chi\in \widehat{H}$. Here character $\chi$ of $H$ we mean always a \textbf{linear} 
character, i.e., $\chi:H\to \mathbb{C}^\times$. 

Now define a function $\mathcal{F}:R_1(G) \rightarrow \mathcal{A}$, where $\mathcal{A}$ is a multiplicative abelian  group with
\begin{equation}\label{eqn 2.2.1}
\mathcal{F}(H,1_H)=1
\end{equation}
 and 
\begin{equation}\label{eqn 2.2.2}
\mathcal{F}(H^g,\chi^g)=\mathcal{F}(H,\chi)
\end{equation}
for all $(H,\chi)$, where $1_H$ denotes the trivial representation of $H$.\\
Here a function $\mathcal{F}$ on $R_1(G)$ means a function which satisfies the equation (\ref{eqn 2.2.1}) 
and (\ref{eqn 2.2.2}).

A function $\mathcal{F}$ is said to be extendible if $\mathcal{F}$ can be extended to 
an $\mathcal{A}$-valued 
function
on $R(G)$ satisfying: 
\begin{equation}\label{eqn 2.2.3}
 \mathcal{F}(H,\rho_1+\rho_2)=\mathcal{F}(H,\rho_1)\mathcal{F}(H,\rho_2)
\end{equation}
for all $(H,\rho_i)\in R(G),i=1,2$, and if $(H,\rho)\in R(G)$ with $\mathrm{dim}\,\rho=0$, and $\Delta$ is a subgroup of
$G$ containing 
$H$, then
\begin{equation}\label{eqn 2.2.4}
 \mathcal{F}(\Delta,\mathrm{Ind}_{H}^{\Delta}\rho)=\mathcal{F}(H,\rho),
\end{equation}
where $\mathrm{Ind}_{H}^{\Delta}\rho$ is the virtual representation of $\Delta$ induced from $\rho$. In general, 
let $\rho$ be a 
representation of $H$ with $\mathrm{dim}\,\rho\neq0$.
We can define a zero dimensional representation of $H$ by $\rho$ and which is:
  $\rho_0:=\rho-\mathrm{dim}\,\rho\cdot 1_H$. So $\mathrm{dim}\,\rho_0$ is zero, then now we use the equation (\ref{eqn 2.2.4}) for
$\rho_0$ and we have,
\begin{equation}\label{eqn 2.2.5}
 \mathcal{F}(\Delta,\mathrm{Ind}_{H}^{\Delta}\rho_0)=\mathcal{F}(H,\rho_0).
 \end{equation}
 Now replace $\rho_0$ by $\rho-\mathrm{dim}\rho\cdot 1_H$ in the above equation (\ref{eqn 2.2.5}) and we have
 \begin{align*}
   \mathcal{F}(\Delta,\mathrm{Ind}_{H}^{\Delta}(\rho-\mathrm{dim}\rho \cdot 1_H))
   &=\mathcal{F}(H,\rho-\mathrm{dim}\rho\cdot1_H)\\\implies
   \frac{\mathcal{F}(\Delta,\mathrm{Ind}_{H}^{\Delta}\rho)}
   {\mathcal{F}(\Delta,\mathrm{Ind}_{H}^{\Delta}1_H)^{\mathrm{dim}\rho}}
   &=\frac{\mathcal{F}(H,\rho)}
   {\mathcal{F}(H,1_H)^{\mathrm{dim}\rho}}.
 \end{align*}
Therefore,
\begin{align}
 \mathcal{F}(\Delta,\mathrm{Ind}_{H}^{\Delta}\rho)\nonumber
 &=\left\{\frac{\mathcal{F}(\Delta,\mathrm{Ind}_{H}^{\Delta}1_H)}{\mathcal{F}(H,1_H)}\right\}^{\mathrm{dim}\rho}\cdot\mathcal{F}(H,\rho)\\
 &=\lambda_{H}^{\Delta}(\mathcal{F})^{\mathrm{dim}\rho}\mathcal{F}(H,\rho), \label{eqn 2.6}
\end{align}
where
\begin{equation}
 \lambda_{H}^{\Delta}(\mathcal{F}):=\frac{\mathcal{F}(\Delta,\mathrm{Ind}_{H}^{\Delta}1_H)}{\mathcal{F}(H,1_H)}.\label{eqn 2.7}
\end{equation}
But by the definition of $\mathcal{F}$, we have 
$\mathcal{F}(H,1_H)=1$, so we can write 
\begin{equation}
 \lambda_{H}^{\Delta}(\mathcal{F})={\mathcal{F}(\Delta,\mathrm{Ind}_{H}^{\Delta}1_H}).\label{eqn 2.8}
\end{equation}
This $\lambda_{H}^{\Delta}(\mathcal{F})$
is called \textbf{Langlands $\lambda$-function} (or simply $\lambda$-function) which is independent of $\rho$.
A extendible function $\mathcal{F}$ is called \textbf{strongly} extendible if it satisfies
equation (\ref{eqn 2.2.3}) and fulfills equation (\ref{eqn 2.2.4}) for all $(H,\rho)\in R(G)$, and if the equation (\ref{eqn 2.2.4})
is fulfilled
only when $\mathrm{dim}\,\rho=0$, then
$\mathcal{F}$ is called \textbf{weakly} extendible function. The extendible functions are \textbf{unique}, if they exist 
(cf. \cite{JT1}, p. 103).

\begin{exm}
 Langlands proves the local constants are weakly extendible functions (cf. \cite{JT1}, p. 105, Theorem 1). The Artin root numbers
(also known as global constants) 
are strongly 
extendible functions (for more examples and details about extendible function, see \cite{JT1} and \cite{HK}).
\end{exm}

 The next lemma is due to Langlands \cite{RL}. This is very important for this thesis. Group theoretically it is not hard to 
 see its proof. But its number theoretical proof is very very difficult and long which can be found in \cite{RL}. 
 
 \begin{lem}\label{Lemma 4.1.1}
  Let $H$ be a subgroup of a group $G$ and $\mathcal{F}$ an extendible function on $R_{1}(G)$. Then we have the following
properties of $\lambda$-factor.
\begin{enumerate}
 \item If $g\in G$, then $\lambda_{g^{-1}Hg}^{G}(\mathcal{F})=\lambda_{H}^{G}(\mathcal{F})$, where $H\subseteq G$. 
 \item If $H'$ is a subgroup of $H$ then 
 $\lambda_{H'}^{G}(\mathcal{F})=\lambda_{H'}^{H}(\mathcal{F})\lambda_{H}^{G}(\mathcal{F})^{[H:H']}$, 
 where $[H:H']$ is index
 of $H'$ in $H$.
 \item If $H'$ is a normal subgroup of $G$ contained in $H$, then 
 $\lambda_{H}^{G}(\mathcal{F})=\lambda_{H/{H'}}^{G/{H'}}(\mathcal{F})$. 
\end{enumerate}
 \end{lem}

\section{\textbf{Local Constants}}

Let $F$ be a non-archimedean local field and $\chi$ be a character of $F^\times$.
 The $L(\chi)$-functions are defined as follows:
\begin{align*}
 L(\chi)
 &=\begin{cases}
            (1-\chi(\pi_F))^{-1} & \text{if $\chi$ is unramified},\\
            1 & \text{if $\chi$ is ramified}.
           \end{cases}
\end{align*}

 We denote by $dx$ a Haar measure on $F$, by $d^\times x$ a Haar measure on $F^\times$ and the relation between these two 
 Haar measure is: 
 \begin{center}
  $d^\times x=\frac{dx}{|x|}$,
 \end{center}
for arbitrary Haar measure $dx$ on $F$. For given additive character $\psi$ of $F$ and Haar measure $dx$ on $F$, we have a 
\textbf{Fourier transform} as:
\begin{equation}
 \hat{f}(y)=\int f(x)\psi(xy)dx.
\end{equation}
where $f\in L^{1}(F^{+})$ (that is, $|f|$ is integrable) and the Haar measure is normalized 
such that $\hat{\hat{f}}(y)=f(-y)$, i.e., $dx$ is self-dual with respect to $\psi$.
By Tate (cf. \cite{JT2}, p. 13), for any character $\chi$ of $F^\times$, there exists 
a number $W(\chi,\psi,dx)\in\mathbb{C}^\times$ such that it satisfies the following local functional equation:
\begin{equation}\label{eqn 2.3.2}
 \frac{\int\hat{f}(x)w_1\chi^{-1}(x)d^\times x}{L(w_1\chi^{-1})}=W(\chi,\psi,dx)
 \frac{\int f(x)\chi(x)d^\times x}{L(\chi)}.
\end{equation}
for any such function $f$ for which the both sides make sense. This number $W(\chi,\psi, dx)$ is called the
{\bf local epsilon factor or local constant} of $\chi$.


For a nontrivial multiplicative character $\chi$ of $F^\times$ and nontrivial additive character $\psi$ of $F$,
we have (cf. \cite{RL}, p. 4)
\begin{equation}
 W(\chi,\psi,c)=\chi(c)\frac{\int_{U_F}\chi^{-1}(x)\psi(x/c) dx}{|\int_{U_F}\chi^{-1}(x)\psi(x/c) dx|}\label{label1}
\end{equation}
where the Haar measure $dx$ is normalized such that the measure of $O_F$ is $1$ and 
where $c\in F^\times$ with valuation $n(\psi)+a(\chi)$.
The \textbf{modified} formula of epsilon factor (cf. \cite{JT1}, p. 94, for proof see \cite{SAB4}) is:
\begin{equation}\label{eqn 2.2}
 W(\chi,\psi,c)=\chi(c)q^{-a(\chi)/2}\sum_{x\in\frac{U_F}{U_{F}^{a(\chi)}}}\chi^{-1}(x)\psi(x/c).
\end{equation}
where $c=\pi_{F}^{a(\chi)+n(\psi)}$. Now if $u\in U_F$ is unit and replace $c=cu$, then we have 
\begin{equation}
 W(\chi,\psi,cu)=\chi(c)q^{-\frac{a(\chi)}{2}}\sum_{x\in\frac{U_F}{U_{F}^{a(\chi)}}}\chi^{-1}(x/u)\psi(x/cu)=W(\chi,\psi,c).
\end{equation}
Therefore $W(\chi,\psi,c)$ \textbf{depends} only on the exponent $\nu_{F}(c)=a(\chi)+n(\psi)$. Therefore we can 
simply write $W(\chi,\psi, c)=W(\chi,\psi)$, because $c$ is determined by 
$\nu_F(c)=a(\chi)+n(\psi)$ up to a unit $u$ which has \textbf{no influence on} $W(\chi,\psi,c)$.
If $\chi$ is unramified, i.e., $a(\chi)=0$, therefore $\nu_F(c)=n(\psi)$. Then from the formula of $W(\chi,\psi,c)$, we can write
\begin{equation}\label{eqn 2.3.5}
 W(\chi,\psi,c)=\chi(c),
\end{equation}
and therefore $W(1,\psi,c)=1$ if $\chi=1$ is the trivial character.

\subsection{\textbf{Some properties of $W(\chi,\psi)$}}

 \begin{enumerate}
 
  \item Let $b\in F^\times$ be the uniquely determined element such that $\psi'=b\psi$. Then 
  \begin{equation}
   W(\chi,\psi',c')=\chi(b)\cdot W(\chi,\psi,c).
  \end{equation}
\begin{proof}
 Here $\psi'(x)=(b\psi)(x):=\psi(bx)$ for all $x\in F$. It is an additive character of $F$. The existence and uniqueness of $b$
 is clear. From the definition of conductor of an 
 additive character we have 
 \begin{center}
  $n(\psi')=n(b\psi)=n(\psi)+\nu_F(b)$.
 \end{center}
Here $c'\in F^\times$ is of valuation 
$$\nu_F(c')=a(\chi)+n(\psi')=a(\chi)+\nu_F(b)+n(\psi)=\nu_F(b)+\nu_F(c)=\nu_F(bc).$$
Therefore $c'=bcu$ where $u\in U_F$ is some unit. Now 
\begin{align*}
 W(\chi,\psi',c')
 &=W(\chi,b\psi,bcu)\\
 &=W(\chi,b\psi,bc)\\
 &=\chi(bc)q_{F}^{-\frac{a(\chi)}{2}}\sum_{x\in\frac{U_F}{U_{F}^{a(\chi)}}}\chi^{-1}(x)((bc)^{-1}(b\psi))(x)\\
 &=\chi(b)\cdot\chi(c)q_{F}^{-\frac{a(\chi)}{2}}\sum_{x\in\frac{U_F}{U_{F}^{a(\chi)}}}\chi^{-1}(x)\psi(xc^{-1})\\
 &=\chi(b)\cdot W(\chi,\psi,c).
\end{align*}

\end{proof}

\item Let $F/\bbQ_p$ be a local field inside $\overline{\bbQ_p}$.
Let $\chi$ and $\psi$ be a character of $F^\times$ and $F^{+}$ respectively, and $c\in F^\times$ with 
valuation $\nu_F(c)=a(\chi)+n(\psi)$. If $\sigma\in\mathrm{Gal}(\overline{\mathbb{Q}_p}/\mathbb{Q}_p)$ is an automorphism, then:
\begin{center}
 $W_{F}(\chi,\psi,c)=W_{\sigma^{-1}(F)}(\chi^{\sigma},\psi^{\sigma},\sigma^{-1}(c))$,
\end{center}
where $\chi^{\sigma}(y):=\chi(\sigma(y))$, $\psi^{\sigma}(y):=\psi(\sigma(y))$, for all $y\in\sigma^{-1}(F)$.

\begin{proof}
Let $L:=\sigma^{-1}(F)$. Since $\sigma$ is an automorphism of $\overline{\bbQ_p}$, then we have $O_F/P_F\cong O_L/P_L$, hence 
$q_F=q_L$. We also can see that $a(\chi^{\sigma})=a(\chi)$ and $n(\psi^{\sigma})=n(\psi)$. Then from the formula of local constant
we have 
\begin{align*}
 W_{\sigma^{-1}(F)}(\chi^{\sigma},\psi^{\sigma},\sigma^{-1}(c))
 &=W_{L}(\chi^{\sigma},\psi^{\sigma},\sigma^{-1}(c))\\
 &=\chi^{\sigma}(\sigma^{-1}(c))q_{L}^{-\frac{a(\chi^{\sigma})}{2}}\sum_{y\in\frac{U_{L}}{U_{L}^{a(\chi^{\sigma})}}}
 (\chi^{\sigma})^{-1}(y)\cdot((\sigma^{-1}(c))^{-1}\psi^{\sigma}(y)\\
 &=\chi(c)q_{F}^{-\frac{a(\chi)}{2}}\sum_{x\in \frac{U_F}{U_{F}^{a(\chi)}}}\chi^{-1}(x)\psi(\frac{x}{c})\\
 &=W_{F}(\chi,\psi,c).
\end{align*}
Here we put 
$y=\sigma^{-1}(x)$ and use $(\sigma^{-1}(c))^{-1}\psi^{\sigma}=(c^{-1}\psi)^{\sigma}$.

\end{proof}
\begin{rem}
 We can simply write as before $W_{F}(\chi,\psi)=W_{\sigma^{-1}(F)}(\chi^{\sigma},\psi^{\sigma})$.
 Tate in his paper \cite{JT1} on local constants defines the local root number as:
\begin{center}
 $W_{F}(\chi):=W_F(\chi,\psi_F)=W_F(\chi,\psi_F,d)$,
\end{center}
where $\psi_F$ is the canonical character of $F^\times$ and $d\in F^\times$ with
$\nu_F(d)=a(\chi)+n(\psi_F)$.
Therefore after fixing canonical additive character $\psi=\psi_F$, we can rewrite 
\begin{center}
 $W_F(\chi)=\chi(d(\psi_F))$, if $\chi$ is unramified,\\
 $W_F(\chi)=W_{\sigma^{-1}(F)}(\chi^{\sigma})$.
\end{center}
The last equality follows because the canonical character $\psi_{\sigma^{-1}(F)}$ is related to the canonical character 
$\psi_F$ as: $\psi_{\sigma^{-1}(F)}=\psi_{F}^{\sigma}$.\\

So we see that 
\begin{equation*}
 (F,\chi)\to W_F(\chi)\in\mathbb{C}^\times
\end{equation*}
is a function with the properties (\ref{eqn 2.2.1}), (\ref{eqn 2.2.2}) of extendible functions.
\end{rem}

\item If $\chi\in\widehat{F^\times}$ and $\psi\in\widehat{F}$, then 
\begin{equation*}
 W(\chi,\psi)\cdot W(\chi^{-1},\psi)=\chi(-1).
\end{equation*}
Furthermore if the character $\chi:F^\times\to\mathbb{C}^{\times}$ is unitary (in particular, if $\chi$ is of finite order), then 
 \begin{center}
  $|W(\chi,\psi)|^{2}=1$.
  \end{center}

\begin{proof}
We prove this properties by using equation (\ref{eqn 2.2}). 
We know that the additive characters are always unitary, hence
\begin{center}
 $\psi(-x)=\psi(x)^{-1}=\overline{\psi}(x)$.
\end{center}
On the other hand we write $\psi(-x)=((-1)\psi)(x)$, where $-1\in F^\times$. Therefore $\overline{\psi}=(-1)\psi$.
We also have $a(\chi)=a(\chi^{-1})$. Therefore by using equation (\ref{eqn 2.2}) we have 
\begin{align*}
 W(\chi,\psi)\cdot W(\chi^{-1},\psi)
 &=\chi(-1)\cdot q_{F}^{-a(\chi)}\sum_{x,y\in \frac{U_F}{U_{F}^{a(\chi)}}}\chi^{-1}(x)\chi(y)\psi(\frac{x-y}{c})\\
 &=\chi(-1)\cdot q_{F}^{-a(\chi)}\sum_{x,y\in \frac{U_F}{U_{F}^{a(\chi)}}}\chi^{-1}(x)\psi(\frac{xy-y}{c}),
 \quad\text{replacing $x$ by $xy$}\\
 &=\chi(-1)\cdot q_{F}^{-a(\chi)}\sum_{x\in \frac{U_F}{U_{F}^{a(\chi)}}}\chi^{-1}(x)\varphi(x),
\end{align*}
where
\begin{equation}
\varphi(x)=\sum_{y\in\frac{U_F}{U_{F}^{a(\chi)}}}\psi(y\frac{x-1}{c}).
\end{equation}
Since $\frac{U_F}{U_{F}^{a(\chi)}}=(\frac{O_F}{P_{F}^{a(\chi)}})^\times=
\frac{O_F}{P_{F}^{a(\chi)}}\setminus\frac{P_F}{P_{F}^{a(\chi)}}$, therefore $\varphi(x)$
can be written as the difference 
\begin{align*}
\varphi(x)
&=\sum_{y\in\frac{U_F}{U_{F}^{a(\chi)}}}\psi(y\frac{x-1}{c})\\
&=\sum_{y\in\frac{O_F}{P_{F}^{a(\chi)}}}\psi(y\frac{x-1}{c})-
\sum_{y\in\frac{P_F}{P_{F}^{a(\chi)}}}\psi(y\frac{x-1}{c})\\
&=\sum_{y\in\frac{O_F}{P_{F}^{a(\chi)}}}\psi(y\frac{x-1}{c})-
\sum_{y\in\frac{O_F}{P_{F}^{a(\chi)-1}}}\psi(y\frac{(x-1)\pi_F}{c})\\
&=A-B,
\end{align*}
where $A=\sum_{y\in\frac{O_F}{P_{F}^{a(\chi)}}}\psi(y\frac{x-1}{c})$
and $B=\sum_{y\in\frac{O_F}{P_{F}^{a(\chi)-1}}}\psi(y\frac{(x-1)\pi_F}{c})$. It is easy to see that (cf. \cite{M}, p. 28, Lemma 2.1)
\begin{align*}
 \sum_{y\in\frac{O_F}{P_{F}^{a(\chi)}}}\psi(y\alpha)=\begin{cases}
                                                      q_{F}^{a(\chi)} & \text{when $\alpha\in P_{F}^{-n(\psi)}$}\\
                                                      0 & \text{otherwise}
                                                     \end{cases}
\end{align*}
Therefore $A=q_{F}^{a(\chi)}$ when $x\in U_{F}^{a(\chi)}$ and $A=0$ otherwise. Similarly
$B=q_{F}^{a(\chi)-1}$ when $x\in U_{F}^{a(\chi)-1}$
 and $B=0$ otherwise. Therefore we have
 \begin{align*}
  W(\chi,\psi)\cdot W(\chi^{-1},\psi)
  &=\chi(-1)\cdot q_{F}^{-a(\chi)}\cdot\{q_{F}^{a(\chi)}-q_{F}^{a(\chi)-1}\sum_{x\in\frac{U_{F}^{a(\chi)-1}}{U_{F}^{a(\chi)}}}\chi^{-1}(x)\}\\
  &=\chi(-1)-\chi(-1)\cdot q_{F}^{-1}\sum_{x\in\frac{U_{F}^{a(\chi)-1}}{U_{F}^{a(\chi)}}}\chi^{-1}(x).
 \end{align*}
Since the conductor of $\chi$ is $a(\chi)$, then it can be proved that
$\sum_{x\in\frac{U_{F}^{a(\chi)-1}}{U_{F}^{a(\chi)}}}\chi^{-1}(x)=0$.
Thus we obtain
\begin{equation}\label{eqn 2.3.9}
 W(\chi,\psi)\cdot W(\chi^{-1},\psi)=\chi(-1).
\end{equation}
\vspace{.3cm}

The right side of equation (\ref{eqn 2.3.9}) is a sign, hence we may rewrite (\ref{eqn 2.3.9}) as 
\begin{center}
 $W(\chi,\psi)\cdot\chi(-1)W(\chi^{-1},\psi)=1$.
\end{center}
But we also know from our earlier property that 
$$\chi(-1)W(\chi^{-1},\psi)=W(\chi^{-1},(-1)\psi)=W(\chi^{-1},\overline{\psi}).$$
So the identity (\ref{eqn 2.3.9}) rewrites as 
\begin{center}
 $W(\chi,\psi)\cdot W(\chi^{-1},\overline{\psi})=1$.
\end{center}
Now we assume that $\chi$ is unitary, hence 
\begin{center}
 $W(\chi^{-1},\overline{\psi})=W(\overline{\chi},\overline{\psi})=\overline{W(\chi,\psi)}$
\end{center}
where the last equality is obvious. Now we see that for unitary $\chi$ the identity (\ref{eqn 2.3.9}) rewrites as 
\begin{center}
 $|W(\chi,\psi)|^{2}=1$.
\end{center}
\end{proof}

\begin{rem}
From the functional equation (\ref{eqn 2.3.2}), we can directly see the first part of the above property of local constant. 
Denote 
\begin{equation}\label{eqn 66}
 \zeta(f,\chi)=\int f(x)\chi(x)d^\times x.
\end{equation}
Now replacing  $f$ by $\hat{\hat{f}}$ in equation (\ref{eqn 66}), and we get
\begin{equation}\label{eqn 77}
 \zeta(\hat{\hat{f}},\chi)=\int \hat{\hat{f}}(x)\chi(x)d^\times x=\chi(-1)\cdot\zeta(f,\chi),
\end{equation}
because $dx$ is self-dual with respect to $\psi$, hence $\hat{\hat{f}}(x)=f(-x)$ for all $x\in F^{+}$.

Again the functional equation (\ref{eqn 2.3.2}) can be written as follows:
\begin{equation}\label{eqn 88}
 \zeta(\hat{f},w_1\chi^{-1})=W(\chi,\psi,dx)\cdot\frac{L(w_1\chi^{-1})}{L(\chi)}\cdot\zeta(f,\chi).
\end{equation}

 Now we replace $f$ by $\hat{f}$, and $\chi$ by $w_1\chi^{-1}$ in equation (\ref{eqn 88}), and we obtain
 \begin{equation}\label{eqn 99}
  \zeta(\hat{\hat{f}},\chi)=W(w_1\chi^{-1},\psi,dx)\cdot\frac{L(\chi)}{L(w_1\chi^{-1})}\cdot\zeta(\hat{f},w_1\chi^{-1}).
 \end{equation}
Then by using equations (\ref{eqn 77}), (\ref{eqn 88}), from the above equation (\ref{eqn 99}) we obtain:
\begin{equation}\label{eqn 100}
 W(\chi,\psi,dx)\cdot W(w_1\chi^{-1},\psi,dx)=\chi(-1).
\end{equation}
Moreover, the convention $W(\chi,\psi)$ is actually as follows (cf. \cite{JT2}, p. 17, equation (3.6.4)):
$$W(\chi w_{s-\frac{1}{2}},\psi)=W(\chi w_s,\psi,dx).$$
By using this relation from equation (\ref{eqn 100}) we can conclude
$$W(\chi,\psi)\cdot W(\chi^{-1},\psi)=\chi(-1).$$

\end{rem}

\item \textbf{Twisting formula of abelian local constants:}

\begin{enumerate}
 \item If $\chi_1$ and $\chi_2$ are two unramified characters of $F^\times$ and $\psi$ is
 a nontrivial additive character of $F$, then from equation (\ref{eqn 2.3.5}) we have
 \begin{equation}
  W(\chi_1\chi_2,\psi)=W(\chi_1,\psi)W(\chi_2,\psi).
 \end{equation}
 \item  Let $\chi_1$ be ramified and $\chi_2$ unramified then (cf. \cite{JT2}, (3.2.6.3))
\begin{equation}
 W(\chi_1\chi_2,\psi)=\chi_2(\pi_F)^{a(\chi_1)+n(\psi)}\cdot W(\chi_1,\psi).
\end{equation}
\begin{proof}
  By the given condition
$a(\chi_1)>a(\chi_2)=0$. Therefore $a(\chi_1\chi_2)=a(\chi_1)$. Then we have
\begin{align*}
 W(\chi_1\chi_2,\psi)
 &=\chi_1\chi_2(c)q_{F}^{-a(\chi_1)/2}\sum_{x\in\frac{U_F}{U_{F}^{a(\chi)}}}(\chi_1\chi_2)^{-1}(x)\psi(x/c)\\
 &=\chi_1(c)\chi_2(c)q_{F}^{-a(\chi_1)/2}\sum_{x\in\frac{U_F}{U_{F}^{a(\chi)}}}\chi_{1}^{-1}(x)\chi_{2}^{-1}(x)\psi(x/c)\\
 &=\chi_2(c)\chi_1(c)q_{F}^{-a(\chi_1)/2}\sum_{x\in\frac{U_F}{U_{F}^{a(\chi)}}}\chi_{1}^{-1}(x)\psi(x/c),
 \quad\text{since $\chi_2$ unramified}\\
 &=\chi_2(c)W(\chi_1,\psi)\\
 &=\chi_2(\pi_F)^{a(\chi_1)+n(\psi)}\cdot W(\chi_1,\psi).
\end{align*}
\end{proof}
\item 
 We also have twisting formula of epsilon factor by Deligne (cf. \cite{D1}, Lemma 4.16)
 under some special condition and which is as follows (for proof, see Corollary \ref{Corollary 6.1.2}(3)):\\
Let $\alpha$ and $\beta$ be two multiplicative characters of a local field $F$ such that $a(\alpha)\geq 2\cdot a(\beta)$.
Let $\psi$ be an additive character of $F$.
Let $y_{\alpha,\psi}$ be an element of $F^\times$ such that 
$$\alpha(1+x)=\psi(y_{\alpha,\psi}x)$$
for all $x\in F$ with valuation $\nu_F(x)\geq\frac{a(\alpha)}{2}$ (if $a(\alpha)=0$, $y_{\alpha,\psi}=\pi_{F}^{-n(\psi)}$). Then 
\begin{equation}\label{eqn 2.3.17}
 W(\alpha\beta,\psi)=\beta^{-1}(y_{\alpha,\psi})\cdot W(\alpha,\psi).
\end{equation}

\end{enumerate}

\end{enumerate}

\subsection{\textbf{Connection of different conventions for local constants} }

Mainly there are two conventions for local constants. They are due to Langlands (\cite{RL}) and Deligne (\cite{D1}). 
Recently Bushnell and Henniart (\cite{BH}) also give a convention of local constants. In this subsection we shall show the connection
among all three conventions for local constants\footnote{The convention $W(\chi,\psi)$ is actually due to Langlands \cite{RL},
and it is:
\begin{center}
 $\epsilon_{L}(\chi,\psi,\frac{1}{2})=W(\chi,\psi).$
\end{center}
See equation (3.6.4) on p. 17 of \cite{JT2} for $V=\chi$.} 
We denote $\epsilon_{BH}$ as local constant of Bushnell-Henniart (introduced in
Bushnell-Henniart, \cite{BH}, Chapter 6).

On page 142 of \cite{BH}, the authors define a rational function 
$\epsilon_{BH}(\chi,\psi,s)\in\mathbb{C}(q_{F}^{-s})$. From Theorem 23.5 on p. 144 of \cite{BH} for ramified character 
$\chi\in\widehat{F^\times}$ and conductor\footnote{The definition of level of an additive character $\psi\in\widehat{F}$ 
in \cite{BH} on p. 11 is the negative sign with our conductor $n(\psi)$, i.e., level of $\psi=-n(\psi)$.} $n(\psi)=-1$ we have 
\begin{equation}\label{eqn 2.3.12}
 \epsilon_{BH}(\chi,s,\psi)=
 q_{F}^{n(\frac{1}{2}-s)}\sum_{x\in\frac{U_F}{U_{F}^{n+1}}}\chi(\alpha x)^{-1}\psi(\alpha x)/q_{F}^{\frac{n+1}{2}},
\end{equation}
where  $n=a(\chi)-1$, and $\alpha\in F^\times$ with $\nu_{F}(\alpha)=-n$.

Also from the Proposition 23.5 of \cite{BH} on p. 143 for unramified character $\chi\in\widehat{F^\times}$ and $n(\psi)=-1$ we have 
\begin{equation}\label{eqn 2.3.13}
\epsilon_{BH}(\chi,s,\psi)=q_{F}^{s-\frac{1}{2}}\chi(\pi_F)^{-1}. 
\end{equation}

\begin{enumerate}
 \item \textbf{Connection between $\epsilon_{BH}$ and $W(\chi,\psi)$.}
 \begin{center}
 $W(\chi,\psi)=\epsilon_{BH}(\chi,\frac{1}{2},\psi)$.
\end{center}
 \begin{proof}
 From \cite{BH}, p. 143, Lemma 1 we see:
 \begin{center}
  $\epsilon_{BH}(\chi,\frac{1}{2},b\psi)=\chi(b)\epsilon_{BH}(\chi,\frac{1}{2},\psi)$
 \end{center}
for any $b\in F^\times$. But we have seen already that $W(\chi,b\psi)=\chi(b)W(\chi,\psi)$ has the same transformation rule. If we fix
one nontrivial $\psi$ then all other nontrivial $\psi'$ are uniquely given as $\psi'=b\psi$ for some $b\in F^\times$. Because of the
parallel transformation rules it is now enough to verify our assertion for a single $\psi$. Now we take $\psi\in\widehat{F^{+}}$
with $n(\psi)=-1$, hence $\nu_F(c)=a(\chi)-1$. Then we obtain
\begin{equation*}
 W(\chi,\psi)=W(\chi,\psi,c)=\chi(c)q_{F}^{-\frac{a(\chi)}{2}}\sum_{x\in\frac{U_F}{U_{F}^{a(\chi)}}}\chi^{-1}(x)\psi(c^{-1}x).
\end{equation*}
We compare this to the equation (\ref{eqn 2.3.12}). There the notation is $n=a(\chi)-1$ and the assumption is $n\geq 0$. This means
we have $\nu_F(c)=n$, hence we may take $\alpha=c^{-1}$ and then comparing our formula with equation (\ref{eqn 2.3.12}), we see that
\begin{center}
 $W(\chi,\psi)=\epsilon_{BH}(\chi,\frac{1}{2},\psi)$
\end{center}
in the case when $n(\psi)=-1$.\\
We are still left to prove our assertion if $\chi$ is unramified, i.e., $a(\chi)=0$. Again we can reduce to the case where 
$n(\psi)=-1$. Then our assertion follows from equation \ref{eqn 2.3.13}.

\end{proof}
\begin{rem}
From Corollary $23.4.2$ of \cite{BH}, on p. 142, for $s\in\mathbb{C}$, we 
 can write 
 \begin{align}
  \epsilon_{BH}(\chi,s,\psi)=q_{F}^{(\frac{1}{2}-s)n(\chi,\psi)}\cdot\epsilon_{BH}(\chi,\frac{1}{2},\psi),
 \end{align}
 for some $n(\chi,\psi)\in\mathbb{Z}$. In fact here $n(\chi,\psi)=a(\chi)+n(\psi)$.
From above connection, we just see $W(\chi,\psi)=\epsilon_{BH}(\chi,\frac{1}{2},\psi)$. Thus for arbitrary $s\in\mathbb{C}$, we 
obtain
\begin{equation}\label{eqn 2.3.10}
\epsilon_{BH}(\chi,s,\psi)=q_{F}^{(\frac{1}{2}-s)(a(\chi)+n(\psi))}\cdot W(\chi,\psi).
\end{equation}
This equation (\ref{eqn 2.3.10}) is very important for us. We shall use this to connect with Deligne's convention.

In \cite{JT2} there is defined a number $\epsilon_{D}(\chi,\psi,dx)$ depending on $\chi$, $\psi$ and a Haar measure
$dx$ on $F$. This notion is due to Deligne \cite{D1}. We write 
 $\epsilon_{D}$ for Deligne's convention in order to distinguish it from the $\epsilon_{BH}(\chi,\frac{1}{2},\psi)$ introduced
in Bushnell-Henniart \cite{BH}.

In the next Lemma we give the connection between Bushnell-Henniart and Deligne conventions for local constants.
\end{rem}

\item {\bf The connection between $\epsilon_D$ and $\epsilon_{BH}$}:
\begin{lem}
We have the relation
 \begin{center}
 $\epsilon_{BH}(\chi,s,\psi)=\epsilon_{D}(\chi\cdot\omega_{s},\psi,dx_{\psi})$,
\end{center}
where $\omega_{s}(x)=|x|_{F}^{s}=q^{-s\nu_F(x)}$ is unramified character of $F^\times$ corresponding to complex number $s$, and where 
$dx_{\psi}$ is the self-dual Haar measure corresponding to the additive character $\psi$.
\end{lem}

\begin{proof}

From equation equation (3.6.4) of \cite{JT2}, we know that
\begin{equation}\label{eqn 2.3.11}
 \epsilon_{L}(\chi,s,\psi):=\epsilon_{L}(\chi\omega_{s-\frac{1}{2}},\psi)=\epsilon_{D}(\chi\omega_{s},\psi,dx_{\psi}).
\end{equation}
 We prove this connection by using the relations (\ref{eqn 2.3.10}) and (\ref{eqn 2.3.11}). From equation (\ref{eqn 2.3.11}) we can write 
 our $W(\chi,\psi)=\epsilon_{D}(\chi\omega_{\frac{1}{2}},\psi,dx_{\psi})$. Therefore when $s=\frac{1}{2}$, we have the 
 relation:
 \begin{equation}
  \epsilon_{BH}(\chi,\frac{1}{2},\psi)=\epsilon_{D}(\chi\omega_{\frac{1}{2}},\psi,dx_{\psi}),
 \end{equation}
since $W(\chi,\psi)=\epsilon_{BH}(\chi,\frac{1}{2},\psi)$.

We know that $\omega_{s}(x)=q_{F}^{-s\nu_F(x)}$ is an unramified character of $F^\times$.
So when $\chi$ is also unramified, we can write 
\begin{equation}
 W(\chi\omega_{s-\frac{1}{2}},\psi)=
 \omega_{s-\frac{1}{2}}(c)\cdot\chi(c)=q_{F}^{(\frac{1}{2}-s)n(\psi)}\epsilon_{BH}(\chi,\frac{1}{2},\psi)
 =\epsilon_{BH}(\chi,s,\psi).
\end{equation}
And when $\chi$ is ramified character, i.e., conductor $a(\chi)>0$,
from Tate's theorem for unramified twist (see property 2.3.1(4b)) , we can write 
\begin{align*}
 W(\chi\omega_{s-\frac{1}{2}},\psi)
 &=\omega_{s-\frac{1}{2}}(\pi_{F}^{a(\chi)+n(\psi)})\cdot W(\chi,\psi)\\
 &=q_{F}^{-(s-\frac{1}{2})(a(\chi)+n(\psi))}\cdot W(\chi,\psi)\\
 &=q_{F}^{(\frac{1}{2}-s)(a(\chi)+n(\psi))}\cdot \epsilon_{BH}(\chi,\frac{1}{2},\psi)\\
 &=\epsilon_{BH}(\chi,s,\psi).
\end{align*}
Furthermore from equation (\ref{eqn 2.3.11}), we have 
\begin{equation}
 W(\chi\omega_{s-\frac{1}{2}},\psi)=\epsilon_{D}(\chi\omega_{s},\psi,dx_{\psi}).
\end{equation}
Therefore finally we can write 
\begin{equation}
 \epsilon_{BH}(\chi,s,\psi)=\epsilon_{D}(\chi\omega_{s},\psi,dx_{\psi}).
\end{equation}
\end{proof}

 \begin{cor}\label{Corollary 4.1}
  For our $W$ we have :
  \begin{center}
   $W(\chi,\psi)=\epsilon_{BH}(\chi,\frac{1}{2},\psi)=\epsilon_{D}(\chi\omega_{\frac{1}{2}},\psi,dx_{\psi})$\\
   $W(\chi\omega_{s-\frac{1}{2}},\psi)=\epsilon_{BH}(\chi,s,\psi)=\epsilon_{D}(\chi\omega_{s},\psi,dx_{\psi})$.
  \end{center}
\end{cor}
\begin{proof}
 From the equations (3.6.1) and (3.6.4) of \cite{JT2} for $\chi$ and above two connections the assertions follow.
\end{proof}
\end{enumerate}

\subsection{{\bf Local constants for virtual representations}}

\begin{enumerate}
 \item To extend the concept of local constant, we need to go from 1-dimensional to other virtual representations $\rho$ of 
 the Weil groups $W_F$ of non-archimedean local field $F$.
 According to Tate \cite{JT1}, the root 
number $W(\chi):=W(\chi,\psi_F)$ extensions to $W(\rho)$, where $\psi_F$ is the canonical additive character of $F$.
More generally, $W(\chi,\psi)$ extends to $W(\rho,\psi)$, and if
$E/F$ is a finite separable extension then one has to take 
$\psi_{E}=\psi_{F}\circ \mathrm{Tr}_{E/F}$ for the extension field $E$. 

According to Bushnell-Henniart \cite{BH}, Theorem on p. 189, the functions
$\epsilon_{BH}(\chi,s,\psi)$ extend to $\epsilon_{BH}(\rho,s,\psi_E)$, where $\psi_E=\psi\circ\rm{Tr}_{E/F}$
\footnote{ Note that they fix a base field $F$ and a nontrivial $\psi=\psi_F$
(which not to be the canonical character used in Tate \cite{JT1}) but then if $E/F$ is an extension they always use 
$\psi_E=\psi\circ\mathrm{Tr}_{E/F}$.}. According to Tate \cite{JT2}, Theorem (3.4.1) the functions $\epsilon_{D}(\chi,\psi,dx)$
extends to $\epsilon_{D}(\rho,\psi,dx)$. In order to get {\bf weak inductivity}  we have again to use $\psi_E=\psi\circ\mathrm{Tr}_{E/F}$
if we consider extensions. Then according to Tate \cite{JT2} (3.6.4) the Corollary \ref{Corollary 4.1} turns into 
\begin{cor}\label{Corollary 4.2}
 For the virtual representations of the Weil groups we have 
 \begin{center}
  $W(\rho\omega_{E,s-\frac{1}{2}},\psi_E)=\epsilon_{BH}(\rho,s,\psi_E)=\epsilon_{D}(\rho\omega_{E,s},\psi_E,dx_{\psi_E})$.\\
  $W(\rho,\psi_E)=\epsilon_{BH}(\rho,\frac{1}{2},\psi_E)=\epsilon_{D}(\rho\omega_{E,\frac{1}{2}},\psi_E,dx_{\psi_E})$.
 \end{center}
\end{cor}
Note that on the level of field extension $E/F$ we have to use $\omega_{E,s}$ which is defined as 
\begin{center}
 $\omega_{E,s}(x)=|x|_{E}^{s}=q_{E}^{-s\nu_E(x)}.$
\end{center}
We also know that $q_{E}=q_{F}^{f_{E/F}}$ and $\nu_E=\frac{1}{f_{E/F}}\cdot\nu_F(N_{E/F})$ (cf. \cite{FV}, p. 41, Theorem 2.5), therefore
we can easily see that 
$$\omega_{E,s}=\omega_{F,s}\circ N_{E/F}.$$

Since the norm map $N_{E/F}:E^\times\to F^\times$ corresponds via class field theory to the injection map $G_E\hookrightarrow G_F$,
Tate \cite{JT2} beginning from (1.4.6), simply writes $\omega_{s}=||^s$ and consider $\omega_{s}$ as an unramified character of the Galois
group (or of the Weil group) instead as a character on the field.
Then Corollary \ref{Corollary 4.2} turns into 
\begin{equation}\label{eqn 2.3.15}
 W(\rho\omega_{s-\frac{1}{2}},\psi_E)=\epsilon_{BH}(\rho,s,\psi_E)=\epsilon_{D}(\rho\omega_{s},\psi_E,dx_{\psi_E}),
\end{equation}
for all field extensions, where $\omega_{s}$ is to be considered as 1-dimensional representation of the Weil group $W_E\subset G_E$
if we are on the $E$-level. The left side equation (\ref{eqn 2.3.15}) is the $\epsilon$-factor of Langlands
(see \cite{JT2}, (3.6.4)). 

\item The functional equation (\ref{eqn 2.3.9}) extends to 
\begin{equation}\label{eqn 2.3.23}
 W(\rho,\psi)\cdot W(\rho^{V},\psi)=\mathrm{det}_{\rho}(-1),
\end{equation}
where $\rho$ is any virtual representation of the Weil group $W_F$, $\rho^{V}$ is the contragredient and $\psi$ is any nontrivial additive
character of $F$. This is formula (3) on p. 190 of \cite{BH} for $s=\frac{1}{2}$.

\item Moreover, the transformation law \cite{JT2} (3.4.5) can (on the $F$-level) be written as:\\
\textbf{unramified character twist}
\begin{equation}
 \epsilon_{D}(\rho\omega_{s},\psi,dx)=\epsilon_{D}(\rho,\psi,dx)\cdot \omega_{F,s}(c_{\rho,\psi})
\end{equation}
for any $c=c_{\rho,\psi}$ such that $\nu_F(c)=a(\rho)+n(\psi)\mathrm{dim}(\rho)$. 
It implies that also for the root number on the $F$-level
we have 
\begin{equation}
 W(\rho\omega_{s},\psi)=W(\rho,\psi)\cdot \omega_{F,s}(c_{\rho,\psi}).
\end{equation}
\end{enumerate}


\section{\textbf{Classical Gauss sums}}

Let $k_q$ be a finite field. Let $p$ be the characteristic of $k_q$; then the prime  field 
contained in $k_q$ is $k_p$. 
The structure of the {\bf canonical} additive character $\psi_q$ of $k_q$ is the same as the structure of the canonical  character
$\psi_F$, namely {\bf it comes by trace} from the canonical character of the base field, i.e., 
\begin{center}
 $\psi_q=\psi_p\circ \rm{Tr}_{k_q/k_p}$,
\end{center}
where 
\begin{center}
 $\psi_p(x):=e^{\frac{2\pi i x}{p}}$ \hspace{.3cm} for all $x\in k_p$.
\end{center}

\textbf{Gauss sums:} Let $\chi$ be a multiplicative and $\psi$ an additive character of $k_q$. Then the Gauss sum $G(\chi,\psi)$ is define
by 
\begin{equation}
 G(\chi,\psi)=\sum_{x\in k_{q}^{\times}}\chi(x)\psi(x).
\end{equation}
In general, computation of this Gauss is very difficult, but for certain characters, the associated Gauss sums can be evaluated explicitly.
In the following theorem for quadratic characters of $k_q$ we can give explicit formula of Gauss sums.
\begin{thm}[\cite{LN}, p. 199, Theorem 5.15]\label{Theorem 3.5}
Let $k_q$ be a finite field with $q=p^s$, where $p$ is an odd prime and $s\in\mathbb{N}$. Let $\chi$ be the quadratic character of 
$k_q$ and let $\psi$ be the canonical additive character of $k_q$. Then
\begin{equation}
 G(\chi,\psi)=\begin{cases}
               (-1)^{s-1}q^{\frac{1}{2}} & \text{if $p\equiv 1\pmod{4}$},\\
               (-1)^{s-1}i^sq^{\frac{1}{2}} & \text{if $p\equiv 3\pmod{4}$}.
              \end{cases}
\end{equation}
\end{thm}

Let $\psi$ be an additive and $\chi$ a multiplicative character of $k_q$. Let $E$ be a finite extension field of $k_q$. Then 
$\psi$ and $\chi$ can be \textbf{lifted} to $E$ by the setting 
\begin{center}
 $\psi'(x)=\psi(\mathrm{Tr}_{E/k_q}(x))$ for all $x\in E$ and $\chi'(x)=\chi(N_{E/k_q}(x))$ for all $x\in E^\times$.
\end{center}
From the additivity of the trace and multiplicativity of the norm it follows that $\chi'$ is an additive character and $\chi'$ is a
multiplicative character of $E$. The following theorem gives the relation between the Gauss sum $G(\chi,\psi)$ in $k_q$ and the 
Gauss sum $G(\chi',\psi')$ in $E$.

\begin{thm}[Davenport-Hasse, \cite{LN}, p. 197, Theorem 5.14]\label{Davenport-Hasse}
 Let $\chi$ be a multiplicative and $\psi$ an additive character of $k_q$, not both of them trivial. Suppose $\chi$ and 
 $\psi$ are lifted to character $\chi'$ and $\psi'$, respectively, of the finite extension field $E$ of $k_q$ with 
 $[E:k_q]=s$. Then 
 \begin{equation}
  G(\chi',\psi')=(-1)^{s-1}\cdot G(\chi,\psi)^{s}.
 \end{equation}
\end{thm}


\section{\textbf{Witt ring and square class group of a local field}}

Let $F$ be a field. Let $M$ be any commutative cancellation monoid under addition.
We define a relation $\sim$ on $M\times M$ by 
\begin{center}
 $(x,y)\sim(x',y')\Longleftrightarrow x+y'=x'+y\in M$.
\end{center}
The cancellation law in $M$ implies that $\sim$ is an equivalence relation on $M\times M$. We define the {\bf Grothendieck group}
of $M$ to be $\rm{Groth}(M)=(M\times M)/\sim$ (the set of equivalence classes) with addition induced by 
\begin{center}
 $(x,y)+(x',y')=(x+x',y+y')$.
\end{center}
We can prove that $\rm{Groth}(M)$ is the additive group generated by $M$.

Now let $M(F)$ be the set of all isometry classes of (nonsingular) quadratic forms of $F$, and replace $M$ by $M(F)$ in the 
definition of Grothendieck group. Then we denote $\widehat{W}(F)=\rm{Groth}(M(F))$ which is called the {\bf Witt-Grothendieck}
ring of quadratic forms over the field $F$. Every element of $\widehat{W}(F)$ has the formal expression $q_1-q_2$, where $q_1,\,q_2$
are nonsingular quadratic forms,
or rather, isometry classes of such forms.

Now, consider the dimension map $\rm{dim}:M(F)\to\bbZ$, which is a semiring homomorphism on $M(F)$. This extends uniquely
(via the universal property) to a ring homomorphism $\rm{dim}:\widehat{W}(F)\to\bbZ$, by 
$$\rm{dim}(q_1-q_2)=\rm{dim}(q_1)-\rm{dim}(q_2).$$
The kernel of this ring homomorphism, denoted by $\widehat{I}F$, is called the \textbf{the fundamental ideal of $\widehat{W}(F)$}.
We have $\widehat{W}(F)/\widehat{I}F=\bbZ$. 

Let $\bbZ\cdot\mathbb{H}$ be the set which consists of all hyperabolic spaces and their additive inverses, and they form an ideal
of $\widehat{W}(F)$. The vector ring
\begin{center}
 $W(F)=\widehat{W}(F)/\bbZ\cdot\mathbb{H}$
\end{center}
is called the \textbf{Witt ring} of $F$. 
The image of the ideal $\widehat{I}F$ under the natural projection 
$\widehat{W}(F)\to W(F)$ is denoted by $IF$; this is called the \textbf{fundamental ideal} of $W(F)$. It can be shown that
$W(F)/IF\cong\bbZ/2\bbZ$ (cf. \cite{TYM}, p. 30, Corollary 1.6).

The group $F^\times/{F^\times}^2$ is called the {\bf square class group} of $F$, and $Q(F)=\bbZ_2\times F^\times/{F^\times}^2$ 
is called the \textbf{extended} square class group of $F$. We also have $W(F)/I^2F\cong Q(F)$ (cf. \cite{TYM}, p. 31, Proposition 2.1),
where $I^2F$ denotes the square of $IF$. By {\bf Pfister's result} (cf. \cite{TYM}, p. 32, Corollary 2.3), we have
the square class group $F^\times/{F^\times}^2\cong IF/I^2F$.

Now we come to our local field case. Let $F/\bbQ_p$ be a local field. When $F/\bbQ_p$ is a local field with $p\ne 2$, from 
Theorem 2.2(1) of \cite{TYM} we have $F^\times/{F^\times}^2\cong V$, where $V$ is Klein's $4$-group. More generally, we need the 
following results for computing $\lambda$-function in the wild case.

\begin{thm}[\cite{TYM}, p. 162, Theorem 2.22]\label{Theorem 2.9}
 Let $F/\bbQ_p$ be a local field with $q_F$ as the cardinality of the residue field of $F$. Let $s=\nu_F(2)$.
 Then $|F^\times/{F^\times}^2|=4\cdot q_{F}^{s}$.\\
 In particular, if $p\ne 2$, i.e., $s=\nu_{F}(2)=0$, we have $|F^\times/{F^\times}^2|=4$.
\end{thm}
When $p=2$ and $F/\bbQ_2$ is a local field of degree $n$ over $\bbQ_2$, we have $s=\nu_{F}(2)=e_{F/\bbQ_2}$
because $2$ is a uniformizer in $\bbQ_2$. We also know that $q_F=q_{\bbQ_2}^{f_{F/\bbQ_2}}=2^{f_{F/\bbQ_2}}$.
We also know that $e_{F/\bbQ_2}\cdot f_{F/\bbQ_2}=n$. Then from the above Theorem \ref{Theorem 2.9} we obtain  
\begin{equation}
 |F^\times/{F^\times}^2|=4\cdot q_{F}^{e_{F/\bbQ_2}}=4\cdot(2^{f_{F/\bbQ_2}})^{e_{F/\bbQ_2}}=4\cdot 2^n=2^{2+n}.
\end{equation}


\begin{thm}[\cite{TYM}, p. 165, Theorem 2.29]\label{Theorem 2.11}
 Let $F/\bbQ_2$ be a dyadic local field, with $|F^\times/{F^\times}^2|=2^m$ ($m\ge 3$). Then:
 \begin{enumerate}
  \item Case 1: When $-1\in{F^\times}^2$, we have $W(F)\cong\bbZ_{2}^{m+2}$ (here $\bbZ_{n}^{k}$ denotes the direct product of $k$
  copies of $\bbZ_n=\bbZ/n\bbZ$).
  \item Case 2: When $-1\not\in{F^\times}^2$, but $-1$ is a sum of two squares, we have $W(F)\cong \bbZ_{4}^{2}\times\bbZ_{2}^{m-2}$.
  \item Case 3: When $-1$ is not a sum of two squares, we have $W(F)=\bbZ_8\times\bbZ_{2}^{m-1}$.
 \end{enumerate}

\end{thm}

\section{\textbf{Heisenberg representations}}

Let $\rho$ be an irreducible representation of a (pro-)finite group $G$. Then $\rho$ is called a \textbf{Heisenberg 
representation} if it represents commutators by 
scalar matrices. Therefore higher commutators are represented by $1$.
We can see that the linear characters of $G$ are Heisenberg representations as the degenerate special case.
To classify Heisenberg representations we need to mention two invariants of an irreducible representation 
$\rho\in\rm{Irr}(G)$:
\begin{enumerate}
 \item Let $Z_\rho$ be the \textbf{scalar} group of $\rho$, i.e., $Z_\rho\subseteq G$ and $\rho(z)=\text{scalar matrix}$
for every $z\in Z_\rho$. If $V/\bbC$ is a representation space of $\rho$ we get $Z_\rho$ as the kernel of the composite map 
\begin{equation}\label{eqn 2.6.1}
 G\xrightarrow{\rho}GL_{\bbC}(V)\xrightarrow{\pi} PGL_{\bbC}(V)=GL_{\bbC}(V)/\bbC^\times E,
\end{equation}
where $E$ is the unit matrix and denote $\overline{\rho}:=\pi\circ\rho$.
Therefore $Z_\rho$ is a normal subgroup of $G$.
\item Let $\chi_\rho$ be the character of $Z_\rho$ which is given as $\rho(g)=\chi_\rho(g)\cdot E$ for all $g\in Z_\rho$. 
Apparently $\chi_\rho$ is a $G$-invariant character of $Z_\rho$ which we call the central 
character of $\rho$.
\end{enumerate}
Let $A$ be a profinite abelian group. Then we know that (cf. \cite{Z5}, p. 124, Theorem 1 and Theorem 2)
the set of isomorphism classes $\rm{PI}(A)$ of projective irreducible representations (for 
projective representation, see \cite{CR}, \S  51) of $A$ is in bijective correspondence with the 
set of continuous alternating characters $\rm{Alt}(A)$. If $\rho\in\rm{PI}(A)$ corresponds to $X\in\rm{Alt}(A)$ then 
\begin{center}
 $\rm{Ker}(\rho)=\rm{Rad}(X)$ \hspace{.4cm} and \hspace{.2cm}$[A:\rm{Rad}(X)]=\rm{dim}(\rho)^2$,
\end{center}
where $\rm{Rad}(X):=\{a\in A|\, X(a,b)=1,\,\text{for all}\, b\in A\}$, the {\bf radical of $X$}.

Let $A:=G/[G,G]$, so $A$ is abelian. 
We also know from the  composite map (\ref{eqn 2.6.1})
$\overline{\rho}$ is a projective irreducible representation of $G$ and $Z_\rho$ is the kernel of $\overline{\rho}$.
Therefore \textbf{modulo commutator group $[G,G]$}, we can consider that 
$\overline{\rho}$ is in $\rm{PI}(A)$ which corresponds an alternating 
character $X$ of $A$ with kernel of $\overline{\rho}$ is $Z_\rho/[G,G]=\rm{Rad}(X)$.
We also know that 
$$[A:\rm{Rad}(X)]=[G/[G,G]:Z_\rho/[G,G]]=[G:Z_\rho].$$
Then we observe that 
$$\rm{dim}(\overline{\rho})=\rm{dim}(\rho)=\sqrt{[G:Z_\rho]}.$$

Let $H$ be a subgroup of $A$, then we define the orthogonal complement of $H$ in $A$ with respect to $X$
$$H^\perp:=\{a\in A:\quad X(a, H)\equiv1\}.$$
An {\bf isotropic} subgroup $H\subset A$ is a subgroup such that $H\subseteq H^\perp$ (cf. \cite{EWZ}, p. 270, Lemma 1(v)).
And when isotropic subgroup $H$ is maximal,
we call $H$ is a \textbf{maximal isotropic} for $X$. Thus when $H$ is maximal isotropic we have 
$H=H^\perp$.

We also can show that the Heisenberg representations $\rho$ are fully characterized by the corresponding pair 
$(Z_{\rho},\chi_{\rho})$.

\begin{prop}[\textbf{\cite{Z3}, Proposition 4.2}]\label{Proposition 3.1}
The map $\rho\mapsto(Z_\rho,\chi_\rho)$ is a bijection between equivalence 
classes of Heisenberg representations of $G$ and the pairs $(Z_\rho,\chi_\rho)$ such that 
\begin{enumerate}
 \item[(a)] $Z_\rho\subseteq G$ is a coabelian normal subgroup,
 \item[(b)] $\chi_\rho$ is a $G$-invariant character of $Z_\rho$,
 \item[(c)] $X(\hat{g_1},\hat{g_2}):=\chi_\rho(g_1g_2g_1^{-1}g_2^{-1})$ is a nondegenerate 
 \textbf{alternating character} on $G/Z_\rho$ where $\hat{g_1},\hat{g_2}\in G/Z_{\rho}$ and their 
 corresponding lifts $g_1,g_2\in G$.
\end{enumerate}
\end{prop}
For pairs $(Z_\rho,\chi_\rho)$ with the properties $(a)-(c)$, the corresponding Heisenberg representation $\rho$ is determined 
by the identity:
\begin{equation}\label{eqn 322}
 \sqrt{[G:Z_\rho]}\cdot\rho=\mathrm{Ind}_{Z_\rho}^{G}\chi_\rho.
\end{equation}
Moreover, the character $\mathrm{tr}_\rho$ of the Heisenberg representation $\rho$ is
\begin{align*}
 \mathrm{tr}_\rho(g)=\begin{cases}
                      0 & \text{if $g\in G - Z_\rho$}\\
                      \mathrm{dim}(\rho)\cdot\chi_\rho(g) & \text{if $g\in Z_\rho$} 
                     \end{cases}
\end{align*}

Let
$C^1G=G$, $C^{i+1}G=[C^iG,G]$ denote the 
descending central series of $G$. Now assume that every projective representation of $A$ lifts to an ordinary representation 
of $G$. Then by I. Schur's results (cf. \cite{CR}, p. 361, Theorem 53.7) we have (cf. \cite{Z5}, p. 124, Theorem 2):
\begin{enumerate}
 \item Let $A\wedge_\bbZ A$ denote the alternating square of the $\bbZ$-module $A$. The commutator map 
 \begin{equation}\label{eqn 2.6.3}
  A\wedge_\bbZ A\cong C^2G/C^3G, \hspace{.3cm} a\wedge b\mapsto [\hat{a},\hat{b}]
 \end{equation}
is an isomorphism.
\item The map $\rho\to X_\rho\in\rm{Alt}(A)$ from Heisenberg representations to alternating characters on $A$ is 
surjective. 
\end{enumerate}

\begin{proof}({\bf of the equation (\ref{eqn 322})})

 Let $H$ be a maximal \textbf{isotropic} subgroup of $G$ for the Heisenberg representation $\rho$
 and choose a character $\chi_H:H\rightarrow\mathbb{C}^\times$
such that $\chi_H\lvert_{Z_\rho}=\chi_\rho$. Then  we have (cf. \cite{Z3}, p. 193, Proposition 5.3):
\begin{equation}\label{eqn 2.3.3}
 \rho=\mathrm{Ind}_{H}^{G}\chi_H.
\end{equation}
This induced representation from $\chi_H$ does not depend on the choice of $H$ and the extension of $\chi_\rho$ to $H$.
 We also know from the transitivity of induction:
\begin{equation}
  \mathrm{Ind}_{Z_\rho}^{G}\chi_\rho=\mathrm{Ind}_{H}^{G}\mathrm{Ind}_{Z_\rho}^{H}\chi_\rho.\label{eqn 3.4}
\end{equation}
Furthermore, we can also write
\begin{align}
 \mathrm{Ind}_{Z_\rho}^{H}\chi_\rho
 &=\chi_{H}\otimes Ind_{Z_\rho}^{H}1_{Z_\rho},\quad \text{where $1_{Z_\rho}$ is the trivial representation of $Z_\rho$,}\\\nonumber
 &=\sum_{\text{all $\chi'_H$ which  are extension of $\chi_\rho$}}\chi'_{H}.
\end{align}
Here the total number of $\chi'_H$ is exactly equal to $[H:Z_\rho]$.
Putting this above result in the equation (\ref{eqn 3.4}), we have:
\begin{align}
 \mathrm{Ind}_{Z_\rho}^{G}\chi_\rho\nonumber
 &=\rm{Ind}_{H}^{G}(\rm{Ind}_{Z_\rho}^{H}\chi_\rho)\\\nonumber
 &=\mathrm{Ind}_{H}^{G}(\sum_{\text{all $\chi'_H$ which are extension of $\chi_\rho$}}\chi'_{H}),\\\nonumber
 &=\{\text{no. of $\chi_H$ which are extended from $\chi_\rho$}\}\times\mathrm{Ind}_{H}^{G}\chi'_H,\\\nonumber
 &=[H:Z_\rho]\cdot\mathrm{Ind}_{H}^{G}\chi'_H,\\
 &=[H:Z_\rho]\cdot\rho \quad \text{since $\mathrm{Ind}_{H}^{G}\chi'_H=\rho$}.\label{eqn 3.6}
\end{align}
We also know that
\begin{equation}
 [G:H]=[H:Z_\rho]=\sqrt{[G:Z_\rho]}=\mathrm{dim}\,\rho.\label{eqn 37}
\end{equation}
From the equations (\ref{eqn 3.6}) and (\ref{eqn 37}) we have our desired result which is:
\begin{align}
 \mathrm{Ind}_{Z_\rho}^{G}\chi_\rho
 &=\sqrt{[G:Z_\rho]}\cdot\rho\\  \nonumber
 &=\mathrm{dim}\,\rho\cdot\rho\\\nonumber  
\end{align}
Therefore, the equation (\ref{eqn 322}) is proved.
\end{proof}

\begin{rem}\label{Remark 3.2}
 Let $\chi_\rho$ be a character of $Z_\rho$. All extensions $\chi_H\supset\chi_\rho$ are conjugate with respect  
to $G/H$. This can be easily seen, since we know $\chi_H\supset\chi_\rho$ and $\chi_{H}^{g}(h)=\chi_{H}(ghg^{-1})$. If we take 
$z\in Z_\rho$,
then we obtain
\begin{center}
 $\chi_{H}^{g}(z)=\chi_{H}(gzg^{-1})=\chi_{\rho}(gzg^{-1})=\chi_{\rho}(gzg^{-1}z^{-1}z)$\\
 $=\chi_\rho([g,z]z)=X(g,z)\cdot\chi_\rho(z)=\chi_\rho(z)$,
\end{center}
since $Z_\rho$ is a normal subgroup of $G$ and the radical of $X$ 
(i.e., $X(g,z)=\chi_\rho([g,z])=1$ for all $z\in Z_\rho$ and 
$g\in G$).
Therefore, $\chi_{H}^{g}$ are extensions of $\chi_\rho$ for all $g\in G/H$. It can also be seen that the conjugates $\chi_{H}^{g}$ are 
all different, because $\chi_{H}^{g_1}=\chi_{H}^{g_2}$ is the same as $\chi_{H}^{g_1g_{2}^{-1}}=\chi_H$.
So it is enough to see that $\chi_{H}^{g-1}\not\equiv 1$ if $g\neq1\in G/H$. But
\begin{center}
 $\chi_{H}^{g-1}(h)=\chi_\rho(ghg^{-1}h^{-1})=X(g,h)$,
\end{center}
and therefore $\chi_{H}^{g-1}\equiv 1$ on $H$ implies $g\in H^{\bot}=H$, where $``\bot"$ denotes the 
orthogonal complement with respect to $X$. Then for a given one extension $\chi_H$ of $\chi_\rho$
all other extensions are of the form $\chi_{H}^{g}$ for $g\in G/H$.

\end{rem}

\begin{rem}
 Let $F$ be a non-archimedean local field and $G_F:=\rm{Gal}(\overline{F}/F)$ be the absolute Galois group.
 The Heisenberg representations of $G_F$ have arithmetic structure due to E.-W. Zink \cite{Z4}, \cite{Z5}.
 For Chapter 4, we just need its group theoretical structure, that is why here we discuss this group theoretical 
 definition. But in Chapter 5 we also need to see its arithmetic structure and we will study them in Chapter 5.
\end{rem}

\section{\textbf{Some useful results from finite Group Theory}}

Let $G$ be a finite abelian group and put $\alpha=\prod_{g\in G}g$. By the following theorem we can compute $\alpha$.
It is very much essential for our computation. In the Heisenberg setting for computing transfer map we have to deal with abelian
group $G/H$ and $\prod_{t\in G/H}t$, where $H$ is a normal subgroup of $G$.

\begin{thm}[\cite{PC}, Theorem 6 (Miller)]\label{Theorem Miller}
 Let $G$ be a finite abelian group and $\alpha=\prod_{g\in G}g$.
 \begin{enumerate}
  \item If $G$ has no element of order $2$, then $\alpha=e$.
  \item If $G$ has a unique element $t$ of order $2$, then $\alpha=t$.
  \item If $G$ has at least two elements of order $2$, then $\alpha=e$.
 \end{enumerate}

\end{thm}

Let $G$ be a two-step nilpotent group\footnote{Its derived subgroup, i.e., commutator subgroup
$[G,G]$ is contained in its center. 
In other worlds, $[G,[G,G]]=\{1\}$, i.e., any triple commutator gives identity. If $\rho$ is a Heisenberg representation
of a finite group $G$, then $G/\rm{Ker}(\rho)$ is a two-step nilpotent group (cf. p. 6).}.
For this  two-step nilpotent group, we have the 
following lemma.

\begin{lem}[\cite{AB}, p. 77, Lemma 9]\label{Lemma 22}
 Let $G$ be a two-step nilpotent group and let $x,y\in G$.  Then 
 \begin{enumerate}
  \item $[x^n, y]=[x,y]^n$, and 
  \item $x^ny^n=(xy)^{n}[x,y]^{\frac{n(n-1)}{2}}$,
 \end{enumerate}
 for any $n\in\mathbb{N}$.
\end{lem}

We also need the elementary divisor theorem for this article which we take from \cite{DF}. Let $G$ be a finite abelian group.
So $G$ is finitely generated.
\begin{thm}[\cite{DF}, p. 160, Theorem 3 (Invariant form)]\label{Theorem 22.4}
Let $G$ be a finite abelian group. Then
\begin{equation}\label{eqn 21}
 G\cong \mathbb{Z}/n_1\times\mathbb{Z}/n_2\times\cdots\times\mathbb{Z}/n_s.
\end{equation}
for some integers $n_1,n_2,\cdots,n_s$ satisfying the following conditions:
\begin{enumerate}
 \item[(a)] $n_j\geq 2$ for all $j\in\{1,2,\cdots,s\}$, and
 \item[(b)] $n_{i+1}|n_i$ for all $1\leq i\leq s-1$.
\end{enumerate}
And the expression in (\ref{eqn 21}) is unique: if $G\cong\mathbb{Z}/m_1\times\mathbb{Z}/m_2\times\cdots\times\mathbb{Z}/m_r$, where 
$m_1,m_2,\cdots,m_r$ satisfies conditions $(a)$ and $(b)$, i.e., $m_j\geq 2$ for all $j$ and $m_{i+1}|m_i$ for all $1\leq i\leq r-1$,
then $s=r$ and $m_i=n_i$ for all $i$.
\end{thm}
This theorem is known as the \textbf{elementary divisor theorem} of a finite abelian group.
Moreover, since $G$ is direct product of $\mathbb{Z}/n_i$, $1\leq i\leq s$, then we can write 
\begin{center}
 $|G|=n_1n_2\cdots n_s$.
\end{center}

We also need a structure theorem for finite abelian groups which come provided with an alternating character:

\begin{lem}[\cite{EWZ}, p. 270, Lemma 1(VI)]\label{Theorem 2.4}
 Let $G$ be a finite abelian group and assume the existence of an alternating bi-character 
 $X:G\times G\to\bbC^\times$  ( $X(g,g)=1$ for all $g\in G$, hence $1=X(g_1g_2,g_1g_2)=X(g_1,g_2)\cdot X(g_2,g_1)$) 
 which is nondegenerate. Then there will exist elements $t_1, t_1',\cdots,t_s,t_s'\in G$ 
such that 
 \begin{enumerate}
  \item 
   $G=<t_1>\times<t_1'>\times\cdots\times<t_s>\times<t_s'>$\\
  $ \cong\bbZ/m_1\times\bbZ/m_1\times\cdots\times\bbZ/m_s\times\bbZ/m_s$
  and $m_1|\cdots|m_s$.
 
\item For all $i=1,2,\cdots,s$ we have $X(t_i,t_i')=\zeta_{m_i}$ a primitive $m_i$-th root of unity.
  \item If we say $g_1\perp g_2$ if $X(g_1,g_2)=1$, then $(<t_i>\times<t_i'>)^\perp=\prod_{j\ne i}(<t_j>\times<t_j'>)$.
 \end{enumerate}

\end{lem}

\section{\textbf{Transfer map}}

Let $H$ be a subgroup of a finite group $G$. Let $\{t_1,t_2,\cdots,t_n\}$ be a left transversal for $H$ in $G$. If $g\in G$ then 
for all $i=1,2,\cdots,n$ we obtain, 
\begin{equation}
 g t_i\in t_{g(i)} H,
\end{equation}
where the map $i\mapsto g(i)$ is a well-defined permutation of the set $\{1,2,\cdots,n\}$. Assume that $f:H\to A$ is a homomorphism from $H$ 
to an abelian
group $A$. Then \textbf{transfer} of $f$, written $T_{f}$, is a mapping 
\begin{center}
 $T_f:G\to A$ \hspace{.5cm}given by \\
 $T_{f}(g)=\prod_{i=1}^{n}f(t_{g(i)}^{-1}g t_{i})$ \hspace{.4cm}for all $g\in G$.
\end{center}
Since $A$ is abelian, the order of the factors in the product is irrelevant. Since $f$ is a homomorphism
from $H$ to $A$, from above we can see that $T_f$ is a homomorphism $G$ with abelian image, and therefore always:
$[G,G]\subseteq\rm{Ker}(T_f)$.
Now we take $f$ the canonical homomorphism, i.e.,
\begin{center}
 $f:H\to H/[H,H]$, where $[H,H]$ is the commutator subgroup of $H$.
\end{center}
And we denote $T_{f}=T_{G/H}$. Thus by definition of transfer map $T_{G/H}:G\to H/[H,H]$, given by 
\begin{equation}
 T_{G/H}(g)=\prod_{i=1}^{n}f(t_{g(i)}^{-1}g t_{i})=\prod_{i=1}^{n}t_{g(i)}^{-1}g t_{i}[H,H],
\end{equation}
for all $g\in G$.

Moreover, if $H$ is any subgroup
of finite index in $G$, then (cf. \cite{AJ}, Chapter 13, p. 183)  
\begin{equation}\label{eqn 1.3}
 T_{G/gHg^{-1}}(g')=g T_{G/H}(g') g^{-1}, 
\end{equation}
for all $g,g'\in G$.
Now let $H$ be an abelian normal subgroup of $G$. Let $H^{G/H}$ be the set 
consisting the elements which are invariant under conjugation. So it is clear that these elements
are central elements and $H^{G/H}\subseteq Z(G)$, the center of $G$. When $H$ is abelian normal subgroup of $G$,
from equation (\ref{eqn 1.3}) we can conclude that (cf. \cite{AJ}, Chapter 13, p. 183) that
\begin{equation}\label{relation 2.6}
 \mathrm{Im}(T_{G/H})\subseteq H^{G/H}\subseteq Z(G). 
\end{equation}

We also need to mention the generalized {\bf Furtw\"{a}ngler}'s theorem for this thesis.
\begin{thm}[\cite{MI}, p. 320, Theorem 10.25]\label{Furtwangler's Theorem}
 Let $G$ be a finite group, and let $T_{G/K}: G\to K/[K,K]$ be the transfer homomorphism, where 
 $[G,G]\subseteq K\subseteq G$. Then $T_{G/K}(g)^{[K:[G,G]]}=1$ for all elements $g\in G$.
\end{thm}
Now if $[K:[G,G]]=1$, i.e., $K=[G,G]$, we have $T_{G/[G,G]}(g)=1$ for all $g\in G$, i.e.,
the transfer homomorphism of a finite group to its commutator is {\bf trivial}. This is due to 
Furtw\"{a}ngler. This is also known as {\bf Principal Ideal Theorem} (cf. \cite{AJ}, p. 194).

 To compute the determinant of an induced representation of a finite group, we need the following theorem.
 \begin{thm}[Gallagher, \cite{GK}, Theorem $30.1.6$]\label{Theorem Gall}
 Let $G$ be a finite group and $H$ a subgroup of $G$. Let $\rho$ be a representation of $H$ and denote 
 $\Delta_{H}^{G}=\mathrm{det}(\mathrm{Ind}_{H}^{G}1_H)$. Then 
  \begin{equation}
   \mathrm{det}(\mathrm{Ind}_{H}^{G}\rho)(g)=(\Delta_{H}^{G})^{\mathrm{dim}(\rho)}(g)\cdot
   (\mathrm{det}(\rho)\circ T_{G/H})(g), \quad\text{for all $g\in G$}.
  \end{equation}

\end{thm}

Let $T$ be a left transversal for $H$ in $G$. Here $\mathrm{Ind}_{H}^{G}\rho$
is a block monomial representation (cf. \cite{GK}, p. 956) with block positions indexed by pairs $(t,s)\in T\times T$. For 
$g\in G$, the $(t,s)$-block of $\mathrm{Ind}_{H}^{G}\rho$ is zero unless $gt\in sH$, i.e., $s^{-1}gt\in H$ and in which case
the block equal to $\rho(s^{-1} g t)$. 
Then we can write for $g\in G$
\begin{equation}
 T_{G/H}(g)=\prod_{t\in T}s^{-1}gt[H,H].
\end{equation}
Thus we can write for all $g\in G$
\begin{align}
 \mathrm{det}(\mathrm{Ind}_{H}^{G}\rho)(g)\nonumber
&=(\Delta_{H}^{G})^{\mathrm{dim}(\rho)}(g)\cdot(\mathrm{det}(\rho)\circ T_{G/H}(g)\\\nonumber
&=(\Delta_{H}^{G})^{\mathrm{dim}(\rho)}(g)\cdot(\mathrm{det}(\rho)(\prod_{t\in T}s^{-1}gt[H,H])\\
&=(\Delta_{H}^{G})^{\mathrm{dim}(\rho)}(g)\cdot\prod_{t\in T}\mathrm{det}(\rho)(s^{-1}gt [H,H]),\label{eqn 1.7}
\end{align}
where in each factor on the right, $s=s(t)$ is uniquely determined by $gt\in sH$.

\chapter{\textbf{Computation of $\lambda$-functions}}

In this chapter we give explicit computation of $\lambda_{K/F}$, where $K/F$ is a finite local Galois extension.
Two different ways we can define $\lambda$-functions: One is directly from local constants, and another one is via 
Deligne's constants. We will use both of them according to our convenience. 
In Section 3.2 we first compute $\lambda$-function for odd degree Galois extension.
And in Section 3.4 we compute $\lambda$-functions for even degree tamely ramified Galois extensions. The whole 
computation is based on the article \cite{SAB1}.

\section{\textbf{Deligne's Constants}}

Let $K/F$ be a finite Galois extension of a local field $F$ of characteristic zero. Let $G=\mathrm{Gal}(K/F)$, and let 
$\rho: G\to\mathrm{Aut}_{\mathbb{C}}(V)$ be a representation. Then for this representation, Deligne (cf. \cite{JT1}, p. 119)
defines:
\begin{equation}
 c(\rho):=\frac{W(\rho,\psi)}{W(\mathrm{det}(\rho),\psi)},
\end{equation}
where $\psi$ is some additive character of $F$. If we change the additive character $\psi$ to $\psi'=b\psi$, where 
$b\in F^\times$, then from \cite{BH}, p. 190,
part (2) of the Proposition, we see:
\begin{equation}\label{eqn 3.1.1}
W(\rho,b\psi)=\epsilon(\rho,\frac{1}{2},b\psi)=\mathrm{det}_{\rho}(b)\cdot \epsilon(\rho,\frac{1}{2},\psi)
 =\mathrm{det}_{\rho}(b)W(\rho,\psi).
\end{equation}
Also, from the property of abelian local constants we have 
$W(\mathrm{det}(\rho), b\psi)=\mathrm{det}_{\rho}(b)\cdot W(\mathrm{det}_\rho, \psi)$, hence
\begin{center}
 $\frac{W(\rho,\ b\psi)}{W(\mathrm{det}_\rho,\ b\psi)}=\frac{W(\rho,\ \psi)}{W(\mathrm{det}_\rho,\ \psi)}=c(\rho)$.
\end{center}
This shows that the Deligne's constant $c(\rho)$ does not depend on the choice of the additive character $\psi$.
We also have the following properties of Deligne's constants:
\begin{prop}[\cite{JT1}, p. 119, Proposition 2]\label{Proposition 3.1.1.1}
 \begin{enumerate}
  \item[(i)] If $\rm{dim}(\rho)=1$, then $c(\rho)=1$.
  \item[(ii)] 
  \begin{equation}
   c(\rho_1+\rho_2)=c(\rho_1)c(\rho_2)W(\det(\rho_1))W(\det(\rho_2))\cdot W(\det(\rho_1)\cdot\det(\rho_2))^{-1}. 
  \end{equation}
\item[(iii)] $c(\rho+\overline{\rho})=\det(\rho)(-1)$.
\item[(iv)] $c(\overline{\rho})=\overline{c(\rho)}$, and $|c(\rho)|=1$.
\item[(v)] Suppose $\rho=\overline{\rho}$. Then $c(\rho)=\pm 1$.
 \end{enumerate}
\end{prop}
\vspace{.3cm}
Now $G$ be a finite group. Let $\rho$ be an orthogonal representation of $G$, i.e., $\rho:G\to O(n)$. We denote the $i$-th 
{\bf Stiefel-Whitney class} of $\rho$ by 
$$s_i(\rho)\in H^{i}(G,\bbZ/2\bbZ).$$
 In low dimensions $i$, the Stiefel-Whitney class is given algebraically as follows (cf. \cite{JBC}, \cite{D2}):
Under the canonical isomorphism:
$$H^{1}(G,\bbZ/2\bbZ)\cong\rm{Hom}(G,\{\pm 1\}),$$
the image of $s_1(\rho)$ is $\det_\rho$. If $s_1(\rho)$ is trivial, i.e., $\det_\rho\equiv 1$, then $s_2(\rho)$ is the element
of $H^2(G,\{\pm 1\})=H^2(G,\bbZ/2\bbZ)$ which is inverse image under $\rho:G\to SO(n)$ of the class of the extension:
$$1\to\{\pm 1\}\to \rm{Spin}(n)\to SO(n)\to 1,$$
where $SO$ denotes the special orthogonal group and $\rm{Spin}$ the spinor group.

Now take $G=\rm{Gal}(K/F)$, for some finite Galois extension of local fields. In the following theorem due to Deligne 
for an orthogonal representation $\rho:G\to O(n)$,
we know a procedure how to obtain out of $s_2(\rho)$ the constant $c(\rho)$.

\begin{thm}[Deligne, \cite{JT1}, p. 129, Theorem 3]\label{Theorem 4.1}
 Let $\rho$ be an \textbf{orthogonal representation} of the finite group $G$ and let $s_2(\rho)\in H^{2}(G,\mathbb{Z}/2\mathbb{Z})$
 be the second Stiefel-Whitney class of $\rho$. The Galois group $G=\mathrm{Gal}(K/F)$ is a quotient group of the full Galois group 
 $G_{F}=\mathrm{Gal}(\overline{F}/F)$ which induces an inflation map
 \begin{equation}\label{eqn 32}
  \mathrm{Inf}: H^{2}(G,\mathbb{Z}/2\mathbb{Z})\rightarrow H^{2}(G_{F},\mathbb{Z}/2\mathbb{Z})\cong \{\pm 1\}.
 \end{equation}
Then 
\begin{equation}
 c(\rho)=cl(s_2(\rho))\in\{\pm 1\}
\end{equation}
is the image of the second Stiefel-Whitney class $s_2(\rho)$ under the inflation map (\ref{eqn 32}).\\
In particular, we have $c(\rho)=1$ if $s_2(\rho)=0\in H^{2}(G,\mathbb{Z}/2\mathbb{Z})$.
\end{thm}

 \section{\textbf{When $K/F$ is an odd degree Galois extension}}

We know that our local constant satisfies the following functional equation (cf. equation (\ref{eqn 2.3.23}))
\begin{equation}\label{eqn 4.5}
 W(\rho)W(\widetilde{\rho})=\mathrm{det}_{\rho}(-1),
\end{equation}
where $\rho$ is a representation of $G$ and $\widetilde{\rho}$ is the contragredient representation of $\rho$. 
In this equation (\ref{eqn 4.5}), we plug the orthogonal representation
$\rho=\mathrm{Ind}_{H}^{G}1_H$ 
in the place of 
$\rho$ and we have
\begin{equation}
  W(\mathrm{Ind}_{H}^{G}1_H)W(\widetilde{\mathrm{Ind}_{H}^{G}1_H})
 =\mathrm{det}_{\mathrm{Ind}_{H}^{G}1_H}(-1)
\end{equation}
Now we have from the definition of $\lambda$-factor,
\begin{align*}
\lambda_{H}^{G}(W)\lambda_{H}^{G}(W)
&=\mathrm{det}_{\mathrm{Ind}_{H}^{G}1_H}(-1),\quad
 \text{since $\widetilde{\mathrm{Ind}_{H}^{G}1_H}=\mathrm{Ind}_{H}^{G}1_H$}\\
 (\lambda_{H}^{G}(W))^{2}
 &=\mathrm{det}_{\mathrm{Ind}_{H}^{G}1_H}(-1)\\
  (\lambda_{H}^{G}(W))^{4}
  &=(\mathrm{det}_{\mathrm{Ind}_{H}^{G}1_H}(-1))^{2}\\
 (\lambda_{H}^{G}(W))^{4}
 &=1,\quad\text{since $\mathrm{det}_{\mathrm{Ind}_{H}^{G}1_H}(-1)$ is a sign}
\end{align*}
Therefore, our $\lambda$-factor  $\lambda_{H}^{G}(W)$ is always a \textbf{fourth root of unity}. 

We also know that
\begin{center}
$\mathrm{det}(\mathrm{Ind}_{H}^{G}\chi_{H})(s)=\varepsilon_{G/H}(s)\cdot(\chi_H\circ T_{G/H}(s))$, 
\end{center}
where $T_{G/H}$ is the transfer map from $G/[G,G]$ to $H/[H,H]$ and 
$\varepsilon_{G/H}(s)$ is the sign of $s\in G$ understood as
permutation of finite set $G/H:\{gH\mapsto sgH\}$. 
If we take $\chi_H=1_H$ the trivial character of $H$ then in particular see that 
\begin{center}
 $\Delta_{H}^{G}(s):=\mathrm{det}(\mathrm{Ind}_{H}^{G}1_H)(s)$
\end{center}
is a character of $G$ of order $2$. If $H\subset G$ is a normal subgroup then it is a character of the factor group 
$G/H$ and therefore it is \textbf{trivial} if $G/H$ is of odd order.

More generally, from Gallagher's Theorem \ref{Theorem Gall}, we have:\\ 
If $\rho$ is a (virtual) representation of $H$, then 
\begin{equation}\label{eqn 2.3}
 \mathrm{det}(\mathrm{Ind}_{H}^{G}\rho)(s)=\Delta_{H}^{G}(s)^{\mathrm{dim}(\rho)}\cdot (\mathrm{det}\rho\circ T_{G/H}(s)),
\end{equation}
for $s\in G$.

We assume now that the Galois groups $H\subset G$ have the fields $K\supset F$ as their base fields. Then by class field theory
we may interpret $\mathrm{det}(\rho)$ of equation (\ref{eqn 2.3})
as a character of $K^\times$ and $\mathrm{det}(\mathrm{Ind}_{H}^{G}\rho)$ as a character of $F^\times$, and then the equation 
(\ref{eqn 2.3}) turns into an equality of two characters of $F^\times$:
\begin{equation}\label{eqn 2.12}
 \mathrm{det}(\mathrm{Ind}_{H}^{G}\rho)=\Delta_{K/F}^{\mathrm{dim}(\rho)}\cdot \mathrm{det}\rho|_{F^\times},\quad
 \text{where}\quad \Delta_{K/F}:F^\times\to\{\pm1\}
\end{equation}
is the discriminant character\footnote{
From \textbf{example (III) on p. 104 of \cite{JT1}}, if $H=\mathrm{Gal}(L/K)<G=\mathrm{Gal}(L/F)$ corresponds to an 
extension $F\subset K\subset L$ of local fields then
\begin{center}
 $\mathrm{det}\circ\mathrm{Ind}_{H}^{G}1_H=\Delta_{K/F}$
\end{center}
can be interpreted by class-field theory as a character of $F^\times$. It is then character of $F^\times$ corresponding
to the quadratic extension $F(\sqrt{d_{K/F}})/F$, which is obtained by adjoining the square root of the discriminant
$d_{K/F}$ of $K/F$.} with respect to the extension $K/F$. If we consider $Z\subset H\subset G$ corresponding to the base fields
$E\supset K\supset F$ then we have 
\begin{center}
 $\Delta_{E/F}=\mathrm{det}(\mathrm{Ind}_{H}^{G}(\mathrm{Ind}_{Z}^{H}1_Z)$,
\end{center}
and with $\rho=\mathrm{Ind}_{Z}^{H}1_Z$ we conclude from (\ref{eqn 2.12}) that 
\begin{equation}\label{eqn 2.13}
 \Delta_{E/F}=\Delta_{E/K}|_{F^\times}\cdot\Delta_{K/F}^{[E:K]}.
\end{equation}

Moreover, in terms of Deligne's constant, we can write:
\begin{equation}
 \lambda_{H}^{G}:=W(\mathrm{Ind}_{H}^{G}1_H)=
 c(\mathrm{Ind}_{H}^{G}1_H)\cdot W(\mathrm{det}\circ\mathrm{Ind}_{H}^{G}1_H).\label{eqn 4.3}
\end{equation}

Replacing Galois groups by the corresponding local fields we may write the lambda function of finite extension $K/F$ as 
\begin{equation}
 \lambda_{K/F}=c(\mathrm{Ind}_{K/F}1)\cdot W(\Delta_{K/F}),\label{eqn 4.4}
\end{equation}
where $c(\mathrm{Ind}_{K/F}1)$ is Deligne's sign, and $\Delta_{K/F}$ is a quadratic character of $F^\times$ related to 
the discriminant.

\begin{lem}\label{Lemma 4.1}
Let $L/F$ be a finite Galois extension of a non-archimedean local field $F$ and $G=\rm{Gal}(L/F)$, $H=\rm{Gal}(L/K)$. 
If $H\leq G$ is a normal subgroup and if $[G:H]$ is odd, then $\Delta_{K/F}\equiv 1$ and 
$\lambda_{K/F}^{2}=1.$
\end{lem}

\begin{proof}
 If $H$ is a normal subgroup, then $\mathrm{Ind}_{H}^{G}1_H=\mathrm{Ind}_{\{1\}}^{G/H}1$ is the regular representation of 
 $G/H$, hence $\mathrm{det}\circ\mathrm{Ind}_{H}^{G}1_H=\Delta_{K/F}$ is the quadratic character of the group $G/H$.
 By the given condition order of $G/H$ is odd, then $\Delta_{K/F}\equiv 1$, hence
  $\lambda_{K/F}^{2}=\Delta_{K/F}(-1)$. Thus $\lambda_{K/F}^{2}=1$. 
\end{proof}
{\bf Note:} Since $\Delta_{K/F}\equiv 1$, then $W(\Delta_{K/F})=1$. We also know that $c(\rm{Ind}_{K/F}(1))\in\{\pm 1\}$.
Then from equation (\ref{eqn 4.4}) we can easily see that $\lambda_{K/F}^{2}=1$.

In the next lemma we state 
some important results for our next Theorem \ref{General Theorem for odd case}. These are the 
consequences of Deligne's formula for the local constant of orthogonal representations.

\begin{lem}\label{Lemma 4.2}
\begin{enumerate}
 \item If $H\leq G$ is a normal subgroup of odd index $[G:H]$, then $\lambda_{H}^{G}=1$.
 \item If there exists a normal subgroup $N$ of $G$ such that $N\leq H\leq G$ and $[G:N]$ odd, then $\lambda_{H}^{G}=1$.
\end{enumerate}

\end{lem}
\begin{proof}
 \begin{enumerate}
 
\item To prove (1) we use the equation (\ref{eqn 4.3})
\begin{equation}
 \lambda_{H}^{G}=W(\mathrm{Ind}_{H}^{G}1_H)=
 c(\mathrm{Ind}_{H}^{G}1_H)\cdot W(\mathrm{det}\circ\mathrm{Ind}_{H}^{G}1_H).\label{eqn 4.7}
\end{equation}
Since $\rho=\mathrm{Ind}_{H}^{G}1_H$ is orthogonal we may compute $c(\rho)$ by using the second  Stiefel-Whitney class $s_2(\rho)$
\footnote{This Stiefel-Whitney class $s_2(\rho)$ is easy accessible only if 
$\mathrm{det}_\rho\equiv 1$, and this is in general wrong for $\rho=\mathrm{Ind}_{H}^{G}1_H$. But it is true for 
$\rho=\mathrm{Ind}_{H}^{G}1_H$ if $H\leq G$ is a normal subgroup and $[G:H]$ is odd (by using Lemma \ref{Lemma 4.1}).}. From Proposition
\ref{Proposition 3.1.1.1}(v)
we know that $c(\rho)=W(\rho)/W(det_\rho)$ is a sign. If $cl(s_2(\rho))$
is the image of $s_2(\rho)$ under inflation map (which is injective), then according to Deligne's theorem \ref{Theorem 4.1}, we have:
\begin{center}
 $c(\rho)=cl(s_2(\rho))$
\end{center}
if $\rho$ is orthogonal.  
Moreover, we have 
\begin{center}
 $s_2(\mathrm{Ind}_{H}^{G}1_H)\in H^2(G/H,\mathbb{Z}/2\mathbb{Z})=\{1\}$,
\end{center}
which implies that in equation (\ref{eqn 4.7}) both factors are $=1$, hence $\lambda_{H}^{G}=1$.
\item From $N\leq H\leq G$ we obtain 
\begin{equation}
 \lambda_{N}^{G}=\lambda_{N}^{H}\cdot(\lambda_{H}^{G})^{[H:N]}
\end{equation}
From (1) we obtain $\lambda_{N}^{G}=\lambda_{N}^{H}=1$ because $N$ is normal and the index $[G:N]$ is odd, hence 
$(\lambda_{H}^{G})^{[H:N]}=1$. Finally this implies $\lambda_{H}^{G}=1$ because $\lambda_{H}^{G}$ is 4th root of unity
and $[H:N]$ is odd.
 \end{enumerate}

\end{proof}
\textbf{Note:} In the other words we can state this above Lemma \ref{Lemma 4.2} as follows:\\ 
Let $H'=\cap_{x\in\Delta}xHx^{-1}\subset\Delta$ be 
the largest subgroup of $H$ which is normal in $\Delta\subseteq G$. Then $\lambda_{H}^{\Delta}(W)=1$ if the index $[\Delta:H']$ is odd, in 
particular if $H$ itself is a normal subgroup of $\Delta$ of odd index.

Now we are in a position to state the main theorem for odd degree Galois extension of a non-archimedean local field.

\begin{thm}\label{General Theorem for odd case} 
Let $F$ be a non-archimedean local field and $E/F$ be an odd degree Galois extension. If 
$L\supset K\supset F $ be any finite extension inside $E$, then $\lambda_{L/K}=1$. 
\end{thm}

\begin{proof}
 By the given condition $|\mathrm{Gal}(E/F)|$ is odd. Then the degree of extension $[E:F]$ of $E$ over $F$ is odd. 
 Let $L$ be any 
 arbitrary intermediate field of $E/F$ which contains $K/F$. Therefore, here we have the 
 tower of fields $E\supset L\supset K\supset F$. Here the degree of extensions are all odd since $[E:F]$ is odd.
 By assumption $E/F$ is Galois, then also the extension $E/L$ and $E/K$
 are Galois and $H=\mathrm{Gal}(E/L)$ is subgroup of $G=\mathrm{Gal}(E/K)$. 
 
 By the definition 
 we have $\lambda_{L/K}=\lambda_{H}^{G}$.
 If $H$ is a normal subgroup of $G$ then $\lambda_{H}^{G}=1$ because $|G/H|$ is odd. But $H$ need not be normal subgroup
 of $G$ therefore $L/K$ need not be Galois extension.  
  Let $N$ be the \textbf{largest} normal subgroup of $G$ contained in $H$ and $N$ can be written as:
 \begin{equation*}
  N=\cap_{g\in G}gHg^{-1}
 \end{equation*}
 Therefore, the fixed field $E^{N}$ is the \textbf{smallest normal} extension of $K$ containing $L$. Now we have 
 from properties of $\lambda$-function(cf. 2.2(2)),
 \begin{equation}
  \lambda_{N}^{G}=\lambda_{N}^{H}\cdot(\lambda_{H}^{G})^{[H:Z]}.
 \end{equation}
This implies $(\lambda_{H}^{G})^{[H:N]}=1$ since $[H:N]$ and $[G:N]$ are odd and $N$ is normal subgroup of $G$ contained in 
$H$. Therefore $\lambda_{H}^{G}=1$ because $\lambda_{H}^{G}$ is 4th root of unity and $[H:N]$ is odd. Moreover, if $N=\{1\}$, it is then
clear that $N$ is a normal subgroup of $G$ which sits in $H$ and $[G:N]=|G|$ is odd. Therefore Lemma \ref{Lemma 4.2}(2) implies
that $\lambda_{H}^{G}=1$.

Then we may say $\lambda_{L/K}=1$ all possible cases if $[E^{N}:K]$ is odd. When the big extension $E/F$ is odd then all intermediate
extensions will be odd. Therefore, the theorem is proved for all possible cases.

\end{proof}

\begin{rem}\label{Remark 4.4} 
{\bf (1).} If the Galois extension $E/F$ is infinite then we say it is \textbf{odd} if 
 $[K:F]$ is odd for all sub-extensions of finite degree. This means the pro-finite group $\mathrm{Gal}(E/F)$ can 
 be realized as the projective limit of finite groups which are all odd order.
 If $E/F$ is Galois extension of odd order in this more general sense, then again we will have 
 $\lambda_{L/K}=1$ in all cases where $\lambda$-function is defined.\\
 {\bf (2).} When the order of a finite local Galois group is odd, all weak extensions are strong extensions, because
 from the above Theorem \ref{General Theorem for odd case} we have
 $\lambda_{1}^{G}=1$, where $G$ is the odd order local Galois group.
 Let $H$ be a any arbitrary subgroup of $G$, then from the properties of $\lambda$-functions we have
 $$\lambda_{1}^{G}=\lambda_{1}^{H}\cdot(\lambda_{H}^{G})^{|H|}.$$
 This implies $\lambda_{H}^{G}=1$, because $|H|$ is odd, hence $\lambda_{1}^{H}=1$ and $\lambda$-functions are fourth
roots of unity.\\
{\bf (3).}
But this above Theorem \ref{General Theorem for odd case} is not true if $K/F$ is not {\bf Galois}. 
Guy Henniart gives {\bf``An amusing formula"} 
 (cf. \cite{GH}, p. 124, Proposition 2) for 
 $\lambda_{K/F}$, when $K/F$ is arbitrary odd degree extension, and this formula is:
 \begin{equation}
  \lambda_{K/F}=W(\Delta_{K/F})^{n}\cdot\left(\frac{2}{q_F}\right)^{a(\Delta_{K/F})},
 \end{equation}
where $K/F$ is an extension in $\overline{F}$ with finite odd degree $n$, and $\left(\frac{2}{q_F}\right)$ is the Legendre symbol if 
$p$ is odd and is $1$ if $p=2$. Here $a$ denotes the exponent Artin-conductor. 

\end{rem}

\section{\textbf{Computation of $\lambda_{1}^{G}$, where $G$ is a finite local Galois group}}

Let $G$ be a finite local Galois group of a non-archimedean local field $F$. Now we consider \textbf{Langlands'} $\lambda$-function:
\begin{equation}\label{eqn 4.6}
 \lambda_{H}^{G}:=W(\rm{Ind}_{H}^{G}1_H)=c_{H}^{G}\cdot W(\Delta_{H}^{G}),
\end{equation}
where 
\begin{center}
 $c_{H}^{G}:=c(\rm{Ind}_{H}^{G}1_H)$.
\end{center}
From the equation (\ref{eqn 4.6}) we observe that to compute $\lambda_{H}^{G}$ we need to compute the Deligne's constant $c_{H}^{G}$
and $W(\Delta_{H}^{G})$. By the following theorem due to Bruno Kahn we will get our necessary information for our
further requirement.

\begin{thm}[\cite{BK}, S\'{e}rie 1-313, Theorem 1]\label{Theorem 4.2}
 Let $G$ be a finite group, $r_G$ its regular representation. Let $S$ be any $2$-Sylow subgroup of $G$. Then $s_2(r_G)=0$
 in the following cases:
 \begin{enumerate}
  \item $S$ is cyclic group of order $\geq 8$;
  \item $S$ is generalized quaternion group;
  \item $S$ is not metacyclic group.
 \end{enumerate}
\end{thm}

We also need Gallagher's result. 

\begin{thm}[Gallagher, \cite{GK}, Theorem $30.1.8$]\label{Theorem 1.3}
 Assume that $H$ is a normal subgroup of $G$, hence $\Delta_{H}^{G}=\Delta_{1}^{G/H}$, then 
 \begin{enumerate}
  \item $\Delta_{H}^{G}=1_G$, where $1_G$ is the trivial representation of $G$, unless the Sylow $2$-subgroups of $G/H$ are cyclic and 
  nontrivial.
  \item If the Sylow $2$-subgroups of $G/H$ are cyclic and nontrivial, then $\Delta_{H}^{G}$ is the only linear character of $G$
  of order $2$.
 \end{enumerate}
\end{thm}

\begin{dfn}[\textbf{$2$-rank of a finite abelian group}]
Let $G$ be a finite abelian group. Then from elementary divisor Theorem \ref{Theorem 22.4} we can write
\begin{equation}
 G\cong\bbZ_{m_1}\times\bbZ_{m_2}\times\cdots\times\bbZ_{m_s}
\end{equation}
where $m_1|m_2|\cdots|m_s$ and $\prod_{i=1}^{s}m_i=|G|$. We define 
\begin{center}
the $2$-rank of $G:=$the number of $m_i$-s which are even
\end{center}
 and we set 
\begin{center}
 $\mathrm{rk}_{2}(G)=$ $2$-rank of $G$.
\end{center}
When the order of an abelian group $G$ is odd, from the structure of $G$ we have $\mathrm{rk}_2(G)=0$, i.e., there is no 
even $m_i$-s for $G$.
We also denote 
\begin{center}
 $G[2]:=\{x\in G|\quad 2x=0\}$, i.e., set of all elements of order at most $2$.
\end{center}
If $G=G_1\times G_2\times \cdots\times G_r$, $r\in\bbN$, is an abelian group, then we can show that 
\begin{equation}\label{eqn 444}
 |G[2]|=\prod_{i=1}^{r}|G_i[2]|.
\end{equation}

\end{dfn}

\begin{rem}[\textbf{Remark on Theorem \ref{Theorem 1.3}}]\label{Remark 3.4}
If $G$ is a finite group with subgroups $H'\subset H\subset G$ then for $\Delta_{H}^{G}=\det(\rm{Ind}_{H}^{G}1_H)$ we know from  
Gallagher's Theorem \ref{Theorem Gall}
\begin{align}
 \Delta_{H'}^{G}
 &=\det(\rm{Ind}_{H'}^{G}1_{H'})=\det(\rm{Ind}_{H}^{G}(\rm{Ind}_{H'}^{H}1_{H'}))\nonumber\\
 &=(\Delta_{H}^{G})^{[H:H']}\cdot\det(\rm{Ind}_{H'}^{H}1_{H'})\circ T_{G/H}\nonumber\\
 &=(\Delta_{H}^{G})^{[H:H']}\cdot(\Delta_{H'}^{H}\circ T_{G/H}).\label{eqn 3.28}
\end{align}
Now we use equation (\ref{eqn 3.28}) for $H'=\{1\}$ and $H=[G,G]=G'$. Then we have 
\begin{equation}\label{eqn 3.29}
 \Delta_{1}^{G}=(\Delta_{G'}^{G})^{|G'|}\cdot\Delta_{1}^{G'}\circ T_{G/G'}=(\Delta_{G'}^{G})^{|G'|},
\end{equation}
because by Theorem \ref{Furtwangler's Theorem},  $T_{G/G'}$ is the trivial map.



We also know that $G'$ is a normal subgroup of $G$, then we can write $\rm{Ind}_{G'}^{G}1_{G'}\cong\rm{Ind}_{1}^{G/G'}1$, hence 
$\Delta_{G'}^{G}=\Delta_{1}^{G/G'}$. So we have 
\begin{equation}\label{eqn 3.30}
 \Delta_{1}^{G}=(\Delta_{G'}^{G})^{|G'|}=(\Delta_{1}^{G/G'})^{|G'|}.
\end{equation}
From the above equation (\ref{eqn 3.30}) we observe that $\Delta_{1}^{G}$ always reduces to the \textbf{abelian case} because 
$G/G'$ is abelian. Moreover, we know that:\\
\emph{If $G$ is abelian then $\rm{Ind}_{1}^{G}1=r_G$ is the sum of all characters of $G$, hence from Miller's result 
(cf. \cite{PC}, Theorem 6) for the abelian
group $\widehat{G}$ we obtain:
\begin{align}
 \Delta_{1}^{G}
 &=\det(\rm{Ind}_{1}^{G}1)=\det(\sum_{\chi\in\widehat{G}}\chi)\nonumber\\
 &=\prod_{\chi\in\widehat{G}}\det(\chi)=\prod_{\chi\in\widehat{G}}\chi\nonumber\\
 &=\begin{cases}
    \alpha & \text{if $\rm{rk}_2(G)=1$}\\
    1 & \text{if $\rm{rk}_2(G)\ne 1$},
   \end{cases}\label{eqn 3.31}
\end{align}
where $\alpha$ is the uniquely determined quadratic character of $G$.}
\end{rem}

\begin{lem}\label{Corollary 3.3.7}
The lambda function for a finite unramified extension of a non-archimedean local field is always a sign. 
\end{lem}
 \begin{proof}
  Let $K$ be a finite unramified extension of a non-archimedean local field $F$. 
  We know that the unramified extensions are Galois, and their corresponding
  Galois groups are cyclic. 
 Let $G=\rm{Gal}(K/F)$, hence $G$ is cyclic. 
 
 When the degree of $K/F$ is odd, from Theorem \ref{General Theorem for odd case} 
 we have $\lambda_1^G=\lambda_{K/F}=1$ because $K/F$ is Galois.
 
 When the degree of $K/F$ is even, we have $\rm{rk}_2(G)=1$ because $G$ is cyclic. 
  So from equation (\ref{eqn 3.31}) we can write 
  $\Delta_{1}^{G}=\alpha$, where $\alpha$
  corresponds to the quadratic unramified extension. Then $\Delta_{1}^{G}(-1)=\alpha(-1)=1$, because 
  $-1$ is a norm, hence from the functional equation (\ref{eqn 2.3.23}) we have
  $$(\lambda_{1}^{G})^2=1.$$
 \end{proof}


Moreover, since $G/G'$ is abelian, by using equation (\ref{eqn 3.31}) for $G/G'$, from equation (\ref{eqn 3.30}) we obtain:

\begin{lem}\label{Lemma 3.41}
 Let $G$ be a finite group and $S$ be a Sylow 2-subgroup of $G$. Then the following are equivalent:
 \begin{enumerate}
  \item $S<G$ is nontrivial cyclic;
  \item $\Delta_{1}^{G}\ne 1$, is the unique quadratic character of $G$;
  \item $\rm{rk}_2(G/G')=1$ and $|G'|$ is odd.
 \end{enumerate}

\end{lem}

\begin{proof}
 Take $H=\{1\}$ in Gallagher's Theorem \ref{Theorem 1.3} and we can see that $(1)$ and $(2)$ are equivalent. From  
 equation (\ref{eqn 3.30}) we can see $(2)$ implies the condition $(3)$.
 
Now we are left to show that $(3)$ implies $(1)$. Let $S'$ be a Sylow 2-subgroup of $G/G'$. Since $\rm{rk}_2(G/G')=1$, hence 
$\rm{rk}_2(S')=1$, and therefore $S'$ is cyclic. Moreover, 
 $|G'|$ is odd, hence $|S|=|S'|$. Let $f:G\to G/G'$ be the canonical group homomorphism. Since $|G'|$ is odd, 
 and $\rm{rk}_2(G/G')=1$,
 $f|_{S}$ is an isomorphism from $S$ to $S'$. Hence $S$ is a nontrivial cyclic Sylow 2-subgroup of $G$. 

This completes the proof. 

\end{proof}

\begin{thm}[\textbf{Schur-Zassenhaus}]\label{Theorem 3.61}
  If $H\subset G$ is a normal subgroup such that 
 $|H|$ and $[G:H]$ are relatively prime, then $H$ will have a complement $S$ that is a subgroup of $G$ such that 
 \begin{center}
  $G=H\rtimes S$
 \end{center}
is a semidirect product.
\end{thm}
Let $G$ be a local Galois group. Then it is known that $G$ has Hall-subgroups (because $G$ is solvable), $H\subset G$ of all types 
such that $[G:H]$ and $|H|$ are relatively prime. In particular, $G$ will have an odd Hall subgroup $H\subset G$ such that $|H|$ 
is odd and $[G:H]$ is power of $2$. 
From this we conclude
\begin{prop}\label{Proposition 3.5}
Let $G$ be a finite local Galois group.
 Let $H\subset G$ be an odd order Hall subgroup of $G$ (which is unique up to conjugation). Then we have 
 \begin{equation}\label{eqn 3.11}
  \lambda_{1}^{G}=(\lambda_{H}^{G})^{|H|}.
 \end{equation}
Hence $\lambda_{1}^{G}=\lambda_{H}^{G}$ if $|H|\equiv 1\pmod{4}$ and $\lambda_{1}^{G}=(\lambda_{H}^{G})^{-1}$ if 
$|H|\equiv 3\pmod{4}$.\\
If the local base field $F$ has residue characteristic $p\ne 2$, then the odd order Hall subgroup $H\subset G$ is a normal subgroup
and therefore $\lambda_{H}^{G}=\lambda_{1}^{G/H}$, where $G/H\cong S$ is isomorphic to a Sylow $2$-subgroup of $G$. For $G=\rm{Gal}(E/F)$
this means that we have a unique normal extension $K/F$ in $E$ such that $\rm{Gal}(K/F)$ is isomorphic to a Sylow $2$-subgroup
of $G$, and we will have 
\begin{center}
 $\lambda_{E/F}=\lambda_{K/F}^{[E:K]}$.
\end{center}
\end{prop}

\begin{proof}
We know that our local Galois group $G$ is solvable, then $G$ has an odd order Hall subgroup $H\subset G$.
 Then the formula (\ref{eqn 3.11}) follows because $\lambda_{1}^{H}=1$ (here $|H|$ is 
 odd and $H$ is a subgroup of the local Galois group $G$).
 
 Let now $p\ne 2$ and let $H$ be an odd order Hall subgroup of $G$. The 
 ramification subgroup $G_1\subset G$ is a normal subgroup of order a power of $p$, hence $G_1\subset H$, and $H/G_1\subset G/G_1$
 will be an odd order Hall subgroup of $G/G_1$. But the group $G/G_1$ is \textbf{supersolvable}. It is also well known 
 that the odd order Hall subgroup of a supersolvable group is normal. Therefore $H/G_1$ is normal in $G/G_1$, and this implies that 
 $H$ is normal in $G$. Now we can use Theorem \ref{Theorem 3.61} and we obtain $G/H\cong S$, where
 $S$ must be a Sylow $2$-subgroup. Therefore when $p\ne 2$ we have 
 \begin{center}
  $\lambda_{1}^{G}=\lambda_{E/F}=(\lambda_{1}^{G/H})^{|H|}=\lambda_{K/F}^{[E:K]}$,
 \end{center}
where $G=\rm{Gal}(E/F)$, $H=\rm{Gal}(E/K)$ and $G/H=\rm{Gal}(K/F)\cong S$.
 
 \end{proof}

Let $F/\bbQ_p$ be a local field with $p\ne 2$. Let $K/F$ be the extension such that $\rm{Gal}(K/F)=V$ Klein's $4$-group.
In the following lemma we give explicit formula for $\lambda_{1}^{V}=\lambda_{K/F}$.

\begin{lem}\label{Lemma 4.6}
Let $F/\bbQ_p$ be a local field with $p\ne 2$. Let $K/F$ be the uniquely determined extension with $V=\rm{Gal}(K/F)$, Klein's $4$-group. 
Then \\
 $\lambda_{1}^{V}=\lambda_{K/F}=-1$ if $-1\in F^\times$ is a square, i.e., $q_F\equiv 1\pmod{4}$, and\\
 $\lambda_{1}^{V}=\lambda_{K/F}=1$ if $-1\in F^\times$ is not a square , i.e., if $q_F\equiv 3\pmod{4}$,\\
 where $q_F$ is the cardinality of the residue field of $F$.
\end{lem}

\begin{proof}
If $p\ne 2$ then from Theorem \ref{Theorem 2.9} the square class group $F^\times/{F^\times}^2$ is Klein's 4-group, and $K/F$
is the unique abelian extension such that $N_{K/F}(K^\times)={F^\times}^2$, hence 
$$\rm{Gal}(K/F)\cong F^\times/{F^\times}^2=V.$$
Since $V$ is abelian, we can write $\widehat{V}\cong V$. This implies that
there are exactly three nontrivial characters of $V$ and they are quadratic. 
By class field theory we can consider them as quadratic characters of $F^\times$.
This each quadratic character determines a quadratic extension of $F$. 
Thus there are three quadratic subextensions $L_i/F$ in $K/F$, where $i=1,2,3$. We denote $L_1/F$ the 
unramified extension whereas $L_2/F$ and $L_3/F$ are tamely ramified. Then we can write 
\begin{equation}
 \lambda_{K/F}=\lambda_{K/L_i}\cdot\lambda_{L_i/F}^{2} 
\end{equation}
for all $i\in\{1,2,3\}$.
 The group $V$ has four characters $\chi_i$, $i=0,\cdots, 3$, where $\chi_0\equiv 1$ and 
 $\chi_i(i=1,2,3)$ are three characters of $V$ such that $\mathrm{Gal}(K/L_i)$ is the kernel of $\chi_i$, in other
words, $\chi_i$ is the quadratic character of $F^\times/N_{L_i/F}(L_{i}^{\times})$.

Let $r_V=\mathrm{Ind}_{\{1\}}^{V} 1$,
 then 
 \begin{center}
  $\Delta_{1}^{V}=\mathrm{det}(r_V)=\prod_{i=0}^{3}\chi_i\equiv 1$,
 \end{center}
because $\chi_3=\chi_1\cdot\chi_2$. Therefore $W(\Delta_{1}^{V})=1$ and 
\begin{center}
 $\lambda_{K/F}=c(r_V)$
\end{center}
is Deligne's constant. More precisely we have 
\begin{equation}\label{eqn 4.26}
 \lambda_{K/F}=W(\chi_1)\cdot W(\chi_2)\cdot W(\chi_1\chi_2).
\end{equation}
But here $\chi_1$ is unramified and therefore $W(\chi_1)=\chi_1(c_1)$ (see equation (\ref{eqn 2.3.5})) 
and by using unramified character twisting formula, $W(\chi_1\chi_2)=\chi_1(c_2)\cdot W(\chi_2)$, where 
$c_2=\pi_F c_1$ because $a(\chi_2)=1+a(\chi_1)=1$. Therefore the equation (\ref{eqn 4.26}) implies:
\begin{equation}
 \lambda_{K/F}=\chi_1(c_1)^{2}\cdot\chi_1(\pi_F)\cdot W(\chi_2)^{2}=-\chi_2(-1),
\end{equation}
since $\chi_1(\pi_F)=-1$.
Similarly putting $\chi_2=\chi_{1}^{-1}\chi_3=\chi_1\chi_3$ and $\chi_1\chi_2=\chi_3$ in the equation (\ref{eqn 4.26}) we have 
\begin{equation}
 \lambda_{K/F}=-\chi_3(-1).
\end{equation}
Therefore we have $\lambda_{K/F}=-\chi_i(-1)$ for $i=2,3$.

Moreover, we know that 
\begin{equation*}
 \chi_i(-1)=\begin{cases}
             1 & \text{if $-1\in F^\times$ is a square, i.e., $q_F\equiv 1\pmod{4}$}\\
             -1 & \text{if $-1\in F^\times$ is not a square , i.e., $q_F\equiv 3 \pmod{4}$}\\
            \end{cases}
\end{equation*}
Thus finally we conclude that 
\begin{equation*}
 \lambda_{K/F}=-\chi_i(-1)=\begin{cases}
             -1 & \text{if $-1\in F^\times$ is a square, i.e., $q_F\equiv 1\pmod{4}$}\\
             1 & \text{if $-1\in F^\times$ is not a square, i.e., $q_F\equiv 3 \pmod{4}$}\\
            \end{cases}
\end{equation*}

\end{proof}

In the following theorem we give a general formula of $\lambda_{1}^{G}$, where $G$ is a finite local Galois group. 

\begin{thm}\label{Theorem 4.3}
 Let $G$ be a finite local Galois group of a non-archimedean local field $F$. Let $S$ be a Sylow 2-subgroup of $G$.
 \begin{enumerate}
 \item If $S=\{1\}$, then we have $\lambda_{1}^{G}=1$. 
  \item If the Sylow 2-subgroup $S\subset G$ is nontrivial cyclic (\textbf{exceptional case}), then
  \begin{equation}
   \lambda_{1}^{G}=\begin{cases}
                    W(\alpha) & \text{if $|S|=2^n\ge 8$}\\
                    c_{1}^{G}\cdot W(\alpha) & \text{if $|S|\le 4$},
                   \end{cases}
\end{equation}
where $\alpha$ is a uniquely determined quadratic 
  character of $G$.
  \item If $S$  is metacyclic but not cyclic ({\bf invariant case}), then 
  \begin{equation}
   \lambda_{1}^{G}=\begin{cases}
                    \lambda_{1}^{V} & \text{if $G$ contains Klein's $4$ group $V$}\\
                    1 &  \text{if $G$ does not contain Klein's $4$ group $V$}.
                   \end{cases}
\end{equation}
  \item If $S$ is nontrivial and not metacyclic, then $\lambda_{1}^{G}=1$.
 \end{enumerate}
\end{thm}

\begin{proof}
 {\bf (1).} When $S=\{1\}$, i.e., $|G|$ is odd, we know from Theorem \ref{General Theorem for odd case} that $\lambda_{1}^{G}=1$.\\
 {\bf (2).} When $S=<g>$ is a nontrivial cyclic subgroup of $G$, $\Delta_{1}^{G}$ is nontrivial
 (because $\Delta_{1}^{G}(g)=(-1)^{|G|-\frac{|G|}{|S|}}=-1$) and 
 by Lemma \ref{Lemma 3.41}, $\Delta_{1}^{G}=\alpha$,
 where $\alpha$ is a uniquely determined quadratic character of $G$. Then we obtain
 \begin{center}
  $\lambda_{1}^{G}=c_{1}^{G}\cdot W(\Delta_{1}^{G})=c_{1}^{G}\cdot W(\alpha)$.
 \end{center}
 If $S$ is cyclic of order $2^n\ge 8$, then by Theorem \ref{Theorem 4.2}(case 1) and Theorem \ref{Theorem 4.1} 
 we have $c_{1}^{G}=1$, hence $\lambda_{1}^{G}=W(\alpha)$. \\
{\bf (3).} The Sylow 2-subgroup $S\subset G$ is metacyclic but not cyclic (invariant case):\\
If $G$ contains Klein's $4$-group $V$, then $V\subset S$ because all Sylow $2$-subgroups are conjugate to each other. Then we have
$V<S<G$.
So from the properties of $\lambda$-function we have 
\begin{center}
 $\lambda_{1}^{G}=\lambda_{1}^{V}\cdot(\lambda_{V}^{G})^{4}=\lambda_{1}^{V}$.
\end{center}

Now assume that $G$ does not contain Klein's $4$-group. Then by assumption $S$ is metacyclic, 
not cyclic and does not contain Klein's $4$-group. We are going to see that this
implies: $S$ is generalized quaternion, and therefore by Theorem \ref{Theorem 4.2}, $s_2(\rm{Ind}_{1}^{G}(1))=0$, hence
$c_{1}^{G}=1$.

We use the following criterion for generalized quaternion groups: A finite
$p$-group in which there is a unique subgroup of order $p$ is either cyclic or generalized
quaternion (cf. \cite{MH}, p. 189, Theorem 12.5.2).\\
So it is enough to show: If $S$ does not contain Klein's $4$-group then $S$ has precisely one
subgroup of order $2$. Indeed, we consider the center $Z(S)$ which is a nontrivial abelian
$2$-group. If it would be non-cyclic then $Z(S)$, hence $S$ would contain Klein's $4$-group. So
$Z(S)$ must be cyclic, hence we have precisely one subgroup $Z_2$ of order $2$ which sits in
the center of $S$. Now assume that $S$ has any other subgroup $U\subset S$ which is of order $2$.
Then $Z_2$ and $U$ would generate a Klein-$4$-group in $S$ which by our assumption cannot
exist. Therefore $Z_2\subset S$ is the only subgroup of order $2$ in $S$. But $S$ is not cyclic, so it is
generalized quaternion.\\
Thus we can write $\lambda_{1}^{G}=c_{1}^{G}\cdot W(\Delta_{1}^{G})=W(\Delta_{1}^{G})$. Now to complete the proof we need to show 
that $W(\Delta_{1}^{G})=1$. This follows from Lemma \ref{Lemma 3.41}.\\
{\bf (4).} The Sylow 2-subgroup $S$ is nontrivial and not metacyclic.\\
 We know that every cyclic group is also a metacyclic group. Therefore when $S$ is nontrivial and not metacyclic, we are \textbf{not}
 in the position: $\rm{rk}_2(G/G')=1$ and $|G'|$ is odd. This gives $\Delta_{1}^{G}=1$, hence $W(\Delta_{1}^{G})=1$. Furthermore
by using the Theorem \ref{Theorem 4.2} and Theorem \ref{Theorem 4.1} we obtain the
 second Stiefel-Whitney class $s_2(\rm{Ind}_{1}^{G}(1))=0$, hence $\lambda_{1}^{G}=c_{1}^{G}\cdot W(\Delta_{1}^{G})=1$.

This completes the proof.
 
\end{proof}

In the above Theorem \ref{Theorem 4.3} we observe that if we are in the \textbf{Case 3}, then by using Lemma \ref{Lemma 4.6} we can 
give complete formula of $\lambda_{1}^{G}$ for $p\ne 2$. Moreover, by using Proposition \ref{Proposition 3.5} in \textbf{case 2}, we can
come down the computation of $\lambda_{K/F}$, where $K/F$ is quadratic.

\begin{cor}\label{Lemma 3.10}
 Let $G=\rm{Gal}(E/F)$ be a finite local Galois group of a non-archimedean local field $F/\bbQ_p$ with $p\ne 2$. 
 Let $S\cong G/H$ be a nontrivial Sylow 2-subgroup of $G$, where $H$ is a uniquely determined Hall subgroup of odd order. Suppose that 
 we have a tower $E/K/F$
 of fields such that $S\cong \rm{Gal}(K/F)$, $H=\rm{Gal}(E/K)$ and $G=\rm{Gal}(E/F)$.
 \begin{enumerate}
  \item If $S\subset G$ is cyclic, then 
  \begin{enumerate}
 \item
 \begin{equation*}
 \lambda_{1}^{G}=\lambda_{K/F}^{\pm 1}=\begin{cases}
                  \lambda_{K/F}=W(\alpha) & \text{if $[E:K]\equiv 1\pmod{4}$}\\
                  \lambda_{K/F}^{-1}=W(\alpha)^{-1} & \text{if $[E:K]\equiv -1\pmod{4}$},
                 \end{cases}
 \end{equation*}
 (here $\alpha=\Delta_{K/F}$ corresponds to the unique quadratic subextension in $K/F$) 
if $[K:F]=2$, hence $\alpha=\Delta_{K/F}$.
 \item 
 \begin{equation*}
\lambda_{1}^{G}=\beta(-1)W(\alpha)^{\pm 1}=\beta(-1)\times\begin{cases}
                  W(\alpha) & \text{if $[E:K]\equiv 1\pmod{4}$}\\
                  W(\alpha)^{-1} & \text{if $[E:K]\equiv -1\pmod{4}$}
                 \end{cases}
 \end{equation*}
if $K/F$ is cyclic of order $4$ with generating character $\beta$ such that 
 $\beta^2=\alpha=\Delta_{K/F}$.
 \item 
 \begin{equation*}
  \lambda_{1}^{G}=\lambda_{K/F}^{\pm 1}=\begin{cases}
                  \lambda_{K/F}=W(\alpha) & \text{if $[E:K]\equiv 1\pmod{4}$}\\
                  \lambda_{K/F}^{-1}=W(\alpha)^{-1} & \text{if $[E:K]\equiv -1\pmod{4}$}
                 \end{cases}
 \end{equation*}
 if $K/F$ is cyclic of order $2^n\ge 8$.
\end{enumerate}
  And if the $4$th roots of unity are in the $F$, we have the same formulas as above but with $1$ instead of $\pm 1$.
Moreover, when $p\ne 2$, a precise formula for $W(\alpha)$ will be obtained in Theorem \ref{Theorem 3.21}.
\item If $S$ is metacyclic but not cyclic and the $4$th roots of unity are in $F$, then 
\begin{enumerate}
 \item $\lambda_{1}^{G}=-1$ if $V\subset G$,
 \item $\lambda_{1}^{G}=1$ if $V\not\in G$.
\end{enumerate}
\item The {\bf Case 4} of Theorem \ref{Theorem 4.3}  will not occur in this case.

\end{enumerate}

\end{cor}
 
\begin{proof}
{\bf (1).}
In the case when $p\ne 2$ we know from Proposition \ref{Proposition 3.5} that the odd Hall-subgroup $H<G$ is actually a normal subgroup
with quotient $G/H\cong S$. So if $G=\rm{Gal}(E/F)$ and $K/F$ is the maximal $2$-extension inside $E$ then $\rm{Gal}(K/F)=G/H\cong S$.
And we obtain:
\begin{equation}\label{eqn 3.8}
 \lambda_{1}^{G}=(\lambda_{1}^{G/H})^{|H|}=\begin{cases}
                                            \lambda_{K/F} & \text{if $[E:K]=|H|\equiv 1\pmod{4}$}\\
                                            \lambda_{K/F}^{-1} & \text{if $[E:K]=|H|\equiv -1\pmod{4}$}.
                                           \end{cases}
\end{equation}
So it is enough to compute $\lambda_{K/F}$ for $\rm{Gal}(K/F)\cong S$, i.e., we can reduce the computation to the case where $G=S$.

We know that $\lambda_{K/F}=W(\rm{Ind}_{K/F}(1))=\prod_{\chi}W(\chi)$, where $\chi$ runs over all characters of the cyclic group 
$\rm{Gal}(K/F)$. 
If $[K:F]=2$ then $\rm{Ind}_{K/F}(1)=1+\alpha$, where $\alpha$ 
is a quadratic character of $F$ associated to $K$ by class field theory, hence $\alpha=\Delta_{K/F}$.
Thus $\lambda_{K/F}=W(\alpha)$.

If $[K:F]=4$ then $\rm{Ind}_{K/F}(1)=1+\beta+\beta^2+\beta^3$, where $\beta^2=\alpha=\Delta_{K/F}$ and 
$\beta^3=\beta^{-1}$, hence by the functional equation of local constant we have:
\begin{center}
 $W(\beta)W(\beta^{-1})=\beta(-1)$.
\end{center}
We then obtain:
\begin{center}
 $\lambda_{K/F}=W(\rm{Ind}_{K/F}(1))=W(\beta)W(\beta^2)W(\beta^3)=\beta(-1)\times W(\alpha)$.
\end{center}
If $S$ is cyclic of order $2^n\ge 8$, then by Theorem \ref{Theorem 4.1}, $c_{1}^{S}=1$. Again from equation (\ref{eqn 3.31}) 
we have $W(\Delta_{1}^{S})=W(\alpha)$ because $\rm{rk}_2(S)=1$,
where $\alpha$ is the uniquely determined quadratic character of $F$. Thus we obtain
$$\lambda_{K/F}=c_{1}^{S}\cdot W(\Delta_{1}^{S})=W(\alpha).$$

Finally by using the equation (\ref{eqn 3.8}) we obtain our desired results.

Now we denote $i=\sqrt{-1}$ and consider it in the algebraic closure of $F$. If $i\not\in F$ then $p\ne 2$ implies that 
$F(i)/F$ is unramified extension of degree $2$. Then we reach the case $i\in F$ which we have assumed.\\
Then first of all we know that $\lambda_{H}^{G}$ is always is a {\bf sign} because 
\begin{center}
 $(\lambda_{H}^{G})^2=\Delta_{H}^{G}(-1)=\Delta_{H}^{G}(i^2)=1$.
\end{center}
Then the formula (\ref{eqn 3.8}) turns into 
\begin{center}
 $\lambda_{1}^{G}=(\lambda_{1}^{G/H})^{|H|}=\lambda_{1}^{G/H}$,
\end{center}
where $G/H=\rm{Gal}(K/F)\cong S$. Therefore in {\bf Case 2} of Theorem \ref{Theorem 4.3} we have now same formulas as 
above but with $1$ instead of $\pm 1$.

{\bf (2).}
Moreover, when $p\ne 2$
we know that always $\lambda_{1}^{V}=-1$ if $i\in F$ (cf. Lemma \ref{Lemma 4.6}). Again if $V\subseteq S$, hence $V\subseteq G$, and
we have 
$$\lambda_{1}^{G}=\lambda_{1}^{V}\cdot(\lambda_{V}^{G})^4=\lambda_{1}^{V}.$$
Therefore, when $S$ is metacyclic but not cyclic we can simply say:\\
$\lambda_{1}^{G}=\lambda_{1}^{V}=-1$, if $V\subset G$,\\
$\lambda_{1}^{G}=1$, if $V\not\subset G$.

{\bf (3).}
If the base field $F$ is $p$-adic for $p\ne 2$ then as a Galois group $S$ corresponds to
a tamely ramified extension (because the degree $2^n$ is prime to $p$), and therefore $S$ must
be metacyclic. Therefore the \textbf{Case 4} of Theorem \ref{Theorem 4.3} can never occur if $p\ne 2$.

\end{proof}


\begin{rem}
If $S$ is cyclic of order $2^n\ge 8$, then we have two formulas:\\
$\lambda_{1}^{G}=W(\alpha)$ as obtained in Theorem \ref{Theorem 4.3}(2), and $\lambda_{1}^{G}=W(\alpha)^{\pm 1}$ in 
Corollary \ref{Lemma 3.10}. So we observe that for 
$|S|=2^n\ge 8$ and $|H|\equiv -1\pmod{4}$ the value of $W(\alpha)$ must be a sign for $p\ne 2$.

In {\bf Case 3} of Theorem \ref{Theorem 4.3} we notice that $\Delta_{1}^{G}\equiv 1$, hence $\lambda_{1}^{G}=c_{1}^{G}$.
We know also that this Deligne's constant $c_{1}^{G}$ takes values $\pm 1$ (cf. Proposition \ref{Proposition 3.1.1.1}(v)). 
Moreover, we also notice that the Deligne's constant of a representation is independent of the choice of the additive character.
Therefore in Case 3 of Theorem \ref{Theorem 4.3}, 
$\lambda_{1}^{G}=c_{1}^{G}\in\{\pm 1\}$ will {\bf not} depend on the choice of the additive character. Since in Case 3 the 
computation of $\lambda_{1}^{G}$ does not depend on the choice of the additive characters, hence we call this case as 
{\bf invariant case}.

Furthermore, Bruno Kahn in his second paper (cf. \cite{BK}) deals with $s_2(r_G)$, where $r_G$ is a regular representation of $G$ in the
invariant case. For metacyclic $S$ of order $\ge 4$,
we have the presentation
\begin{center}
 $S\cong G(n,m,r,l)=<a,b: a^{2^n}=1, b^{2^m}=a^{2^r}, bab^{-1}=a^l>$\\
 with $n, m\ge 1$, $0\le r\le n$, $l$ an integer $\equiv 1\pmod{2^{n-r}}$, $l^{2^m}\equiv l\pmod{2^n}$.
\end{center}
When $S$ is \textbf{metacyclic but not cyclic} with $n\ge 2$, then $s_2(r_G)=0$ if and only if $m=1$ and 
$l\equiv-1\pmod{4}$ (cf. \cite{BK}, on p. 575 of the second paper). In this case our $\lambda_{1}^{G}=c_{1}^{G}=1$.

\end{rem}

\begin{cor}\label{Corollary 3.13}
 Let $G$ be a finite abelian local Galois group of $F/\bbQ_p$, where $p\ne 2$. Let $S$ be a Sylow 2-subgroup of $G$.
 \begin{enumerate}
  \item If $\rm{rk}_2(S)=0$, we have $\lambda_{1}^{G}=1$.
  \item If $\rm{rk}_2(S)=1$, then 
   \begin{equation}
   \lambda_{1}^{G}=\begin{cases}
                    W(\alpha) & \text{if $|S|=2^n\ge 8$}\\
                    c_{1}^{G}\cdot W(\alpha) & \text{if $|S|\le 4$},
                   \end{cases}
\end{equation}
where $\alpha$ is a uniquely determined quadratic 
  character of $G$.
  \item If $\rm{rk}_2(S)=2$, we have 
  \begin{equation}
   \lambda_{1}^{G}=\begin{cases}
                    -1 & \text{if $-1\in F^\times$ is a square element}\\
                    1 & \text{if $-1\in F^\times$ is not a square element}.
                   \end{cases}
  \end{equation}
 \end{enumerate}
\end{cor}
\begin{proof}
 This proof is straightforward from Theorem \ref{Theorem 4.3} and Corollary \ref{Lemma 3.10}.
 Here $S$ is abelian and normal because $G$ is abelian. When $\rm{rk}_2(S)=0$, $G$ is of odd order, hence 
 $\lambda_{1}^{G}=1$. When $\rm{rk}_2(S)=1$, $S$ is a cyclic group because $S\cong\bbZ_{2^n}$ for some $n\ge 1$.
 Then we are in the Case 2 of Theorem \ref{Theorem 4.3}. From the Case 4 of Corollary \ref{Lemma 3.10}, 
 we can say that the case $\rm{rk}_2(S)\ge 3$
 will not occur here because $p\ne 2$ and $G$ tame Galois group\footnote{Here tame Galois group means a Galois group of a
 tamely ramified extension.}.
 
 So we are left to check the case $\rm{rk}_2(S)=2$. In this case $S$ is metacyclic and contains Klein's 4-group, i.e., 
 $V\subseteq S\subseteq G$.
 Then we have from the properties of $\lambda$-functions and Lemma \ref{Lemma 4.6} we obtain
 \begin{equation}
  \lambda_{1}^{G}=\lambda_{1}^{V}\cdot(\lambda_{V}^{G})^4=\lambda_{1}^{V}=
  \begin{cases}
                    -1 & \text{if $-1\in F^\times$ is a square element}\\
                    1 & \text{if $-1\in F^\times$ is not a square element}.
                   \end{cases}
 \end{equation}

\end{proof}

\begin{rem}
 In Corollary \ref{Corollary 3.13} we observe that except the case $\rm{rk}_2(S)=1$, the computation of $\lambda_{1}^{G}$
 is explicit. Now let $G=\rm{Gal}(L/F)$, where $L/F$ is a finite abelian Galois extension, and $K/F$ be a subextension in $L/F$
 for which $\rm{Gal}(L/K)=S$. Since $S$ is normal, $K/F$ is Galois extension of odd degree. Then when $\rm{rk}_2(S)=1$ we have 
 \begin{equation}
  \lambda_{1}^{G}=\lambda_{1}^{S}\cdot(\lambda_{S}^{G})^{|S|}.
 \end{equation}
Again $S$ is normal subgroup of $G$, hence $G/S\cong \rm{Gal}(K/F)$. Hence 
$$\lambda_{S}^{G}=\lambda_{1}^{G/S}=\lambda_{1}^{\rm{Gal}(K/F)}=\lambda_{K/F}=1.$$
Moreover, $\lambda_{1}^{S}=\lambda_{L/K}$. Then when $\rm{rk}_2(S)=1$ we have 
$$\lambda_{1}^{G}=\lambda_{1}^{S}=\lambda_{L/K},$$
where $[L:K]=2^n$ ($n\ge 1$).
\end{rem}

From Theorem \ref{Theorem 4.3} and Corollaries \ref{Lemma 3.10}, \ref{Corollary 3.13} 
for the \textbf{case $p\ne 2$} we realize that the explicit computation
of $W(\alpha)$ gives a complete computation of $\lambda_{1}^{G}$ in the case $p\ne 2$. 
In the following section we give explicit computation of $\lambda_{K/F}$, when $K/F$ is an even degree Galois extension.

\section{\textbf{Explicit computation of $\lambda_{K/F}$, where $K/F$ is an even degree Galois extension }}

Let $K/F$ be a quadratic extension of the field $F/\bbQ_p$. Let $G=\mathrm{Gal}(K/F)$
be the Galois group of the extension $K/F$. Let $t$ be the \textbf{ramification break or jump} (cf. \cite{JPS}) of the Galois group
$G$ (or of the extension $K/F$). Then it can be
proved that the conductor of 
$\omega_{K/F}$ (a quadratic character of $F^\times$ associated to $K$ by class field theory) is $t+1$ (cf. \cite{SAB}, Lemma 3.1). 
When $K/F$ is unramified we have $t=-1$, therefore the conductor of a quadratic character $\omega_{K/F}$ of $F^\times$
is zero, i.e., 
$\omega_{K/F}$ is unramified.
And when $K/F$ is tamely ramified we have $t=0$, then $a(\omega_{K/F})=1$.
In the wildly ramified case (which occurs if $p=2$) it can be proved that $a(\omega_{K/F})=t+1$ is,
{\bf up to the exceptional case $t=2\cdot e_{F/\bbQ_2}$}, always an \textbf{even number} which 
can be seen by the following filtration of $F^\times$. 

Let $K/F$ be a quadratic wild ramified extension, hence $p=2$.
In $F$, we have the sequence 
 $F^\times\supset U_F\supset U_{F}^{1}\supset U_{F}^{2}\supset \cdots$ of higher principal unit subgroups. Since $\omega_{K/F}$ is a 
 quadratic character, it will be trivial on $F^{\times^2}$ therefore we need to consider the induced sequence
 \begin{equation*}
  F^\times\supset U_{F}F^{\times^2}\supset U_{F}^{1}F^{\times^2}\supset U_{F}^{2}F^{\times^2}\supset\cdots\supset F^{\times^2}.
  \tag{\bf{S2}}
 \end{equation*}
In general, for any prime $p$, for $F/\mathbb{Q}_p$ we have the following series 
 \begin{equation*}
  F^\times\supset U_{F}F^{\times^p}\supset U_{F}^{1}F^{\times^p}\supset U_{F}^{2}F^{\times^p}\supset\cdots\supset F^{\times^p}.
  \tag{\bf{Sp}}
 \end{equation*}
 Now we use Corollary 5.8 of \cite{FV} on p. 16,
 for the following: Let $e=\nu_F(p)=e_{F/\mathbb{Q}_p}$ be the absolute ramification exponent of $F$.
 Then
 \begin{enumerate}
  \item[(i)] If $i>\frac{p e}{p-1}$ then $U_{F}^{i}\subset F^{\times^p}$, hence $U_{F}^{i}F^{\times^p}=F^{\times^p}$.
  \item[(ii)] If $i<\frac{pe}{p-1}$ and $i$ is prime to $p$, then in
  \begin{equation}
   1\to U_{F}^{i}\cap F^{\times^p}/U_{F}^{i+1}\cap F^{\times^p}\xrightarrow{(1)} U_{F}^{i}/U_{F}^{i+1}\xrightarrow{(2)}
   U_{F}^{i}\cdot F^{\times^p}/U_{F}^{i+1}\cdot F^{\times^p}\to 1,
  \end{equation}
the arrow (1) is trivial and (2) is an isomorphism.
\item[(iii)] If $i<\frac{pe}{p-1}$ and $p$ divides $i$, then arrow (1) is an isomorphism and (2) is trivial
 \end{enumerate}
Therefore the jumps in \textbf{(Sp)} occur at $U_{F}^{i}F^{\times^p}$, where $i$ is prime to $p$ and $i<\frac{pe}{p-1}$.\\
(If $\frac{pe}{p-1}$ is an integer, then $i=\frac{pe}{p-1}$ is an {\bf exceptional case.})

We now take $p=2$, hence $\frac{pe}{p-1}=2e$. Then the sequence \textbf{(Sp)} turns into \textbf{(S2)}
 for all odd numbers $t<2 e$ or for $t=2e$ (the exceptional case).
This means that in the wild ramified case 
the conductor $a(\omega_{K/F})=t+1$ will always be an \textbf{even number} (except when $t=2e$). 
\vspace{.3cm}

From the following lemma we can see that $\lambda$-function can change by a sign if we change the additive character. 
\begin{lem}\label{Lemma 3.2}
The $\lambda$-function can change by sign if we change the additive character.
\end{lem}
\begin{proof}
Let $K/F$ be a finite separable extension of the field $F$ and $\psi$ be a nontrivial additive character of $F$. 
We know that the local constant $W(\rho,\psi)$ is well defined for all pairs consisting of a virtual representation $\rho$ 
of the Weil group
 $W_F$ and a nontrivial additive character $\psi$ of $F$. If we change the additive character $\psi$ to $b\psi$, where $b\in F^\times$
 is a unique element for which $b\psi(x):=\psi(bx)$ for all $x\in F$, then from
 equation (\ref{eqn 3.1.1}), we have 
 \begin{equation}\label{eqn 3.13}
  W(\rho,b\psi)=\mathrm{det}_\rho(b)\cdot W(\rho,\psi).
 \end{equation}
In the definition of $\lambda$-function $\rho=\mathrm{Ind}_{K/F} 1$, therefore by using equation (\ref{eqn 3.13}), we have
\begin{equation}\label{eqn 3.121}
 \lambda_{K/F}(b\psi)=W(\mathrm{Ind}_{K/F} 1, b\psi)=\Delta_{K/F}(b)W(\mathrm{Ind}_{K/F} 1,\psi)=\Delta_{K/F}(b)\lambda_{K/F}(\psi),
\end{equation}
where $\Delta_{K/F}=\mathrm{det}(\mathrm{Ind}_{K/F} (1))$ is a quadratic character (a sign function),
i.e., $\Delta_{K/F}(b)\in\{\pm 1\}$. 
 
\end{proof}

In the following lemma we compute an explicit formula for $\lambda_{K/F}(\psi_F)$, where $K/F$ is a quadratic unramified extension of 
$F$. In general for {\bf any quadratic extension} $K/F$,
we can write $\mathrm{Ind}_{K/F}1=1_F\oplus\omega_{K/F}$, where $\omega_{K/F}$ is a quadratic character of 
$F^\times$ associated to $K$ by class field theory and $1_F$ is the trivial character of $F^\times$.
Now by the definition of $\lambda$-function we have:
\begin{equation}\label{eqn 3.333}
 \lambda_{K/F}=W(\mathrm{Ind}_{K/F}1)=W(\omega_{K/F}).
\end{equation}
So, $\lambda_{K/F}$ is the local constant of the quadratic character $\omega_{K/F}$ corresponding to $K/F$. 
\begin{lem}\label{Lemma 3.13}
 Let $K$ be the quadratic unramified extension of $F/\bbQ_p$ and let $\psi_F$ be the canonical additive character of $F$ with 
 conductor $n(\psi_F)$. Then 
 \begin{equation}
  \lambda_{K/F}(\psi_F)=(-1)^{n(\psi_F)}.
 \end{equation}
\end{lem}
\begin{proof}
When $K/F$ is quadratic unramified extension, 
 it is easy to see that in equation (\ref{eqn 3.333})
 $\omega_{K/F}$ is an unramified character because here the ramification break $t$ is $-1$.
 Then from equation (\ref{eqn 2.3.5}) have:
 \begin{equation*}
  W(\omega_{K/F})=\omega_{K/F}(c).
 \end{equation*}
Here $\nu_F(c)=n(\psi_F)$. Therefore from equation (\ref{eqn 3.333}) we obtain:
\begin{equation}\label{eqn 3.14}
 \lambda_{K/F}=\omega_{K/F}(\pi_F)^{n(\psi_F)}.
\end{equation}
We also know that  $\pi_F\notin N_{K/F}(K^\times)$, and hence 
$\omega_{K/F}(\pi_F)=-1$. Therefore from equation (\ref{eqn 3.14}),
we have
\begin{equation}\label{eqn 0.4}
 \lambda_{K/F}=(-1)^{n(\psi_F)}.
\end{equation}

\end{proof}

For giving the generalized formula for $\lambda_{K/F}$, where $K/F$ is an even degree unramified extension, we need the following
lemma, and here we give an alternative proof by using Lemma \ref{Lemma 2.21}.

 \begin{lem}[\cite{AW}, p. 142, Corollary 3]\label{Lemma 3.4}
 Let $K/F$ be a finite extension and let $\mathcal{D}_{K/F}$ be the different of this extension $K/F$. Let $\psi$ be an additive
 character of $F$. Then
 \begin{equation}
  n(\psi\circ\mathrm{Tr}_{K/F})=e_{K/F}\cdot n(\psi)+ \nu_K(\mathcal{D}_{K/F}).
 \end{equation}
\end{lem}
\begin{proof}
Let the conductor of the character $\psi\circ\mathrm{Tr}_{K/F}$ be $m$. This means from the definition of conductor of additive character
 we have 
 \begin{center}
   $\psi\circ\mathrm{Tr}_{K/F}|_{P_{K}^{-m}}=1$ but $\psi\circ\mathrm{Tr}_{K/F}|_{P_{K}^{-m-1}}\neq 1$,\\
   i.e.,  $\psi(\mathrm{Tr}_{K/F}(P_{K}^{-m}))=1$ but $\psi(\mathrm{Tr}_{K/F}(P_{K}^{-m-1}))\neq 1$.
 \end{center}
This implies 
 \begin{center}
  $\mathrm{Tr}_{K/F}(P_{K}^{-m})\subseteq P_{F}^{-n(\psi)}$,
 \end{center}
 since $n(\psi)$ is the conductor of $\psi$.
Then by Lemma \ref{Lemma 2.21} we have 
\begin{equation}\label{eqn 3.20}
 P_{K}^{-m}\cdot\mathcal{D}_{K/F}\subseteq P_{F}^{-n(\psi)}\Leftrightarrow
 P_{K}^{-m+d_{K/F}}O_K\subseteq P_{F}^{-n(\psi)},
\end{equation}
since $\mathcal{D}_{K/F}=\pi_{K}^{d_{K/F}}O_K$. From the definition of ramification index we know that 
\begin{equation*}
 \pi_{K}^{e_{K/F}}O_K=\pi_{F}O_K.
\end{equation*}
 Therefore from the equation (\ref{eqn 3.20}) we obtain:
 \begin{equation*}
  P_{K}^{-m+d_{K/F}}O_K\subseteq P_{K}^{-n(\psi)e_{K/F}}O_K.
 \end{equation*}
This implies
\begin{equation}
 m\leq n(\psi)\cdot e_{K/F}+d_{K/F}=n(\psi)\cdot e_{K/F}+\nu_K(\mathcal{D}_{K/F}),
\end{equation}
since $d_{K/F}=\nu_{K}(\mathcal{D}_{K/F})$.

Now we have to prove that the equality $m=n(\psi)\cdot e_{K/F}+\nu_K(\mathcal{D}_{K/F})$. \\
Let $m=n(\psi)\cdot e_{K/F}+\nu_{K}(\mathcal{D}_{K/F})-r$, where $r\geq 0$. Then we have 
\begin{center}
 $\mathrm{Tr}_{K/F}(P_{K}^{-n(\psi)\cdot e_{K/F}-d_{K/F}+d_{K/F}+r})\subseteq P_{K}^{-n(\psi)\cdot e_{K/F}}$.
\end{center}
This implies $r\leq 0$. Therefore $r$ must be $r=0$, because by assumption $r\geq 0$. This proves that 
\begin{center}
 $m=n(\psi\circ\mathrm{Tr}_{K/F})=e_{K/F}\cdot n(\psi)+ \nu_K(\mathcal{D}_{K/F})$.
\end{center}

\end{proof}

\begin{rem}
 If $K/F$ is unramified, then $e_{K/F}=1$ and $\mathcal{D}_{K/F}=O_K$, hence $\nu_K(\mathcal{D}_{K/F})=\nu_K(O_K)=0$, therefore from the
 above Lemma \ref{Lemma 3.4} we have $n(\psi\circ\mathrm{Tr}_{K/F})=n(\psi)$. 
 Moreover, if $\psi_F=\psi_{\mathbb{Q}_p}\circ\mathrm{Tr}_{F/\mathbb{Q}_p}$, is the canonical
 additive character of $F$,  then $n(\psi_{\mathbb{Q}_p})=0$ 
 and therefore 
 $$n(\psi_F)=\nu_F(\mathcal{D}_{F/\mathbb{Q}_p})$$ 
 is the exponent of the absolute different.
\end{rem}

\begin{thm}\label{Theorem 3.6}
 Let $K/F$ be a finite unramified extension with even degree and let $\psi_F$ be the canonical additive character of $F$ 
 with conductor $n(\psi_F)$. Then
 \begin{equation}
  \lambda_{K/F}=(-1)^{n(\psi_F)}.
\end{equation}
\end{thm}
\begin{proof}
When $K/F$ is a quadratic unramified extension, by Lemma \ref{Lemma 3.13}, we have $\lambda_{K/F}=(-1)^{n(\psi_F)}$.
We also know that if $K/F$ is unramified of even degree then 
 we have precisely one subextension $K'/F$ in $K/F$ such that $[K:K']=2$. Then
 \begin{center}
  $\lambda_{K/F}=\lambda_{K/K'}\cdot(\lambda_{K'/F})^2=\lambda_{K/K'}=(-1)^{n(\psi_{K'})}=(-1)^{n(\psi_F)}$,
 \end{center}
because in the unramified case $\lambda$-function is always a sign (cf. Lemma \ref{Corollary 3.3.7}), and 
from Lemma \ref{Lemma 3.4}, $n(\psi_{K'})=n(\psi_F)$.
 
This completes the proof.

\end{proof}

In the following corollary, we show that the above Theorem \ref{Theorem 3.6} is true for any nontrivial arbitrary additive character. 
\begin{cor}\label{Corollary 3.7}
 Let $K/F$ be a finite unramified extension of even degree and let $\psi$ be any nontrivial additive character of $F$ 
 with conductor $n(\psi)$. Then
 \begin{equation}
  \lambda_{K/F}(\psi)=(-1)^{n(\psi)}.
\end{equation}
\end{cor}
\begin{proof}
 We know that any nontrivial additive character
 $\psi$ is of the form, for some unique $b\in F^\times$, $\psi(x):=b\psi_F(x)$, for all $x\in F$. 
 By the definition of conductor of additive character of $F$, we obtain:
 \begin{center}
  $n(\psi)=n(b\psi_F)=\nu_F(b)+n(\psi_F)$.
 \end{center}
 Now let $G=\rm{Gal}(K/F)$ be the Galois group of the extension $K/F$. Since $K/F$ is unramified, then $G$ is {\bf cyclic}.
 Let $S$ be a Sylow $2$-subgroup of $G$. Here $S$
is nontrivial cyclic because the degree of $K/F$ is even and $G$ is cyclic. 
Then from Lemma \ref{Lemma 3.41} we have $\Delta_{1}^{G}=\Delta_{K/F}\not\equiv 1$.
Therefore $\Delta_{K/F}(b)=(-1)^{\nu_F(b)}$ is the uniquely determined unramified quadratic character of $F^\times$.
Now from equation (\ref{eqn 3.121}) we have:
\begin{align*}
 \lambda_{K/F}(\psi)
 &=\lambda_{K/F}(b\cdot\psi_F)\\
 &=\Delta_{K/F}(b)\lambda_{K/F}(\psi_F)\\
 &=(-1)^{\nu_F(b)}\times(-1)^{n(\psi_F)},\quad \text{from Theorem $\ref{Theorem 3.6}$}\\
 &=(-1)^{\nu_F(b)+n(\psi_F)}\\
 &=(-1)^{n(\psi)}.
\end{align*}
Therefore
when $K/F$ is an unramified extension of even degree, we have
\begin{equation}
 \lambda_{K/F}(\psi)=(-1)^{n(\psi)},
\end{equation}
where $\psi$ is any nontrivial additive character of $F$.
 
\end{proof}

In the following theorem we give an
explicit formula of $\lambda_{K/F}$, when $K/F$ is an even degree Galois extension
with odd ramification index.

\begin{thm}\label{Theorem 3.8}
 Let $K$ be an even degree Galois extension of a non-archimedean local field $F$ of odd ramification index. 
 Let $\psi$ be a nontrivial additive character of $F$. Then 
 \begin{equation}
  \lambda_{K/F}(\psi)=(-1)^{n(\psi)}.
 \end{equation}

\end{thm}
\begin{proof}
In general,  
any extension $K/F$ of local fields has a uniquely determined
maximal subextension  $F'/F$ in $K/F$ which is unramified. If the degree of $K/F$ is even, then certainly we have 
$e_{K/F}=[K:F']$ because $e_{K/F}=e_{F'/F}\cdot e_{K/F'}=e_{K/F'}$ and $K/F'$ is a totally ramified extension. 
By the given condition, here $K/F$ is an even degree Galois extension with odd ramification index $e_{K/F}$, 
hence $K/F'$ is an odd degree Galois extension.
Now from the properties of $\lambda$-function and Theorem \ref{General Theorem for odd case} we have 
\begin{center}
 $\lambda_{K/F}=\lambda_{K/F'}\cdot(\lambda_{F'/F})^{e_{K/F}}=(-1)^{e_{K/F}\cdot n(\psi_F)}=(-1)^{n(\psi_F)}$,
\end{center}
because $K/F'$ is an odd degree Galois extension and $F'/F$ is an unramified extension.

\end{proof}

 \subsection{\textbf{Computation of $\lambda_{K/F}$, where $K/F$ is a tamely ramified quadratic extension}}

The existence of a tamely ramified quadratic character (which is not unramified)
 of a local field $F$ implies $p\ne2$ for the residue characteristic. But then by Theorem \ref{Theorem 2.9}
 \begin{center}
  $F^\times/{F^\times}^2\cong V$
 \end{center}
is isomorphic to Klein's $4$-group. So we have only $3$ nontrivial quadratic characters in that case, corresponding to $3$
quadratic extensions $K/F$. One is unramified and other two are ramified. The unramified case 
is already settled (cf. Lemma \ref{Lemma 3.13}). Now we have to address the quadratic ramified characters. 
These two ramified quadratic characters determine two different 
quadratic ramified extensions of $F$.\\
In the ramified case we have $a(\chi)=1$ because it is tame,
and we take $\psi$ of conductor $-1$. Then we have $a(\chi)+n(\psi)=0$ and therefore in the definition of $W(\chi,\psi)$ 
(cf. equation (\ref{eqn 2.2})) we can take $c=1$. So we obtain:
\begin{equation}\label{eqn 3.33}
 W(\chi,\psi)=q_{F}^{-\frac{1}{2}}\sum_{x\in U_F/U_{F}^{1}}\chi^{-1}(x)\psi(x)=
 q_{F}^{-\frac{1}{2}}\sum_{x\in k_{F}^{\times}}\overline{\chi}^{-1}(x)\overline{\psi}(x),
\end{equation}
where $\overline{\chi}$ is the quadratic character of the residue field $k_{F}^{\times}$, and 
$\overline{{\psi}}$ is an additive character of $k_F$.
When $n(\psi)=-1$, we observe that both {\bf the ramified characters $\chi$ give the same $\overline{\chi}$, hence the same 
$W(\chi,\psi)$}, 
because one is different from other by a quadratic unramified character twist.
To compute an explicit formula for $\lambda_{K/F}(\psi_{-1})$, where $K/F$ is a tamely ramified quadratic extension and 
$\psi_{-1}$ is an additive character of $F$ with conductor $-1$, we need to use classical quadratic Gauss sums.


Now let $F$ be a non-archimedean local field. Let $\psi_{-1}$ be an additive character of $F$ of conductor $-1$, i.e.,
 $\psi_{-1}:F/P_F\to\bbC^\times$. Now restrict 
$\psi_{-1}$ to $O_F$, it will be one of the characters $a\cdot\psi_{q_F}$, for some $a\in k_{q_F}^{\times}$ and usually it will not be 
$\psi_{q_F}$ itself.
Therefore choosing $\psi_{-1}$ is very important and we have to choose $\psi_{-1}$ such a way that its restriction to $O_F$
is exactly $\psi_{q_F}$. Then we will be able to use the quadratic classical Gauss sum in our $\lambda$-function computation.
We also know that there exists an element $c\in F^\times$ such that 
\begin{equation}\label{eqn 3.34}
 \psi_{-1}=c\cdot\psi_F
\end{equation}
induces the canonical character $\psi_{q_F}$ on the residue field $k_F$.

Now question is: {\bf Finding proper $c\in F^\times$ for which $\psi_{-1}|_{O_F}=c\cdot\psi_F|_{O_F}=\psi_{q_F}$}, i.e., 
the canonical character of the residue field $k_F$.

From the definition of conductor of the additive character $\psi_{-1}$ of $F$, we obtain from the construction (\ref{eqn 3.34})
\begin{equation}\label{eqn 3.35}
 -1=\nu_F(c)+n(\psi_F)=\nu_F(c)+d_{F/\bbQ_p}.
\end{equation}
In the next two lemmas we choose the proper $c$ for our requirement.

\begin{lem}\label{Lemma 3.20}
 Let $F/\bbQ_p$ be a local field and let $\psi_{-1}$ be an additive character of $F$ of conductor $-1$. 
 Let $\psi_F$ be the canonical
 character of $F$. Let $c\in F^\times$ be any element such that $-1=\nu_F(c)+d_{F/\bbQ_p}$, and 
 \begin{equation}\label{eqn 3.36}
  \rm{Tr}_{F/F_0}(c)=\frac{1}{p},
 \end{equation}
where $F_0/\bbQ_p$ is the maximal unramified subextension in $F/\bbQ_p$. Then 
the restriction of $\psi_{-1}=c\cdot\psi_F$ to $O_F$ is the canonical character $\psi_{q_F}$ of the residue field $k_F$ of $F$.
\end{lem}
\begin{proof}
 Since $F_0/\bbQ_p$ is the maximal unramified subextension in $F/\bbQ_p$, we have $\pi_{F_0}=p$, and the residue fields of 
 $F$ and $F_0$ are 
 isomorphic, i.e., $k_{F_0}\cong k_{F}$, because $F/F_0$ is totally ramified extension.
 Then every element of $O_F/P_F$ can be considered as an element of $O_{F_0}/P_{F_0}$.
 Moreover, since $F_0/\bbQ_p$ is the maximal unramified extension, then from Proposition 2 of \cite{AW} on p. 140,
 for $x\in O_{F_0}$ we have 
 $$\rho_{p}(\rm{Tr}_{F_0/\bbQ_p}(x))=\rm{Tr}_{k_{F_0}/k_{\bbQ_p}}(\rho_0(x)),$$
 where $\rho_0,\,\rho_p$ are the canonical homomorphisms of $O_{F_0}$ onto $k_{F_0}$, and of $O_{\bbQ_p}$ onto $k_{\bbQ_p}$,
 respectively. Then 
 for $x\in k_{F_0}$ we can write
 \begin{equation}\label{eqn 3.50}
  \rm{Tr}_{F_0/\bbQ_p}(x)=\rm{Tr}_{k_{F_0}/k_{\bbQ_p}}(x).
 \end{equation}
Furthermore, since $F/F_0$ is totally ramified, we have $k_{F}=k_{F_0}$, then the trace map 
for the tower of the residue fields $k_{F}/k_{F_0}/k_{\bbQ_p}$ is:
\begin{equation}\label{eqn 3.51}
 \rm{Tr}_{k_F/k_{\bbQ_p}}(x)=\rm{Tr}_{k_{F_0}/k_{\bbQ_p}}\circ\rm{Tr}_{k_F/k_{F_0}}(x)=\rm{Tr}_{k_{F_0}/k_{\bbQ_p}}(x),
\end{equation}
for all $x\in k_F$. Then from the equations (\ref{eqn 3.50}) and (\ref{eqn 3.51}) we obtain 
\begin{equation}\label{eqn 3.52}
 \rm{Tr}_{F_0/\bbQ_p}(x)=\rm{Tr}_{k_F/k_{\bbQ_p}}(x)
\end{equation}
 for all $x\in k_F$.

Since the conductor of $\psi_{-1}$ is $-1$, for $x\in O_F/P_F(=O_{F_0}/P_{F_0}$ because $F/F_0$ is totally ramified)  we have  
\begin{align*}
  \psi_{-1}(x)
  &=c\cdot\psi_F(x)\\
  &=\psi_{F}(cx)\\
  &=\psi_{\bbQ_p}(\rm{Tr}_{F/\bbQ_p}(cx))\\
  &=\psi_{\bbQ_p}(\rm{Tr}_{F_0/\bbQ_p}\circ\rm{Tr}_{F/F_0}(cx))\\
  &=\psi_{\bbQ_p}(\rm{Tr}_{F_0/\bbQ_p}(x\cdot\rm{Tr}_{F/F_0}(c)))\\
  &=\psi_{\bbQ_p}(\rm{Tr}_{F_0/\bbQ_p}(\frac{1}{p}x)), \quad\text{since $x\in O_F/P_F=O_{F_0}/P_{F_0}$ and $\rm{Tr}_{F/F_0}(c)=\frac{1}{p}$}\\
  &=\psi_{\bbQ_p}(\frac{1}{p}\rm{Tr}_{F_0/\bbQ_p}(x)), \quad\text{because $\frac{1}{p}\in\bbQ_p$}\\
 &=e^{\frac{2\pi i \rm{Tr}_{F_0/\bbQ_p}(x)}{p}},\quad\text{because $\psi_{\bbQ_p}(x)=e^{2\pi i x}$}\\
&=e^{\frac{2\pi i\rm{Tr}_{k_F/k_{\bbQ_p}}(x)}{p}},\quad \text{using equation $(\ref{eqn 3.52})$}\\
&=\psi_{q_F}(x).
 \end{align*}
This competes the lemma.

\end{proof}

The next step is to produce good elements $c$ more explicitly. By using Lemma \ref{Lemma 3.20}, 
in the next lemma we see more general choices of $c$.

\begin{lem}\label{Lemma 3.21}
 Let $F/\bbQ_p$ be a tamely ramified local field and let $\psi_{-1}$ be an additive character of $F$ 
 of conductor $-1$. Let $\psi_F$ be the canonical
 character of $F$. Let $F_0/\bbQ_p$ be the maximal unramified subextension in $F/\bbQ_p$.
 Let $c\in F^\times$ be any element such that $-1=\nu_F(c)+d_{F/\bbQ_p}$, then 
 \begin{center}
  $c'=\frac{c}{\rm{Tr}_{F/F_0}(pc)}$, 
 \end{center}
 fulfills conditions (\ref{eqn 3.35}), (\ref{eqn 3.36}), and hence $\psi_{-1}|_{O_F}=c'\cdot\psi_F|_{O_F}=\psi_{q_F}$.
\end{lem}
\begin{proof}
 By the given condition 
 we have $\nu_{F}(c)=-1-d_{F/\bbQ_p}=-1-(e_{F/\bbQ_p}-1)=-e_{F/\bbQ_p}$. Then we can write 
 $c=\pi_{F}^{-e_{F/\bbQ_p}}u(c)=p^{-1}u(c)$ for some $u(c)\in U_F$ because $F/\bbQ_p$ is tamely ramified, hence $p=\pi_{F}^{e_{F/\bbQ_p}}$. 
Then we can write 
 $$\rm{Tr}_{F/F_0}(pc)=p\cdot\rm{Tr}_{F/F_0}(c)=p\cdot p^{-1}u_0(c)=u_0(c)\in U_{F_0}\subset U_{F},$$ 
 where $u_0(c)=\rm{Tr}_{F/F_0}(u(c))$,
 hence $\nu_{F}(\rm{Tr}_{F/F_0}(pc))=0$. Then the valuation of $c'$ is:
 \begin{center}
  $\nu_F(c')=\nu_F(\frac{c}{\rm{Tr}_{F/F_0}(pc)})=\nu_F(c)-\nu_F(\rm{Tr}_{F/F_{0}}(pc))$\\
  $=\nu_F(c)-0=\nu_{F}(c)=-1-d_{F/\bbQ_p}$.
 \end{center}
 Since $\rm{Tr}_{F/F_0}(pc)=u_0(c)\in U_{F_0}$, we have 
 \begin{center}
  $\rm{Tr}_{F/F_0}(c')=\rm{Tr}_{F/F_0}(\frac{c}{\rm{Tr}_{F/F_0}(pc)})
  =\frac{1}{\rm{Tr}_{F/F_0}(pc)}\cdot\rm{Tr}_{F/F_0}(c)=\frac{1}{p\cdot\rm{Tr}_{F/F_0}(c)}\cdot\rm{Tr}_{F/F_0}(c)=\frac{1}{p}$.
 \end{center}

Thus we observe that here $c'\in F^\times$ satisfies equations (\ref{eqn 3.35}) and (\ref{eqn 3.36}). 
Therefore from Lemma \ref{Lemma 3.20} we can see that $\psi_{-1}|_{O_F}=c'\cdot\psi_{F}|_{O_F}$ is the canonical additive 
character of $k_F$.
\end{proof}

By Lemmas \ref{Lemma 3.20} and \ref{Lemma 3.21} we get many good (in the sense that $\psi_{-1}|_{O_F}=c\cdot\psi_{F}|_{O_F}=\psi_{q_F}$)
elements $c$ which we will use in our next theorem to calculate $\lambda_{K/F}$, where $K/F$ is a tamely ramified quadratic extension.

\begin{thm}\label{Theorem 3.21}
 Let $K$ be a tamely ramified quadratic extension of $F/\bbQ_p$ with $q_F=p^s$. Let $\psi_F$ be the canonical additive character of $F$.
 Let $c\in F^\times$ with $-1=\nu_F(c)+d_{F/\bbQ_p}$, and $c'=\frac{c}{\rm{Tr}_{F/F_0}(pc)}$, where $F_0/\bbQ_p$ is the maximal unramified
 extension in $F/\bbQ_p$. Let $\psi_{-1}$ be an additive character of $F$ with conductor $-1$, of the form $\psi_{-1}=c'\cdot\psi_F$.
 Then 
 \begin{equation*}
  \lambda_{K/F}(\psi_F)=\Delta_{K/F}(c')\cdot\lambda_{K/F}(\psi_{-1}),
 \end{equation*}
where 
 \begin{equation*}
 \lambda_{K/F}(\psi_{-1})=\begin{cases}
                                               (-1)^{s-1} & \text{if $p\equiv 1\pmod{4}$}\\
                                                  (-1)^{s-1}i^{s} & \text{if $p\equiv 3\pmod{4}$}.
                                            \end{cases}
\end{equation*}
If we take $c=\pi_{F}^{-1-d_{F/\bbQ_p}}$, where $\pi_F$ is a norm for $K/F$, then 
\begin{equation}
 \Delta_{K/F}(c')=\begin{cases}
                   1 & \text{if $\overline{\rm{Tr}_{F/F_0}(pc)}\in k_{F_0}^{\times}=k_{F}^{\times}$ is a square},\\
                   -1 & \text{if $\overline{\rm{Tr}_{F/F_0}(pc)}\in k_{F_0}^{\times}=k_{F}^{\times}$ is not a square}.
                  \end{cases}
\end{equation}
Here "overline" stands for modulo $P_{F_0}$.
\end{thm}

\begin{proof}
From equation (\ref{eqn 3.121}) we have 
$$\lambda_{K/F}(\psi_{-1})=\lambda_{K/F}(c'\psi_{F})=\Delta_{K/F}(c')\cdot \lambda_{K/F}(\psi_F).$$
Since $\Delta_{K/F}$ is quadratic, we can write $\Delta_{K/F}=\Delta_{K/F}^{-1}$. So we obtain
\begin{equation*}
 \lambda_{K/F}(\psi_F)=\Delta_{K/F}(c')\cdot\lambda_{K/F}(\psi_{-1}).
\end{equation*}
Now we have to compute $\lambda_{K/F}(\psi_{-1})$, and which we do in the following:\\
  Since $[K:F]=2$, we have $\rm{Ind}_{K/F}(1)=1_F\oplus\omega_{K/F}$.
 The conductor of $\omega_{K/F}$ is $1$ because $K/F$ is a tamely ramified quadratic extension, and hence $t=0$, so 
 $a(\omega_{K/F})=t+1=1$. Therefore we can consider 
$\omega_{K/F}$ as a character of $F^\times/U_{F}^{1}$. So the restriction of $\omega_{K/F}$ to $U_{F}$, 
$\rm{res}(\omega_{K/F}):=\omega_{K/F}|_{U_F}$, we may consider as the uniquely determined
character of $k_{F}^{\times}$ of order $2$.
Since $c'$ satisfies equations (\ref{eqn 3.35}), (\ref{eqn 3.36}), then from Lemma \ref{Lemma 3.21} we have 
$\psi_{-1}|_{O_F}=c'\cdot\psi_F|_{O_F}=\psi_{q_F}$, and this is the canonical
character of $k_F$. Then from equation (\ref{eqn 3.33}) we can write
\begin{align*}
 \lambda_{K/F}(\psi_{-1})
 &=q_{F}^{-\frac{1}{2}}\sum_{x\in k_{F}^{\times}}\rm{res}(\omega_{K/F})(x)\psi_{q_F}(x)\\
 &=q_{F}^{-\frac{1}{2}}\cdot G(\rm{res}(\omega_{K/F}),\psi_{q_F}).
\end{align*}
Moreover, by Theorem \ref{Theorem 3.5} we have 
\begin{equation}
 G(\rm{res}(\omega_{K/F}),\psi_{q_F})=\begin{cases}
                    (-1)^{s-1}q_{F}^{\frac{1}{2}} & \text{if $p\equiv 1\pmod{4}$}\\
                    (-1)^{s-1}i^{s}q_{F}^{\frac{1}{2}} & \text{if $p\equiv 3\pmod{4}$}.
                   \end{cases}
\end{equation}
By using the classical quadratic Gauss sum we obtain
\begin{equation}
 \lambda_{K/F}(\psi_{-1})=\begin{cases}
                    (-1)^{s-1} & \text{if $p\equiv 1\pmod{4}$}\\
                    (-1)^{s-1}i^{s} & \text{if $p\equiv 3\pmod{4}$}.
                   \end{cases}
\end{equation}

 We also can write 
$\Delta_{K/F}=\det(\rm{Ind}_{K/F}(1))=\det(1_F\oplus \omega_{K/F})=\omega_{K/F}.$
So we have 
 $$\Delta_{K/F}(\pi_F)=\omega_{K/F}(\pi_F)=1,$$
because $\pi_F\in N_{K/F}(K^\times)$.

Under the assumption of the Theorem \ref{Theorem 3.21} we have $\pi_F\in N_{K/F}(K^\times)$, $\Delta_{K/F}=\omega_{K/F}$ and 
$c'=\frac{c}{\rm{Tr}_{F/F_0}(pc)}$, where $c\in F^\times$ with $\nu_F(c)=-1-d_{F/\bbQ_p}$. Then we can write 
\begin{align*}
 \Delta_{K/F}(c')
 &=\omega_{K/F}(c')\\
 &=\omega_{K/F}\left(\frac{c}{\rm{Tr}_{F/F_0}(pc)}\right)\\
 &=\omega_{K/F}\left(\frac{\pi_{F}^{-e_{F/\bbQ_p}}u(c)}{u_0(c)}\right),
 \quad\text{where $c=\pi_{F}^{-e_{F/\bbQ_p}}u(c)$, $\rm{Tr}_{F/F_0}(pc)=u_0(c)\in U_{F_0}$}\\
 &=\omega_{K/F}(\pi_{F}^{-e_{F/\bbQ_p}})\omega_{K/F}(v),\quad\text{where $v=\frac{u(c)}{u_0(c)}\in U_F$}\\
 &=\omega_{K/F}(x)\\
 &=\begin{cases}
                  1 & \text{when $x$ is a square element in $k_{F}^{\times}$}\\
                  -1 & \text{when $x$ is not a square element in $k_{F}^{\times}$},
                 \end{cases}
\end{align*}
where $v=xy$, with $x=x(\omega_{K/F},c)\in U_{F}/U_{F}^{1}$, and $y\in U_{F}^{1}$.

In particular, if we choose $c$ such a way that $u(c)=1$, i.e., $c=\pi_{F}^{-1-d_{F/\bbQ_p}}$, then 
we have 
$\Delta_{K/F}(c')=\Delta_{K/F}(\rm{Tr}_{F/F_0}(pc)).$
Since $\rm{Tr}_{F/F_0}(pc)\in O_{F_0}$ is a unit and $\Delta_{K/F}=\omega_{K/F}$ induces the quadratic character of 
$k_{F}^{\times}=k_{F_0}^{\times}$, then for this particular choice of $c$ we obtain
\begin{equation*}
 \Delta_{K/F}(c')=\begin{cases}
                   1 & \text{if $\overline{\rm{Tr}_{F/F_0}(pc)}$ is a square in $k_{F_0}^{\times}$}\\
                   -1 & \text{if $\overline{\rm{Tr}_{F/F_0}(pc)}$ is not a square in $k_{F_0}^{\times}$}.\\
                  \end{cases}
\end{equation*}

\end{proof}

\textbf{Note:} When we are in the Case 1 of Corollary \ref{Lemma 3.10}, by using this above Theorem \ref{Theorem 3.21}
we can give explicit formula for $\lambda_{K/F}$, because here the quadratic extensions are tamely ramified (since $p\ne 2$).

\begin{rem}\label{Remark 3.4.11}
When $p\ne 2$, for any non-archimedean local field $F/\bbQ_p$ we know that 
the square class group $F^\times/{F^\times}^2\cong V$ Klein's $4$-group.
If $K$ is the abelian extension of $F$ with $N_{K/F}(K^\times)={F^\times}^2$, then  
from Lemma \ref{Lemma 4.6} we can see that $\lambda_{K/F}$ depends on the choice of the 
base field $F$.
Now we put $\lambda_{i}=\lambda_{L_i/F}$, where $L_i/F$ are quadratic extension of $F$, $i=1,2,3$.
Let $\psi$ be an additive character of $F$.
From a direct computation we can write
\begin{equation}
 \lambda_{K/F}(\psi)=W(\mathrm{Ind}_{K/F} 1,\psi)=W(\chi_1,\psi)W(\chi_2,\psi) W(\chi_3,\psi)=
 \lambda_{1}(\psi)\lambda_{2}(\psi)\lambda_{3}(\psi).
\end{equation}
On the other hand from Lemma \ref{Lemma 4.6} we see:
\begin{equation}
 \lambda_{K/F}(\psi)=-\lambda_{2}(\psi)^{2}=-\lambda_{3}(\psi)^{2}.
\end{equation}
Comparing these two expressions we obtain:
\begin{center}
 $\lambda_1(\psi)\lambda_3(\psi)=-\lambda_2(\psi)$,  and $\lambda_1(\psi)\lambda_2(\psi)=-\lambda_3(\psi)$.
\end{center}
Moreover, since $L_1/F$ is unramified, therefore from Lemma \ref{Lemma 3.13} we have $\lambda_1(\psi)=(-1)^{n(\psi)}$.
So we observe:
\begin{enumerate}
 \item[(a)] The three conditions: $n(\psi)$ odd, $\lambda_1(\psi)=-1$, and $\lambda_2(\psi)=\lambda_3(\psi)$ are equivalent.
 \item[(b)] In the same way, the conditions: $n(\psi)$ even, $\lambda_1(\psi)=1$ and $\lambda_2(\psi)=-\lambda_3(\psi)$ are equivalent.
\end{enumerate}
Let $\mu_4$ denote the group generated by a fourth root of unity. Then we have more equivalences:
\begin{enumerate}
 \item[(c)] $q_F\equiv 1\pmod{4}$, $\mu_4\subset F^\times$ and $\lambda_{2}(\psi)^{2}=\lambda_{3}(\psi)^{2}=1$ are three equivalent conditions.
 \item[(d)] $q_F\equiv 3\pmod{4}$, $\mu_4\not\subset F^\times$ and $\lambda_{2}(\psi)^{2}=\lambda_{3}(\psi)^{2}=-1$ are also three equivalent conditions.
\end{enumerate}
\textbf{This gives us four disjoint cases:} We can have $(a)$ and $(c)$ or $(a)$ and $(d)$ or $(b)$ and $(c)$ or 
$(b)$ and $(d)$.\\

We can put all these four disjoint cases into the following table:
\begin{center}
$\begin{array}{|l|l|l|l|l|l|}
 \underline{q_F} & \underline{n(\psi)} & \underline{\lambda_{K/F}} &\underline{\lambda_1} &\underline{\lambda_2} &\underline{\lambda_3}\\
q_F\equiv 1\pmod{4} & odd & -1 & -1 & \pm 1 & \pm 1\\
q_F\equiv 1\pmod{4} & even & -1 & 1 & 1 & -1\\
q_F\equiv 3\pmod{4} & odd & 1 & -1 & \pm i & \pm i\\
q_F\equiv 3\pmod{4} & even & 1 & 1 & +i & -i
\end{array}$
\end{center}
For the two cases where $\lambda_2,\;\lambda_3$ have different sign, and  we take (up to permutation) $\lambda_2$ with positive sign.

\end{rem}

\subsection{\textbf{Computation of $\lambda_{K/F}$, where $K/F$ is a wildly ramified quadratic extension}}

In the case $p=2$, the square class group of $F$, i.e., $F^\times/{F^\times}^2$ can be large, 
so we can have many quadratic characters but they are 
wildly ramified, not tame. 
Let $F=\bbQ_2$, then we have $\bbQ_{2}^{\times}/{\bbQ_{2}^{\times}}^2\cong\bbZ_2\times\bbZ_2\times\bbZ_2$. If $K/\bbQ_2$ is the 
abelian extension for which $N_{K/\bbQ_2}(K^\times)={\bbQ_{2}^{\times}}^2$, then we have the following lemma.
\begin{lem}\label{Lemma 3.25}
 Let $K$ be the finite abelian extension of $\bbQ_2$ for which $N_{K/\bbQ_2}(K^\times)={\bbQ_{2}^{\times}}^2$. 
 Then $\lambda_{K/\bbQ_2}=1$.
\end{lem}
\begin{proof}
 Let $G=\rm{Gal}(K/\bbQ_2)$. We know that $\bbQ_2/{\bbQ_{2}^{\times}}^2\cong\bbZ_2\times\bbZ_2\times\bbZ_2$. Therefore from class 
 field theory $G\cong\bbZ_2\times\bbZ_2\times\bbZ_2$. So the $2$-rank of $G$ is $3$, i.e., $\rm{rk}_2(G)=3$, and hence from the equation
 \ref{eqn 3.31} we have $\Delta_{1}^{G}=1$. Moreover, it is easy to see that $G$ is not metacyclic, because $\bbZ_2\times\bbZ_2$
 is not cyclic. So from Theorems \ref{Theorem 4.2}, and \ref{Theorem 4.1} we have $c_{1}^{G}=1$. Then finally we obtain
 \begin{center}
  $\lambda_{K/\bbQ_2}=\lambda_{1}^{G}=c_{1}^{G}\cdot W(\Delta_{1}^{G})=1$.
 \end{center}
\end{proof}
 
 Moreover, from Theorem \ref{Theorem 2.11}, if $F/\bbQ_2$, we have $|F^\times/{F^\times}^2|=2^m,\,(m\ge 3)$, therefore more generally 
 we obtain the following result.
 
\begin{thm}\label{Theorem 3.26}
 Let $F$ be an extension of $\bbQ_2$. Let $K$ be the abelian extension for which $N_{K/F}(K^\times)={F^\times}^2$.
 Then $\lambda_{K/F}=1$.
\end{thm}
\begin{proof}
This proof is same as the above Lemma \ref{Lemma 3.25}.
Let $G=\rm{Gal}(K/F)$. From Theorem \ref{Theorem 2.11} we have $\rm{rk}_2(G)\ne 1$
 and $G$ is not metacyclic. Therefore by using Theorems \ref{Theorem 4.2}, and \ref{Theorem 4.1} we can conclude that 
 $\lambda_{K/F}=\lambda_{1}^{G}=1$.
\end{proof}

\begin{exm}[{\bf Computation of $\lambda_{L/\bbQ_2}$, where $L/\bbQ_2$ is a quadratic extension}]\label{Example wild}

 Let $F=\bbQ_2$. For the principal unit filtration we write $U^i:=U_{\bbQ_2}^{i}$. Then we have
\begin{center}
 $\bbQ_{2}^{\times}\supset U^0=U^1\supset U^2\supset U^3$,
 \end{center}
 and $U^3\subset {\bbQ_{2}^{\times}}^2$, 
 therefore we can write 
 $$\bbQ_2^{\times}\supset U^0{\bbQ_2^\times}^2=U^1{\bbQ_2^\times}^2\supset U^2{\bbQ_2^\times}^2\supset 
 U^3{\bbQ_2^\times}^2={\bbQ_2^\times}^2.$$
 Now take 
  modulo ${\bbQ_{2}^{\times}}^2$ we have \\
 $$\bbQ_{2}^{\times}/{\bbQ_{2}^{\times}}^2>U^1{\bbQ_2^\times}^2/{\bbQ_{2}^{\times}}^2>U^2{\bbQ_2^\times}^2/{\bbQ_{2}^{\times}}^2>\{1\},$$
 and the index is always $2$.
So we have 
\begin{center}
 $2-1=1$ character $\chi_1$ with $a(\chi_1)=0$, $\chi_1\ne\chi_0$, the trivial character,\\
 $2^2-2=2$ characters $\chi_2,\,\chi_3$ with $a(\chi_i)=2$, $i=2,3$,\\
 $2^3-2^2=4$ characters $\chi_4,\cdots,\chi_7$ with $a(\chi_i)=3$, $i=4,\cdots,7$.
\end{center}
The last case is the \textbf{exceptional case} (cf. p. 50) because $p=2,\,e=1$ gives $i=\frac{pe}{p-1}=2$. Here we will have odd 
conductor. We can simplify as follows:
\begin{center}
 $\chi_1,\,\chi_2,\,\chi_3=\chi_1\chi_2$,
 $\chi_4,\,\chi_5=\chi_1\chi_4,\,\chi_6=\chi_2\chi_4$, and $\chi_7=\chi_1\chi_2\chi_4$.
\end{center}
We denote
$G=\rm{Gal}(K/F)\cong \bbQ_{2}^{\times}/{\bbQ_{2}^{\times}}^2$. Since $G$ is abelian, then $G\cong\widehat{G}$, namely 
$\widehat{G}=\{1=\chi_0,\chi_1,\chi_2,\cdots,\chi_7\}$, where $\chi_{i}^{2}=1$, $i=1,2,\cdots,7$.
So we can write 
\begin{equation*}
 \rm{Ind}_{K/F}(1)=1\oplus\sum_{i=1}^{7}\chi_i.
\end{equation*}
Again from Lemma \ref{Lemma 3.25} we have $\lambda_{K/F}=1$.
Thus we can write 
\begin{align*}
 \lambda_{1}^{G}
 &=\lambda_{K/F}
 =W(\rm{Ind}_{K/F}(1))
 =\prod_{i=1}^{7}W(\chi_i)\\
 &=\prod_{i=1}^{7}\lambda_i=1,
\end{align*}
where $\lambda_{i}=\lambda_{K_i/F}$ and $K_i/F$ is the corresponding quadratic extension of character $\chi_i$ for  
$i=1,2,\cdots,7$. Moreover, there is an unramified quadratic extension of $\bbQ_2$, namely $K_1/\bbQ_2$ which corresponds $\chi_1$.
Then $\lambda_{1}=\lambda_{K_1/\bbQ_2}=(-1)^{n(\psi_{\bbQ_2})}=1$, because the conductor $n(\psi_{\bbQ_2})=0$. 
We also have 
\begin{center}
 $\lambda_{2}=W(\chi_2),\,\lambda_3=W(\chi_3), \cdots,\lambda_7=W(\chi_7)$.
\end{center}
Then we obtain
\begin{align*}
 \lambda_{K/\bbQ_2}(\psi_{\bbQ_2})
 &=\prod_{i=1}^{7}\lambda_i\\
 &=W(\chi_1)W(\chi_2)W(\chi_3)W(\chi_4)W(\chi_5)W(\chi_6)W(\chi_7)\\ 
 &=W(\chi_1)W(\chi_2)W(\chi_1\chi_2)W(\chi_4)W(\chi_1\chi_4)W(\chi_2\chi_4)W(\chi_1\chi_2\chi_4)\\
 &=(-1)^{n(\psi_{\bbQ_2})}\cdot W(\chi_2)\cdot\chi_1(p)^{2}W(\chi_2)\cdot W(\chi_4)\cdot\chi_1(p)^3W(\chi_4)\cdot W(\chi_2\chi_4)\cdot
 \chi_1(p)^3 W(\chi_2\chi_4)\\
 &=W(\chi_2)^2\cdot W(\chi_4)^2\cdot W(\chi_2\chi_4)^2\\
 &=1,
\end{align*}
since $n(\psi_{\bbQ_2})=0$ and $\lambda_{K/\bbQ_2}=1$.

Now we have to give explicit computation of $\lambda_i$, where $i=1,\cdots, 7$. For this particular example directly 
 we can give explicit
computation of $\lambda_i$ by using the modified formula (\ref{eqn 2.2}) of abelian local constant. 
Before going to our explicit computation we need to recall 
few facts. Suppose that $\chi$ is a multiplicative character of a non-archimedean local field $F/\bbQ_p$ 
of conductor $n$. Then we can write 
\begin{equation}
 W(\chi,\psi_F)=\chi(\pi_F^{n+n(\psi_F)})\cdot q_F^{-\frac{n}{2}}\cdot
 \sum_{x\in U_F/U_F^n}\chi^{-1}(x)\psi(\frac{x}{\pi_F^{n+n(\psi_F)}}),
\end{equation}
where $\pi_F$ is a uniformizer of $F$. By definition we have $\psi_F(x)=e^{2\pi i\rm{Tr}_{F/\bbQ_p}(x)}$, 
and any element $x\in U_F/U_F^n$
can be written as 
$$x=a_0+a_1\pi_F+a_2\pi_F^2+\cdots+a_{n-1}\pi_F^{n-1}, \quad\text{where $a_i\in k_F$ and $a_0\ne 0$}.$$
Then we can consider the following set
$$\{ a_0+a_1\pi_F+a_2\pi_F^2+\cdots+a_{n-1}\pi_F^{n-1}\;|\; \quad\text{where $a_i\in k_F$ and $a_0\ne 0$}\}$$
is a representative of $U_F/U_F^n$.

When $F=\bbQ_2$, we have $a(\chi_2)=a(\chi_3)=2$ and $a(\chi_i)=3, \;(i=4,\cdots, 7)$. Therefore we can write 
$$U_{\bbQ_2}/U_{\bbQ_2}^2=\{1, 1+2\}=\{1, 3\}, \qquad U_{\bbQ_2}/U_{\bbQ_2}^3=\{1, 1+2, 1+2^2, 1+2+2^2\}=\{1,3,5, 7\}.$$
We also know that any square element $x$ in $\bbQ_2^\times$ is of the form $x=4^m(1+8 n)$, where $m\in\bbZ$ and $n$ is a $2$-adic 
integer. This tells us $\pm2, -1$ and $\pm 5$ are not square in $\bbQ_2^\times$. Again, if we fix $\pi_F=2$ as a uniformizer, then 
we can write 
$$\bbQ_2^\times=<2>\times U_{\bbQ_2}^1=<2>\times <\eta>\times U_{\bbQ_2}^2=<2>\times<-1>\times U_{\bbQ_2}^2,$$
where $\eta^2=1$. Then we have the following list of seven quadratic extensions of $\bbQ_2$ (cf. \cite{JP3}, p. 18, Corollary
of Theorem 4 and \cite{GG}, pp. 83-84):
\begin{center}
 $\bbQ_2(\sqrt{5}), \bbQ_2(\sqrt{-1}),\bbQ_2(\sqrt{-5}), \bbQ_2(\sqrt{2}), \bbQ_2(\sqrt{-2}), \bbQ_2(\sqrt{10}), \bbQ_2(\sqrt{-10}).$
\end{center}
Now our next job is to see the norm groups of the above quadratic extensions of $\bbQ_2$. For any finite extension $K/F$, we denote 
$\cN_{K/F}:=N_{K/F}(K^\times)$, the norm group of the extension $K/F$. So we can write:
$$\cN_{\bbQ_2(\sqrt{5})/\bbQ_2}=<2^2>\times U_{\bbQ_2}=<2^2>\times<-1>\times U_{\bbQ_2}^2,$$
$$\cN_{\bbQ_2(\sqrt{-1})/\bbQ_2}=<2>\times U_{\bbQ_2}^2,$$
$$\cN_{\bbQ_2(\sqrt{-5})/\bbQ_2}=<-2>\times U_{\bbQ_2}^2,$$
$$\cN_{\bbQ_2(\sqrt{2})/\bbQ_2}=<2>\times<-1>\times U_{\bbQ_2}^3,$$
$$\cN_{\bbQ_2(\sqrt{-2})/\bbQ_2}=<2>\times U_{\bbQ_2}^3,$$
$$\cN_{\bbQ_2(\sqrt{10})/\bbQ_2}=<2\times 5>\times<-1>\times U_{\bbQ_2}^3,$$
$$\cN_{\bbQ_2(\sqrt{-10})/\bbQ_2}=<-2>\times U_{\bbQ_2}^3.$$
From the above norm groups, we can conclude that:
\begin{enumerate}
 \item the extension $\bbQ_2(\sqrt{5})$ is unramified, hence it corresponds the character $\chi_1$.
 \item the extensions $\bbQ_2(\sqrt{-1}), \bbQ_2(\sqrt{-5})$ are two wild quadratic extensions which
 correspond the characters $\chi_2, \chi_3$ respectively.
 \item and the extensions $\bbQ_2(\sqrt{2}), \bbQ_2(\sqrt{10}),\bbQ_2(\sqrt{-2})$, and $\bbQ_2(\sqrt{-10})$
 correspond the characters $\chi_4,\chi_5,\chi_6$ and $\chi_7$ respectively. 
\end{enumerate}
Now we have all necessary informations for giving explicit formula of $\lambda$-functions, and they are:
\begin{enumerate}
 \item $\lambda_{\bbQ_2(\sqrt{5})/\bbQ_2}=W(\chi_1,\psi_{\bbQ_2})=(-1)^{n(\psi_{\bbQ_2})}=1$.
 \item 
 \begin{align*}
  \lambda_{\bbQ_2(\sqrt{-1})/\bbQ_2}
&=W(\chi_2,\psi_{\bbQ_2})=\chi_2(2^2)\cdot\frac{1}{2}\cdot\sum_{x\in U_{\bbQ_2}/U_{\bbQ_2}^2}\chi_2(x)\psi_{\bbQ_2}(\frac{x}{4})\\
&=\frac{1}{2}\cdot \left(\psi_{\bbQ_2}(\frac{1}{4})+\chi_2(3)\cdot\psi_{\bbQ_2}(\frac{3}{4})\right)\\
&=\frac{1}{2}\cdot \left(e^{\frac{\pi i}{2}}-e^{\frac{3\pi i}{2}}\right),\quad\text{since $3\not\in\cN_{\bbQ_2(\sqrt{-1})/\bbQ_2}$
and $\psi_{\bbQ_2}(x)=e^{2\pi i x}$}\\
&=\frac{1}{2}\cdot (i+i)=i.
 \end{align*}
\item $\lambda_{\bbQ_2(\sqrt{-5})/\bbQ_2}=W(\chi_3,\psi_{\bbQ_2})=W(\chi_1\chi_2,\psi_{\bbQ_2})
=\chi_1(2^2)\cdot W(\chi_2,\psi_{\bbQ_2})=i$.
\item 
\begin{align*}
 \lambda_{\bbQ_2(\sqrt{2})/\bbQ_2}
 &=W(\chi_4,\psi_{\bbQ_2})=\chi_4(2^3)\cdot\frac{1}{2\sqrt{2}}\cdot\sum_{U_{\bbQ_2}/U_{\bbQ_2}^3}\chi_4(x)\psi_{\bbQ_2}(\frac{x}{8})\\
 &=\frac{1}{2\sqrt{2}}\cdot \left(\psi_{\bbQ_2}(\frac{1}{8})+\chi_4(3)\cdot\psi_{\bbQ_2}(\frac{3}{8})+
 \chi_4(5)\cdot \psi_{\bbQ_2}(\frac{5}{8})+\chi_4(7)\cdot\psi_{\bbQ_2}(\frac{7}{8})\right)\\
 &=\frac{1}{2\sqrt{2}}\left(e^{\frac{\pi i}{4}}-e^{\frac{3\pi i}{4}}-e^{\frac{5\pi i}{4}}+e^{\frac{7\pi i}{4}}\right)\\
 &=\frac{1}{2\sqrt{2}}\cdot ( 2\sqrt{2} + 0\cdot i)\\
 &=1,
\end{align*}
since $3,5\not\in\cN_{\bbQ_2(\sqrt{2})/\bbQ_2}$ but $7\in\cN_{\bbQ_2(\sqrt{2})/\bbQ_2}$.
\item $\lambda_{\bbQ_2(\sqrt{10})/\bbQ_2}=W(\chi_5,\psi_{\bbQ_2})=W(\chi_1\chi_4,\psi_{\bbQ_2})
=\chi_1(2^3)\cdot W(\chi_4,\psi_{\bbQ_2})=(-1)\cdot 1=-1$.
\item Again $2, 3\in \cN_{\bbQ_2(\sqrt{-2})/\bbQ_2}$ but $5,7\not\in\cN_{\bbQ_2(\sqrt{-2})/\bbQ_2}$, so similarly we can write
\begin{align*}
 \lambda_{\bbQ_2(\sqrt{-2})/\bbQ_2}
 &=W(\chi_6,\psi_{\bbQ_2})\\
 &=\chi_6(2^3)\cdot\frac{1}{2\sqrt{2}}\left(e^{\frac{\pi i}{4}}+e^{\frac{3\pi i}{4}}-e^{\frac{5\pi i}{4}}-e^{\frac{7\pi i}{4}}\right)\\
 &=\frac{1}{2\sqrt{2}}\cdot 2\sqrt{2}=i.
\end{align*}

\item $\lambda_{\bbQ_2(\sqrt{-10})/\bbQ_2}=W(\chi_7,\psi_{\bbQ_2})=W(\chi_1\chi_6,\psi_{\bbQ_2})=(-1)\cdot W(\chi_6,\psi_{\bbQ_2})=-i$.

\end{enumerate}

\end{exm}

\begin{rem}
Finally we observe that Theorem \ref{Theorem 4.3} and Corollary \ref{Lemma 3.10} are the general results on
$\lambda_{1}^{G}=\lambda_{E/F}$, where $E/F$ is a Galois extension with Galois group $G=\rm{Gal}(E/F)$. And {\bf the general
results leave open} the computation of $W(\alpha)$, where $\alpha$ is a quadratic character of $G$. For such a quadratic character we
can have three cases:
\begin{enumerate}
 \item unramified, this is the Theorem \ref{Theorem 3.6}, 
 \item tamely ramified, this is the Theorem \ref{Theorem 3.21}, 
 \item wildly ramified, its explicit computation is still open.
\end{enumerate}
We also observe from the above example \ref{Example wild} that giving explicit formula for $\lambda_{K/F}$,
where $K/F$ is a wildly ramified quadratic extension, is very subtle. In particular, when $F=\bbQ_2$, in the above 
Example (\ref{Example wild})  
we have the explicit computation of $\lambda_{K/\bbQ_2}$.

\end{rem}

\chapter{\textbf{Determinant of Heisenberg representations}}

In this chapter we give an invariant formula of determinant of a Heisenberg representation $\rho$ of a finite group 
$G$ modulo $\rm{Ker}(\rho)$. The group $G$ need not be a two-step nilpotent group, but under modulo $\rm{Ker}(\rho)$, $G$ is a two-step nilpotent group.
In this chapter firstly we compute transfer map for two-step nilpotent group. Then we compute $\det(\rho)$ modulo
$\rm{Ker}(\rho)$, because $G$ is always a two-step nilpotent group under modulo $\rm{Ker}(\rho)$.
This chapter is based on the article \cite{SAB3}.

\section{{\bf Explicit computation of the transfer map for two-step nilpotent group}}

Let $G$ be a finite group with $[G,[G,G]]=\{1\}$. Let $H$ be a normal subgroup of $G$,
with abelian quotient group $G/H$ of order $d$. If $d$ is odd,
then in the following lemma we compute $T_{G/H}(g)$ for all $g\in G$.

\begin{lem}\label{Lemma 2.2}
 Assume that $G$ is a finite group and $H$ a normal subgroup such that 
 \begin{enumerate}
  \item $H$ is abelian,
  \item $G/H$ is abelian of odd order $d$,
  \item $[G,[G,G]]=\{1\}$. 
 \end{enumerate}
Then we have $T_{G/H}(g)=g^{d}$ for all $g\in G$.\\
As a consequence one has $[G,G]^d=\{1\}$, in other words, $G^d$ is contained in the center of $G$.
\end{lem}
\begin{proof}
In general, we know that transfer map is independent of the choice of the left transversal for $H$ in $G$. 
So we take $T$ as a  
transversal\footnote{Since $H$ is normal, left cosets and right cosets are the same, so we can 
simply call transversal instead of specifying 
left or right transversal.} for $H$ in $G$. 
By the given condition $H$ is normal, we have $T\cong G/H$ and 
$G=TH$, where $TH=\{th\;|\; t\in T, h\in H\}$. This shows that every element $g\in G$ can uniquely be written as $g=th$, where 
$t\in T$ and $h\in H$.

 First assume $g=h\in H$. Then we have 
 \begin{center}
  $ht=t\cdot t^{-1}ht\in tH$,
 \end{center}
because $H$ is a normal subgroup of 
 $G$. Hence $s=t$, where $s=s(t)$ is a function of $t$ which is uniquely determined by $gt\in sH$, for some 
 $g\in G$. Therefore:
 \begin{align}
  T_{G/H}(h)\nonumber
  &=\prod_{t\in T}s^{-1}ht=\prod_{t\in T}t^{-1}ht=\prod_{t\in T}hh^{-1}t^{-1}ht=\prod_{t\in T}(h\cdot[h^{-1},t^{-1}])\\
  &=h^{d}\prod_{t\in T}[h^{-1},t^{-1}]=h^{d}[h^{-1},\prod_{t\in T}t^{-1}].\label{eqn 2.1}
 \end{align}
We have used the condition (3) in the last two equalities which means that commutators are in the center. Now we use that $G/H$ is of odd 
order, hence $x=1$ if $x\in G/H$ is an element such that $x=x^{-1}$, i.e., $G/H$ has no self-inverse element. Therefore from 
Theorem \ref{Theorem Miller}, we have 
$\prod_{t\in T}t^{-1}=\prod_{t\in T}t=1\in G/H$, hence $T_{G/H}(h)=h^d$. Proceeding with the proof of the Lemma we have now
\begin{equation}
 T_{G/H}(th)=T_{G/H}(t)\cdot T_{G/H}(h)=T_{G/H}(t)\cdot h^d.
\end{equation}
Moreover, from Lemma \ref{Lemma 22} we can write
\begin{center}
 $t^dh^d=(th)^{d}[t^{\frac{d(d-1)}{2}},h]=(th)^{d}[e,h]=(th)^d$,
\end{center}
since $[G,G]\subseteq Z(G)$ and $d$ is odd\footnote{Here $d$ divides $\frac{d(d-1)}{2}$ and the order of group $G/H$ is $d$.
So for any $t\in G/H$, $t^{\frac{d(d-1)}{2}}=e$, the identity in $G/H$.}.
So we are left to show that $T_{G/H}(t)=t^d$ for all $t\in T$.

Since $G/H$ is an abelian group of odd order, hence we may write
\begin{center}
 $G/H=C\times U$,
\end{center}
where $C$ is cyclic group of odd order $m|d$, and we assume $t\in T$ such that $tH$ is a generator of $C$. Then our transversal system
can be chosen as
\begin{center}
 $T=\{t^{i}u|i=0,1,\cdots,m-1, uH\in U\}$.
\end{center}
Now if $i\leq m-2$ we have $t\cdot t^{i}u=t^{i+1}\cdot u=s$, hence $s^{-1}\cdot t\cdot t^{i}u=1$. But $i=m-1$ we obtain
\begin{center}
 $t(t^{m-1}u)=t^{m}u\in uH$,\hspace{.5cm} $u^{-1}t(t^{m-1}u)=u^{-1}t^{m}u$,
\end{center}
hence 
\begin{align}
 T_{G/H}(t)\nonumber
 &=\prod_{u\in U}u^{-1}t^{m}u=\prod_{u\in U}t^{m}[t^{-m},u^{-1}]=t^{d}\prod_{u\in U}[t^{-m},u^{-1}]\\
 &=t^{d}[t^{-m},\prod_{u\in U}u^{-1}]=t^{d}[t^{-m},e]=t^{d}\label{eqn 2.21},
\end{align}
since $d$ is odd, then the order of $U$ is also odd and by Theorem \ref{Theorem Miller} we have 
$\prod_{u\in U}u^{-1}=\prod_{u\in U}u=e\in U$.

We also know that any $g\in G$ can uniquely be written as $g=th$, where $t\in T$ and $h\in H$. Then finally we obtain:
\begin{equation}
 T_{G/H}(g)=T_{G/H}(th)=T_{G/H}(t)\cdot T_{G/H}(h)= t^{d}\cdot h^{d}=(th)^{d}[t^{\frac{d(d-1)}{2}},h]=g^{d}.
\end{equation}

 Moreover, by our assumption (2), we have $G/H$ is an abelian group, therefore $[G,G]\subseteq H$, in particular, $T_{G/H}(h)=h^d$ for 
 $h\in [G,G]\subseteq H$. On the other hand, from Theorem \ref{Furtwangler's Theorem} 
 $T_{G/[G,G]}$ is trivial.
 So under the above Lemma's conditions we conclude $[G,G]^d=1$, in other words, $G^d$ is in the center
 because due to condition (3) the commutator is bilinear.
 
 \end{proof}
\begin{rem}
 From Lemma \ref{Lemma 2.2} we have $T_{G/H}(g)=g^d$ for all $g\in G$. This implies for $g_1,g_2\in G$
 \begin{center}
  $T_{G/H}(g_1g_2)=(g_1g_2)^d$, on the other hand, \\
   $T_{G/H}(g_1g_2)=T_{G/H}(g_1)\cdot T_{G/H}(g_2)=g_1^d\cdot g_2^d$,
 \end{center}
 because $T_{G/H}$ is a homomorphism.
Hence for all $g_1,\, g_2\in G$ we have
 $$(g_1 g_2)^d = g_1^d g_2^d.$$
This implies $G^d$ is actually a subgroup of $G$ not only a subset.
\end{rem}

By combining Lemma \ref{Lemma 2.2} and the elementary divisor theorem, we have the following result.

\begin{lem}\label{Lemma 2.10}
 Assume that $G$ is a finite group and $H$ a normal subgroup such that 
 \begin{enumerate}
  \item $H$ is abelian
  \item $G/H$ is abelian of order $d$, such that (according to the elementary divisor theorem):
  \begin{center}
   $G/H\cong\mathbb{Z}/m_1\times\cdots\times\mathbb{Z}/m_s$
  \end{center}
where $m_1|\cdots|m_s$ and $\prod_{i}m_i=d$. Moreover, we fix elements $t_1,t_2,\cdots,t_s\in G$ such that $t_iH\in G/H$ generates the 
cyclic factor $\cong\mathbb{Z}/m_i$, hence $t_{i}^{m_i}\in H$.
\item $[G,[G,G]]=\{1\}$. In particular, $[G,G]$ is in the center $Z(G)$ of $G$.
 \end{enumerate}
Then each $g\in G$ has a unique decomposition 
\begin{enumerate}
 \item[(i)] 
 \begin{align*}
  g=t_{1}^{a_1}\cdots t_{s}^{a_s}\cdot h, \hspace{.5cm} T_{G/H}(g)=\prod_{i}^{s}T_{G/H}(t_i)^{a_i}\cdot T_{G/H}(h),
 \end{align*}
where $0\leq a_i\leq m_i-1$, $h\in H$, and
\item[(ii)] 
\begin{align*}
 T_{G/H}(t_i)=t_{i}^{d}\cdot[t_{i}^{m_i},\alpha_i], \quad\quad T_{G/H}(h)=h^{d}\cdot[h,\alpha],
\end{align*}
where $\alpha_i\in G/H$ is the product over all elements from $C_i\subset G/H$, the subgroup which is complementary to the cyclic subgroup
$<t_i>$ mod $H$, and where $\alpha\in G/H$ is product over all elements from $G/H$.\\
Here we mean $[t_{i}^{m_i},\alpha_i]:=[t_{i}^{m_i},\widehat{\alpha_i}]$, $[h,\alpha]:=[h,\widehat{\alpha}]$ for any representatives 
$\widehat{\alpha_i},\widehat{\alpha}\in G$. The commutators
are independent of the choice of the representatives and are always elements of order $\leq 2$ because 
$\widehat{\alpha_i}^{2},\widehat{\alpha}^{2}\in H$, and $H$ is abelian. As a consequence of $(i)$ and $(ii)$ we always obtain
\item[(iii)]
\begin{align*}
 T_{G/H}(g)=g^{d}\cdot\varphi_{G/H}(g),
\end{align*}
where $\varphi_{G/H}(g)\in Z(G)$ is an element of order $\leq 2$.
\end{enumerate}
As a consequence of the second equality in $(ii)$ combined with $[G,G]\subseteq H\cap\mathrm{Ker}(T_{G/H})$, one has $[G,G]^d=\{1\}$,
in other words, $G^d$ is contained in the center $Z(G)$ of $G$. 
\end{lem}

\begin{proof}

By the given conditions, we have the abelian group $G/H$ of order $d$ with 
 \begin{center}
   $G/H\cong\mathbb{Z}/m_1\times\cdots\times\mathbb{Z}/m_s$
  \end{center}
where $m_1|\cdots|m_s$ and $\prod_{i}m_i=d$. 
Moreover, we fix elements $t_1,t_2,\cdots,t_s\in G$ such that $t_iH\in G/H$ generates the 
cyclic factor $\cong\mathbb{Z}/m_i$, hence $t_{i}^{m_i}\in H$. Therefore for a fixed $i\in\{1,2,\cdots,s\}$
 we can define a subgroup $C_i\subset G/H$ such that $C_i$ is complementary to the 
 cyclic subgroup $<t_i>$ of order $m_i$ mod $H$, i.e., $G/H=<t_iH>\times C_i$.
 
Then for a fixed $i\in\{1,2,\cdots,s\}$, we can choose a transversal system for $H$ in $G$ and which is:
 \begin{center}
  $T=\left\{t_{i}^{j}c|\quad 0\leq j\leq m_{i}-1, cH\in C_i\right\}$.
 \end{center}
Therefore from equation (\ref{eqn 2.21}) we can write
\begin{align}\label{eqn 281}
 T_{G/H}(t_i)=t_{i}^{d}\cdot[t_{i}^{-m_i},\prod_{c\in C_i}c]=t_{i}^{d}\cdot[t_{i}^{-m_i},\alpha_i],
\end{align}
where $\alpha_i=\prod_{c\in C_i}c$.

For $h\in H$, from equation (\ref{eqn 2.1})
we have 
\begin{equation}\label{eqn 291}
 T_{G/H}(h)=h^{d}\cdot[h^{-1},\alpha],
\end{equation}
where $\alpha=\prod_{t\in T}t$.

We also have $\widehat{\alpha}^2, \widehat{\alpha_i}^2\in H$,
and the commutator $[.,.]$ is bilinear by assumption (3),
  hence $1 = [h,\widehat{\alpha}^2] = [h, \widehat\alpha]^2$
and therefore
\begin{center}
 $[h,\widehat{\alpha}] = [h,\widehat\alpha] ^{-1} = [h^{-1}, \widehat\alpha]$.
\end{center}
Similarly, we have
\begin{center}
 $[t_i^{m_i},\widehat{\alpha_i}]= [t_{i}^{m_i},\widehat{\alpha_i}] ^{-1} = [t_{i}^{-m_i}, \widehat{\alpha_i}]$.
\end{center}
Thus we can rewrite the equations (\ref{eqn 281}) and (\ref{eqn 291}) as:
\begin{align}\label{eqn 28}
 T_{G/H}(t_i)=t_{i}^{d}\cdot[t_{i}^{-m_i},\prod_{c\in C_i}c]=t_{i}^{d}\cdot[t_{i}^{m_i},\alpha_i],
\end{align}
and 
\begin{equation}\label{eqn 29}
 T_{G/H}(h)=h^{d}\cdot[h,\alpha].
\end{equation}
Here $[t_{i}^{m_i},\alpha_i]:=[t_{i}^{m_i},\widehat{\alpha_i}]$ and $[h,\alpha]:=[h,\widehat{\alpha}]$ for any representatives
$\widehat{\alpha_i},\widehat{\alpha}\in G$.

We also know that every $g\in G$ can be uniquely written as $th$, where $t\in G/H$ and $h\in H$. Again, since 
$G/H$ is abelian, therefore by using elementary divisor decomposition of $G/H$, we can also uniquely express 
$t$ as $t=t_{1}^{a_1}t_{2}^{a_2}\cdots t_{s}^{a_s}$, where $0\leq a_i\leq m_i-1$. Thus each $g$ has a unique decomposition
\begin{center}
 $g=th=t_{1}^{a_1}t_{2}^{a_2}\cdots t_{s}^{a_s}\cdot h$.
\end{center}

Then we have
\begin{align*}
 T_{G/H}(g)
 &=T_{G/H}(th)=T_{G/H}(t)\cdot T_{G/H}(h)\\
 &=T_{G/H}(t_{1}^{a_1} t_{2}^{a_2}\cdots t_{s}^{a_s})\cdot T_{G/H}(h)\\
 &=\prod_{i=1}^{s} T_{G/H}(t_i)^{a_i}\cdot T_{G/H}(h).
\end{align*}
By the assumption (2), $G/H$ is an abelian group, hence $[G,G]\subseteq H$. And from equation 
(\ref{eqn 29}) we have $T_{G/H}(h)=h^d[h,\alpha]$. This implies  for 
 $[G,G]\subseteq \mathrm{Ker}(T_{G/H})$, hence $[G,G]\subseteq H\cap\mathrm{Ker}(T_{G/H})$.
 On the other hand in general $T_{G/H}:G\to H/[H,H]$ is a homomorphism with values in an abelian group,
 hence it is trivial on commutators. So under the assumptions we can say $[G,G]^d=\{1\}$, in other words, $G^d$ is in the center
 $Z(G)$ because due to assumption (3) the commutator is bilinear.
Let $Z_2$ be the set of all elements of $Z(G)$ of order $\le 2$. Since $G^d\subseteq Z(G)$, then by using Lemma \ref{Lemma 22}(2)
with $n=d$, we obtain
\begin{equation}\label{eqn 222}
 \prod_{i=1}^{s}t_{i}^{a_i\cdot d}\equiv (\prod_{i=1}^{s}t_{i}^{a_i})^d \pmod{Z_2}\equiv t^d \pmod{Z_2}
\end{equation}
because combining $G^d\subseteq Z(G)$ and Lemma \ref{Lemma 22}(2) we can write 
\begin{center}
 $x^dy^d=(xy)^d\cdot[x,y]^{\frac{d(d-1)}{2}}\equiv (xy)^d\pmod{Z_2}$ for all $x, y\in G$.
\end{center}

Moreover, since $\alpha_{i}^2=1$ and $\alpha^{2}=1$, therefore  we have $[t_{i}^{m_i},\alpha_i]^{a_i}\in Z_2$
for all $i\in\{1,2,\cdots,s\}$ and $[h,\alpha]\in Z_2$. Again by using Lemma \ref{Lemma 22}(2) we can write 
\begin{equation}\label{eqn 2.133}
 \prod_{i=1}^{s}t_{i}^{a_i\cdot d}\cdot [t_{i}^{m_i},\alpha_i]^{a_i}\cdot h^{d}[h,\alpha]
 \equiv g^d \pmod{Z_2}.
\end{equation}

Now by using equations (\ref{eqn 28}) and (\ref{eqn 29}), we obtain:
\begin{align*}
 T_{G/H}(g)
&=\prod_{i=1}^{s} T_{G/H}(t_i)^{a_i}\cdot T_{G/H}(h)\\
&=\prod_{i=1}^{s}t_{i}^{a_i\cdot d}\cdot [t_{i}^{m_i},\alpha_i]^{a_i}\cdot h^{d}[h,\alpha]\\
&\equiv g^d \pmod{Z_2}\quad \text{by equation $(\ref{eqn 2.133})$}\\
&=g^{d}\cdot\varphi_{G/H}(g),
\end{align*}
where $\varphi_{G/H}$ is a correcting function with values in $Z_2$.


\end{proof}

\begin{rem}[\textbf{Properties of the correcting function $\varphi_{G/H}$}]
{\bf (i)} The correcting function $\varphi_{G/H}$ is a function on $G/G^2[G,G]$ with values in $Z_2$.
\begin{proof}
 From Lemma \ref{Lemma 2.10} we have 
 \begin{equation}
  T_{G/H}(g)=g^d\varphi_{G/H}(g),
 \end{equation}
where $\varphi_{G/H}(g)$ is the correcting function. 

We have here $[G,[G,G]]=\{1\}$. This implies $[g,z]=1$ for all $g\in G$ and $z\in [G,G]$.
Since $[G,G]\subseteq\mathrm{Ker}(T_{G/H})$, then for all $x\in [G,G]$ we have 
\begin{center}
 $T_{G/H}(gx)=T_{G/H}(g)T_{G/H}(x)=T_{G/H}(g)$ \hspace{.4cm}for all $g\in G$.
\end{center}
Also here we have $[G,G]^d=\{1\}$, then by using Lemma \ref{Lemma 22}(2) for $x\in [G,G]$ we can write  
\begin{center}
 $(gx)^d=g^dx^d[g,x]^{-\frac{d(d-1)}{2}}=g^d$ \hspace{.4cm}for all $g\in G$.
\end{center}
From Lemma \ref{Lemma 2.10} we also have $T_{G/H}((gx)^d)=(gx)^d\varphi_{G/H}(gx)$. By comparing these above equations
for $x\in[G,G]$ we obtain
\begin{center}
 $\varphi_{G/H}(gx)=\varphi_{G/H}(g)$ \hspace{.4cm}for all $g\in G$.
\end{center}
Moreover, if $x\in G^2$, from Lemma \ref{Lemma 2.10} we have 
\begin{center}
 $T_{G/H}(x)=x^d\varphi_{G/H}(x)=x^d$ since $\varphi_{G/H}(x)=1\in Z_2$\\
 So $T_{G/H}(gx)=T_{G/H}(g)\cdot T_{G/H}(x)=T_{G/H}(g)\cdot x^d$ \hspace{.3cm}for all $g\in G$.
\end{center}
Again from Lemma \ref{Lemma 22}(2) we have for $x\in G^2$
\begin{center}
 $(gx)^d=g^dx^d[g,x]^{-\frac{d(d-1)}{2}}=g^dx^d$ \hspace{.3cm}for all $g\in G$.
\end{center}
By comparing we can see that $\varphi_{G/H}(gx)=\varphi_{G/H}(g)$ for all $g\in G$ and $x\in G^2$.

Thus we can conclude that the correcting function $\varphi_{G/H}$ is a function on $G/G^2[G,G]$ with values in $Z_2$.
\end{proof}
{\bf (ii)} $G^d\subset Z(G)$ if and only if $\rm{Im}(\varphi_{G/H})\subset Z(G)$.
 \begin{proof}
   From Lemma \ref{Lemma 2.10}(iii) we have $T_{G/H}(g)=g^d\cdot\varphi_{G/H}(g)$. From relation (\ref{relation 2.6}) we also know that
 $\rm{Im}(T_{G/H})\subseteq H^{G/H}\subseteq Z(G)$, hence $g^d\cdot\varphi_{G/H}(g)\in Z(G)$.
 Now if $G^d\subset Z(G)$, then $\varphi_{G/H}(g)\in Z(G)$ for all $g\in G$. Hence $\rm{Im}(\varphi_{G/H})\subset Z(G)$.\\
 Conversely, if $\rm{Im}(\varphi_{G/H})\subset Z(G)$, then from $g^d\varphi_{G/H}(g)\in Z(G)$ we can conclude that
 $G^d\subset Z(G)$.
 \end{proof}
 
 {\bf (iii)} When $d$ is odd (resp. even), $\varphi_{G/H}$ is a homomorphism (resp. not a homomorphism).
 \begin{proof}  
Since $T_{G/H}$ is a homomorphism we obtain the identity:
  \begin{equation}\label{eqn 213}
   (g_1g_2)^{d}\varphi_{G/H}(g_1g_2)=g_{1}^{d}g_{2}^{d}\varphi_{G/H}(g_1)\varphi_{G/H}(g_2).
  \end{equation}
This implies
\begin{equation}\label{eqn 214}
 \frac{\varphi_{G/H}(g_1g_2)}{\varphi_{G/H}(g_1)\varphi_{G/H}(g_2)}=\frac{g_{1}^{d}g_{2}^{d}}{(g_1g_2)^{d}}=
 \frac{(g_1g_2)^{d}[g_1,g_2]^{\frac{d(d-1)}{2}}}{(g_1g_2)^{d}}=[g_1,g_2]^{\frac{d(d-1)}{2}}.
\end{equation}
We also have here $[G,G]^d=1$.
When $d$ is odd, then $d$ divides $\frac{d(d-1)}{2}$, hence the right side of equation (\ref{eqn 214}) is equal to $1$.
Thus when $d$ is odd, $\varphi$ is a homomorphism, and exactly $\varphi\equiv 1$. This follows from Lemma \ref{Lemma 2.2}.\\
But when $d$ is even $d$ does not divide $\frac{d(d-1)}{2}$, hence the right side of equation (\ref{eqn 214}) is not equal to 
$1$.
This shows that $\varphi_{G/H}$ is {\bf not} a homomorphism when $d$ is even.
 \end{proof}
 {\bf (iv)} If $H'\subset G$ is another normal subgroup such that $H'$ is abelian and $G/H'$ is abelian of order $d$, then $\varphi_{G/H'}$
is again a function on $G/G^2[G,G]$ with values in $Z_2$ which satisfies the same identity (\ref{eqn 214}), hence we will have 
\begin{center}
 $\varphi_{G/H'}=\varphi_{G/H}\cdot f_{H,H'}$
\end{center}
for some homomorphism $f_{H,H'}\in\mathrm{Hom}(G/G^2[G,G], Z_2)$.

\end{rem}

\section{\textbf{Invariant formula of determinant for Heisenberg representations}}

In general, for the Heisenberg setting $G$ need not be two-step nilpotent group. But $\overline{G}=G/\rm{Ker}(\rho)$
is always  a two-step nilpotent group, where $\rho$ is a Heisenberg representation of $G$.
The Lemmas \ref{Lemma 2.2} and \ref{Lemma 2.10} hold for two-step nilpotent groups. Therefore to use them in our 
Heisenberg setting, we have to do our computation under {\bf modulo $\rm{Ker}(\rho)$}. 
{\bf Our determinant computation is under modulo $\rm{Ker}(\rho)$}. And we drop {\bf modulo $\rm{Ker}(\rho)$}
from our remaining part of this chapter.

Before going to our next proposition we need this following result.
\begin{prop}\label{Proposition 212}
 Let $G$ be an abelian group of $\mathrm{rk}_2(G)=n$. Then $G$ has $2^n-1$ nontrivial elements of order $2$. 
\end{prop}
\begin{proof}
 We know that $(\bbZ_n,+)$ is a cyclic group, where $n\in \bbN$. If $n$ is odd, $\bbZ_n$ does not have any nontrivial 
 element of order $2$.
 But when $n$ is even, it is clear that $\frac{n}{2}\in \bbZ_n$ is the only one nontrivial element of order $2$. So this tells us when 
 $n$ is even, $\bbZ_n$ has a unique element of order $2$. 
 
 Any given abelian group $G$ of $\mathrm{rk}_2(G)=n$ can be written as 
 \begin{equation*}
 G\cong\bbZ_{m_1}\times\bbZ_{m_2}\times\cdots\times\bbZ_{m_s}
\end{equation*}
 where $m_1|m_2|\cdots|m_s$ and $m_{s-n+1},m_{s-n+2},\cdots,m_s$ are $n$ even, and rest of the $m_i$-s are odd.
Therefore from equation (\ref{eqn 444}) we conclude that  
\begin{center}
 $|G[2]|=|\bbZ_{m_1}\times\cdots\times\bbZ_{m_s}[2]|=\prod_{i=1}^{n}|\bbZ_{m_{s-n+i}}[2]|=
 \substack{2\times\cdots\times 2\\\text{$n$-times}}=2^n$.
\end{center}
Hence we can conclude that when $G$ is abelian with $\mathrm{rk}_{2}(G)=n$, it has $2^n-1$ nontrivial elements of order $2$.
 
\end{proof}

\begin{prop}\label{Proposition 2.13}
 Let $\rho=(Z,\chi_\rho)$ be a Heisenberg representation of $G$, of dimension $d$, and put $X_\rho(g_1,g_2):=\chi_\rho\circ [g_1,g_2]$.
 Then we obtain
 \begin{equation}\label{eqn 2.16}
  (\mathrm{det}(\rho))(g)=\varepsilon(g)\cdot\chi_\rho(g^d),
 \end{equation}
where $\varepsilon$ is a function on $G$ with the following properties:
\begin{enumerate}
 \item $\varepsilon$ has values in $\{\pm 1\}$.
 \item $\varepsilon(gx)=\varepsilon(g)$ for all $x\in G^2\cdot Z$, hence $\varepsilon$ is a function on the factor group
 $G/G^2\cdot Z$, and in particular, $\varepsilon\equiv 1$ if $[G:Z]=d^2$ is odd.
 \item If $d$ is even, then the function $\varepsilon$ need not be a homomorphism but:
 \begin{center}
  $\frac{\varepsilon(g_1)\varepsilon(g_2)}{\varepsilon(g_1g_2)}=X_\rho(g_1,g_2)^{\frac{d(d-1)}{2}}=X_\rho(g_1,g_2)^{\frac{d}{2}}$.
 \end{center}
 Furthermore,
 \begin{enumerate}
  \item \textbf{When $\mathrm{rk}_2(G/Z)\ge 4$:} $\varepsilon$ is a homomorphism, and exactly $\varepsilon\equiv 1$.
  \item \textbf{When $\mathrm{rk}_2(G/Z)=2$:}  $\varepsilon$ is not a homomorphism and $\varepsilon$ is a function
  on $G/G^2Z$ such that
  \begin{center}
   $(\det\rho)(g)=\varepsilon(g)\cdot\chi_\rho(g^d)=\begin{cases}
                                                     \chi_\rho(g^d) & \text{for $g\in G^2Z$}\\
                                                     -\chi_\rho(g^d) & \text{for $g\notin G^2Z$.}
                                                    \end{cases}
$
  \end{center}

 \end{enumerate}

\end{enumerate}

\end{prop}

\begin{proof}
  By the given condition, $\rho=(Z,\chi_\rho)$ is a Heisenberg representation of $G$. Let $H$ be a maximal isotopic subgroup for $X_\rho$,
 then we have $\rho=\mathrm{Ind}_{H}^{G}(\chi_H)$, where $\chi_H$ is a linear character of $H$ which extends $\chi_\rho$.
 Then modulo 
$\mathrm{Ker}(\chi_\rho)=\mathrm{Ker}(\rho)\subset Z$ the assumptions of the Lemma \ref{Lemma 2.10} are fulfilled, and therefore:
\begin{align*}
 (\det\rho)(g)
 &=\Delta_{H}^{G}(g)\cdot\chi_H(T_{G/H}(g))\\
 &=\Delta_{H}^{G}(g)\cdot\chi_\rho(T_{G/H}(g))\quad\text{because the values of $T_{G/H}$ are in $Z$}\\
 &=\Delta_{H}^{G}(g)\cdot\chi_\rho(g^d)\chi_{\rho}(\varphi_{G/H}(g))\quad\text{from Lemma $\ref{Lemma 2.10}$}\\
 &=\varepsilon(g)\cdot\chi_\rho(g^d),
\end{align*}
where 
\begin{equation}
 \varepsilon(g):=\Delta_{H}^{G}(g)\cdot \chi_\rho(\varphi_{G/H}(g)).
\end{equation}
Since $\Delta_{H}^{G}$ is the quadratic determinant character of $G$, hence for every $g\in G$, we have $\Delta_{H}^{G}(g)\in\{\pm 1\}$.
And $\varphi_{G/H}(g)\in Z_2$, then $\chi_\rho(\varphi_{G/H}(g))\in\{\pm 1\}$.
Therefore for every $g\in G$,
\begin{equation*}
 \varepsilon(g)=\Delta_{H}^{G}(g)\cdot \chi_\rho(\varphi_{G/H}(g))\in\{\pm 1\},
\end{equation*}
 which does not depend on $H$ because $\Delta_{H}^{G}=\Delta_{1}^{G/H}$, and $\chi_\rho$ does not depend on $H$.
 
Here $Z$ is the scalar group of the irreducible representation $\rho$ of dimension $d$, then by definition of scalar 
group, elements $z\in Z$ are 
represented by scalar matrices, i.e., 
\begin{equation*}
 \rho(z)=\chi_\rho(z)\cdot I_d, \quad \text{where $I_d$ is the $d\times d$ identity matrix}.
\end{equation*}
This implies 
\begin{center}
 $(\det\rho)(z)=\chi_\rho(z)^{d}=\chi_\rho(z^d)$.
\end{center}
We also know that $Z$ is the radical of $X_\rho$, therefore 
\begin{center}
 $X_\rho(z,g)=\chi_\rho([z,g])=1$ for all $z\in Z$ and $g\in G$.
\end{center}
Moreover, we can consider $\det\rho$ as a linear character of $G$, therefore 
\begin{equation}\label{eqn 2.17}
 (\det\rho)(gz)=(\det\rho)(g)\cdot(\det\rho)(z)=\varepsilon(g)\chi_\rho(g^d)\chi_\rho(z^d).
\end{equation}
On the other hand 
\begin{equation}\label{eqn 2.18}
 (\det \rho)(gz)=\varepsilon(gz)\chi_\rho((gz)^d)=\varepsilon(gz)\chi_\rho(g^dz^d[g,z]^{-\frac{d(d-1)}{2}})=\varepsilon(gz)\chi_\rho(g^d)
 \chi_\rho(z^d).
\end{equation}
On comparing equations (\ref{eqn 2.17}) and (\ref{eqn 2.18}) we get 
\begin{center}
 $\varepsilon(gz)=\varepsilon(g)$ for all $g\in G$ and $z\in Z$.
\end{center}
Moreover, since $\varepsilon(g)$ is a sign, we have
\begin{equation*}
 (\det\rho)(g^2)=(\det\rho)(g)^2=\varepsilon(g)^2\chi_\rho(g^d)^2=\chi_\rho(g^{2d}).
\end{equation*}
Therefore
\begin{equation}\label{eqn 2.19}
 (\det\rho)(gx^2)=(\det\rho)(g)\cdot(\det\rho)(x^2)=\varepsilon(g)\chi_\rho(g^d)\chi_\rho(x^{2d}).
\end{equation}
On the other hand
\begin{equation}\label{eqn 2.201}
 (\det\rho)(gx^2)=\varepsilon(gx^2)\chi_\rho((gx^2)^d)=\varepsilon(gx^2)\chi_\rho(g^d)\chi_\rho(x^{2d}).
\end{equation}
So we see from equations (\ref{eqn 2.19}) and (\ref{eqn 2.201}) $\varepsilon(gx^2)=\varepsilon(g)$, hence 
$\varepsilon$ is a function on $G/G^2Z$.
 
 In particular, when $[G:Z]=d^2$ is odd, i.e.,
$|G/H|=d$ is odd, we have $\varphi_{G/H}(g)=1$ as well $\Delta_{H}^{G}(g)=1$ because $H$ is normal subgroup of odd index in $G$.
This shows that $\varepsilon\equiv 1$ when $[G:Z]=d^2$ is odd.

For checking property (iii), we use equation (\ref{eqn 2.16}) and $[G,G]^d=\{1\}$.
Since $[G,G]^d=\{1\}$, we have for $g_1,g_2\in G$
\begin{center}
 $([g_1,g_2]^{\frac{d(d-1)}{2}})^{2}=1$, i.e.,
$[g_1,g_2]^{\frac{d(d-1)}{2}}=\frac{1}{[g_1,g_2]^{\frac{d(d-1)}{2}}}$. Also,\\
$[g_1,g_2]^{d-1}=[g_1,g_2]^{-1}$ and $[g_1,g_2]^{\frac{d}{2}}=[g_1,g_2]^{-\frac{d}{2}}$.
\end{center}
From equation (\ref{eqn 2.16}) we obtain
\begin{equation*}
 \frac{(\det\rho)(g_1)\cdot(\det\rho)(g_2)}{(\det\rho)(g_1g_2)}=
 \frac{\varepsilon(g_1)\chi_\rho(g_{1}^{d})\cdot\varepsilon(g_2)\chi_\rho(g_{2}^{d})}{\varepsilon(g_1g_2)\chi_\rho((g_1g_2)^{d})}.
\end{equation*}
This implies
\begin{align}
 \frac{\varepsilon(g_1)\varepsilon(g_2)}{\varepsilon(g_1g_2)}\nonumber
 &=\frac{\chi_\rho((g_1g_2)^d)}{\chi_\rho(g_{1}^{d}g_{2}^{d})}=\chi_\rho([g_1,g_2])^{\frac{d(d-1)}{2}}\\
 &=X_\rho(g_1,g_2)^{\frac{d(d-1)}{2}}=X_\rho(g_1,g_2)^{\frac{d}{2}}.\label{eqn 2.22}
\end{align}
This shows that $\varepsilon$ need not be a homomorphism when $d$ is even.

But when $|G/Z|=d^2$ and $d$ is even we can write
\begin{align*}
 G/Z
 &\cong(\mathbb{Z}/m_1\times\mathbb{Z}/m_1)\times\cdots\times(\mathbb{Z}/m_s\times\mathbb{Z}/m_s)\\
 &\cong(<t_1>\times<t'_1>)\perp\cdots\perp(<t_s>\times<t'_s>),
\end{align*}
such that $m_1|\cdots|m_s$ and $\prod_{i=1}^{s}m_{i}^{2}=d^2$, $X_\rho(t_i,t'_i)=\chi_\rho([t_i,t'_i])=\zeta_{m_i}$, a 
primitive $m_i$-th root of unity because $[t_i,t'_i]^{m_i}=1$.
If $m_{s-1}$, $m_s$ are both even\footnote{Here $d=m_1\cdots m_{s-1}m_s$, if both $m_{s-1}, m_s$ are even, then 
$\frac{d}{2}=m_1\cdots(\frac{m_{s-1}}{2})\cdot m_s$. This shows that $m_i|\frac{d}{2}$ for all $i\in \{1,\cdots,s\}$.
Therefore, $X_\rho(t_i,t'_i)^{\frac{d}{2}}=(\zeta_{m_i})^{\frac{d}{2}}=1$ for all $i\in\{1,\cdots,s\}$.}, 
which means $2$-rank of $G/Z$ is $\geq 4$ then $\frac{d}{m_s}$ is even and therefore $X_\rho(x,y)^{\frac{d}{2}}\equiv 1$, hence
from equation (\ref{eqn 2.22}) we see that $\varepsilon$ is a homomorphism.

Moreover, from the above we see that 
\begin{center}
 $H/Z=<t_1>\times\cdots\times <t_s>\cong H'/Z=<t_1'>\times\cdots<t_s'>$
\end{center}
are two maximal isotropic which are isomorphic. We have $H\cap H'=Z$, hence $G$ is not the direct product of $H$ and $H$ but 
nevertheless $G=H\cdot H'$. So for any $g\in G$ there must exist a decomposition $g=h\cdot h'$, where $h\in H$ and $h'\in H'$.

Now we assume $\mathrm{rk}_{2}(G/Z)\ne 2$, hence $\mathrm{rk}_2(H/Z)=\mathrm{rk}_2(H'/Z)\ne 1$. And since $G/H\cong H/Z$
and $G/H'\cong H'/Z$, then $\mathrm{rk}_2(G/H)=\mathrm{rk}_2(G/H')\ne 1$. Then from Proposition \ref{Proposition 212} we can 
say both $G/H$ and $G/H'$ have at least $3$ elements of order $2$. Then from Theorem \ref{Theorem Miller} we have 
 $\alpha_{G/H}=1$ and  $\alpha_{G/H'}=1$.
Furthermore from formula (\ref{eqn 29}) we obtain
\begin{equation*}
 T_{G/H}(h)=h^d\cdot[h,\alpha_{G/H}]=h^d,\quad \text{and}\quad T_{G/H}(h')=h'^d\cdot[h',\alpha_{G/H'}]=h'^d.
\end{equation*}

So we can write 
\begin{align*}
 (\det\rho)(g)
 &=(\det\rho)(h)\cdot(\det\rho)(h'),\quad\text{here $g=h\cdot h'$ is a decomposition of $g$ with $h\in H,$ $h'\in H'$},\\
 &=\chi_\rho(h^d)\cdot\chi_\rho(h'^d),\quad\text{because $\mathrm{rk}_2(G/H')=\mathrm{rk}_2(G/H')\ne 1$},\\
 &=\chi_\rho(h^d\cdot h'^d)\\
 &=\chi_\rho((h\cdot h')^d[h,h']^{\frac{d(d-1)}{2}})\quad\text{using Lemma $\ref{Lemma 22}(2)$}\\
 &=\chi_\rho(g^d)\cdot X_\rho(h,h')^{\frac{d(d-1)}{2}}\\
 &=\chi_\rho(g^d),
\end{align*}
because all $m_i|\frac{d}{2}$, $i\in\{1,2,\cdots,s\}$, and then 
\begin{center}
 $X_\rho(h,h')^{\frac{d(d-1)}{2}}=\chi_\rho([h,h'])^{\frac{d(d-1)}{2}}=\zeta_m^{\frac{d(d-1)}{2}}=1$,
\end{center}
where $\zeta_m$ is a primitive $m$-th root of unity and $m$ is some positive integer (which is the order of $[h,h']$) 
 which divides $\frac{d}{2}$.
This shows that when $\mathrm{rk}_2(G/Z)\ne 2$ we 
have $\varepsilon\equiv 1$.

If on the other hand only $m_s$ is even, i.e., $G/Z$ has $2$ rank$=2$, then $\frac{d}{m_s}$ is odd. Therefore 
$X_\rho(t_s,t'_s)^{\frac{d}{2}}=(\zeta_{m_s})^{\frac{d}{2}}=-1$, since $m_s$ does not divide $\frac{d}{2}$ and 
$(\zeta_{m_s}^{\frac{d}{2}})^2=1$. Therefore $\varepsilon$ cannot be a homomorphism when 
$\mathrm{rk}_{2}(G/Z)=2$.

But since $\varepsilon$ is a function on $G/G^2Z$, hence 
$\varepsilon|_{G^2Z}\equiv 1$. Therefore when $g\in G^2Z$ we have $(\det\rho)(g)=\chi_\rho(g^d)$. So now we are left to show
that for $g\notin G^2Z$, $\varepsilon(g)=-1$, i.e., $(\det\rho)(g)=-\chi_\rho(g^d)$. 
Also, for $\mathrm{rk}_2(G/Z)=2$,
$G/G^2Z$ is Klein's $4$-group\footnote{Since $G/Z$ is an abelian group, we have $G/Z\cong\widehat{G/Z}$. When $\rm{rk}_2(G/Z)=2$,
by Proposition \ref{Proposition 212}, there are exactly three elements of order $2$ in $G/Z$, and this each element (i.e., 
self-inverse element) corresponds a 
quadratic character of $G/Z$. Hence the group $G/G^2Z$ has exactly three quadratic characters. Furthermore,
$G/G^2Z$ is a quotient group of the abelian group $G/Z$, hence $G/G^2Z$ is abelian. Therefore $G/G^2Z$ is isomorphic to 
the Klein's 4-group.} 
and $\varepsilon$ is a sign function on that group. So up to permutation the possibilities are 
\begin{enumerate}
 \item $+ + + +$
 \item $+ + + -$
 \item $+ + - -$
 \item $+ - - -$
\end{enumerate}
The cases $(1)$, $(3)$ can be excluded because we know that $\varepsilon$ is not a homomorphism. So we have to exclude the case $(2)$
and for this it is enough to see that we must have "$-$'' more than once. 

If we restrict $\varepsilon$ to a maximal isotropic subgroup $H$, then from equation (\ref{eqn 2.22}) we can say $\varepsilon$ is a
homomorphism on $H$, because $X_\rho|_{H\times H}=1$. 
We also have from equation (\ref{eqn 29}) $T_{G/H}(h)=h^d\cdot[h,\alpha_{G/H}]=h^d\varphi_{G/H}(h)$.
This implies $\varphi_{G/H}(h)=[h,\alpha_{G/H}]$ for all $h\in H$. 
Moreover, since $\Delta_{H}^{G}|_{H}\equiv 1$, then for $h\in H$ we obtain:
\begin{center}
 $\varepsilon(h)=\Delta_{H}^{G}(h)\cdot\chi_\rho(\varphi_{G/H}(h))=\chi_\rho([h,\alpha_{G/H}])$.
\end{center}
If there exists a maximal isotropic subgroup $H$ of $\rm{rk}_2(H/Z)=1$, then from the Proposition \ref{Proposition 212} can say that 
$H/Z$ has a unique element of order $2$. We also know $G/H\cong H/Z$ because $G/H$ and $H/Z$ are both  
finite abelian groups of same order
$d$, hence $\rm{rk}_2(H/Z)=\rm{rk}_2(G/H)=1$. Then from the Proposition \ref{Proposition 212}, $G/H$ has a unique element of order $2$,
and therefore by Miller's theorem we have $\alpha_{G/H}\ne 1$. Thus for the case $\rm{rk}_2(H/Z)=1$
we have 
\begin{equation}\label{eqn 3.25}
 \varepsilon(h)=\Delta_{H}^{G}(h)\cdot\chi_\rho([h,\alpha_{G/H}])=\chi_\rho([h,\alpha_{G/H}])=-1
\end{equation}
 for all nontrivial $h\in H$.

Moreover,
if $\rm{rk}_2(G/Z)=2$, then from the Lemma \ref{Theorem 2.4}
there exists subgroups $H$, $H'$ with the following properties
\begin{enumerate}
 \item $\rm{rk}_2(H/Z)=\rm{rk}_2(H'/Z)=1$
 \item $G=H\cdot H'$
 \item $Z=H\cap H'$
\end{enumerate}
Then $H/G^2Z$ and $H'/G^2Z$ are two different subgroups of order $2$ in Klein's $4$-group. Now take the nontrivial elements of these 
subgroups are $h$ and $h'$ respectively. Then by using equation (\ref{eqn 3.25}) we have 
\begin{center}
 $\varepsilon(h)=\varepsilon(h')=-1$,\\
 i.e., the nontrivial elements of $H/G^2Z$ and $H'/G^2Z$ give the two "$-$'' signs for $\varepsilon$.
\end{center}
Therefore the only possibility is $+ - - -$, i.e., $\varepsilon$ takes $1$ on the trivial coset and $-1$ on the three other cosets.

 This completes the proof.
\end{proof}

\begin{cor}
\begin{enumerate}
 \item Let $\rho=(Z,\chi_\rho)\in\mathrm{Irr}(G)$ be a Heisenberg representation of odd dimension $d$. Then 
 $G^d\subseteq Z$ and 
 \begin{center}
$\mathrm{det}(\rho)(g)=\chi_\rho(g^d)$, \hspace{.4cm} for all $g\in G$.  
 \end{center}
In particular, $\mathrm{det}(\rho)\equiv 1$ if and only if $\chi_\rho$ is a character of $Z/G^d$. This is only possible if 
$[G,G]\not\subseteq G^d$ and if $\chi_\rho$ is a nontrivial character on $G^d[G,G]/G^d\subseteq Z/G^d$.
\item Let $\omega$ be a linear character of $G$, then $\rho\otimes\omega=(Z,\chi_{\rho\otimes\omega})$, where:
\begin{center}
 $\chi_{\rho\otimes\omega}=\chi_\rho\cdot\omega_Z$, \hspace{.4cm} $\mathrm{det}(\rho\otimes\omega)=\mathrm{det}(\rho)\cdot\omega^d$,
\end{center}
where $\omega_Z=\omega|_{Z}$.
Therefore it is possible to find $\omega$ such that $\mathrm{det}(\rho\otimes\omega)\equiv 1$, equivalently 
$\chi_\rho=\omega_{Z}^{-1}$ on $G^d$, if and only if $\chi_\rho$ is trivial on $G^d\cap[G,G]$.
\end{enumerate}
\end{cor}

\begin{proof}
{\bf (1).}
 We consider $H$ such that $Z\subset H\subset G$ and $H$ is maximal isotropic with respect to 
 \begin{center}
  $X(g_1,g_2):=\chi_\rho\circ[g_1,g_2]$.
 \end{center}
By definition, $Z/[G,G]$ is radical of $X$, hence $\mathrm{Ker}(\rho)=\mathrm{Ker}(\chi_\rho)\subset Z$, and factorizing by 
$\mathrm{Ker}(\chi_\rho)$ we obtain a group $\overline{G}=G/\mathrm{Ker}(\chi_\rho)$ 
which satisfies the assumptions of the Lemma \ref{Lemma 2.2}. Moreover, 
$\rho=\mathrm{Ind}_{H}^{G}\chi_H$ for any extension $\chi_H$ of $\chi_\rho$, hence 
\begin{center}
 $\mathrm{det}(\rho)=\Delta_{H}^{G}\cdot(\chi_H\circ T_{G/H})=\chi_H\circ T_{G/H}$,
\end{center}
because $H\subset G$ is a normal subgroup of odd index $d$. Applying the Lemma \ref{Lemma 2.2} we obtain for all $g\in G$:
\begin{equation}\label{eqn 2.33}
 \mathrm{det}(g)=\chi_H\circ T_{G/H}(g)=\chi_H(g^d)=\chi_\rho(g^d),\quad\text{since $g^d\in Z$},
\end{equation}
here we have used $g^d=T_{G/H}(g)\in\mathrm{Im}(T_{G/H})\subseteq H^{G/H}\subseteq Z$ because our 
computation is modulo $\mathrm{Ker}(\chi_\rho)$.

If $\mathrm{det}(\rho)\equiv 1$, then from equation (\ref{eqn 2.33}), we have $\chi_\rho(g^d)=1$, i.e.,
$g^d\in\mathrm{Ker}(\chi_\rho)$ for all $g\in G$. This shows $G^d\subseteq \mathrm{Ker}(\chi_\rho)$.
Again if $G^d\subseteq\mathrm{Ker}(\chi_\rho)$, then it is easy to see $\det\rho\equiv 1$.

Now if $\chi_\rho:Z/G^d\to\mathbb{C}^\times$ is a character, then we see that $G^d\subseteq\mathrm{Ker}(\chi_\rho)$, i.e., 
$g^d\in \mathrm{Ker}(\chi_\rho)$. Thus from equation (\ref{eqn 2.33}), we conclude $\mathrm{det}(\rho)\equiv 1$.

If $[G,G]\subseteq G^d$ this would imply 
$[G,G]\subset\mathrm{Ker}(\chi_\rho)$ which means $Z=G$, hence $\rho$ is of dimension $1$. Also, if $\chi_\rho$ is trivial on 
$G^d[G,G]/G^d$, i.e., $[G,G]\subseteq \mathrm{Ker}(\rho)$, hence dimension of $\rho$ is $1$. Therefore when $\det\rho\equiv 1$
and $[G,G]\not\subseteq G^d$ and $\chi_\rho$ is a nontrivial character on 
$G^d[G,G]/G^d\subseteq Z/G^d$, then we can extend $\chi_\rho$ to $Z/G^d$ because $\chi_\rho$ is $G$-invariant and $G^d[G,G]/G^d$ is a
normal subgroup of $Z/G^d$.\\
{\bf (2).}
Let $\omega$ be a linear character of $G$ and $\omega_Z=\omega|_{Z}$. Then we can write (cf. \cite{JP}, p. 57, Remark (3))
\begin{center}
 $\omega\otimes\mathrm{Ind}_{Z}^{G}\chi_\rho=\mathrm{Ind}_{Z}^{G}(\chi_\rho\otimes\omega_Z)=
 \mathrm{Ind}_{Z}^{G}\chi_{\rho\otimes\omega}=d\cdot\rho\otimes\omega$,
\end{center}
where $\chi_{\rho\otimes\omega=\chi_\rho\cdot\omega_Z}$ and $d=\mathrm{dim}(\rho)$. Moreover, it is easy to see 
$\chi_\rho\otimes\omega_Z$ is a $G$-invariant. Therefore we can write $\rho\otimes\omega=(Z,\chi_{\rho\otimes\omega})$.
Now we are left to compute determinant of $\rho\otimes\omega$, which follows from the properties of determinant function
(cf. \cite{GK}, p. 955, Lemma 30.1.3). Since $\mathrm{dim}(\rho)=d$ and $\omega$ is linear, we have 
\begin{equation*}
 \mathrm{det}(\rho\otimes\omega)=\mathrm{det}(\rho)^{\omega(1)}\cdot\mathrm{det}(\omega)^{\rho(1)}=\mathrm{det}(\rho)\cdot \omega^{d}.
\end{equation*}
Here $d$ is odd and we know $g^d\in Z$,
then for every $g\in G$, we have 
\begin{equation}\label{eqn 2.44}
 \mathrm{det}(\rho\otimes\omega)(g)=\chi_\rho(g^d)\cdot\omega^d(g)=\chi_\rho\cdot\omega_Z(g^d).
\end{equation}
Now if $\mathrm{det}(\rho\otimes\omega)\equiv 1$, then we have $\chi_\rho=\omega_{Z}^{-1}$ on $G^d$. This implies,
it is possible to find a linear character $\omega$ such that $\mathrm{det}(\rho\otimes\omega)\equiv 1$.

Now let $\chi_\rho=\omega_{Z}^{-1}$ on $G^d$. Since $G^d\cap [G,G]\subseteq G^d$, $\chi_\rho\cdot\omega_Z(g)=1$ for $g\in G^d\cap[G,G]$.
Then $\chi_\rho$ is trivial on $G^d\cap[G,G]$.

Conversely, if $\chi_\rho$ is trivial on $G^d\cap [G,G]$, then we are left to show that we can find an $\omega$ such that  
$G^d\subseteq \mathrm{Ker}(\chi_\rho\cdot\omega_Z)$.

Put  $Z_1= G^d\cdot [G,G]$, and  $Z_0= G^d\cap [G,G]$.  Then we have
$Z\supset Z_1\supset Z_0$,  and
$Z_1/Z_0  =  G^d/Z_0  \times [G,G]/Z_0$        is a direct product.
Now assume that  $\chi_\rho$  is a character of  $Z/Z_0$.
Then the restriction   $\chi_{Z_1} = \chi_1\cdot \chi_2$
comes as the product of two characters of  $Z_1$, where  $\chi_1$ is trivial on $[G,G]$
and $\chi_2$  is trivial on $G^d$. But then we can find  $\omega$ of  $G/[G,G]$
such that  $\omega_{Z_1}= \chi_1$,  hence
  $\omega^{-1}\chi_\rho$   restricted to $Z_1$  is equal $\chi_2$.
In particular,   $\omega^{-1}\chi_\rho$   is trivial on  $G^d$
and therefore   $\omega^{-1}\otimes \rho$  has  $\mathrm{det}(\omega^{-1}\otimes \rho) \equiv 1$.
\end{proof}

\begin{cor}\label{Corollary 2.6}
 If $\rho=(Z,\chi_\rho)$ is a Heisenberg representation (of dimension $>1$) 
 for a nonabelian group of order $p^3, (p\neq 2,)$ then $Z=[G,G]$
 is cyclic group of order $p$, and $G^p=Z$ or $G^p=\{1\}$ depending on the isomorphism type of $G$. So we have 
 $\mathrm{det}(\rho)\not\equiv 1 $ and $\mathrm{det}(\rho)\equiv 1$ depending on the isomorphism type of $G$.
\end{cor}
\begin{proof}
 In this particular case for Heisenberg setting we have $|Z|=p$ and $G/Z$ is abelian. This implies $[G,G]\subseteq Z$.
 Here $G$ is nonabelian and $p$ is prime, therefore $[G,G]=Z$ is a cyclic group of order $p$.
 Let $\Psi: G\to G$ be a $p$-power map, i.e., $g\mapsto g^p$.
 It can be proved that this map $\Psi$ is a surjective group homomorphism (by using Lemma \ref{Lemma 22}) and the image is in 
 $Z$ (because from Lemma \ref{Lemma 2.2} and relation (\ref{relation 2.6}) we have $g^p=T_{G/H}(g)\in Z$),
 hence from the first isomorphism theorem
 we have 
 \begin{center}
  $G/\mathrm{Ker}(\Psi)\cong\mathrm{Im}(\Psi)=G^{p}$.
 \end{center}
 Thus we can write
 \begin{center}
  $p^3=|\mathrm{Ker}(\Psi)|\cdot|G^p|$.
 \end{center}
So we have the possibility 
\begin{center}
 $|\mathrm{Ker}(\Psi)|=p^3$ or $p^2$ corresponding to $G^p=1$ or $G^p=Z$.
\end{center}
Both the cases are possible depending on the isomorphism type of $G$.
So we can conclude that 
 $\mathrm{det}(\rho)\not\equiv 1 $ and $\mathrm{det}(\rho)\equiv 1$ depending on the isomorphism type of $G$.

\end{proof}

\begin{rem}
 We know that there are two non-abelian group of order $p^3$, up to isomorphism (for details see \cite{KC1}). Now put 
\begin{equation*}
 G_{p}=\left\{\begin{pmatrix}
  1 & a & b  \\
  0 & 1 & c  \\
  0 & 0 & 1
        \end{pmatrix}: a,b,c\in \mathbb{Z}/p\mathbb{Z}\right\}
\end{equation*}
We observe that this $G_p$ is a non-abelian group under matrix multiplication with order $p^3$. We also see that 
$G_{p}^{p}=\{I_3\}$, the identity
in $G_{p}$. Now if $\rho$ is a Heisenberg representation of group $G_p$, then we will have $\mathrm{det}(\rho)\equiv 1$.

And when $G$ is extraspecial group of order $p^3$, where $p\neq 2$ with $G^p=Z$, we will have 
 $\mathrm{Ker}(\Psi)\cong C_p\times C_p$, where $C_p$ is the cyclic group of order $p$.
 Therefore $\mathrm{det}(\rho)(g)=\chi_{Z}(g^p)$. This shows that 
$\mathrm{det}(\rho)\not\equiv 1$.

 From  Corollary \ref{Corollary 2.6}, we observe that 
for non-abelian group of order $p^3$, where $p$ is prime, the determinant of Heisenberg representation of  $G$ gives the information
about the isomorphism type of $G$.

\end{rem}

\begin{rem}
 When the dimension $d$ of $\rho$ is odd, the $2$-rank of $G/Z$ is $0$, and we notice that in this case $\varepsilon\equiv 1$. Therefore
 we could rephrase our above Proposition \ref{Proposition 2.13} as follows:
 \begin{enumerate}
  \item If the $2$-rank of $G/Z$ is different from $2$, we have 
  \begin{center}
    $(\det\rho)(g)=\chi_\rho(g^d)$.
  \end{center}
\item If the $2$-rank of $G/Z$ is equal to $2$, then $G/G^2Z$ is Klein's $4$-group, and we
  have a sign function $\varepsilon$ on $G/G^2Z$ such that 
  \begin{center}
   $(\det\rho)(g)=\varepsilon(g)\cdot\chi_\rho(g^d)$.
  \end{center}
Moreover, the function $\varepsilon$ is not a homomorphism and it takes $1$ on the trivial coset and $-1$ on the three other cosets.
Thus when $\rm{rm}_2(G/Z)=2$, we can write
\begin{center}
 $(\det\rho)(g)=\begin{cases}
                 \chi_\rho(g^d) & \text{for $g\in G^2Z$}\\
                 -\chi_\rho(g^d) & \text{for $g\notin G^2Z$}.
                \end{cases}
$
\end{center}

 \end{enumerate}

\end{rem}

\begin{exm}

Let $G$ be a dihedral group of order $8$ and we write 
\begin{center}
 $G=\{e,b,a,a^2,a^3,ab,a^2b,a^3b\;|\quad a^4=b^2=e, bab^{-1}=a^{-1}\}$. 
\end{center}
It is easy to see that $[G,G]=\{e,a^2\}=Z=Z(G)=G^2$.
Thus $G^2Z=\{e,a^2\}$. We also have
\begin{center}
 $G/G^2Z\cong\{G^2Z, aG^2Z, abG^2Z, bG^2Z\}$, and the subgroups of order $4$ are:\\
 $H_1=\{e,b,a^2,a^2b\}$, $H_2=\{e,ab,a^2,a^3b\}$ and $H_3=\{e,a,a^2,a^3\}$.\\ And the factor groups are:
 $G/H_1=\{H_1,aH_1\}$, $G/H_2=\{H_2,aH_2\}$ and $G/H_3=\{H_3,bH_3\}$.
\end{center}
Let $\rho$ be a Heisenberg representation of $G$. The dimension
of $\rho$ is $2$. In this case the $2$-rank of $G/Z$ is $2$, then $G/Z=G/G^2Z$ is Klein's $4$-group.

For this group we can see that $\varepsilon$ is a function on $G/Z\cong\{Z,aZ,bZ,ab Z\}$.
When $g=a\in G$ , we have $H_1$ for which $g=a\in H_1$, then we have
\begin{center}
 $\varepsilon(a)=\Delta_{H_1}^{G}(a)\cdot\chi_\rho([a,\alpha_{G/H_1}])=\chi_\rho([a,\alpha_{G/H_1}])=-1$.
\end{center}
Similarly, we can see when $g\in bZ$ and $g\in abZ$ we have $\varepsilon(g)=-1$. Thus we can conclude for dihedral group
of order $8$, that
\begin{center}
 $(\det\rho)(g)=\begin{cases}
                 \chi_\rho(g^2) & \text{for $g\in Z$}\\
                 -\chi_\rho(g^2) & \text{for $g\notin Z$}.
                \end{cases}
$
\end{center}

\end{exm}


For our next remark we need the following lemma.

\begin{lem}\label{Lemma 2.13}
 Let $\rho=(Z,\chi_\rho)$ be a Heisenberg representation of $G$ and put $X_\rho(g_1,g_2):=\chi_\rho([g_1,g_2])$.
 Then for every element $g\in G$, there exists a maximal isotropic subgroup $H$ for $X_\rho$ such that 
 $g\in H$.
\end{lem}
\begin{proof}

Let $g$ be a nontrivial element in $G$. Now we take a cyclic subgroup $H_0$ generated by $g$, i.e., $H_0=<g>$. Then 
$X_\rho(g,g)=1$ implies $H_0\subseteq H_{0}^{\perp}$. If $H_0$ is not maximal isotropic, then the inclusion is proper and $H_0$ together
with some $h\in H_{0}^{\perp}\setminus H_0$ generates some larger isotropic subgroup $H_1\supset H_0$. Again we have 
$H_1\subseteq H_{1}^{\perp}$, and if $H_1$ is not maximal then the inclusion is proper, then again we proceed same method and will 
have another isotopic subgroup and we continue this process step by step  come to maximal isotropic subgroup $H$.

Therefore for every element $g\in G$, we would have a maximal subgroup $H$ such that $g\in H$.    
\end{proof}

\begin{rem}
  
From equation (\ref{eqn 2.22}) we can say that $\varepsilon$ is a homomorphism when it restricts to $H$
because $X_\rho|_{H\times H}=1$.
Also from Lemma
\ref{Lemma 2.13}, if $g\in G$, then there always exists a maximal isotropic subgroup
$H$ such that $g=h\in H$. Since $H$ is normal, then 
 $\Delta_{H}^{G}=\Delta_{1}^{G/H}$ is trivial on $H$ and we see in 
particular for $h\in H$
\begin{equation}\label{eqn 2.24}
 (\mathrm{det}\rho)(h)=\chi_H(h)^{d}\chi_H(\varphi_{G/H}(h))=\chi_\rho(h^d)\cdot\chi_\rho([h,\alpha_{G/H}]),
\end{equation}
since $[h,\alpha_{G/H}]\in Z$.

The formula (\ref{eqn 2.24}) reformulates as:
\begin{equation}\label{eqn 2.25}
 (\det\rho)(h)=\chi_\rho(h^d)\cdot\chi_\rho([h,\alpha_{G/H}])=\chi_\rho(h^d)\cdot X_\rho(h,\alpha_{G/H}),
\end{equation}
if $g=h$ sits in some maximal isotropic $H$, and $\alpha_{G/H}\in G/H$ is as above (the product over all elements from $G/H$).
Of course some $g\in G$ can sit in several different maximal isotropic $H$. So it is a little mysterious that the result 
(=left side of the formula (\ref{eqn 2.25})) is independent from that $H$. Moreover, if $H$ is maximal isotropic then 
\begin{center}
 $G/H\cong\widehat{H/Z}$, \hspace{.5cm} $g\mapsto\{h\mapsto X(h,g)\}$.
\end{center}
Now $\alpha_{G/H}\ne 1$ means that $G/H$ has precisely one element of order $2$, equivalently $H/Z$ has precisely one 
character of order $2$, equivalently $H/H^2Z$ is of order $2$. Therefore equation (\ref{eqn 2.25}) can be reformulated as to
say that in the critical case:
\begin{center}
 $(\det\rho)(h)=\chi_\rho(h^d)$ or $=-\chi_\rho(h^d)$ depending on $h\in H^2Z$ or $h\not\in H^2Z$.
\end{center}

 \end{rem}

\chapter{\textbf{The local constants for Heisenberg representations}}

In this chapter we give an invariant formula of local constant for a Heisenberg representation $\rho$ of the absolute Galois group 
$G_F$ of a non-archimedean local field $F/\bbQ_p$ (cf. Theorem \ref{Theorem invariant odd}).
But for giving more explicit invariant formula for $W(\rho)$, we need to know the full information about the 
dimension of a Heisenberg representation. In Theorem \ref{Dimension Theorem}, we compute the dimension of a Heisenberg representation
$\rho$ of $G_F$.

In Section 5.1, we define U-isotopic Heisenberg representations and study their properties (e.g., dimensions, Artin conductors,
Swan conductors). In Theorem \ref{invariant formula for minimal conductor representation}, we give an invariant formula
of local constant of a minimal conductor Heisenberg representation $\rho$ of dimension prime to $p$. And when 
$\rho$ is not minimal conductor but dimension is prime to $p$, we have 
Theorem 
\ref{Theorem using Deligne-Henniart}.

In Section 5.3, we also discuss Tate's root-of-unity criterion, and by applying this Tate's criterion we give the information
when $W(\rho)$ will be a root of unity or not. This chapter is based on the article \cite{SAB2}.

\section{\textbf{Arithmetic description of Heisenberg representations}}

In Section 2.6 of Chapter 2, we became familiar with the notion of Heisenberg representations of a (pro-)finite group.
These Heisenberg representations have arithmetic structure due to E.-W. Zink (cf. \cite{Z2}, \cite{Z4}, \cite{Z5}).
For this chapter we need to describe the arithmetic structure of Heisenberg representations.

Let $F/\bbQ_p$ be a local field, and $\overline{F}$ be an algebraic closure of $F$. Denote $G_F=\rm{Gal}(\overline{F}/F)$ the 
absolute Galois group for $\overline{F}/F$. We know that (cf. \cite{HK2}, p. 197) each representation $\rho:G_F\to GL(n,\bbC)$ corresponds 
to a projective 
representation $\overline{\rho}:G_F\to GL(n,\bbC)\to PGL(n,\bbC)$. On the other hand, each projective representation 
$\overline{\rho}:G_F\to PGL(n,\bbC)$ can be lifted to a representation $\rho:G_F\to GL(n,\bbC)$.
Let $A_F=G_{F}^{ab}$ be the factor commutator group of $G_F$. Define 
\begin{center}
 $FF^\times:=\varprojlim(F^\times/N\wedge F^\times/N)$
\end{center}
where $N$ runs over all open subgroups of finite index in $F^\times$. Denote by $\rm{Alt}(F^\times)$ as the set of 
all alternating characters $X:F^\times\times F^\times\to\bbC^\times$ such that $[F^\times:\rm{Rad}(X)]<\infty$. Then the local 
reciprocity map gives an isomorphism between $A_F$ and the profinite completion of $F^\times$, and induces a natural bijection 
\begin{equation}
 \rm{PI}(A_F)\xrightarrow{\sim}\rm{Alt}(F^\times),
\end{equation}
where $\rm{PI}(A_F)$ is the set of isomorphism classes of projective irreducible representations of $A_F$.
By using class field theory from the commutator map (\ref{eqn 2.6.3}) (cf. p. 125 of \cite{Z5}) we obtain 
\begin{equation}\label{eqn 5.1.2}
 c:FF^\times\cong [G_F,G_F]/[[G_F,G_F], G_F].
\end{equation}
 
Let $K/F$ be an abelian extension corresponding to the norm subgroup $N\subset F^\times$ and if $W_{K/F}$ denotes the relative Weil 
group, the commutator map for $W_{K/F}$ induces an isomorphism (cf. p. 128 of \cite{Z5}):
\begin{equation}\label{eqn 5.1.3}
 c: F^\times/N\wedge F^\times/N\to K_{F}^{\times}/I_{F}K^\times,
\end{equation}
where 
\begin{center}
 $K_{F}^{\times}:=\{x\in K^\times|\quad N_{K/F}(x)=1\}$, i.e., the norm-1-subgroup of $K^\times$,\\
 $I_FK^\times:=\{x^{1-\sigma}|\quad x\in K^{\times}, \sigma\in \rm{Gal}(K/F)\}<K_{F}^{\times}$, the augmentation with respect to $K/F$. 
\end{center}
Taking the projective limit over all abelian extensions $K/F$ the isomorphisms (\ref{eqn 5.1.3}) induce:
\begin{equation}\label{eqn 5.1.4}
 c:FF^\times\cong \varprojlim K_{F}^{\times}/I_FK^\times,
\end{equation}
where the limit on the right side refers to norm maps. This gives an arithmetic description of Heisenberg representations of the 
group $G_F$.

\begin{thm}[Zink, \cite{Z2}, p. 301, Corollary 1.2]\label{Theorem 5.1.1}
 The set of Heisenberg representations $\rho$ of $G_F$ is in bijective correspondence with the set of all pairs $(X_\rho,\chi_\rho)$
 such that:
 \begin{enumerate}
  \item $X_\rho$ is a character of $FF^\times$,
  \item $\chi_\rho$ is a character of $K^{\times}/I_FK^\times$, where the abelian extension $K/F$ corresponds to the radical 
  $N\subset F^\times$ of $X_\rho$, and 
  \item via (\ref{eqn 5.1.3}) the alternating character $X_\rho$ corresponds to the restriction of $\chi_\rho$ to $K_{F}^{\times}$.
 \end{enumerate}

\end{thm}
 Given a pair $(X,\chi)$, we can construct the Heisenberg representation $\rho$ by induction from $G_K:=\rm{Gal}(\overline{F}/K)$ to 
 $G_F$:
 \begin{equation}\label{eqn 5.1.5}
 \sqrt{[F^\times:N]}\cdot\rho=\rm{Ind}_{K/F}(\chi),
 \end{equation}
where $N$ and $K$ are as in (2) of the above Theorem \ref{Theorem 5.1.1}
and where the induction of $\chi$ (to be considered as a character of $G_K$ by class
field theory) produces a multiple of $\rho$. From 
$[F^\times:N]=[K:F]$ we obtain the {\bf dimension formula:}
\begin{equation}\label{eqn dimension formula}
 \rm{dim}(\rho)=\sqrt{[F^\times:N]},
\end{equation}
where $N$ is the radical of $X$.

Let $K/E$ be an extension of $E$, and $\chi_K:K^\times\to\bbC^\times$ be a character of $K^\times$. In the following lemma, we 
give the conditions of the existence of characters $\chi_E\in\widehat{E^\times}$ such that $\chi_E\circ N_{K/E}=\chi_K$, 
and the solutions set 
of this $\chi_E$. 

\begin{lem}\label{Lemma 5.1.4}
Let $K/E$ be a finite extension of a field $E$, and $\chi_K: K^\times\to\bbC^\times$.  
 \begin{enumerate}
  \item[(i)] The existence of characters $\chi_E: E^\times\to\bbC^\times$ such that $\chi_E\circ N_{K/E}=\chi_K$
  is equivalent to $K_{E}^{\times}\subset\rm{Ker}(\chi_K)$.
  \item[(ii)] In case (i) is fulfilled, we have a well defined character 
  \begin{equation}
   \chi_{K/E}:=\chi_K\circ N_{K/E}^{-1}:\mathcal{N}_{K/E}\to \bbC^\times,
  \end{equation}
on the subgroup of norms $\mathcal{N}_{K/E}:=N_{K/E}(K^\times)\subset E^\times$, and the solutions $\chi_E$ such that 
$\chi_E\circ N_{K/E}=\chi_K$ are precisely the extensions of $\chi_{K/E}$ from $\mathcal{N}_{K/E}$ to a character of 
$E^\times$.
 \end{enumerate}
\end{lem}
\begin{proof}
{\bf (i)}
Suppose that an equation $\chi_K=\chi_E\circ N_{K/E}$ holds.
Let $x\in K_{E}^{\times}$, hence $N_{K/E}(x)=1$. Then 
$$\chi_K(x)=\chi_E\circ N_{K/E}(x)=\chi_E(1)=1.$$
So $x\in\rm{Ker}(\chi_K)$, and hence $K_{E}^{\times}\subset \rm{Ker}(\chi_K)$.

Conversely assume that $K_{E}^{\times}\subset\rm{Ker}(\chi_K)$. 
 Then $\chi_K$ is actually a character of $K^\times/K_{E}^{\times}$. Again we have
 $K^\times/K_{E}^{\times}\cong \mathcal{N}_{K/E}\subset E^\times$, 
 hence $\widehat{K^\times/K_{E}^{\times}}\cong \widehat{\mathcal{N}_{K/E}}$.
 Now suppose that $\chi_K$ corresponds to the character $\chi_{K/E}$ of $\mathcal{N}_{K/E}$. Hence 
 we can write $\chi_K\circ N_{K/F}^{-1}=\chi_{K/E}$. Thus 
 the character $\chi_{K/E}:\mathcal{N}_{K/E}\to\bbC^\times$
 is well defined. Since $E^\times$ is an abelian group and $\mathcal{N}_{K/E}\subset E^\times$ is a subgroup of finite index
 (by class field theory) $[K:E]$,
 we can extend $\chi_{K/E}$ to $E^\times$, and $\chi_K$ is of the form $\chi_K=\chi_E\circ N_{K/E}$ 
 with $\chi_E|_{\cN_{K/E}}=\chi_{K/E}$.\\
 {\bf (ii)}
 If condition (i) is satisfied, then this part is obvious. 
 If $\chi_E$ is a solution of $\chi_K=\chi_E\circ N_{K/E}$, 
 with $\chi_{K/E}:=\chi_K\circ N_{K/E}^{-1}:\mathcal{N}_{K/E}\to\bbC^\times$, then certainly $\chi_E$ is an extension of 
 the character $\chi_{K/E}$. 
 
 Conversely, if $\chi_E$ extends $\chi_{K/E}$, then it is a solution of $\chi_K=\chi_E\circ N_{K/E}$ with 
 $\chi_K\circ N_{K/E}^{-1}=\chi_{K/E}:\mathcal{N}_{K/E}\to\bbC^\times$.
\end{proof}

\begin{rem}
Now take Heisenberg representation $\rho=\rho(X,\chi_K)$ of $G_F$. Let $E/F$ be any extension corresponding to a maximal 
isotropic for $X$. In this Heisenberg setting, from Theorem \ref{Theorem 5.1.1}(2), we know $\chi_K$ is a character of 
$K^\times/I_FK^\times$, and from the first commutative diagram on p. 302 of \cite{Z2} we have 
$N_{K/E}:K_F^\times/I_FK^\times\to E_F^\times/I_F\cN_{K/E}$. Thus in the Heisenberg setting,
 we have more information than Lemma \ref{Lemma 5.1.4}(i), that $\chi_K$ is a character of 
 \begin{equation}
  K^\times/K_{E}^{\times}I_FK^\times\xrightarrow{N_{K/E}}\mathcal{N}_{K/E}/I_F\mathcal{N}_{K/E}\subset E^\times/I_F\mathcal{N}_{K/E},
 \end{equation}
and therefore $\chi_{K/F}$ is actually a character of $\mathcal{N}_{K/E}/I_F\mathcal{N}_{K/E}$, or in other words, it is a 
$\rm{Gal}(E/F)$-invariant character of the $\rm{Gal}(E/F)$-module $\mathcal{N}_{K/E}\subset E^\times$. And if $\chi_E$ is one of 
the solution of Lemma \ref{Lemma 5.1.4}(ii), then the complete solutions is the set $\{\chi_E^\sigma\,|\,\sigma\in \rm{Gal}(E/F)\}$.

{\bf 
We know that $W(\chi_E,\psi\circ\rm{Tr}_{K/E})$ has the same value for all solutions $\chi_E$ of $\chi_E\circ N_{K/E}=\chi_K$,
which means for all $\chi_E$ which extend the character $\chi_{K/E}$}.

Moreover, from the above Lemma \ref{Lemma 5.1.4}, we also can see that $\chi_E|_{\mathcal{N}_{K/E}}=\chi_{K}\circ N_{K/E}^{-1}$.

\end{rem}

Let $\rho=\rho(X,\chi_K)$ be a Heisenberg representation of $G_F$. Let $E/F$ be any extension corresponding to a maximal 
isotropic for $X$. Then by using the above Lemma \ref{Lemma 5.1.4}, we have the following lemma. 

\begin{lem}\label{Lemma 5.1.44}
  Let $\rho=\rho(Z,\chi_\rho)=\rho(\rm{Gal}(L/K),\chi_K)$ be a Heisenberg representation of a finite local Galois group 
  $G=\rm{Gal}(L/F)$, where $F$ is a non-archimedean local field. Let $H=\rm{Gal}(L/E)$ be a maximal isotropic for 
  $\rho$. Then we obtain
  \begin{equation}
   \rho=\rm{Ind}_{E/F}(\chi_{E}^{\sigma})\quad\text{for all $\sigma\in\rm{Gal}(E/F)$},
  \end{equation}
 where $\chi_E:E^\times/I_F\cN_{K/E}\to\bbC^\times$ with $\chi_K=\chi_E\circ N_{K/E}$.\\
 Moreover, for a fixed base field $E$ of a maximal isotropic for $\rho$, this construction of 
 $\rho$ is independent of the choice of this character $\chi_E$.
 \end{lem}
\begin{proof}

From the group theoretical construction of Heisenberg representation (cf. see Section 2.6), we can write 
\begin{equation}
 \rho=\rm{Ind}_{H}^{G}(\chi_{H}^{g}), \quad\text{for all $g\in G/H$},
\end{equation}
where $\chi_H:H\to\bbC^\times$ is an extension of $\chi_\rho$. From Remark \ref{Remark 3.2} we know that all extensions of 
character $\chi_\rho$ are conjugate with respect to $G/H$, and they are different. If we fix $H$, then $\rho$ is independent
of the choice of character $\chi_H$. For every extension of $\chi_\rho$ we will have same $\rho$. The assertion of the lemma
is the arithmetic expression of this group theoretical facts, and which we will prove in the following.

By the given conditions,
 $L/F$ is a finite Galois extension of the local field $F$ and $G=\rm{Gal}(L/F)$, and 
 $H=\mathrm{Gal}(L/E)$, $Z=\mathrm{Gal}(L/K)$ and $\{1\}=\mathrm{Gal}(L/L)$.
Then by class
field theory, equation (\ref{eqn 5.1.3}), and the condition $X:=\chi_K\circ [-,-]$, 
$\chi_\rho$ identifies with a character 
\begin{center}
 $\chi_K: K^\times/I_FK^\times\to\mathbb{C}^\times$.
\end{center}
Moreover, for the Heisenberg representations we also have the following commutative diagram 

\begin{equation}
\begin{CD}
K^\times_E/I_EK^\times                    @>inclusion>>                         K^\times_F/I_FK^\times\\
@AAcA                                                                  @AAcA \\
E^\times/\cN_{K/E} \wedge E^\times/\cN_{K/E}  @>{N_{E/F}\wedge N_{E/F}}>>  F^\times/\cN_{K/F}\wedge  F^\times/\cN_{K/F}
\end{CD}
\end{equation}
where $N_{E/F}\wedge N_{E/F}(a\wedge b)=N_{E/F}(a)\wedge N_{E/F}(b)$ for all $a,b\in E^\times$, and
the vertical isomorphisms in upward direction are given as the
commutator maps (cf. equation (\ref{eqn 5.1.3})) in the Weil groups $W_{K/E}/I_EK^\times$ and
$W_{K/F}/I_FK^\times$ respectively. 
Under the right vertical $\chi_K$ corresponds (cf. Theorem \ref{Theorem 5.1.1}(3)) to the alternating
character $X$ which is trivial on $N_{E/F}\wedge N_{E/F},$ because $H$
corresponding to $E^\times$ is isotropic.
The commutative diagram  now shows that $\chi_K$ must be trivial on the
image of the upper horizontal, i.e., $\chi_K$ is trivial on the subgroups $K_{E}^{\times}$ for all maximal isotropic $E$. 
Hence $\chi_K$ is actually a character of $K^\times/K_{E}^{\times}$. 

Then from Lemma \ref{Lemma 5.1.4} we can say that there exists a character $\chi_E:E^\times/I_F\cN_{K/E}\to \bbC^\times$ such that 
$\chi_K=\chi_E\circ N_{K/E}$. And this $\chi_E$ is determined by the character $\chi_H$. 
Since $\chi_E$ is trivial on $I_F\mathcal{N}_{K/E}$,
for $\sigma\in G/H=\mathrm{Gal}(E/F)$ we have $\chi_{E}^{\sigma}\circ N_{K/E}=\chi_E\circ N_{K/E}=\chi_K$ because
$\chi_{E}^{\sigma-1}\circ N_{K/E}\equiv 1$.

Therefore instead of $\rho=\mathrm{Ind}_{H}^{G}(\chi_{H}^{g})$ for all $g\in G/H$, we obtain
\begin{center}
 $\rho=\rm{Ind}_{E/F}(\chi_{E}^{\sigma})$, for all $\sigma\in\rm{Gal}(E/F)$,
\end{center}
independently of the choice of $\chi_E$.

\end{proof}

\begin{rem}\label{Remark 5.1.5}
Moreover, we have the exact sequence
\begin{align}\label{sequence 5.1.2}
 K^\times/I_FK^\times\xrightarrow{N_{K/E}} E^\times/I_F\mathcal{N}_{K/E}\xrightarrow{N_{E/F}} F^\times/\mathcal{N}_{K/F},
\end{align}
which is only exact in the middle term. For the dual groups this gives
\begin{align}\label{sequence 5.1.3}
 \widehat{K^\times/I_FK^\times}\xleftarrow{N_{K/E}^{*}} \widehat{E^\times/I_F\mathcal{N}_{K/E}}
 \xleftarrow{N_{E/F}^{*}} \widehat{F^\times/\mathcal{N}_{K/F}}.
\end{align}
But $N_{K/E}^{*}(\chi_{E}^{\sigma-1})=\chi_{E}^{\sigma-1}\circ N_{K/E}\equiv 1$, and therefore the exactness of sequence 
(\ref{sequence 5.1.3}) yields
\begin{equation}
\chi_{E}^{\sigma-1}=\chi_F\circ N_{E/F}, \quad\text{ for some $\chi_F\in\widehat{F^\times/\mathcal{N}_{K/F}}$},
\end{equation}
\end{rem}

For our (arithmetic) determinant computation of Heisenberg representation $\rho$ of $G_F$, we need the following lemma regarding 
transfer map.

\begin{lem}\label{Lemma transfer Heisenberg}
 Let $\rho=\rho(Z,\chi_\rho)$ be a Heisenberg representation of a group $G$ and assume that $H/Z\subset G/Z$ is a maximal
 isotropic for $\rho$. Then transfer map $T_{(G/Z)/(H/Z)}\equiv1$ is the trivial map.
\end{lem}
\begin{proof}
 In general, if $H$ is a central subgroup\footnote{A subgroup of a group which lies inside the center of the group, i.e., 
 a subgroup $H$ of $G$ is central if $H\subseteq Z(G)$.} of finite index $n=[G:H]$ of a group $G$, then by Theorem 5.6 on p. 154 of \cite{MI} we have 
 $T_{G/H}(g)=g^n$. If $G$ is abelian, then center $Z(G)=G$. Hence every subgroup of $G$ is central subgroup. Now if we take $G$ as 
 an abelian group and $H$ is a subgroup of finite index, then we can write $T_{G/H}(g)=g^{[G:H]}$.
 
 Now we come to the Heisenberg setting. We know that $G/Z$ is abelian, hence $H/Z\subset G/Z$ is a central subgroup.
 Then we have $T_{(G/Z)/(H/Z)}(g)=g^{[G/Z:H/Z]}=g^d$, where $d$ is the dimension of $\rho$.
 For the Heisenberg setting, we also know (cf. Lemma \ref{Lemma 2.10}) that $G^d\subseteq Z$, hence $g^d\in Z$. This implies 
 $$T_{(G/Z)/(H/Z)}(g)=g^d=1,\quad\text{the identity in $H/Z$},$$
 for all $g\in G$, hence $T_{(G/Z)/(H/Z)}\equiv1$ is a trivial map.
\end{proof}

By using the above Lemma \ref{Lemma 5.1.4} and 
Lemma \ref{Lemma transfer Heisenberg}, in the following, we give the 
arithmetic description of Proposition \ref{Proposition 2.13}.

\begin{prop}\label{Proposition arithmetic form of determinant}
 Let $\rho=\rho(Z,\chi_\rho)=\rho(G_K,\chi_K)$ be a Heisenberg representation of the absolute Galois group $G_F$.
 Let $E$ be a base field of a maximal isotropic for $\rho$. Then $F^\times\subseteq \cN_{K/E}$, and
 \begin{equation}\label{eqn 5.1.12}
  \det(\rho)(x)=\Delta_{E/F}(x)\cdot\chi_K\circ N_{K/E}^{-1}(x)\quad \text{for all $x\in F^\times$},
 \end{equation}
where, for all $x\in F^\times$,
\begin{equation}\label{eqn 5.1.13}
 \Delta_{E/F}(x)=\begin{cases}
                  1 & \text{when $\rm{rk}_2(\rm{Gal}(E/F))\ne 1$}\\
                  \omega_{E'/F}(x) & \text{when $\rm{rk}_2(\rm{Gal}(E/F))= 1$},
                 \end{cases}
\end{equation}
where $E'/F$ is a uniquely determined quadratic subextension in $E/F$, and $\omega_{E'/F}$ is the character of $F^\times$ which 
corresponds to $E'/F$ by class field theory.
\end{prop}

\begin{proof}
 From the given condition, we can write $G/Z=\rm{Gal}(K/F)\supset H/Z=\rm{Gal}(K/E)$. Here both $G/Z$ and $H/Z$ are abelian, then from 
 class field theory we have the following commutative diagram
 \begin{equation}\label{diagram 5.1.13}
  \begin{CD}
  F^\times/\cN_{K/F}   @>inclusion>>  E^\times/\cN_{K/E}\\
  @VV\theta_{K/F}V                      @VV\theta_{K/E}V\\
  \rm{Gal}(K/F) @>T_{(G/Z)/(H/Z)}>> \rm{Gal}(K/E)
 \end{CD}
 \end{equation}
Here $\theta_{K/F}$, $\theta_{K/F}$ are the isomorphism (Artin reciprocity) maps and $T_{(G/Z)/(H/Z)}$ 
is transfer map. From Lemma \ref{Lemma transfer Heisenberg}, we have $T_{(G/Z)/(H/Z)}\equiv1$. Therefore from the above 
diagram (\ref{diagram 5.1.13}) we can say $F^\times\subseteq\cN_{K/E}$, i.e., all elements
\footnote{This condition $F^\times\subseteq\cN_{K/E}$ implies that for every $x\in F^\times$ must have a preimage under the 
$N_{K/E}$, but the preimage is not unique.} from the base field $F$ are norms with 
respect to the extension $K/E$.

Now identify $\chi_\rho=\chi_K:K^\times/I_FK^\times\to\bbC^\times$. Then the map 
$$x\in F^\times\mapsto\chi_K\circ N_{K/E}^{-1}(x)$$
is well-defined character of $F^\times$ because $\chi_K$ is trivial on $K_{E}^{\times}$.

Now by Gallagher's Theorem \ref{Theorem Gall} (arithmetic side) (cf. equation (\ref{eqn 2.12})) we can write for all 
$x\in F^\times$,
\begin{equation}
 \det(\rho)(x)=\Delta_{E/F}(x)\cdot\chi_E(x)=\Delta_{E/F}(x)\cdot\chi_K(N_{K/E}^{-1}(x)),
\end{equation}
since $F^\times\subseteq\cN_{K/E}$, and $\chi_E|_{\cN_{K/E}}=\chi_K\circ N_{K/E}^{-1}$.

Furthermore,
since $E/F$ is an abelian extension, $\rm{Gal}(E/F)\cong\widehat{\rm{Gal}(E/F)}$, and 
from Miller's Theorem \ref{Theorem Miller}, we can write (cf. equation (\ref{eqn 3.31}))
\begin{align*}
 \Delta_{E/F}
 &=\begin{cases}
                  1 & \text{when $\rm{rk}_2(\rm{Gal}(E/F))\ne 1$}\\
                  \omega_{E'/F}(x) & \text{when $\rm{rk}_2(\rm{Gal}(E/F))= 1$},
                 \end{cases}
\end{align*}
where $E'/F$ is a uniquely determined quadratic subextension in $E/F$, and $\omega_{E'/F}$ is the character of $F^\times$ which 
corresponds to $E'/F$ by class field theory.

\end{proof}




\subsection{{\bf Heisenberg representations of $G_F$ of dimensions prime to $p$}}

Let $F/\bbQ_p$ be a non-archimedean local field, and $G_F$ be the absolute Galois group of $F$. In this subsection we construct all 
Heisenberg representations of $G_F$ of dimensions prime to $p$. Studying the construction of this type (i.e., dimension prime to $p$)
Heisenberg representations are important for our next section.

\begin{dfn}[{\bf U-isotropic}]\label{Definition U-isotropic}
Let $F$ be a non-archimedean local field. 
Let $X:FF^\times\to \bbC^\times$ be an alternating character with the property 
 $$X(\varepsilon_1,\varepsilon_2)=1,\qquad \text{for all $\varepsilon_1,\varepsilon_2\in U_F$}.$$
 In other words, $X$ is a character of $FF^\times/U_F\wedge U_F$. Then $X$ is said to be the U-isotropic.  
 These $X$ are easy to classify:
\end{dfn}

\begin{lem}\label{Lemma U-isotropic}
 Fix a uniformizer $\pi_F$ and write $U:=U_F$. Then we obtain an isomorphism 
 $$\widehat{U}\cong \widehat{FF^\times/U\wedge U}, \quad \eta\mapsto X_\eta,\quad \eta_X\leftarrow X$$
 between characters of $U$ and $U$-isotropic alternating characters as follows:
 \begin{equation}\label{eqn 5.1.25}
  X_\eta(\pi_F^a\varepsilon_1,\pi_F^b\varepsilon_2):=\eta(\varepsilon_1)^b\cdot\eta(\varepsilon_2)^{-a},\quad
  \eta_X(\varepsilon):=X(\varepsilon,\pi_F),
 \end{equation}
 where $a,b\in\bbZ$, $\varepsilon,\varepsilon_1,\varepsilon_2\in U$, and $\eta:U\to\bbC^\times$.
 Then 
 $$\rm{Rad}(X_\eta)=<\pi_F^{\#\eta}>\times\rm{Ker}(\eta)=<(\pi_F\varepsilon)^{\#\eta}>\times\rm{Ker}(\eta),$$
 does not depend on the choice of $\pi_F$, where  $\#\eta$ is the order of the character $\eta$, hence 
 $$F^\times/\rm{Rad}(X_\eta)\cong <\pi_F>/<\pi_F^{\#\eta}>\times U/\rm{Ker}(\eta)\cong \bbZ_{\#\eta}\times\bbZ_{\#\eta}.$$
 Therefore all Heisenberg representations of type $\rho=\rho(X_\eta,\chi)$ have dimension $\rm{dim}(\rho)=\#\eta$.
\end{lem}

\begin{proof}
To prove $\widehat{U}\cong \widehat{FF^\times/U\wedge U}$, we have to show that $\eta_{X_\eta}=\eta$ and $X_{\eta_X}=X_\eta$, and that 
the inverse map $X\mapsto \eta_X$ does not depend on the choice of $\pi_F$.

From the above definition of $\eta_X$, we can write:
\begin{align*}
 \eta_{X_\eta}(\varepsilon)
 &=X_\eta(\epsilon,\pi_F)=\eta(\varepsilon)^{1}\cdot \eta(1)^0=\eta(\varepsilon), 
\end{align*}
for all $\varepsilon\in U$, hence $\eta_{X_\eta}=\eta$.

Similarly, from the above definition of $X$, we have:
\begin{align*}
 X_{\eta_X}(\pi_F^a\varepsilon_1,\pi_F^b\varepsilon_2)
 &=\eta_X(\varepsilon_1)^b\cdot\eta_X(\varepsilon_2)^{-a}=X(\varepsilon_1,\pi_F)^b\cdot X(\varepsilon_2,\pi_F)^{-a}\\
 &=X(\varepsilon_1,\pi_F)^b\cdot X(\pi_F,\varepsilon_2)^{a}=X(\varepsilon_1,\pi_F^b)\cdot X(\pi_F^a,\varepsilon_2)\\
 &=X(\pi_F^a\varepsilon_1,\pi_F^b\varepsilon_2).
\end{align*}
This shows that $X_{\eta_X}=X$.

Now we choose a uniformizer $\pi_F\varepsilon$, where $\varepsilon\in U$, instead of choosing $\pi_F$.
Then we can write 
\begin{align*}
 X_\eta((\pi_F\varepsilon)^a\varepsilon_1,(\pi_F\varepsilon)^b\varepsilon_2)
 &=X_\eta(\pi_F^a(\varepsilon^a\varepsilon_1),\pi_F^b(\varepsilon^b\varepsilon_2))\\
 &=\eta(\varepsilon^a\varepsilon_1)^b\cdot \eta(\varepsilon^b\varepsilon_2)^{-a}\\
 &=\eta(\varepsilon_1)^b\cdot \eta(\varepsilon_2)^{-a}\cdot \eta(\varepsilon^{ab-ab})\\
 &=\eta(\varepsilon_1)^b\cdot\eta(\varepsilon_2)^{-a}=X(\pi_F^a\varepsilon_1,\pi_F^b\varepsilon_2).
\end{align*}
This shows that $X_\eta$ does not depend on the choice of the uniformizer $\pi_F$. Similarly since 
$\eta_X(\varepsilon):=X(\varepsilon,\pi_F)$, it is clear that $\eta_X$ is also does not depend on the choice of the 
uniformizer $\pi_F$.

By the definition of the radical of $X_\eta$, we have:
 $$\rm{Rad}(X_\eta)=
 \{\pi_F^a\varepsilon\in F^\times\,|\; X_\eta(\pi_F^{a}\varepsilon,\pi_F^{b}\varepsilon')= 
 \eta(\varepsilon)^{b}\cdot \eta(\varepsilon')^{-a}=1\},$$
 for all $b\in \bbZ$, and $\varepsilon'\in U$. 
 
 Now if we fix a uniformizer $\pi_F\varepsilon'',$ where $\varepsilon''\in U$ instead of $\pi_F$, we can write:
 $$\rm{Rad}(X_\eta)=
 \{(\pi_F\varepsilon'')^a\varepsilon\in F^\times\,|\; X_\eta((\pi_F\varepsilon'')^{a}\varepsilon,(\pi_F\varepsilon'')^{b}\varepsilon')= 
 \eta(\varepsilon''^a\varepsilon)^{b}\cdot \eta(\varepsilon''^b\varepsilon')^{-a}=\eta(\varepsilon)^b\cdot\eta(\varepsilon')^{-a}=1\},$$

 This gives $\rm{Rad}(X_\eta)=<\pi_F^{\#\eta}>\times\rm{Ker}(\eta)=<(\pi_F\varepsilon)^{\#\eta}>\times\rm{Ker}(\eta)$, hence 
 $$F^\times/\rm{Rad}(X_\eta)\cong <\pi_F>/<\pi_F^{\#\eta}>\times U/\rm{Ker}(\eta)\cong \bbZ_{\#\eta}\times\bbZ_{\#\eta}.$$
 
 Then all Heisenberg representations of type $\rho=\rho(X_\eta,\chi)$ have dimension
 $$\rm{dim}(\rho)=\sqrt{[F^\times:\rm{Rad}(X_\eta)]}=\#\eta.$$
 
\end{proof}

\begin{rem}
 
From Proposition 5.2(i) on p. 50 of \cite{Z4}, we know that 
$FF^\times/U_F\wedge U_F\cong U_F$, hence we have $\widehat{U_F}\cong \widehat{FF^\times/U_F\wedge U_F}$.
From the above Lemma \ref{Lemma U-isotropic} we have:
$$\rm{Rad}(X_\eta)=<\pi_F^{\#\eta}>\times\rm{Ker}(\eta).$$
Therefore we can conclude that
a U-isotropic character $X=X_\eta$ has $U_F^i$ contained in its radical if and 
only if $\eta$ is a character of $U_F/U_{F}^{i}$.
\end{rem}

From the above Lemma \ref{Lemma U-isotropic} we know that the dimension of a U-isotropic Heisenberg representation 
$\rho=\rho(X_\eta,\chi)$ of $G_F$ is $\rm{dim}(\rho)=\#\eta$, and $F^\times/\rm{Rad}(X_\eta)\cong \bbZ_{\#\eta}\times\bbZ_{\#\eta}$,
a direct product of two cyclic (bicyclic) groups of the same order $\#\eta$. In general, if $A=\bbZ_m\times\bbZ_m$ is a bicyclic
group of order $m^2$, then by the following lemma we can compute total number of elements of order $m$ in $A$, and 
number of cyclic complementary subgroup of a fixed cyclic subgroup of order $m$. 

\begin{lem}\label{Lemma on bicyclic abelian groups}
 Let $A\cong \bbZ_m\times\bbZ_m$ be a bicyclic abelian group of order $m^2$. Then:
 \begin{enumerate}
  \item Then number $\psi(m)$ of cyclic subgroups $B\subset A$ of order $m$ is a multiplicative arithmetic function 
  (i.e., $\psi(mn)=\psi(m)\psi(n)$ if $gcd(m,n)=1$).
  \item Explicitly we have 
  \begin{equation}
   \psi(m)=m\cdot\prod_{p|m}(1+\frac{1}{p}).
  \end{equation}
And the number of elements of order $m$ in $A$ is:
\begin{equation}
 \varphi(m)\cdot\psi(m)=m^2\cdot\prod_{p|m}(1-\frac{1}{p^2}).
\end{equation}
Here $p$ is a prime divisor of $m$ and $\varphi(n)$ is the Euler's totient function of $n$.
\item Let $B\subset A$ be cyclic of order $m$. Then $B$ has always a complementary subgroup $B'\subset A$ such that $A=B\times B'$,
and $B'$ is again cyclic of order $m$. And for $B$ fixed, the number of all different complementary subgroups 
$B'$ is $=m$.
 \end{enumerate}
\end{lem}

\begin{proof}
To prove these assertions we need to recall the fact: If $G$ is a finite cyclic group of order $m$, then number of generators of 
$G$ is $\varphi(m)=m\prod_{p|m}(1-\frac{1}{p})$.\\  
 {\bf (1).} By the given condition $A\cong \bbZ_m\times\bbZ_m$ and $\psi(m)$ is the number of cyclic subgroup of $A$ of order $m$.
 Then it is clear that $\psi$ is an arithmetic function with $\psi(1)=1\ne 0$, hence $\psi$ is not {\bf additive}.
 Now take $m\ge 2$, and the prime factorization of $m$ is: $m=\prod_{i=1}^{k}p_{i}^{a_i}$.
 To prove this, first we should start with $m=p^n$, hence $A\cong \bbZ_{p^n}\times\bbZ_{p^n}$. Then number of subgroup of $A$ of order 
 $p^n$ is:
 $$\psi(p^n)=\frac{2\varphi(p^n)p^n-\varphi(p^n)^2}{\varphi(p^n)}=2p^n-\varphi(p^n)=p^n(2-1+\frac{1}{p})=p^n(1+\frac{1}{p}).$$
 
Now take $m=p^nq^r$, where $p,q$ are both prime with $gcd(p,q)=1$. 
 We also know that $\bbZ_{p^nq^r}\times\bbZ_{p^nq^r}\cong\bbZ_{p^n}\times\bbZ_{p^n}\times\bbZ_{q^r}\times\bbZ_{q^r}$.
 This gives $\psi(p^nq^r)=\psi(p^n)\cdot\psi(q^r)$. By the similar method we can show that 
 $\psi(m)=\prod_{i=i}^{k}\psi(p_{i}^{a_i})$, where  $m=\prod_{i=1}^{k}p_{i}^{a_i}$.
 This condition implies that $\psi$ is a multiplicative arithmetic function.\\

 {\bf (2).} Since $\psi$ is multiplicative arithmetic function, we have
 \begin{align*}
  \psi(m)
  &=\prod_{i=1}^{k}\psi(p_{i}^{a_i})=\prod_{i=1}^{k}p_{i}^{a_i}(1+\frac{1}{p_i})\quad\text{since $\psi(p^n)=p^n(1+\frac{1}{p})$},\\
  &=p_{1}^{a_1}\cdots p_{k}^{a_k}\prod_{i=1}^{k}(1+\frac{1}{p_i})=m\cdot\prod_{p|m}(1+\frac{1}{p}).
 \end{align*}
We also know that number of generator of a finite cyclic group of order $m$ is $\varphi(m)$, hence number of elements of order 
$m$ is $\varphi(m)$. Then the number of elements of order $m$ in $A$ is:
\begin{equation*}
 \varphi(m)\cdot\psi(m)=m\cdot\prod_{p|m}(1-\frac{1}{p})\cdot m\prod_{p|m}(1+\frac{1}{p})=m^2\cdot\prod_{p|m}(1-\frac{1}{p^2}).
\end{equation*}
{\bf (3).} Let $B\subset A$ be a cyclic subgroup of order $m$. Since $A$ is abelian and bicyclic of order $m^2$, $B$ has always 
a complementary subgroup $B'\subset A$ such that $A=B\times B'$, and $B'$ is again cyclic (because $A$ is cyclic, hence 
$A/B$ and $|A/B|=m$) of order $m$.

To prove the last part of (3), we start with $m=p^n$. Here $B$ is a cyclic subgroup of $A$ of order $p^n$, hence 
$B=<(a,e)>$, where $\# a=p^n$, and $e$ is the identity of $B'$. 
Since $B$ has complementary cyclic subgroup, namely $B'$, of order $p^n$. we can choose 
$B'=<(b,c)>$, where $B\cap B'=(e,e)$. This gives that $c$ is a generator of $B'$,  and $b$ could be any element in $\bbZ_{p^n}$.
Thus total number $\psi_{B'}(p^n)$ of all different complementary subgroups $B'$ is:
$$\psi_{B'}(p^n)=\frac{p^n\varphi(p^n)}{\varphi(p^n)}=p^n=m.$$
Now if we take $m=p^nq^r$, where $q$ is a different prime from $p$. Then by same method we can see that 
$\psi_{B'}(p^nq^r)=\psi_{B'}(p^n)\cdot\psi_{B'}(q^r)=p^nq^r=m$. Thus for arbitrary $m$ we can conclude that 
$\psi_{B'}(m)=m$.

\end{proof}

In the following lemma, we give an equivalent condition for U-isotropic Heisenberg representation.

\begin{lem}\label{Lemma U-equivalent}
Let $G_F$ be the absolute Galois group of a non-archimedean local field $F$.
 For a Heisenberg representation $\rho=\rho(Z,\chi_\rho)=\rho(X,\chi_K)$ the following are equivalent:
 \begin{enumerate}
  \item The alternating character $X$ is U-isotropic.
  \item Let $E/F$ be the maximal unramified subextension in $K/F$. Then $\rm{Gal}(K/E)$ is maximal isotropic for $X$.
  \item $\rho=\rm{Ind}_{E/F}(\chi_E)$ can be induced from a character $\chi_E$ of $E^\times$ (where $E$ is as in (2)).
 \end{enumerate}
\end{lem}
\begin{proof}
 This proof follows from the above Lemma \ref{Lemma U-isotropic}. \\
 First, assume that $X$ is U-isotropic, i.e., $X\in\widehat{FF^\times/U\wedge U}$. We also know that 
 $\widehat{U}\cong\widehat{FF^\times/U\wedge U}$. Then $X$ corresponds a character of $U$, namely $X\mapsto \eta_X$.
 Then from Lemma \ref{Lemma U-isotropic} we have $F^\times /\rm{Rad}(X)\cong \bbZ_{\#\eta_X}\times\bbZ_{\#\eta_X}$, i.e.,
 product of two cyclic groups of same order.
 
 Since $K/F$ is the abelian bicyclic extension which corresponds to $\rm{Rad}(X)$, we can write:
 $$\cN_{K/F}=\rm{Rad}(X),\qquad\rm{Gal}(K/F)\cong F^\times/\rm{Rad}(X).$$
Let $E/F$ be the maximal unramified subextension in $K/F$. Then $[E:F]=\#\eta_K$ because the order of 
maximal cyclic subgroup of $\rm{Gal}(K/F)$ is $\#\eta_X$. Then $f_{E/F}=\#\eta_X$, hence 
$f_{K/F}=e_{K/F}=\#\eta_X$ because $f_{K/F}\cdot e_{K/F}=[K:F]=\#\eta_X^2$ and $\rm{Gal}(K/F)$ is not cyclic group.

Now we have to prove that the extension $E/F$ corresponds to a maximal isotropic for $X$. Let $H/Z$ be a maximal isotropic for 
$X$, hence $[G_F/Z:H/Z]=\#\eta_X$, hence $H/Z=\rm{Gal}(K/E)$, i.e., the maximal unramified subextension $E/F$ in $K/F$ corresponds
to a maximal isotropic subgroup, hence 
\begin{center}
 $\rho(X,\chi_K)=\rm{Ind}_{E/F}(\chi_E)$, for $\chi_E\circ N_{K/E}=\chi_K$.
\end{center}
Finally, since $E/F$ is unramified and the extension $E$ corresponds a maximal isotropic subgroup for $X$, we have 
$U_F\subset\cN_{E/F}$, hence $U_F\subset\cN_{K/F}$ and $X|_{U\times U}=1$ because $U_F\subset F^\times\subset\cN_{K/E}$. 
This shows that $X$ is U-isotropic.
\end{proof}

\begin{cor}\label{Corollary U-isotropic}
 The U-isotropic Heisenberg representation $\rho=\rho(X_\eta,\chi)$ can never be wild because it is induced from 
 an unramified extension $E/F$, 
 but the dimension $\rm{dim}(\rho(X_\eta,\chi))=\#\eta$ can be a power of $p.$\\
The representations $\rho$ of dimension prime to p are precisely given as 
$\rho=\rho(X_\eta,\chi)$ for characters $\eta$ of $U/U^1.$
\end{cor}
\begin{proof}
 This is clear from the above lemma \ref{Lemma U-isotropic} and the fact: $|U/U^1|=q_F-1$ is prime to $p$.
 We know that the dimension $\rm{dim}(\rho)=\sqrt{[K:F]}=\sqrt{[F^\times:\rm{Rad}(X)]}$. If this is prime to $p$ then 
 $K/F$ is tame and $U_F^1\subseteq \rm{Rad}(X)$. But $U/U^1$ is cyclic, hence $X$ is then $U$-isotopic.
\end{proof}


\begin{lem}\label{Lemma dimension equivalent}
 Let $\rho=\rho(X,\chi_K)$ be a Heisenberg representation of 
 the absolute Galois group $G_F$ of a non-archimedean local field 
 $F/\bbQ_p$. Then following are equivalent:
 \begin{enumerate}
  \item $\rm{dim}(\rho)$ is prime to $p$.
  \item $\rm{dim}(\rho)$ is a divisor of $q_F-1$.
  \item The alternating character $X$ is $U$-isotropic and $X=X_\eta$ for a character $\eta$ of 
  $U_F/U_F^1$, i.e., $a(\eta)=1$.
  \item The abelian extension $K/F$ which corresponds to $\rm{Rad}(X)$ is tamely ramified.
 \end{enumerate}
\end{lem}
\begin{proof}
(1) implies (2):
From Corollary \ref{Corollary U-isotropic} we know that all Heisenberg representations of dimensions prime to $p$, are 
 U-isotropic representations of the form $\rho=\rho(X_\eta,\chi)$, where $\eta:U_F/U_F^1\to\bbC^\times$, and the dimensions 
 $\rm{dim}(\rho)=\#\eta$.
 Thus if $\rm{dim}(\rho)$ is prime to $p$, then $\rm{dim}(\rho)=\#\eta$ is a divisor of $q_F-1$. \\
 (2) implies (3):
 If $\rm{dim}(\rho)$
 is a divisor of $q_F-1$, then $gcd(p,\rm{dim}(\rho))=1$. Then from Corollary \ref{Corollary U-isotropic}, the alternating 
 character $X$ is U-isotropic and $X=X_\eta$ for a character $\eta\in\widehat{U_F/U_F^1}$.\\
 (3) implies (4): 
 We know that 
 $$\rm{dim}(\rho)=\sqrt{[F^\times:\rm{Rad}(X_\eta)]}=\sqrt{[F^\times:\cN_{K/F}]}=\#\eta.$$
 Here since $K/F$ is abelian, we have $\rm{dim}(\rho)^2=[K:F]$. Again since $\#\eta=\rm{dim}(\rho)$ is a divisor
 of $q_F-1$, hence $K/F$ is tamely ramified.\\
 (4) implies (1): If $K/F$ is tamely ramified, then we can write $U_F^1\subset \cN_{K/F}\subset F^\times$, and hence 
 $F^\times/\cN_{K/F}$ is a quotient group of $F^\times/U_F^1$. Therefore  
if $K/F$ is the abelian tamely ramified extension and $\cN_{K/F}=\rm{Rad}(X)$, then $X$ must be an alternating character
of $F^\times/U_F^1$.  
We also know that  
$F^\times=<\pi_F>\times<\zeta>\times U_{F}^{1}$, where $\zeta$ is a root of unity of order $q_F-1$. This implies 
$F^\times/U_{F}^{1}=<\pi_F>\times<\zeta>$.
So each element $x\in
F^\times/U_F^1$ can be written as $x= \pi_{F}^a\cdot \zeta^b$, where $a,b\in\bbZ$. 
We now take $x_1=\pi_{F}^{a_1}\zeta^{b_1}, x_2=\pi_{F}^{a_2}\zeta^{b_2}\in F^\times/U_{F}^{1}$, where $a_i,b_i\in\bbZ(i=1,2)$, then 
\begin{align*}
 X(x_1,x_2)
 &= X(\pi_{F}^{a_1}\zeta^{b_1},\; \pi_{F}^{a_2}\zeta^{b_2})\\
 &= X(\pi_{F}^{a_1},\zeta^{b_2})\cdot X(\zeta^{b_1},\pi_{F}^{a_2})\\
 &=\chi_\rho([\pi_{F}^{a_1},\zeta^{b_2}])\cdot\chi_\rho([\zeta^{b_1},\pi_{F}^{a_2}]).
\end{align*}
But this implies  $X^{q_F-1}\equiv 1$ because $\zeta^{q_F-1}=1$,
which means that $X$ is actually an alternating character on  $F^\times/({F^\times}^{(q_F-1)} U_F^1),$ and therefore
$G_F/G_K$ is actually a quotient of $F^\times/({F^\times}^{(q_F-1)} U_F^1).$ 
We also know that $U_F^1$ is a pro-p-group and therefore 
$$U_F^1=(U_F^1)^{q_F-1}\subset F^\times.$$
Thus the cardinality of 
$F^\times/({F^\times}^{(q_F-1)} U_F^1)$ is $(q_F-1)^2$ because 
$$F^\times/({F^\times}^{(q_F-1)} U_F^1)\cong \bbZ/(q_F-1)\bbZ\times<\zeta>\cong \bbZ_{q_F-1}\times\bbZ_{q_F-1}.$$
Therefore $\rm{dim}(\rho)$ divides $q_F-1.$ Hence $\rm{dim}(\rho)$ is prime to $p$.

\end{proof}

\begin{rem}\label{Remark 5.1.14} 
We let $K_\eta|F$ be the abelian bicyclic 
extension which corresponds to $\rm{Rad}(X_\eta):$

$$ \cN_{K_\eta/F}= \rm{Rad}(X_\eta),\qquad \rm{Gal}(K_\eta/F)\cong F^\times/\rm{Rad}(X_\eta).$$
Then we have $f_{K_\eta|F}= e_{K_\eta|F}=\#\eta$ and the maximal unramified subextension 
$E/F\subset K_\eta/F$ corresponds to a maximal isotropic subgroup, hence
$$ \rho(X_\eta,\chi) = \rm{Ind}_{E/F}(\chi_E),\quad\textrm{for}\; \chi_E\circ N_{K_\eta/E} =\chi.$$
We recall here that $\chi:K_\eta^\times/I_FK_\eta^\times\rightarrow\bbC^\times$ is a character such that
(cf. Theorem \ref{Theorem 5.1.1}(3))
$$ \chi|_{(K_\eta^\times)_F} \leftrightarrow X_\eta,\quad\textrm{with respect to}\; 
(K_\eta^\times)_F/I_FK_\eta^\times\cong F^\times/\rm{Rad}(X_\eta)\wedge F^\times/\rm{Rad}(X_\eta).$$
In particular, we see that $(K_\eta^\times)_F/I_FK_\eta^\times$ is cyclic of 
order $\#\eta$ and $\chi|_{(K_\eta^\times)_F}$ must be a faithful character of that cyclic group.
\end{rem}
In the following lemma we see the explicit description of the representation $\rho=\rho(X_\eta,\chi)$.

\begin{lem}[{\bf Explicit Lemma}]\label{Explicit Lemma}
 Let $\rho=\rho(X_\eta,\chi_K)$ be a U-isotropic Heisenberg representation of the absolute Galois group $G_F$ of a local field 
 $F/\bbQ_p$. Let $K=K_\eta$ and let $E/F$ be the maximal unramified subextension in $K/F$. Then: 
 \begin{enumerate}
  \item The norm map induces an isomorphism:
  $$N_{K/E}:K_F^\times/I_FK^\times\stackrel{\sim}{\to}I_FE^\times/I_F\cN_{K/E}.$$
  \item Let $c_{K/F}:F^\times/\rm{Rad}(X_\eta)\wedge F^\times/\rm{Rad}(X_\eta)\cong K_F^\times/I_FK^\times$ be the isomorphism
  which is induced by the commutator in the relative Weil-group $W_{K/F}$. Then for units $\varepsilon\in U_F$ we 
  explicitly have:
  $$c_{K/F}(\varepsilon\wedge\pi_F)=N_{K/E}^{-1}(N_{E/F}^{-1}(\varepsilon)^{1-\varphi_{E/F}}),$$
  where $\varphi_{E/F}$ is the Frobenius automorphism for $E/F$ and where $N^{-1}$ means to take a preimage of the norm map.
  \item The restriction $\chi_K|_{K_F^\times}$ is characterized by:
  $$\chi_K\circ c_{K/F}(\varepsilon\wedge\pi_F)=X_\eta(\varepsilon,\pi_F)=\eta(\varepsilon),$$
  for all $\varepsilon\in U_F$, where $c_{K/F}(\varepsilon\wedge\pi_F)$ is explicitly given via (2).
 \end{enumerate}

\end{lem}

\begin{proof}
 {\bf (1).} By the given conditions we have: $K=K_\eta,$ and $K/F$ is the bicyclic extension with $\rm{Rad}(X_\eta)=\cN_{K/F}$, and 
 $E/F$ is the maximal unramified subextension in $K/F$. So $K/E$ and $E/F$ both are cyclic, hence 
 $$E_F^\times=I_FE^\times,\qquad K_E^\times=I_EK^\times.$$
 From the diagram (3.6.1) on p. 41 of \cite{Z4}, we have 
 $$N_{K/E}: K_F^\times/I_FK^\times\stackrel{\sim}{\to} E_F^\times/I_F\cN_{K/E}.$$
 We also know that $E_F^\times=I_FE^\times$. Thus the norm map $N_{K/E}$ induces an isomorphism:
 $$N_{K/E}:K_F^\times/I_FK^\times\cong I_FE^\times/I_F\cN_{K/E}.$$
 {\bf (2).} By the given conditions, $c_{K/F}$ is the isomorphism which is induced by the commutator in 
 the relative Weil-group  $W_{K/F}$
 (cf. the map (\ref{eqn 5.1.3}). Here $\rm{Rad}(X_\eta)=\cN_{K/F}=:N$. Then from Proposition 1(iii) of \cite{Z5} on p. 128, we have 
 $$c_{K/F}: N\wedge F^\times/N\wedge N\stackrel{\sim}{\to} I_FK^\times/I_FK_F^\times$$
 as an isomorphism by the map:
 $$c_{K/F}(x\wedge y)=N_{K/F}^{-1}(x)^{1-\phi_F(y)},$$
 where $\phi_F(y)\in \rm{Gal}(K/F)$ for $y\in F^\times$ by class field theory.
 If $y=\pi_F$, then by class field theory (cf. \cite{JM}, p. 20, Theorem 1.1(a)), we can write 
 $\phi_F(\pi_F)|_{E}=\varphi_{E/F}$, where $\varphi_{E/F}$ is the Frobenius automorphism for $E/F$.
 
 Now we come to our special case.
Since $E/F$ is unramified, we have $U_F\subset\cN_{E/F}$, and we obtain (cf. \cite{Z4}, pp. 46-47 of Section 4.4 and 
the diagram on p. 302 of \cite{Z2}):
\begin{equation}\label{eqn explicit lemma}
 N_{K/E}\circ c_{K/F}(\varepsilon\wedge\pi_F)=N_{E/F}^{-1}(\varepsilon)^{1-\varphi_{E/F}}.
\end{equation}
We also know (see the first two lines under the upper diagram on p. 302 of \cite{Z2}) that
$E_F^\times\subseteq \cN_{K/E}$. Here 
$$N_{E/F}^{-1}(\varepsilon)^{1-\varphi_{E/F}}\in I_FE^\times/I_F\cN_{K/E}=E_F^\times/I_F\cN_{K/E},$$
because $E/F$ is cyclic, hence $E_F^\times=I_FE^\times$. Therefore from equation (\ref{eqn explicit lemma}) we can conclude:
$$c_{K/F}(\varepsilon\wedge\pi_F)=N_{K/E}^{-1}(N_{E/F}^{-1}(\varepsilon)^{1-\varphi_{E/F}}).$$
{\bf (3.)} We know that the $c_{K/F}(\varepsilon\wedge\pi_F)\in K_F^\times$ and $\chi_K:K^\times/I_FK^\times\to\bbC^\times$. 
Then we can write 
\begin{align*}
 \chi_K\circ c_{K/F}(\varepsilon\wedge\pi_F)
 &=\chi_K(N_{K/E}^{-1}(N_{E/F}^{-1}(\varepsilon)^{1-\varphi_{E/F}})\\
 &=\chi_E\circ N_{K/E}(N_{K/E}^{-1}(N_{E/F}^{-1}(\varepsilon)^{1-\varphi_{E/F}}), \quad\text{since $\chi_K=\chi_E\circ N_{K/E}$}\\
 &=\chi_E(N_{E/F}^{-1}(\varepsilon)^{1-\varphi_{E/F}})=X_\eta(\varepsilon,\pi_F)\\
 &=\eta(\varepsilon).
\end{align*}
This is true for all $\varepsilon\in U_F$. Therefore we can conclude that 
$\chi_K|_{K_F^\times}=\eta$.
\end{proof}

\begin{exm}[{\bf Explicit description of Heisenberg representations of dimension prime to $p$}]\label{Example for Heisenberg reps}

Let $F/\bbQ_p$ be a local field, and $G_F$ be the absolute Galois group of $F$.
Let $\rho=\rho(X,\chi_K)$ be a Heisenberg representation of $G_F$ of dimension $m$ prime to $p$. Then from 
Corollary \ref{Corollary U-isotropic} the alternating character $X=X_\eta$ is $U$-isotropic for a character
$\eta:U_F/U_F^1\to\bbC^\times$. Here from Lemma \ref{Lemma U-isotropic} 
we can say $m=\sqrt{[F^\times:\rm{Rad}(X_\eta)]}=\#\eta$ divides $q_F-1$.

Since $U_F^1$ is a pro-p-group and $gcd(m,p)=1$, we have $(U_F^1)^m=U_F^1\subset {F^\times}^m$, and therefore  
$$F^\times/{F^\times}^m\cong\bbZ_m\times\bbZ_m,$$
is a bicyclic group of order $m^2$. So by class field theory there is precisely one extension $K/F$ such that 
$\rm{Gal}(K/F)\cong\bbZ_m\times\bbZ_m$ and the norm group $\cN_{K/F}:=N_{K/F}(K^\times)={F^\times}^m$.

We know that $U_F/U_F^1$ is a cyclic group of order $q_F-1$, hence $\widehat{U_F/U_F^1}\cong U_F/U_F^1$. By the given condition 
$m|(q_F-1)$, hence $U_F/U_F^1$ has exactly one subgroup of order $m$. Then number of elements of order $m$ in $U_F/U_F^1$ is 
$\varphi(m)$, the Euler's $\varphi$-function of $m$.
In this setting, we have $\eta\in \widehat{U_F/U_F^1}\cong \widehat{FF^\times/U_F^1\wedge U_F^1}$ with 
$\#\eta=m$. This implies that up to $1$-dimensional character twist there are $\varphi(m)$ representations 
corresponding to $X_\eta$ where $\eta:U_F/U_F^1\to\bbC^\times$ is of order $m$.
According to Corollary 1.2 of \cite{Z2}, all dimension-m-Heisenberg 
representations of $G_F=\rm{Gal}(\overline{F}/F)$ are given as 
\begin{equation}
 \rho=\rho(X_\eta,\chi_K),\tag{1H}
\end{equation}
where $\chi_K: K^\times/ I_{F}K^\times\to\mathbb{C}^{\times}$ is a character 
such that the restriction of $\chi_K$
to the subgroup $K_{F}^{\times}$ corresponds to $X_\eta$ under the map (\ref{eqn 5.1.3}), and
\begin{equation}
 F^\times/{F^\times}^m\wedge F^\times/{F^\times}^m\cong K_{F}^{\times}/I_{F}K^\times,\tag{2H}
\end{equation}
which is given via the commutator in the relative Weil-group $W_{K/F}$ (for details arithmetic description of Heisenberg
representations of a Galois group, see \cite{Z2}, pp. 301-304).
The condition (2H) corresponds to (\ref{eqn 5.1.3}). Here the above Explicit Lemma \ref{Explicit Lemma} comes in.

Here due to our assumption both sides of (2H) are groups of order $m$.
And if one choice $\chi_K=\chi_0$ has been fixed, then all other $\chi_K$
are given as
\begin{equation}\label{eqn 4.20}
 \chi_K=(\chi_F\circ N_{K/F})\cdot\chi_0,
\end{equation}
for arbitrary characters of $F^\times$. For an optimal choice $\chi_K=\chi_0$, and order of $\chi_0$ we need the following lemma.

\begin{lem}\label{Lemma 5.3.3}
Let $K/F$ be the extension of $F/\bbQ_p$ for which $\rm{Gal}(K/F)=\bbZ_m\times\bbZ_m$. 
The $K_{F}^{\times}$ and $I_{F}K^\times$ are
as above. Then 
 the sequence 
 \begin{equation}\label{eqn 4.21}
  1\to U_{K}^{1}K_{F}^{\times}/U_{K}^{1}I_{F}K^\times\to U_K/U_{K}^{1}I_{F}K^\times\xrightarrow{N_{K/F}} U_F/U_{F}^{1}\to
  U_F/U_F\cap {F^\times}^m\to 1
 \end{equation}
is exact, and the outer terms are both of order $m$, hence inner terms are both cyclic of order $q_F-1$.
\end{lem}
\begin{proof}
 The sequence is exact because ${F^\times}^m=N_{K/F}(K^\times)$ is the group of norms, and 
 $F^\times/{F^\times}^m\cong \bbZ_m\times\bbZ_m$ implies
 that the right hand term\footnote{Since $gcd(m,p)=1$, we have 
 \begin{center}
  $U_F\cdot{F^\times}^m=(<\zeta>\times U_F^1)(<\pi_F^m>\times<\zeta^m>\times U_F^1)=<\pi_F^m>\times<\zeta>\times U_F^1$,
 \end{center}
where $\zeta$ is a $(q_F-1)$-st root of unity. 
Then 
\begin{center}
 $U_F/U_F\cap {F^\times}^m=U_F\cdot {F^\times}^m/{F^\times}^m=
 <\pi_F^m>\times<\zeta>\times U_F^1/<\pi_F^m>\times<\zeta^m>\times U_F^1\cong\bbZ_m$.
\end{center}
Hence $|U_F/U_F\cap{F^\times}^m|=m$.} is of order $m$. By our assumption the order of $K_{F}^{\times}/I_{F}K^\times$ is $m$. Now 
 we consider the exact sequence
 \begin{equation}\label{sequence 5.1.25}
  1\to U_{K}^{1}\cap K_{F}^{\times}/U_{K}^{1}\cap I_{F}K^\times\to K_{F}^{\times}/I_{F}K^\times\to 
  U_{K}^{1}K_{F}^{\times}/U_{K}^{1}I_{F}K^\times\to 1.
 \end{equation}
Since the middle term has order $m$, the left term must have order $1$, because $U_{K}^{1}$ is a pro-p-group and $gcd(m,p)=1$.
Hence the right term is also of order $m$. So the outer terms of the sequence (\ref{eqn 4.21}) have both order $m$, hence the inner 
terms must have the same order $q_F-1=[U_F:U_{F}^{1}]$, and they are cyclic, because the groups $U_F/U_{F}^{1}$ and $U_K/U_{K}^{1}$
are both cyclic.
\end{proof}

{\bf\large{We now are in a position to choose $\chi_K=\chi_0$ as follows}}: 
\begin{enumerate}
 \item we take $\chi_0$ as a character of $K^\times/U_{K}^{1}I_{F}K^\times$,
 \item we  take it on $U_{K}^{1}K_{F}^{\times}/U_{K}^{1}I_{F}K^\times$ as it is prescribed by the above 
 Explicit Lemma \ref{Explicit Lemma},
 in particular, $\chi_0$ restricted to that subgroup (which is cyclic of order $m$) will be faithful.
 \item we take it trivial on all primary components of the cyclic group $U_{K}/U_{K}^{1}I_{F}K^\times$ which are not $p_i$-primary,
 where $m=\prod_{i=1}^{n}p_i^{a_i}$.
 \item we take it trivial for a fixed prime element $\pi_K$.
\end{enumerate}

Under the above optimal choice of $\chi_0$, we have

\begin{lem}\label{Lemma 5.1.17}
Denote $\nu_p(n):=$ as the highest power of $p$ for which $p^{\nu_p(n)}|n$.
 The character $\chi_0$ must be a character of order 
 $$m_{q_F-1}:=\prod_{l|m}l^{\nu_l(q_F-1)},$$
 which we will call the $m$-primary part of $q_F-1$, so it determines a cyclic
extension $L/K$ of degree $m_{q_F-1}$ which is totally tamely ramified, and we can consider 
the Heisenberg representation $\rho=(X,\chi_0)$ of 
$G_F=\rm{Gal}(\overline{F}/F)$ is a representation of $\rm{Gal}(L/F)$, which is of order $m^2\cdot m_{q_F-1}$.
\end{lem}

\begin{proof}
By the given conditions, $m|q_F-1$. Therefore we can write
$$q_F-1=\prod_{l|m}l^{\nu_l(q_F-1)}\cdot \prod_{p|q_F-1,\; p\nmid m}p^{\nu_p(q_F-1)}=
m_{q_F-1}\cdot \prod_{p|q_F-1,\;p\nmid m}p^{\nu_p(q_F-1)},$$
where $l, p$ are prime, and $m_{q_F-1}=\prod_{l|m}l^{\nu_l(q_F-1)}$.

From the construction of $\chi_0$, $\pi_K\in\rm{Ker}(\chi_0)$, hence the order of $\chi_0$ comes from the restriction to 
$U_K$. Then the order of $\chi_0$ is $m_{q_F-1}$, because from Lemma \ref{Lemma 5.3.3}, the order of $U_K/U_{K}^{1}I_FK$ is
$q_F-1$. Since order of $\chi_0$ is $m_{q_F-1}$, by class field theory $\chi_0$ determines a cyclic 
extension $L/K$ of degree $m_{q_F-1}$, hence 
$$N_{L/K}(L^\times)=\rm{Ker}(\chi_0)=\rm{Ker}(\rho).$$
This means $G_L$ is the kernel of $\rho(X,\chi_0)$, hence $\rho(X,\chi_0)$ is actually a representation of 
$G_F/G_L\cong\rm{Gal}(L/F)$.

Since $G_L$ is normal subgroup of $G_F$, hence $L/F$ is a normal extension of degree $[L:F]=[L:K]\cdot[K:F]=m_{q_F-1}\cdot m^2$.
Thus $\rm{Gal}(L/F)$ is of order $m^2\cdot m_{q_F-1}$.

Moreover, since $[L:K]=m_{q_F-1}$ and $gcd(m,p)=1$, $L/K$ is tame. By construction we have a prime 
$\pi_K\in\rm{Ker}(\chi_0)=N_{L/K}(L^\times)$, hence $L/K$ is totally ramified extension. 

\end{proof}

\begin{lem}(Here $L$, $K$, and $F$ are the same as in Lemma \ref{Lemma 5.1.17})
 Let $F^{ab}/F$ be the maximal abelian extension. Then we have 
$$L\supset L\cap F^{ab}\supset K\supset F, \quad\{1\}\subset G'\subset Z(G)\subset G=\rm{Gal}(L/F),$$
where $[L:L\cap F^{ab}]=|G'|=m$ and $[L:K]=|Z(G)|=m_{q_F-1}$.
\end{lem}

\begin{proof}
 Let $F^{ab}/F$ be the maximal abelian extension. Then we have 
$$L\supset L\cap F^{ab}\supset K\supset F.$$
Here $L\cap F^{ab}/F$ is the maximal abelian in $L/F$. Then from Galois theory we can conclude 
$$\rm{Gal}(L/L\cap F^{ab})=[\rm{Gal}(L/F), \rm{Gal}(L/F)]=: G'.$$
Since $\rm{Gal}(L/F)=G_F/\rm{Ker}(\rho)$, and $[[G_F,G_F],G_F]\subseteq\rm{Ker}(\rho)$, from relation (\ref{eqn 5.1.3}) we have 
$$G'=[G_F,G_F]/\rm{Ker}(\rho)\cap [G_F,G_F]=[G_F,G_F]/[[G_F,G_F],G_F]\cong K_F^\times/I_FK^\times.$$
Again from sequence \ref{sequence 5.1.25} we have $|U_K^1K_F^\times/U_K^1 I_FK^\times|=|K_F^\times/I_FK^\times|=m$.
Hence $|G'|=m$.


From the Heisenberg property of $\rho$, we have 
$[[G_F,G_F],G_F]\subseteq\rm{Ker}(\rho)$, hence $\rm{Gal}(L/F)=G_F/\rm{Ker}(\rho)$ is a two-step nilpotent group.
This gives $[G',G]=1$, hence $G'\subseteq Z:=Z(G)$. Thus $G/Z$ is abelian. 

Moreover, here $Z$ is the scalar group of $\rho$, hence the dimension of $\rho$ is:
$$\rm{dim}(\rho)=\sqrt{[G:Z]}=m$$
Therefore the order of $Z$ is $m_{q_F-1}$ and $Z=\rm{Gal}(L/K)$.

\end{proof}

\begin{rem}[{\bf Special case: $m=2$, hence $p\ne 2$}]

Now if we take $m=2$, hence $p\ne 2$, and choose $\chi_0$ as the above optimal choice, then we will have 
$m_{q_F-1}=2_{q_F-1}=2$-primary factor of the number $q_F-1$, and $\rm{Gal}(L/F)$ is a $2$-group of order 
$4\cdot 2_{q_F-1}$.

 When $q_F\equiv -1\pmod{4}$, $q_F$ is of the form $q_F=4l-1$, where $l\ge 1$. So we can write $q_F-1=2(2l-1)$.
Since $2l-1$ is always odd, therefore when $q_F\equiv-1\pmod{4}$, the order of $\chi_0$ is $2_{q_F-1}=2$. 
Then $\rm{Gal}(L/F)$ will be of order 8 if and only if $q_F\equiv -1\pmod{4}$, i.e., if and only
if $i\not\in F$. And if $q_F\equiv 1\pmod{4}$, then similarly,  we can write $q_F-1=4m$ for some integer $m\ge1$, hence 
$2_{q_F-1}\ge 4$. Therefore when $q_F\equiv 1\pmod{4}$, the order of $\rm{Gal(L/F)}$ will be at least $16$.

\end{rem}

\end{exm}

\subsection{{\bf Artin conductors, Swan conductors, and the dimensions of Heisenberg representations}}

\begin{dfn}[{\bf Artin and Swan conductor}]
 Let $G$ be a finite group and $R(G)$ be the complex representation ring of $G$. For any two representations 
 $\rho_1,\rho_2\in R(G)$ with characters $\chi_1,\chi_2$ respectively, we have the Schur's inner product:
 $$<\rho_1,\rho_2>_G=<\chi_1,\chi_2>_G:=\frac{1}{|G|}\sum_{g\in G}\chi_1(g)\cdot\overline{\chi_2(g)}.$$
 Let $K/F$ be a finite Galois extension with Galois
 group $G:=\rm{Gal}(K/F)$. For an element $g\in G$ different from identity $1$, we define the non-negative integer 
 (cf. \cite{JPS}, Chapter IV, p. 62)
 $$i_G(g):=\rm{inf}\{\nu_K(x-g(x))|\; x\in O_K\}.$$
 By using this non-negative (when $g\ne 1$) integer $i_G(g)$ we define a function $a_G:G\to\bbZ$ as follows:
 \begin{center}
  $a_G(g)=-f_{K/F}\cdot i_G(g)$ when $g\ne 1$, and $a_G(1)=f_{K/F}\sum_{g\ne 1}i_G(g)$.
 \end{center}
Thus from this definition we can see that $\sum_{g\in G}a_G(g)=0$, hence $<a_G, 1_G>=0$. 
It can be proved (cf. \cite{JPS}, p. 99, Theorem 1) that the function $a_G$ is the character of a linear representation of $G$,
and that corresponding linear representation is called the {\bf Artin representation} $A_G$ of $G$.

Similarly, for a nontrivial $g\ne 1\in G$, we define (cf. \cite{VS}, p. 247)
$$s_G(g)=\rm{inf}\{\nu_K(1-g(x)x^{-1})|\;x\in K^\times\},\qquad s_G(1)=-\sum_{g\ne 1}s_G(g).$$
And we can define a function $\rm{sw}_G:G\to\bbZ$ as follows:
$$\rm{sw}_G(g)=-f_{K/F}\cdot s_G(g)$$
It can also be shown that $\rm{sw}_G$ is a character of a linear representation of $G$, and that corresponding representation
is called the {\bf Swan representation} $SW_G$ of $G$.

From \cite{JP}, p. 160 , we have the relation between the Artin and Swan representations (cf. \cite{VS}, p. 248, equation (6.1.9))
\begin{equation}\label{eqn 5.1.22}
 SW_G=A_G+\rm{Ind}_{G_0}^{G}(1)-\rm{Ind}_{\{1\}}^{G}(1),
\end{equation}
$G_0$ is the $0$-th ramification group (i.e., inertia group) of $G$.

Now we are in a position to define the Artin and Swan conductor of a representation $\rho\in R(G)$. The Artin conductor of a 
representation $\rho\in R(G)$ is defined by 
$$a_F(\rho):=<A_G,\rho>_G=<a_G,\chi>_G,$$
where $\chi$ is the character of 
the representation $\rho$. Similarly, for the representation $\rho$, the Swan conductor is:
$$\rm{sw}_F(\rho):=<SW_G,\rho>_G=<\rm{sw}_G,\chi>_G.$$
For more details about Artin and Swan conductor, see Chapter 6 of \cite{VS} and Chapter VI of \cite{JPS}.
\end{dfn}
From equation (\ref{eqn 5.1.22}) we obtain
\begin{equation}\label{eqn 5.1.23}
 a_F(\rho)=\rm{sw}_F(\rho)+\rm{dim}(\rho)-<1,\rho>_{G_0}.
\end{equation}
Moreover, from Corollary of Proposition 4 on p. 101 of \cite{JPS}, for an induced representation 
$\rho:=\rm{Ind}_{\rm{Gal}(K/E)}^{\rm{Gal}(K/F)}(\rho_E)=\rm{Ind}_{E/F}(\rho_E)$, we have
\begin{equation}\label{eqn 5.1.24}
 a_F(\rho)=f_{E/F}\cdot \left( d_{E/F}\cdot \rm{dim}(\rho_E)+\textrm{a}_E(\rho_E)\right).
\end{equation}
We apply this formula (\ref{eqn 5.1.24}) for $\rho_E=\chi_E$ of dimension $1$ and then conversely 
$$a(\chi_E)=\frac{a_F(\rho)}{f_{E/F}}-d_{E/F}.$$
So if we know $a_F(\rho)$ then we can compute $a(\chi_E)$. 


Let $\{G^i\}$, where $i\ge 0,\in\bbQ$ be the ramification subgroups (in the upper numbering) of a local Galois group $G$.
Now let $\rho$ be an irreducible representation of $G$. For this irreducible $\rho$ we define 
$$j(\rho):=\rm{max}\{ i\;|\; \rho|_{G^i}\not\equiv 1\}.$$
Now if $\rho$ is an irreducible representation of $G$ which is not an unramified character, 
then $\rho|_{I}\not\equiv 1$, where $I=G^0=G_0$ is the inertia subgroup
of $G$. Thus from the definition of $j(\rho)$ we can say, if $\rho$ is irreducible, then we always have 
$j(\rho)\ge 0$, i.e., $\rho$ is nontrivial on the inertia group $G_0$. Then from the definitions of Swan and Artin 
conductors, and equation (\ref{eqn 5.1.23}), when $\rho$ is irreducible, we have the following relations
\begin{equation}\label{eqn 5.1.281}
 \rm{sw}_F(\rho)=\rm{dim}(\rho)\cdot j(\rho),\qquad a_F(\rho)=\rm{dim}(\rho)\cdot (j(\rho)+1).
\end{equation}
From the Theorem of Hasse-Arf (cf. \cite{JPS}, p. 76), if $\rm{dim}(\rho)=1$, i.e., $\rho$ is a character of $G/[G,G]$, 
we can say that $j(\rho)$ must be an integer, then $\rm{sw}_F(\rho)=j(\rho), a_F(\rho)=j(\rho)+1$.
Moreover, by class field theory, $\rho$ corresponds to a linear character $\chi_F$, hence for linear character $\chi_F$, we can write 
$$j(\chi_F):=\rm{max}\{i\;|\;\chi_F|_{U_F^i}\not\equiv1\},$$
because under class field theory (under Artin isomorphism) 
the upper numbering in the filtration of $\rm{Gal}(F_{ab}/F)$ is compatible with the filtration (descending chain) of the group of units 
$U_F$.

From equation (\ref{eqn 5.1.281}),
it is easy to see that for higher dimensional $\rho$, we have $\rm{sw}_F(\rho), a_F(\rho)$ multiples of $\rm{dim}(\rho)$ if and only 
if $j(\rho)$ is an integer.

Now we come to our Heisenberg representations. For each $X\in\widehat{FF^\times}$ we define
\begin{equation}
 j(X):=\begin{cases}
        0 & \text{when $X$ is trivial}\\
        \rm{max}\{i\;|\; X|_{UU^i}\not\equiv 1\} & \text{when $X$ is nontrivial},
       \end{cases}
\end{equation}
where $UU^i\subseteq FF^\times$ is a subgroup which under (\ref{eqn 5.1.2}) corresponds 
$$G_F^i\cap[G_F,G_F]/G_F^i\cap[[G_F,G_F],G_F]\subseteq[G_F,G_F]/[[G_F,G_F],G_F].$$
Let $\rho=\rho(X_\rho,\chi_K)$ be a {\bf minimal conductor} (i.e., a representation with the smallest Artin conductor) 
Heisenberg representation for $X_\rho$ of the absolute Galois group $G_F$. 
From Theorem 3 on p. 125 of \cite{Z5}, we 
have 
\begin{equation}\label{eqn 5.1.26}
 \rm{sw}_F(\rho)=\rm{dim}(\rho)\cdot j(X_\rho)=\sqrt{[F^\times:\rm{Rad}(X_\rho)]}\cdot j(X_\rho).
\end{equation}
Moreover if $\rho_0=\rho_0(X,\chi_0)$ is a minimal representation
corresponding $X$, then all other Heisenberg
representations of dimension $\rm{dim}(\rho)$ are of the form $\rho=\chi_F\otimes \rho_0=(X, (\chi_F\circ N_{K/F})\chi_0)$,
where $\chi_F:F^\times\to \bbC^\times$. Then 
we have (cf. \cite{Z2}, p. 305, equation (5))
\begin{equation}\label{eqn 5.1.27}
 \rm{sw}_F(\rho)=\rm{sw}_F(\chi_F\otimes\rho_0)=\sqrt{[F^\times:\rm{Rad}(X)]}\cdot\rm{max}\{j(\chi_F), j(X)\}.
\end{equation}


For minimal conductor U-isotopic Heisenberg representation we have the following proposition.

\begin{prop}\label{Proposition conductor}
 Let $\rho=\rho(X_\eta,\chi_K)$ be a U-isotropic Heisenberg representation of $G_F$ of minimal conductor. 
 Then we have the following conductor relation
 \begin{center}
  $j(X_\eta)=j(\eta)$, $\rm{sw}_F(\rho)=\rm{dim}(\rho)\cdot j(X_\eta)=\#\eta\cdot j(\eta)$,
  $a_F(\rho)=\rm{sw}_F(\rho)+\rm{dim}(\rho)=\#\eta(j(\eta)+1)=\#\eta\cdot a_F(\eta)$.
 \end{center}
 
\end{prop}

\begin{proof}
 From \cite{Z5}, on p. 126, Proposition 4(i) and Proposition 5(ii), and $U\wedge U=U^1\wedge U^1$, we see the injection 
$U^i\wedge F^\times\subseteq UU^i$ induces a natural isomorphism 
$$U^i\wedge<\pi_F>\cong UU^{i}/UU^i\cap (U\wedge U)$$
for all $i\ge 0$. 

Now let $j(X_\eta)=n-1$, hence $X_\eta|_{UU^n}=1$ but $X_\eta|_{UU^{n-1}}\ne 1$.
This gives $X_\eta|_{U^n\wedge<\pi_F>}=1$ but $X_\eta|_{U^{n-1}\wedge<\pi_F>}\ne 1$. Now from equation (\ref{eqn 5.1.25})
we can conclude that $\eta(x)=1$ for all $x\in U^n$ but $\eta(x)\ne 1$ for $x\in U^{n-1}$. Hence 
$$j(\eta)=n-1=j(X_\eta).$$
Again from the definition of $j(\chi)$, where $\chi$ is a character of $F^\times$, we can see that 
$j(\chi)=a(\chi)-1$, i.e., $a(\chi)=j(\chi)+1$.

From equation (\ref{eqn 5.1.26}) we obtain:
$$\rm{sw}_F(\rho)=\rm{dim}(\rho)\cdot j(X_\eta)=\#\eta\cdot j(\eta),$$
since $\rm{dim}(\rho)=\#\eta$ and $j(X_\eta)=j(\eta)$. Finally, from  equation (\ref{eqn 5.1.23}) for $\rho$ (here $<1,\rho>_{G_0}=0$),
we have 
\begin{equation}\label{eqn 5.1.28}
 a_F(\rho)=\rm{sw}_F(\rho)+\rm{dim}(\rho)=\#\eta\cdot j(\eta)+\#\eta=\#\eta\cdot (j(\eta)+1)=\#\eta\cdot a_F(\eta).
\end{equation}
\end{proof}

By using the equation (\ref{eqn 5.1.24}) in our Heisenberg setting, we have the following proposition.

\begin{prop}\label{Proposition 5.1.20}
 Let $\rho=\rho(Z,\chi_\rho)=\rho(X,\chi_K)$ be a Heisenberg representation of the absolute Galois group $G_F$ of a field 
 $F/\bbQ_p$ of dimension $m$. Let $E/F$ be any subextension in $K/F$ corresponding to a maximal isotropic subgroup for $X$. Then 
 $$a_F(\rho)=a_F(\rm{Ind}_{E/F}(\chi_E)),\qquad m\cdot a_F(\rho)=a_F(\rm{Ind}_{K/F}(\chi_K)).$$
 As a consequence we have 
 $$a(\chi_K)=e_{K/E}\cdot a(\chi_E)-d_{K/E}.$$
 In particular $a(\chi_K)=a(\chi_E)$ if $K/E$ is unramified.
\end{prop}
\begin{proof}
 We know that $\rho=\rm{Ind}_{E/F}(\chi_E)$ and $m \cdot \rho=\rm{Ind}_{K/F}(\chi_K)$.
 By the definition of Artin conductor we can write 
 $$a_F(\rm{dim}(\rho)\cdot \rho)=\rm{dim}(\rho)\cdot a_F(\rho)=m\cdot a_F(\rm{Ind}_{E/F}(\chi_E)).$$
 Since $K/E/F$ is a tower of Galois extensions with $[K:E]=m=e_{K/E}f_{K/E}$, we have the transitivity relation of 
 different (cf. \cite{JPS}, p. 51,
 Proposition 8)
 $$\mathcal{D}_{K/F}=\mathcal{D}_{K/E}\cdot \mathcal{D}_{E/F}.$$
 Now from the definition of different of a Galois extension, and taking $K$-valuation we obtain:
 \begin{equation}\label{eqn discriminant relation}
  d_{K/F}=d_{K/E}+e_{K/E}\cdot d_{E/F}.
 \end{equation}
 Now by using equation (\ref{eqn 5.1.24}) we have:
 \begin{equation}\label{eqn 44}
  m\cdot a_F(\rm{Ind}_{E/F}(\chi_E))=m\cdot f_{E/F}\left(d_{E/F}+a(\chi_E)\right)=m\cdot f_{E/F}\cdot d_{E/F}+e_{K/E}\cdot f_{K/F}
  \cdot a(\chi_E),
 \end{equation}
and 
\begin{equation}\label{eqn 45}
 a_F(\rm{Ind}_{K/F}(\chi_K))=f_{K/F}\cdot\left(d_{K/F}+a(\chi_K)\right)=f_{K/F}\cdot d_{K/F}+f_{K/F}\cdot a(\chi_K).
\end{equation}
By using equation (\ref{eqn discriminant relation}), from equations (\ref{eqn 44}), (\ref{eqn 45}), we have 
$$a(\chi_K)=e_{K/E}\cdot a(\chi_E)-d_{K/E}.$$
And when $K/E$ is unramified, i.e., $e_{K/E}=1$ and $d_{K/E}=0$, hence $a(\chi_K)=a(\chi_E)$.

\end{proof}

Now by combining Proposition \ref{Proposition 5.1.20} with Proposition \ref{Proposition conductor}, 
we get the following result.

\begin{lem}\label{Lemma general conductor}
 Let $\rho=\rho(X_\eta,\chi_K)$ be a U-isotopic Heisenberg representation 
 of minimal conductor of the absolute Galois group $G_F$ of a non-archimedean
 local field $F$.
 Let $K=K_\eta$ correspond to the radical of $X_\eta$, and let $E_1/F$ be the maximal unramified subextension, and $E/F$
 be any maximal cyclic and totally ramified subextension in $K/F$. Let $m$ denote the order of $\eta$.
 Then $\rho$ is induced by $\chi_{E_1}$ or by 
 $\chi_E$ respectively, and we have 
 \begin{enumerate}
  \item $a_E(\chi_E)=m\cdot a(\eta)-d_{E/F}$,
  \item $a_{E_1}(\chi_{E_1})=a(\eta)$,
  \item and for the character $\chi_K\in\widehat{K^\times}$,
  $$a_K(\chi_K)=m\cdot a(\eta)-d_{K/F}.$$
 \end{enumerate}
Moreover, $a_E(\chi_E)=a_K(\chi_K)$. 
\end{lem}
\begin{proof}
Proof of these assertions follows from equation (\ref{eqn 5.1.24}) and Proposition \ref{Proposition conductor}. When 
$\rho=\rm{Ind}_{E/F}(\chi_E)$, where $E/F$ is a maximal cyclic and totally ramified subextension in $K/F$, from equation 
(\ref{eqn 5.1.24}) we have
\begin{align*}
 a_F(\rho)
 &=m\cdot a(\eta)\quad\text{using Proposition $\ref{Proposition conductor}$},\\
 &=f_{E/F}\cdot\left(d_{E/F}\cdot 1+a_E(\chi_E)\right),\quad\text{since $\rho=\rm{Ind}_{E/F}(\chi_E)$}\\
 &=1\cdot\left(d_{E/F}+a_E(\chi_E)\right).
\end{align*}
because $E/F$ is totally ramified, hence $f_{E/F}=1$.  This implies $a_E(\chi_E)=m\cdot a(\eta)-d_{E/F}$.

Similarly, when $\rho=\rm{Ind}_{E_1/F}(\chi_{E_1})$, where $E_1/F$ is the maximal unramified subextension in $K/F$, hence 
$f_{E_1/F}=m$ and $d_{E_1/F}=0$, by using equation (\ref{eqn 5.1.24}) we obtain $a_{E_1}(\chi_{E_1})= a(\eta)$.

Again from Proposition \ref{Proposition 5.1.20} we have 
$$a_K(\chi_K)=m\cdot a(\chi_{E_1})-d_{K/E_1}=m\cdot a(\eta)-d_{K/F}.$$

Finally, since $E/F$ is a maximal cyclic totally ramified implies $K/E$ is unramified and therefore 
$$d_{E/F}=d_{K/F},\quad\text{and hence}\; a_E(\chi_E)=a_K(\chi_K).$$
\end{proof}

\begin{rem}\label{Remark 5.1.22}
 Assume that we are in the dimension $m=\#\eta$ prime to $p$ case. Then from Corollary \ref{Corollary U-isotropic}, $\eta$
must be a character of $U/U^1$ (for $U=U_F$), hence
$$ a(\eta)=1\qquad  a_F(\rho_0) =m.$$
Therefore in this case the minimal conductor of $\rho$ is $m$, hence it is equal to the dimension of $\rho$. 

From the above Lemma \ref{Lemma general conductor}, in this case we have 
$$a_{E_1}(\chi_{E_1})=a(\eta)=1.$$
And $K/F, E/F$ are tamely ramified of ramification exponent $e_{K/F}=m$, hence
$$ a_E(\chi_E) = a_K(\chi_K) = m\cdot a(\eta)-d_{K/F}=m -(e_{K/F}-1)=m-(m-1)=1.$$
Thus we can conclude that in this case all three characters (i.e., $\chi_{E_1},\chi_E$, and $\chi_K$) are of conductor $1$.

In the general case $a_{E_1}(\chi_{E_1}) = a(\eta)$ and
$$a_E(\chi_E)= a_K(\chi_K) = m\cdot a(\eta)-d,$$
where $d=d_{E/F}=d_{K/F}$, conductors will be different.
\end{rem}

In general, if $\rho=\rho_0\otimes\chi_F$, where $\rho_0$ is a finite dimensional minimal conductor representation of $G_F$, and 
$\chi_F\in\widehat{F^\times}$, then we have the following result.

\begin{lem}\label{Lemma 5.1.23}
 Let $\rho_0$ be a finite dimensional representation of $G_F$ of minimal conductor.
 Then we have 
 \begin{equation}
  a_F(\rho)=\rm{dim}(\rho_0)\cdot a_F(\chi_F),
 \end{equation}
where $\rho=\rho_0\otimes\chi_F=\rho(X_\eta,(\chi_F\circ N_{K/F})\chi_0)$ and $\chi_F\in\widehat{F^\times}$ with 
$a(\chi_F)>\frac{a(\rho_0)}{\rm{dim}(\rho)}$.
\end{lem}
\begin{proof}
From equation (\ref{eqn 5.1.281}) we have $a_F(\rho_0)=\rm{dim}(\rho_0)\cdot (1+j(\rho_0))$.
By the given condition $\rho_0$ is of minimal conductor. So for representation $\rho=\rho_0\otimes\chi_F$, we have 
\begin{align*}
 a_F(\rho)
 &=a_F(\rho_0\otimes\chi_F)=\rm{dim}(\rho_0)\cdot\left(1+\rm{max}\{j(\rho_0),j(\chi_F)\}\right)\\
 &=\rm{dim}(\rho_0)\cdot\rm{max}\{1+j(\chi_F), 1+j(\rho_0)\}\\
 &=\rm{dim}(\rho_0)\cdot\rm{max}\{a(\chi_F), 1+j(\rho_0)\}\\
 &=\rm{dim}(\rho_0)\cdot a_F(\chi_F),
\end{align*}
because by the given condition 
$$a(\chi_F)>\frac{a(\rho_0)}{\rm{dim}(\rho_0)}=\frac{\rm{dim}(\rho_0)\cdot(1+j(\rho_0))}{\rm{dim}(\rho_0)}=1+j(\rho_0).$$

\end{proof}

\begin{prop}\label{Proposition 5.1.23}
 Let $\rho=\rho(X,\chi_K)$ be a Heisenberg representation of dimension $m$ of the absolute Galois group $G_F$ of a 
 non-archimedean local field $F$.
 Then $m| a_F(\rho)$ if and only if:\\
$X$ is $U$-isotropic, or (if $X$ is not $U$-isotropic) $a_F(\rho)$ is with respect to $X$ not the minimal conductor.
\end{prop}
\begin{proof}
From the above Lemma \ref{Lemma 5.1.23} we know that if $\rho$ is not minimal, then $a_F(\rho)$ is always a multiple of the 
dimension $m$. So now we just have to check for minimal conductors. In the U-isotropic case the minimal conductor is multiple
of the dimension (cf. Proposition \ref{Proposition conductor}). 

Finally, suppose that $X$ is not U-isotropic, i.e., $X|_{U\wedge U}=X|_{U^1\wedge U^1}\not\equiv1$, because 
$U\wedge U=U^1\wedge U^1$ (see the Remark on p. 126 of \cite{Z5}). We also know that 
$UU^i=(UU^i\cap U^1\wedge U^1)\times(U^i\wedge<\pi_F>)$ (cf. \cite{Z5}, p. 126, Proposition 5(ii)). 
In Proposition 5 of \cite{Z5}, we observe that all the jumps $v$ in the filtration $\{UU^i\cap (U^1\wedge U^1)\}, i\in\bbR_{+}$
are not {\bf integers with $v>1$}. This shows that $j(X)$ is also not an integer, hence $a_F(\rho_0)$ is not 
multiple of the dimension. This implies the conductor $a_F(\rho)$ is not minimal.

\end{proof}

For giving invariant formula of $W(\rho)$, we need to know the explicit dimension formula of $\rho$.
In the following theorem we give the general dimension formula of a Heisenberg representation.
\begin{thm}[{\bf Dimension}]\label{Dimension Theorem}
Let $F/\bbQ_p$ be a local field and $G_F$ be the absolute Galois group of $F$. If $\rho$ is a Heisenberg representation of 
$G_F$, then $\rm{dim}(\rho)=p^n\cdot d'$, where $n\ge 0$ is an integer and where the prime to $p$ factor $d'$ must divide $q_F-1$.
\end{thm}
\begin{proof}
 By the definition of Heisenberg representation $\rho$ we have the relation 
 $$[[G_F,G_F],G_F]\subseteq\rm{Ker}(\rho).$$
 Then we can consider $\rho$ as a representation of $G:=G_F/[[G_F,G_F],G_F]$. Since 
 $[x,g]\in [[G_F,G_F],G_F]$ for all $x\in [G_F,G_F]$ and $g\in G_F$, we have $[G,G]=[G_F,G_F]/[[G_F,G_F],G_F]\subseteq Z(G)$,
 hence $G$ is a two-step nilpotent group.
 
 We know that each nilpotent group is isomorphic to the direct product of its Sylow subgroups. Therefore we can write 
 $$G=G_p\times G_{p'},$$
 where $G_p$ is the Sylow $p$-subgroup, and $G_{p'}$ is the direct product of all other Sylow subgroups. Therefore each irreducible
 representation $\rho$ has the form $\rho=\rho_{p}\otimes\rho_{p'}$, where $\rho_{p}$ and $\rho_{p'}$ are irreducible representations of 
 $G_p$ and $G_{p'}$ respectively. 
 
 We also know that finite $p$-groups are nilpotent groups, and direct product of a finite number of 
 nilpotent groups is again a nilpotent group.
 So $G_p$ and $G_{p'}$ are both two-step nilpotent group because $G$ is a two-step nilpotent group. Therefore the representations
 $\rho_p$ and $\rho_{p'}$ are both Heisenberg representations of $G_p$ and $G_{p'}$ respectively.
 
 Now to prove our assertion, we have to show that $\rm{dim}(\rho_p)$ can be an arbitrary power of $p$, whereas 
 $\rm{dim}(\rho_{p'})$ must divide $q_F-1$. Since
 $\rho_p$ is an {\bf irreducible} representation of $p$-group $G_p$, so the dimension of $\rho_p$ is some $p$-power.
 
 Again from the construction of $\rho_{p'}$ we can say that $\rm{dim}(\rho_{p'})$ is {\bf prime} to $p$. 
 Then from Lemma \ref{Lemma dimension equivalent} $\rm{dim}(\rho_{p'})$ is a divisor of $q_F-1$.

This completes the proof.

\end{proof}

\begin{rem}\label{Remark 5.1.3}
{\bf (1).}
Let $V_F$ be the wild ramification subgroup of $G_F$.
 We can show that $\rho|_{V_F}$ is irreducible if and only if $Z_\rho=G_K\subset G_F$
 corresponds to an abelian extension $K/F$ which is totally ramified and wildly 
 ramified\footnote{Group theoretically, if $\rho|_{V_F}=\rm{Ind}_{H}^{G_F}(\chi_H)|_{V_F}$ is irreducible, then from 
 Section 7.4 of \cite{JP},
we can say $G_F=H\cdot V_F$. Here $H=G_L$, where $L$ is a certain extension of $F$, and $V_F=G_{F_{mt}}$ where $F_{mt}/F$ is the maximal 
tame extension of $F$. Therefore $G_F=H\cdot V_F$ is equivalent to $F=L\cap F_{mt}$ that means the extension $L/F$ must be totally 
ramified and wildly ramified, and $[G_F:H]=[L:F]=|V_F|$.
We know that the wild ramification subgroup $V_F$ is a pro-p-group (cf. \cite{FV}, p. 106). Then 
 $\rm{dim}(\rho)$ is a power of $p$.} (cf. \cite{Z2}, p. 305). If $N:=N_{K/F}(K^\times)$ is the subgroup
 of norms, then this means that $N\cdot U_{F}^{1}=F^\times$, in other words,
 $$F^\times/N=N\cdot U_{F}^{1}/N=U_{F}^{1}/N\cap U_{F}^{1},$$
 where $N$ can be also considered as the radical of $X_\rho$. So we can consider the alternating character $X_\rho$ on the principal
 units $U_{F}^{1}\subset F^\times$. Then 
 $$\rm{dim}(\rho)=\sqrt{[F^\times:N]}=\sqrt{[U_F^1: N\cap U_F^1]},$$
 is a power of $p$, because $U_F^1$ is a pro-p-group.

 Here we observe: If $\rho=\rho(X,\chi_K)$ with $\rho|_{V_F}$ stays irreducible, then
 $\rm{dim}(\rho)=p^n$, $n\ge 1$ and  
 $K/F$ is a totally and {\bf wildly} ramified. But there is 
 a {\bf big} class of Heisenberg representations $\rho$ such that $\rm{dim}(\rho)=p^n$ is a $p$-power, but which are not 
 wild representations (see the Definition \ref{Definition U-isotropic} of U-isotropic).\\
 {\bf (2).}
Let $\rho=\rho(X,\chi_K)$ be a Heisenberg representation of the absolute Galois group $G_F$ of dimension $d>1$ which is prime 
to $p$. Then from above Lemma \ref{Lemma dimension equivalent}, we have  $d|q_F-1$. For this representation $\rho$, here 
$K/F$ must be tame if $\rm{Rad}(X)=\cN_{K/F}$ (cf. \cite{FV}, p. 115).
\end{rem}

\section{\textbf{Invariant formula for $W(\rho)$}}

\begin{lem}\label{Lemma 5.2.1}
Let $\rho=\rho(Z,\chi_Z)$ be a Heisenberg representation of the local Galois group $G=\mathrm{Gal}(L/F)$ of odd dimension.
Let $H$ be a maximal isotropic subgroup for $\rho$ and $\chi_H\in\widehat{H}$ with $\chi_H|_{Z}=\chi_Z$
then:
 \begin{equation}\label{eqn 4.9}
  W(\rho)=W(\chi_H),\hspace{.5cm} W(\rho)^{\mathrm{dim}(\rho)}=W(\chi_Z),
 \end{equation}
 and 
 \begin{equation}
  W(\chi_H)^{[H:Z]}=W(\chi_Z).
 \end{equation}
 \end{lem}
\begin{proof}
From the construction of Heisenberg representation $\rho$ of $G$ we have 
\begin{center}
 $\rho=\rm{Ind}_{H}^{G}(\chi_H)$, \hspace{.4cm} $\rm{dim}(\rho)\cdot\rho=\rm{Ind}_{Z}^{G}(\chi_Z)$.\\
 This implies that $W(\rho)=\lambda_{H}^{G}\cdot W(\chi_H)$ and $W(\rho)^{\rm{dim}(\rho)}=\lambda_{Z}^{G}\cdot W(\chi_Z).$
\end{center}
Since $\rm{dim}(\rho)$ is odd we may apply now Lemma \ref{Lemma 4.2}, and we obtain 
 \begin{center}
  $\lambda_{H}^{G}=\lambda_{Z}^{G}=1$.
 \end{center}
 So, we have $W(\rho)=\lambda_{H}^{G}(W)\cdot W(\chi_H)=W(\chi_H)$. Similarly, we have 
 $W(\rho)^{\mathrm{dim}(\rho)}=W(\chi_Z)$.
 
Moreover, it is easy to see\footnote{We have 
 \begin{center}
  $d\cdot\rho=\mathrm{Ind}_{Z}^{G}\chi_Z=\mathrm{Ind}_{H}^{G}(\mathrm{Ind}_{Z}^{H}\chi_Z)$,
 \end{center}
 and $\mathrm{Ind}_{Z}^{H}\chi_Z$ of dimension $d=[H:Z]$. Therefore:
 \begin{center}
  $W(\rho)^d=(\lambda_{H}^{G})^d\cdot W(\mathrm{Ind}_{Z}^{H}\chi_Z)$.
 \end{center}
On the other hand $W(\rho)=\lambda_{H}^{G}\cdot W(\chi_H)$ implies
\begin{center}
 $W(\rho)^d=(\lambda_{H}^{G})^d\cdot W(\chi_H)^d$.
\end{center}
Now comparing these two expressions for $W(\rho)^d$ we see that 
\begin{center}
 $W(\chi_H)^d=W(\mathrm{Ind}_{Z}^{H}\chi_Z)$.
\end{center}} that $W(\rm{Ind}_{Z}^{H}(\chi_Z))=W(\chi_H)^{[H:Z]}$.
By the given condition, $[H:Z]=\rm{dim}(\rho)$ is odd, hence $\lambda_{Z}^{H}=1$, 
then we have 
\begin{equation}\label{eqn 5.2.4}
 W(\chi_H)^{[H:Z]}=W(\rm{Ind}_{Z}^{H}(\chi_Z))=W(\chi_Z).
\end{equation}

 \end{proof}
 
\begin{rem}
 Related to $G\supset H\supset Z$ we have the base fields $F\subset E\subset K$, and $\chi_Z$ is the restriction of 
 $\chi_H$. In arithmetic terms this means:
 $$\chi_K=\chi_E\circ N_{K/E}.$$
 So in arithmetic terms of $W(\rm{Ind}_{Z}^{G}(\chi_Z))=W(\rm{Ind}_{H}^{G}(\chi_H))^{[G:H]}$ is as follows:
$$W(\rm{Ind}_{K/F}(\chi_K),\psi)=W(\rm{Ind}_{E/F}(\chi_E),\psi)^{[K:E]}.$$
Then from $\lambda_{K/F}=\lambda_{K/E}\cdot\lambda_{E/F}^{[K:E]}$ we can conclude that 
$$\lambda_{K/E}\cdot W(\chi_K,\psi_K)=W(\chi_E,\psi_E)^{[K:F]}.$$
If the dimension $\rm{dim}(\rho)=[K:E]$ is odd, we have $\lambda_{K/E}=1$, because $K/E$ is Galois.
Then we obtain
\begin{equation}\label{eqn 5.2.5}
  W(\chi_E,\psi_E)^{[K:E]}=W(\chi_E\circ N_{K/E},\psi_E\circ\rm{Tr}_{K/E}).
 \end{equation}
The formula (\ref{eqn 5.2.5}) is known as a {\bf Davenport-Hasse} relation.

\end{rem}

\begin{cor}
 Let $\rho=\rho(Z,\chi_Z)$ be a Heisenberg representation of a local Galois group $G$. Let $\rm{dim}(\rho)=d$ be odd.
 Let the order of $W(\chi_Z)$ be $n$ (i.e., $W(\chi_Z)^n=1$). If $d$ is prime to $n$, then $d^{\varphi(n)}\equiv 1\mod{n}$, and 
 $$W(\rho)=W(\chi_Z)^{\frac{1}{d}}=W(\chi_Z)^{d^{\varphi(n)-1}},$$
where $\varphi(n)$ is the Euler's totient function of $n$.
\end{cor}
 \begin{proof}
 By our assumption, here $d$ and $n$ are coprime. Therefore from {\bf Euler's theorem} we can write 
 $$d^{\varphi(n)}\equiv 1\mod{n}.$$
 This implies $d^{\varphi(n)}-1$ is a multiple of $n$.
 
 Here $d$ is odd, then from the above Lemma \ref{Lemma 5.2.1} we have $W(\rho)^d=W(\chi_Z)$.
 So we obtain
 $$W(\rho)=W(\chi_Z)^{\frac{1}{d}}=W(\chi_Z)^{d^{\varphi(n)-1}},$$
 since $d^{\varphi(n)}-1$ is a multiple of $n$, and by assumption $W(\chi_Z)^n=1$.
\end{proof}

We observe that when $\rm{dim}(\rho)=d$ is odd, if we take second part of the 
equation (\ref{eqn 4.9}), we have $W(\rho)=W(\chi_Z)^{\frac{1}{d}}$, but it is not well-defined in general. 
Here we have to make precise 
which root $W(\chi_Z)^{\frac{1}{d}}$ really occurs. That is why, giving invariant formula of
$W(\rho)$ using $\lambda$-functions computation is difficult. In the following theorem we give an invariant formula of local 
constant for Heisenberg representation modulo certain roots of unity.

\begin{thm}\label{Theorem invariant odd}
 Let $\rho=\rho(X,\chi_K)$ be a Heisenberg representation of the absolute Galois group $G_F$ of a local field $F/\bbQ_p$
 of dimension $d$. Let $\psi$ be a nontrivial canonical additive character of $F$ and $\psi_K:=\psi\circ\rm{Tr}_{K/F}$.
 Denote $\mu_{p^\infty}$ as the group of roots of unity of $p$-power order and $\mu_{n}$ as the group of 
 $n$-th roots of unity, where $n>1$ is an integer.
 \begin{enumerate}
  \item When the dimension $d$ is odd, we have 
   \begin{center}
  $W(\rho,\psi)\equiv W(\chi_\rho)'\mod{\mu_{d}}$,
 \end{center}
where $W(\chi_\rho)'$ is any $d$-th root of 
$W(\chi_K,\psi_K)$.
\item When the dimension $d$ is even, we have 
 \begin{center}
  $W(\rho,\psi)\equiv W(\chi_\rho)'\mod{\mu_{d'}}$,
 \end{center}
 where $d'=\rm{lcm}(4,d)$.
 \end{enumerate}
 
\end{thm}

\begin{proof}
{\bf (1).}
We know that the lambda functions are always fourth roots of unity. In particular, when degree of the Galois extension 
$K/F$ is odd, from Theorem \ref{General Theorem for odd case} we have $\lambda_{K/F}=1$. For proving our assertions we will use these 
facts about $\lambda$-functions.

We know that $\rm{dim}(\rho)\cdot\rho=\rm{Ind}_{K/F}(\chi_K)$, where by class field theory $\chi_K\leftrightarrow\chi_\rho$ 
 is a character of $K^\times$.
When $d$ is odd, we can write 
 $$W(\rho,\psi)^d=\lambda_{K/F}(\psi)\cdot W(\chi_K,\psi_K)= W(\chi_K,\psi_K).$$
 Now let $W(\chi_\rho)'$ be any $d$-th root of $W(\chi_K,\psi_K)$. Then we have 
 $$W(\rho,\psi)^d={W(\chi_\rho)'}^d,$$
 hence $\frac{W(\rho,\psi)}{W(\chi_\rho)'}$ is a $d$-th root of unity. Therefore we have
 $$W(\rho,\psi)\equiv W(\chi_\rho)' \mod{\mu_{d}}.$$
 {\bf (2).}
Similarly, we can give invariant formula for even degree Heisenberg representations. When the dimension $d$ of $\rho$ is even, we have 
\begin{equation}\label{eqn 5.2.10}
 W(\rho,\psi)^d=\lambda_{K/F}(\psi)\cdot W(\chi_K,\psi_K)\equiv W(\chi_K,\psi_K)\mod{\mu_4},
\end{equation}
because $\lambda_{K/F}$ is a fourth root of unity.
Now let $W(\chi_\rho)'$ be any $d$-th root of $W(\chi_K,\psi_K)$, hence $W(\chi_K,\psi_K)=W(\chi_\rho)'^d$. Then from equation 
(\ref{eqn 5.2.10}) we have 
$$\left(\frac{W(\rho,\psi)}{W(\chi_\rho)'}\right)^d\equiv 1\mod{\mu_4}.$$
Therefore we can conclude that 
\begin{equation}
 W(\rho,\psi)\equiv W(\chi_\rho)'\mod{\mu_{d'}},
\end{equation}
where $d'=\rm{lcm}(4, d)$.\\

\end{proof}

In the following theorem we give an invariant formula for $W(\rho,\psi)$, where 
$\rho=\rho(X,\chi_K)$ is a minimal conductor Heisenberg representation of the absolute Galois group $G_F$
of a local field $F/\bbQ_p$ of dimension $m$ which is prime to $p$.

 \begin{thm}\label{invariant formula for minimal conductor representation}
  Let $\rho=\rho(X,\chi_K)$ be a minimal conductor Heisenberg representation of the absolute Galois group $G_F$ of a non-archimedean
  local field $F/\bbQ_p$ of dimension $m$ with $gcd(m,p)=1$. Let $\psi$ be a nontrivial additive character of $F$. Then 
  \begin{equation}
   W(\rho,\psi)=R(\psi,c)\cdot L(\psi,c),
  \end{equation}
where 
$$R(\psi,c):=\lambda_{E/F}(\psi)\Delta_{E/F}(c)=\lambda_{E/F}(c\psi)$$
is a fourth root of unity that depends on $c\in F^\times$ with $\nu_F(c)=1+n(\psi)$ but not on the totally ramified cyclic subextension
$E/F$ in $K/F$, and 
$$L(\psi,c):=\det(\rho)(c)q_F^{-\frac{1}{2}}\sum_{x\in k_F^\times}(\chi_K\circ N_{E_1/F}^{-1})^{-1}(x)\cdot (c^{-1}\psi)(mx),$$
where $E_1/F$ is the unramified extension of $F$ of degree $m$.
 \end{thm}

 Before proving this Theorem \ref{invariant formula for minimal conductor representation} we need to prove the following
 lemma.
 \begin{lem}\label{Lemma 5.2.17} (With the notation of the above theorem)
  \begin{enumerate}
   \item Let $E/F$ be any totally ramified cyclic extension of degree $m$ inside $K/F$. Then:
   $$\Delta_{E/F}(\epsilon)=:\Delta(\epsilon),\qquad\text{for all $\epsilon\in U_F$},$$
does not depend on $E$ if we restrict to units of $F$.
\item We have $L(\psi,\epsilon c)=\Delta(\epsilon)L(\psi,c)$, and therefore changing $c$ by unit we see that 
$$\Delta_{E/F}(\epsilon c)\cdot L(\psi,\epsilon c)=\Delta(\epsilon)^2\Delta_{E/F}(c)\cdot L(\psi,c)=\Delta_{E/F}(c) L(\psi,c).$$
\item We also have the transformation rule $R(\psi,\epsilon c)=\Delta(\epsilon)\cdot R(\psi,c)$.   
   \end{enumerate}
\end{lem}

 \begin{proof}
  {\bf (1).} Denote $G:=\rm{Gal}(E/F)$. From equation (\ref{eqn 3.31}) and by class field theory we know that 
  \begin{equation}
   \Delta_{E/F}=\begin{cases}
                 \omega_{E'/F} & \text{when $\rm{rk}_2(G)=1$}\\
                 1 & \text{when $\rm{rk}_2(G)=0$},
                \end{cases}
  \end{equation}
where $E'/F$ is a uniquely determined quadratic extension inside $E/F$, and $\omega_{E'/F}$ is the quadratic character of $F^\times$
which corresponds to the extension $E'/F$ by class field theory.

When $m$ is odd, i.e., $\rm{rk}_2(G)=0$, hence $\Delta_{E/F}\cong 1$.
So for odd case, the assertion (1) is obvious. 

When $m$ is even, we choose two different totally ramified cyclic subextensions, namely $L_1/F, \;L_2/F$, in $K/F$ of degree $m$.
Then we can write for all $\epsilon\in U_F$,
$$\Delta_{L_1/F}(\epsilon)=\omega_{E'/F}(\epsilon)
=\eta(\epsilon)\cdot\omega_{E'/F}(\epsilon)=\omega_{E'/F}(\epsilon)=\Delta_{L_2/F}(\epsilon),$$
where $\eta$ is the unramified quadratic character of $F^\times$. This proves that $\Delta_{E/F}$ does not depend on $E$ if we restrict 
to $U_F$.\\
{\bf (2).} From Proposition \ref{Proposition arithmetic form of determinant} we know that 
$\det(\rho)(x)=\Delta_{E/F}(x)\cdot \chi_K\circ N_{K/E}^{-1}(x)$ for all $x\in F^\times$.
Let $E_1/F$ be the unramified subextension in $K/F$ of degree $m$. Then we have $EE_1=K$ and 
$$N_{K/E}|_{E_1}=N_{E_1/F}, \qquad (E_1^\times)_{F}\subseteq K_E^\times\subset\rm{Ker}(\chi_K).$$
Moreover, $U_F\subset\cN_{E_1/F}$ and therefore we may write $N_{K/E}^{-1}(\epsilon)=N_{E_1/F}^{-1}(\epsilon)$
for all $\epsilon\in U_F$. Now we can write:
\begin{align*}
 L(\psi,\epsilon c)
 &=\det(\rho)(\epsilon c)q_F^{-\frac{1}{2}}\sum_{x\in k_F^\times}(\chi_K\circ N_{E_1/F}^{-1})^{-1}(x)\cdot (c^{-1}\psi)(mx/\epsilon)\\
 &=\Delta_{E/F}(\epsilon)\chi_K\circ N_{K/E}^{-1}(\epsilon)\det(\rho)(c)q_F^{-\frac{1}{2}}
 \sum_{x\in k_F^\times}(\chi_K\circ N_{E_1/F}^{-1})^{-1}(\epsilon x)\cdot (c^{-1}\psi)(mx)\\
 &=\Delta(\epsilon)\chi_K\circ N_{E_1/F}^{-1}(\epsilon\epsilon^{-1})\det(\rho)(c)q_F^{-\frac{1}{2}}
 \sum_{x\in k_F^\times}(\chi_K\circ N_{E_1/F}^{-1})^{-1}(x)\cdot (c^{-1}\psi)(mx)\\
 &=\Delta(\epsilon)\cdot \det(\rho)(c)q_F^{-\frac{1}{2}}
 \sum_{x\in k_F^\times}(\chi_K\circ N_{E_1/F}^{-1})^{-1}(x)\cdot (c^{-1}\psi)(mx)\\
 &=\Delta(\epsilon)\cdot L(\psi,c).
\end{align*}
This implies that 
$$\Delta_{E/F}(\epsilon c)\cdot L(\psi,\epsilon c)=\Delta(\epsilon)^2\Delta_{E/F}(c)\cdot L(\psi,c)=\Delta_{E/F}(c)L(\psi,c).$$
{\bf (3).} By the definition of $R(\psi,c)$ we can write:
\begin{align*}
 R(\psi,\epsilon c)
 &=\lambda_{E/F}(\psi)\Delta_{E/F}(\epsilon c)=\lambda_{E/F}(\psi)\Delta_{E/F}(\epsilon)\Delta_{E/F}(c)\\
 &=\Delta(\epsilon)\lambda_{E/F}(\psi)\Delta_{E/F}(c)=\Delta(\epsilon)\cdot R(\psi,c).
\end{align*}
 \end{proof}

Now we are in a position to give a proof of Theorem \ref{invariant formula for minimal conductor representation} by using Lemma 
\ref{Lemma 5.2.17}.

\begin{proof}[{\bf Proof of Theorem \ref{invariant formula for minimal conductor representation}}]
By the given conditions: 
 $\rho=\rho(X,\chi_K)$ is a minimal conductor 
 Heisenberg representation of the absolute Galois group $G_F$ of a local field $F/\bbQ_p$ of dimension
 $m$ which is prime to $p$. This means we are in the situation: $\rho=\rho(X,\chi_K)=\rho(X_\eta,\chi_K)$,
 where $\eta$ is a character of $U_F/U_F^1$, and $\rm{dim}(\rho)=\#\eta=m$.
 
 Since $\rho$ is of minimal conductor, we have $a(\rho_0)=m$. Then from Remark \ref{Remark 5.1.22} we have $a(\chi_K)=1$.
 
 Now we choose $E/F\subset K/F$ a totally ramified cyclic subextension of degree $[E:F]=m$, hence $k_E=k_F$
 the same residue fields, and $K/E$ is unramified of degree $m$. Then we can write $\rho=\rm{Ind}_{E/F}(\chi_E)$, and 
 $a(\chi_E)=1$. Again, from Proposition \ref{Proposition arithmetic form of determinant} we have 
 $$\det(\rho)(x)=\Delta_{E/F}(x)\cdot \chi_K\circ N_{K/E}^{-1}(x)\quad\text{for all $x\in F^\times$}.$$
Then for all $x\in F^\times$, we can write 
$$\chi_K\circ N_{K/E}^{-1}(x)=\chi_E(x)=\Delta_{E/F}(x)\cdot \det(\rho)(x).$$
This is true for all subextensions\footnote{In $K/F$ of type $\bbZ_m\times\bbZ_m$ any cyclic subextension $E/F$ in $K/F$
of degree $m$ will correspond to a maximal isotropic subgroup. But we restrict to choosing $E$
totally ramified or unramified.} $E/F$ in $K/F$ which are cyclic of degree $m$.

Now we come to our particular choice: $\rho=\rm{Ind}_{E/F}(\chi_E)$, with $a(\chi_E)=1$ and $E/F$ is totally ramified.
We can write 
\begin{align*}
 W(\rho,\psi)
 &=W(\rm{Ind}_{E/F}(\chi_E),\psi)=\lambda_{E/F}(\psi)\cdot W(\chi_E,\psi\circ\rm{Tr}_{E/F})\\
 &=\lambda_{E/F}(\psi)\cdot q_E^{-\frac{1}{2}}\chi_E(c_E)\sum_{x\in U_E/U_E^1}\chi_E^{-1}(x)(c_E^{-1}\psi\circ\rm{Tr}_{E/F})(x),
\end{align*}
 where $v_E(c_E)=1+n(\psi\circ\rm{Tr}_{E/F})=e_{E/F}(1+n(\psi))$. This implies that we can choose $c_F\in F^\times$ such that 
 $\nu_F(c_F=c_E)=1+n(\psi)$. 
 Let $E_1/F$ be the unramified subextension in $K/F$, then for each $\epsilon\in U_F$, we have 
 $N_{K/E}^{-1}(\epsilon)=N_{E_1/F}^{-1}(\epsilon)$ where $N_{E_1/F}:=N_{K/E}|_{E_1}$. Since $E/F$ is totally ramified, we have 
 $q_E=q_F$. And when $x\in F^\times$, we have $\rm{Tr}_{E/F}(x)=mx$.
 
 Then the above formula rewrites:
 \begin{align*}
  W(\rho,\psi)
  &=\lambda_{E/F}(\psi)\cdot q_F^{-\frac{1}{2}}\chi_K\circ N_{K/E}^{-1}(c_F)\sum_{x\in k_F^\times}
  (\chi_K\circ N_{K/E}^{-1})^{-1}(x)(c_F^{-1}\psi)(mx)\\
  &=\lambda_{E/F}(\psi)\cdot q_F^{-\frac{1}{2}}\Delta_{E/F}(c_F)\det(\rho)(c_F)\sum_{x\in k_F^\times}
  (\chi_K\circ N_{E_1/F}^{-1})^{-1}(x)(c_F^{-1}\psi)(mx)\\
   &=\lambda_{E/F}(\psi)\Delta_{E/F}(c_F)\cdot\left(\det(\rho)(c_F)q_F^{-\frac{1}{2}}\sum_{x\in k_F^\times}
  (\chi_K\circ N_{E_1/F}^{-1})^{-1}(x)(c_F^{-1}\psi)(mx)\right)\\
  &=R(\psi,c)\cdot L(\psi,c),
 \end{align*}
where $c_F=c\in F^\times$ with $\nu_F(c)=1+n(\psi)$,
$R(\psi,c)=\lambda_{E/F}(\psi)\Delta_{E/F}(c)$, and 
$$L(\psi,c)=\det(\rho)(c_F)q_F^{-\frac{1}{2}}\sum_{x\in k_F^\times}
  (\chi_K\circ N_{E_1/F}^{-1})^{-1}(x)(c^{-1}\psi)(mx).$$
Now it is clear that $L(\psi,c)$ depends on $c$ but not on the totally ramified cyclic extension $E/F$ which we have chosen.

Again we know that $\lambda_{E/F}(\psi)$ is a fourth root of unity and $\Delta_{E/F}(c)\in\{\pm 1\}$. Therefore it is easy to see 
that $R(\psi,c)$ is a fourth root of unity. So to call our expression 
$$W(\rho,\psi)=R(\psi,c)\cdot L(\psi,c)$$
is invariant, we are left to show $R(\psi,c)$ does not depend on the the totally ramified cyclic subextension $E/F$ in $K/F$.

From equation (\ref{eqn 3.121}) we can write here 
$$R(\psi,c)=\lambda_{E/F}(\psi)\Delta_{E/F}(c)=\lambda_{E/F}(c\psi)=\lambda_{E/F}(\psi'),$$
where $\psi'=c\psi$, hence $n(\psi')=\nu_F(c)+n(\psi)=1+n(\psi)+n(\psi)=2n(\psi)+1$.

When $m(=[E:F])$ is odd, we have $\lambda_{E/F}(\psi')=1$, hence $R(\psi,c)=\lambda_{E/F}(c\psi)=1$. Thus in the odd case 
$R(\psi,c)$ is independent of the choice of the totally ramified subextension $E/F$ in $K/F$.

When $m$ is even, we have 
\begin{align*}
 R(\psi,c)
 &=\lambda_{E/F}(\psi')=\lambda_{E/E'}(\psi'')\cdot \lambda_{E'/F}^{[E:E'']}\\
 &=\lambda_{E'/F}(\psi')^{\pm 1},
\end{align*}
where $[E',F]$ is the $2$-primary part of $m$, hence $[E:E']$ is odd. Here the sign only depends on 
$m$ but not on $E$. So we can restrict to the case where $m=[E:F]$ is a power of $2$. Let $E_2/F$ be the unique quadratic subextension
in $E/F$. Since $E/F$ is a cyclic tame extension, from Corollary \ref{Lemma 3.10}(1), we obtain:
\begin{equation}
 \lambda_{E/F}(\psi')=\begin{cases}
                       \lambda_{E_2/F}(\psi') & \text{if $[E:F]\ne 4$}\\
                       \beta(-1)\cdot\lambda_{E_2/F}(\psi') & \text{if $[E:F]=4$},
                      \end{cases}
\end{equation}
where $\beta$ is the character of $F^\times/\cN_{E/F}$ of order $4$.

Since here $n(\psi')=2 n(\psi)+1$ is {\bf odd}\footnote{If $n(\psi')$ is even, 
then from the table of the Remark \ref{Remark 3.4.11}, $\lambda_{E_2/F}(\psi')=-\lambda_{E_2'/F}(\psi')$,
where $E_2'/F$ be the totally ramified quadratic extension different from $E_2/F$.
Therefore $\lambda_{E/F}(\psi')$ depends on $\psi'$.},
from Remark \ref{Remark 3.4.11} (see the table of Remark \ref{Remark 3.4.11})
we can tell that $\lambda_{E_2/F}(\psi')$ is invariant.\\
Finally, we have to see that $\beta(-1)$ does not depend on $E$ if $[E:F]=4$. 

Since $E/F$ is totally ramified of degree $4$, 
we have $F^\times=U_F\cdot N$, hence $F^\times/N=U_F N/N=U_F/U_F\cap N\cong\bbZ_4$, where $N=N_{E/F}(E^\times)$.
Again $U_F^1\subset U_F$, and $U_F^1\subset N$, hence $U_{F}^{1}\subset N\cap U_F\subset U_F$. We know that 
$U_F/U_F^1$ is a cyclic group. Therefore $N\cap U_F$ is determined by its index in $U_F$, which does not 
depend on $E$. Hence,
$U_F\cap N$ does not depend on $E$.

We also know that there are two characters of $U_F/U_F\cap N$ of order $4$, and they are inverse to each other. Then 
$$\beta(-1)=\beta(-1)^{-1}=\beta^{-1}(-1)$$ 
is the same in both cases. 
Since 
$\beta$ is the character which corresponds to $E/F$ by class field theory, we can say $\beta$ is a character of 
$F^\times/U_F^1$, hence $a(\beta)=1$. It clearly shows that $\beta(-1)$ does not depend on $E$. So we can conclude that 
$R(\psi,c)$ does not depend on $E$.



Thus our expression $W(\rho,\psi)=R(\psi,c)\cdot L(\psi,c)$ does not depend on the choice of the totally ramified
cyclic subextension $E/F$ in $K/F$. Moreover, we notice that we have the transformation rules 
$$R(\psi,\epsilon c)=\Delta(\epsilon)R(\psi,c),\qquad L(\psi,\epsilon c)=\Delta(\epsilon)L(\psi,c),$$
for all $\epsilon\in U_F$. Again $\Delta(\epsilon)^2=1$, hence the product $R(\psi,\epsilon c)\cdot L(\psi,\epsilon c)=
R(\psi,c)\cdot L(\psi,c)=W(\rho,\psi)$ does not depend on the choice of $c$.

Therefore, finally, we can conclude our formula $W(\rho,\psi)=R(\psi,c)L(\psi,c)$ is an invariant expression.

\end{proof}

Now let $\rho=\rho(X,\chi_K)$ be a Heisenberg representation of dimension prime to $p$ but the conductor of $\rho$
is {\bf not} minimal. In the following theorem we give an invariant formula of $W(\rho,\psi)$. But before this we need this 
following proposition.

\begin{prop}\label{Proposition-A}
 Let $\rho=\rho(X_\eta,\chi_K)$ be a Heisenberg representation of the 
 absolute Galois group $G_F$ of a non-archimedean local field $F/\bbQ_p$ of
 dimension $m$ prime to $p$. Then it is of minimal conductor $a_F(\rho)=m$ if and only if $\rho$ is a representation of 
 $G_F/V_F$, where $V_F$ is the subgroup of wild ramification.
\end{prop}

\begin{proof}
By the given condition, the dimension $\rm{dim}(\rho)=m$ is prime to $p$. Then from Lemma \ref{Lemma dimension equivalent}
we can conclude that $K/F$ is tamely ramified with $f_{K/F}=e_{K/F}=m$ 
(cf. Remark \ref{Remark 5.1.14})and hence $d_{K/F}=e_{K/F}-1=m-1$. Then from the conductor 
formula (\ref{eqn 5.1.24}) we can easily see that {\bf $a(\rho)=m$ is minimal if and only if $a(\chi_K)=1$}.

Further, for some extension $L/K$, if $\cN_{L/K}=\rm{Ker}(\chi_K)$, then by class field theory we can conclude:
{\bf $L/K$ is tamely ramified if and only if $a(\chi_K)=1$.}

Now suppose that $\rho$ is a Heisenberg representation of $G:=G_F/V_F$ of dimension $m$ prime to $p$. 
This implies $V_F\subset\rm{Ker}(\rho)=\rm{Ker}(\chi_K)=\cN_{L/K},$
where $L/K$ is some tamely ramified extension. Then $a(\chi_K)=1$, hence $a(\rho)=m$ is minimal.

Conversely, when conductor $a(\rho)=m$ is minimal, we have $a(\chi_K)=1$. 
 By class field theory this character $\chi_K$ determines
 an extension $L/K$ such that $\cN_{L/K}=\rm{Ker}(\chi_K)$. Since $a(\chi_K)=1$, here $L/K$ must be tamely ramified, hence 
 $L/F$ is tamely ramified.
 This means $V_F$ sits in the kernel $\rm{Ker}(\rho)=G_L$ , therefore $\rho$ is actually a representation of 
 $G_F/V_F$.

\end{proof}

 \begin{thm}\label{Theorem using Deligne-Henniart}
  Let $\rho=\rho(X_\eta,\chi_K)=\rho_0\otimes\widetilde{\chi_F}$ be a Heisenberg representation of the absolute Galois group
  $G_F$ of a non-archimedean local field $F/\bbQ_p$ of dimension $m$ with $gcd(m,p)=1$, where 
  $\rho_0=\rho_0(X_\eta,\chi_0)$ is a minimal conductor Heisenberg representation of $G_F$
  and $\widetilde{\chi_F}:W_F\to\bbC^\times$
  corresponds to $\chi_F:F^\times\to \bbC^\times$ by class field theory.
  If $a(\chi_F)\ge 2$, then we have 
 \begin{equation}\label{eqn 5.4.9}
  W(\rho,\psi)=W(\rho_0\otimes\widetilde{\chi_F},\psi)=W(\chi_F,\psi)^m\cdot\det(\rho_0)(c),
 \end{equation}
where $\psi$ is a nontrivial additive character of $F$, and $c:=c(\chi_F,\psi)\in F^\times$, satisfies 
\begin{center}
 $\chi_F(1+x)=\psi(c^{-1}x)$ for all $x\in P_{F}^{a(\chi_F)-[\frac{a(\chi_F)}{2}]}$.
\end{center}
\end{thm}

\begin{proof}
 {\bf Step-1:}
 By the given conditions, 
 $$\rho=\rho_0\otimes\widetilde{\chi\chi_F},$$
 where $\rho_0$ is a minimal conductor Heisenberg representation of $G_F$ of dimension $m$ prime to $p$ and 
 $\widetilde{\chi_F}:W_F\to\bbC^\times$ corresponds to $\chi_F:F^\times\to\bbC^\times$ by class field theory. And
 \footnote{We also know that there are $m^2$ characters of $F^\times/{F^\times}^m$ such that 
$\rho_0\otimes\widetilde{\chi}=\rho_0$ (cf. \cite{Z2}, p. 303, Proposition 1.4). So we always have:
\begin{equation*}
 \rho=\rho_0\otimes\widetilde{\chi_F}=\rho_0\otimes\widetilde{\chi\chi_F},
\end{equation*}
where $\chi\in\widehat{F^\times/{F^\times}^m}$, and
$\widetilde{\chi_F}:W_F\to\bbC^\times$ corresponds to $\chi_F$ by class field theory.
}
 here $\chi:F^\times/{F^\times}^m\to\bbC^\times$ such that $\rho_0=\rho_0\otimes\widetilde{\chi}$.

 Let $\zeta$ be a $(q_F-1)$-st
root of unity. Since $U_F^1$ is a pro-p-group and $gcd(p,m)=1$, we have 
\begin{equation}\label{eqn 5.2.23}
  F^\times/{F^\times}^m=<\pi_F>\times<\zeta>\times U_F^1/<\pi_F^m>\times<\zeta>^m\times U_F^1\cong \bbZ_m\times\bbZ_m,
\end{equation}
that is, a direct product of two cyclic group of same order. Hence $F^\times/{F^\times}^m\cong\widehat{F^\times/{F^\times}^m}$.
Since ${F^\times}^m=<\pi_F^m>\times<\zeta>^m\times U_F^1$, and
$F^\times/{F^\times}^m\cong \bbZ_m\times\bbZ_m,$
we have $a(\chi)\le 1$ and $\#\chi$ is a divisor of $m$ for all 
$\chi\in\widehat{F^\times/{F^\times}^m}$. Now if we take a character $\chi_F$ of $F^\times$ conductor $\ge 2$, hence  
$a(\chi_F)\ge 2 a(\chi)$ for all $\chi\in \widehat{F^\times/{F^\times}^m}$. Then by using Deligne's 
formula (\ref{eqn 2.3.17}) we have 
\begin{equation}\label{eqn 5.2.216}
 W(\chi\chi_F,\psi)^m=\chi(c)^m\cdot W(\chi_F,\psi)^m=W(\chi_F,\psi)^m,
\end{equation}
where $c\in F^\times$ with $\nu_F(c)=a(\chi_F)+n(\psi)$, satisfies 
$$\chi_F(1+x)=\psi(c^{-1}x),\quad\text{for all $x\in F^\times$ with $2\nu_F(x)\ge a(\chi)$}.$$

 Now from Proposition \ref{Proposition-A} we can consider the representation $\rho_0$ as a representation
 of $G:=\rm{Gal}(F_{mt}/F)$, where $F_{mt}/F$ is the maximal tamely ramified subextension in $\bar{F}/F$.
 Then we can write 
 $$\rho_0=\rm{Ind}_{E/F}(\chi_{E,0}),\quad\rho=\rm{Ind}_{E/F}(\chi_E),$$
 where $E/F$ is a cyclic tamely ramified subextension of $K/F$ of degree $m$ and 
 $$\chi_E:=\chi_{E,0}\otimes(\chi_F\circ N_{E/F}).$$
 
 {\bf Step-2:}
 Let $G$ be a finite group and $R(G)$ be the character ring provided with the tensor product as multiplication and the 
 unit representation as unit element. Then for any zero dimensional representation
 $\Pi\in R(G)$ we have (cf. Theorem 2.1(h) on p. 40
 of \cite{RB}):
 \begin{equation}\label{eqn 5.2.217}
  \Pi=\sum_{H\leq G}n_H\rm{Ind}_{H}^{G}(\chi_H-1_H),
 \end{equation}
 where $n_H\in \bbZ$ (cf. Proposition 2.24 on p. 48 of \cite{RB}) and $\chi_H\in\widehat{H}$. Moreover, from 
 Theorem 2.1 (k) of \cite{RB} we know that $n_H\ne 0$ {\bf only if} $Z(G)\le H$ and $\chi_H|_{Z(G)}=\chi_Z$, where $Z(G)$ is 
 the center of $G$ and $\chi_Z$ is the center character.
 
 Now we will use this above formula (\ref{eqn 5.2.217}) for the representation $\rho_0-m\cdot 1_F$ and we get
 \begin{equation}\label{eqn 5.2.211}
 \rho_0 - m\cdot 1_F=\sum_{i=1}^{r}n_i\rm{Ind}_{E_i/F}(\chi_{E_i} - 1_{E_i}),
\end{equation}
where $E_i/F$ are intermediate fields of $K/F$ and for nonzero $n_i$, we have the relation 
$$\chi_0=\chi_{E_i}\circ N_{K/E_i}.$$
Since $a(\chi_0)=1$ and $K/E_i$ are cyclic tamely ramified\footnote{The subfields $E_i\subseteq K$ are related to Boltje's 
approach by $\rm{Gal}(K/E_i)=H_i/Z(G)$ and the fact that $\chi_{H_i}$ extends the character $\chi_Z$ which translates 
via class field theory to $\chi_{E_i}\circ N_{K/E_i}=\chi_K$. Moreover, $X=\chi_Z\circ [-,-]$ and $\chi_Z$ extendible to 
$H_i$ means that $\rm{Gal}(K/E_i)=H_i/Z(G)$ must be isotropic for $X$, hence in our situation 
$K/E_i$ must be a cyclic extension of degree dividing $m$.}, we have $a(\chi_{E_i})=1$ for all $i=1,2,\cdots,r$.
Then from equation (\ref{eqn 5.2.211}) we have 
\begin{equation}\label{eqn 5.2.212}
 \det(\rho_0)=\prod_{i=1}^{r}(\chi_{E_i}|_{F^\times})^{n_i}.
\end{equation}

{\bf Step-3:} From equation (\ref{eqn 5.2.211}) we also can write
\begin{equation}\label{eqn 5.2.213}
 \rho_0\otimes\widetilde{\chi\chi_F} - m\cdot\chi\chi_F=\sum_{i=1}^{r}n_i\rm{Ind}_{E_i/F}(\chi_{E_i}\theta_i-\theta_i),
\end{equation}
where $\theta_i:=\chi\chi_F\circ N_{E_i/F}$ for all $i\in\{1,2,\cdots,r\}$.
Since $a(\chi_F)\ge 2$, and $E_i/F$ are tamely ramified, 
the conductors $a(\theta_i)\ge 2$ for all $i\in\{1,2,\cdots,r\}$. Then from equation (\ref{eqn 5.2.213}) we can write
\begin{align}
 W(\rho,\psi)\nonumber
 &=W(\chi\chi_F,\psi)^m\cdot\prod_{i=1}^{r}W(\chi_{E_i}\theta_i-\theta_i,\psi_{E_i})^{n_i}\\\nonumber
 &=W(\chi\chi_F,\psi)^m\cdot\prod_{i=1}^{r}\frac{W(\chi_{E_i}\theta_i,\psi_{E_i})^{n_i}}{W(\theta_i,\psi_{E_i})^{n_i}}\\
 &=W(\chi\chi_F,\psi)^m\cdot\prod_{i=1}^{r}\chi_{E_i}(c_i)^{n_i},\label{eqn 5.2.214}
\end{align}
where $\psi_{E_i}=\psi\circ\rm{Tr}_{E_i/F}$ and $c_i\in E_{i}^{\times}$ such that 
\begin{center}
 $\theta_i(1+y)=\psi_{E_i}(\frac{y}{c_i}),$ for all $y\in P_{E_i}^{a(\theta_i)-[\frac{a(\theta_i)}{2}]}$.
\end{center}
Moreover, here $E_i/F$ are tamely ramified extensions, then from Lemma 18.1 of \cite{BH} on p. 123, we have
$$N_{E_i/F}(1+y)\cong 1+\rm{Tr}_{E_i/F}(y)\pmod{P_F^{a(\chi_F)}},$$
and 
$\rm{Tr}_{E_i/F}(y)\in P_F^{a(\chi_F)-[\frac{a(\chi_F)}{2}]}$. Therefore for all 
$y\in P_{E_i}^{a(\theta_i)-[\frac{a(\theta_i)}{2}]}$ we can write (cf. Proposition 18.1 on p. 124 of \cite{BH}):
\begin{align*}
 \theta_i(1+y)
 &=\chi\chi_F\circ N_{E_i/F}(1+y)=\chi_F(1+\rm{Tr}_{E_i/F}(y))\\
 &=\psi(\frac{\rm{Tr}_{E_i/F}(y)}{c})=\psi_{E_i}(\frac{y}{c}),
\end{align*}
where $c:=c(\chi_F,\psi)$ for which $\chi_F(1+x)=\psi(\frac{x}{c})$ for $x\in P_F^{a(\chi_F)-[\frac{a(\chi_F)}{2}]}$.
This varifies that the choice $c_i(\theta_i,\psi_{E_i})=c_i(\chi\chi_F,\psi_{E_i})=c(\chi_F,\psi)\in F^{\times}$ is right 
for applying Tate-Lamprecht formula.

Then by using equation (\ref{eqn 5.2.212}) in equation (\ref{eqn 5.2.214}) we have:
\begin{equation}\label{eqn 5.2.215}
 W(\rho,\psi)=W(\chi\chi_F,\psi)^{m}\cdot \det(\rho_0)(c).
\end{equation}
Finally, by using equation (\ref{eqn 5.2.216}) from equation (\ref{eqn 5.2.215}) we can write 
\begin{align*}
 W(\rho,\psi)
 &=W(\chi\chi_F,\psi)^{m}\cdot\det(\rho_0)(c)\\
 &=W(\chi_F,\psi)^m\cdot\det(\rho_0)(c).
\end{align*}

\end{proof}

\section{\textbf{Applications of Tate's root-of-unity criterion}}

Let $K/F$ be a finite
 Galois extension of the non-archimedean local field $F$, and $\rho:\mathrm{Gal}(K/F)\to \mathrm{Aut}_{\mathbb{C}}(V)$ a 
 representation of $\mathrm{Gal}(K/F)$ on a complex vector space $V$. 
 Let $P(K/F)$ denote the first {\bf wild} ramification group of $K/F$.
 Let $V^{P}$ be the subspace of all elements of $V$ fixed by $\rho(P(K/F))$. Then $\rho$ induces a representation:
 \begin{center}
  $\rho^{P}:\mathrm{Gal}(K/F)/P(K/F)\to\mathrm{Aut}_{\mathbb{C}}(V^{P})$.
 \end{center}
 Let $\overline{F}$ be an algebraic closure of the local field $F$, and $G_F=\rm{Gal}(\overline{F}/F)$ be the absolute Galois
 group for $\overline{F}/F$.
  Let $\rho$ be a representation of $G_F$.\\
{\bf Then by Tate, $W(\rho)/W(\rho^{P})$ is a root of a unity (cf. \cite{JT1}, p. 112, Corollary 4).}\\
Now let $\rho$ be an irreducible representation $G_F$, then either $\rho^P=\rho$, in which case
$\frac{W(\rho)}{W(\rho^P)}=1$, or else $\rho^P=0$, in this case from Tate's result we can say $W(\rho)$ is a root of unity.
Equivalently:\\
If $W(\rho)$ is not a root of unity then $\rho^P\ne 0$, hence $\rho^P=\rho$ because $\rho$ is irreducible. This means that 
all vectors $v\in V$ of the representation space are fixed under $P$ action on $V$. \\
In other words, if we consider $\rho$ as a 
homomorphism $\rho:G_F\to\rm{Aut}_\bbC(V)$ then the elements from $P$ are mapped to the identity, hence 
\begin{center}
 $\rho^P=\rho$ means $P\subset\rm{Ker}(\rho)$.
\end{center}
Therefore we can state the following lemma.
\begin{lem}
 If $\rho$ is an irreducible representation of $G_F$, such that the subgroup  $P\subset G_F$, of wild ramification
does {\bf not trivially} act on the representation space  $V$ (this gives $\rho^P\ne \rho$, i.e., $\rho^P=0$), 
then   $W(\rho)$  is a root of unity.
\end{lem}

Before going to our next results we need to recall some facts from class field theory.
Let $F$ be a non-archimedean local field. Let $F^{ab}$ be the maximal abelian extension of $F$ and $F_{nr}$ be the maximal 
unramified extension of $F$. Then by local class field theory there is a unique homomorphism
$$\theta_{F}:F^\times\to \rm{Gal}(F^{ab}/F)$$
having certain properties (cf. \cite{JM}, p. 20, Theorem 1.1). 
This local reciprocity map $\theta_F$ is continuous and injective with dense image. 
From class field theory we have the following commutative diagram  
$$\begin{array}{ccccccccc} &&& && v_F &&  \\
0 & \to & U_F & \to & F^\times & \to & \bbZ & \to & 0 \\
&& \quad\downarrow \theta_F && \quad\downarrow\theta_F && \quad\downarrow \rm{id} \\
0 & \to & I_F & \to & \rm{Gal}(F^{ab}/F) & \to & \widehat{\bbZ} & \to & 0, 
\end{array}$$
where $I_F:=\rm{Gal}(F^{ab}/F_{nr})$ is the inertia subgroup of $\rm{Gal}(F^{ab}/F)$, and $\rm{Gal}(F_{nr}/F)$ is identified  with 
$\widehat{\bbZ}$ (cf. \cite{CF}, p. 144). We also know that $\theta_F:U_F\to I_F$ is an isomorphism. Moreover, the descending chain 
$$U_F\supset U_{F}^{1}\supset U_{F}^{2}\cdots$$
is mapped isomorphically by $\theta_F$ to the descending chain of ramification subgroups of $\rm{Gal}(F^{ab}/F)$ in the upper numbering.

Now let $I$ be the inertia subgroup of $G_F$. Let $P$ be the wild 
ramification subgroup of $G_F$.
Then we have $G_F\supset I\supset P$. Parallel with this we have 
$F^\times\supset U_F\supset U_{F}^{1}$. Then we have 
\begin{equation}\label{sequence 5.3.1}
 1\to I/P\cdot[G_F,G_F]\to G_F/P\cdot[G_F,G_F]\to G_F/I\to 1,
\end{equation}
and parallel 
\begin{equation}\label{sequence 5.3.2}
 1\to U_F/U_{F}^{1}\to F^\times/U_{F}^{1}\to F^\times/U_F\to 1.
\end{equation}
Now by class field theory the left terms of sequences (\ref{sequence 5.3.1}) and (\ref{sequence 5.3.2}) 
are isomorphic, but for the right terms we have $G_F/I$ is isomorphic to the total
completion of $\bbZ$
(because here $G_F/I$ is profinite group, hence compact). We also have  $F^\times/U_F=<\pi_F>\times U_F/U_F\cong\bbZ$.
Therefore sequence (\ref{sequence 5.3.2}) is dense in (\ref{sequence 5.3.1}) because $\bbZ$
is dense in the total completion $\widehat{\bbZ}$. But $\bbZ$ and $\widehat{\bbZ}$ have the same finite factor groups. 
{\bf As a consequence $F^\times/U_{F}^{1}$ is also dense in $G_F/P\cdot[G_F,G_F]$.}




Let $\rho$ be a Heisenberg representation of the absolute Galois group $G_F$.
In the following proposition we show that if $W(\rho)$ is not a root of unity, then $\rm{dim}(\rho)|(q_F-1)$.
\begin{prop}\label{Proposition 4.12}
 Let $F/\bbQ_p$ be a local field and let $q_F=p^s$ be the order of its finite residue field. If
$\rho=(Z_\rho,\chi_\rho)=\rho(X_\rho,\chi_K)$ is a Heisenberg representation of the absolute 
Galois group $G_F$ such that $W(\rho)$ is not a root of unity, 
then $dim(\rho)|(q_F-1)$.
\end{prop}

\begin{proof}
Let $P$ denote the wild ramification subgroup of $G_F$.
By Tate's root-of-unity criterion, we know that $\gamma:=\frac{W(\rho)}{W(\rho^P)}$ is a root of unity. If $W(\rho)$ is not a 
root of unity, then $\rho=\rho^P$, otherwise $W(\rho)$ must be a root of unity. Again $\rho^P=\rho$ implies 
$P\subset\rm{Ker}(\rho)\subset Z_\rho\subset G_F$. So $G_F/Z_\rho$ is a quotient of $G_F/P$, hence $F^\times/U_F^1$.
 
Moreover,  from the dimension formula (\ref{eqn dimension formula}), we have 
$$\rm{dim}(\rho)=\sqrt{[G_F:Z_\rho]}=\sqrt{[K:F]}=\sqrt{[F^\times:\cN_{K/F}]},$$
where $Z_\rho=G_K$ and $\rm{Rad}(X)=\cN_{K/F}$, hence $F^\times/N$ is a quotient group of $F^\times/U_F^1$.
 Therefore the alternating character $X_\rho$ induces an alternating 
character $X$ on $F^\times/U_F^1.$ 
Then from the proof of Lemma \ref{Lemma dimension equivalent} ((4) to (1)), we can conclude that 
$\rm{dim}(\rho)$ divides $q_F-1$.

\end{proof}

\chapter{{\bf Appendix}}

\section{{\bf Lamprecht-Tate formula for $W(\chi)$}}

\begin{thm}[{\bf Lamprecht-Tate formula}]\label{Theorem 6.1.1}
Let $F$ be a non-archimedean local field.
Let $\chi$ be a character of $F^\times$ of exponential Artin-conductor $a(\chi)=a_F(\chi)$ and let $m$ be a natural number 
such that $2m\le a(\chi)$. Let $\psi_F$ be the canonical additive character of $F$. 
Then there exists $c\in F^\times$, $\nu_F(c)=a(\chi)+d_{F/\bbQ_p}$ such that 
\begin{equation}\label{eqn 5.4.5}
 \chi(1+y)=\psi_F(c^{-1}y)\qquad\text{for all $y\in P_{F}^{a(\chi)-m}$},
\end{equation}
and this will imply:
\begin{equation}\label{eqn 6.0.9}
 W(\chi)=W(\chi,c)=\chi(c)\cdot q_{F}^{-\frac{(a(\chi)-2m)}{2}}\sum_{x\in (1+P_F^m)/(1+P_F^{a(\chi)-m})}\chi^{-1}(x)\psi_F(c^{-1}x).
\end{equation}
{\bf Remark:} Note that the assumption (\ref{eqn 5.4.5}) is obviously fulfilled for $m=0$ because then both sides are $=1$,
and the resulting formula for $m=0$ is the original formula (\ref{eqn 2.2}) for abelian local constant $W(\chi)$.
\end{thm}
\begin{proof}
 We have seen already that for $m=0$ everything is correct. In general, the assumption $2m\le a(\chi)$ implies 
 $2(a(\chi)-m)\ge a(\chi)$ and therefore 
 $$\chi(1+y)\chi(1+y')=\chi(1+y+y')$$
 for $y,y'\in P_{F}^{a(\chi)-m}$. That is, $y\mapsto \chi(1+y)$ is a character of the additive group $P_F^{a(\chi)-m}$.
 This character extends to a character of $F^{+}$ and, by local additive duality, there is some $c\in F^\times$ such that
 $$\chi(1+y)=\psi_F(c^{-1}y)=(c^{-1}\psi_F)(y),\quad\text{for all $y\in P_F^{a(\chi)-m}$}.$$
 Now comparing the conductors of both sides we must have:
 $$a(\chi)=-n(c^{-1}\psi_F)=\nu_F(c)-n(\psi_F),$$
 hence $\nu_F(c)=a(\chi)+n(\psi_F)=a(\chi)+d_{F/\bbQ_p}$ is the right assumption for our formula.\\
 Now we assume $m\ge 1$ (the case $m=0$ we have checked already) and consider the filtration 
 $$O_F^\times\supseteq 1+P_F^{a(\chi)-m}\supseteq 1+P_{F}^{a(\chi)}.$$
 Then we may represent $x\in O_F^\times/(1+P_F^{a(\chi)})$ as $x=z(1+y)$, where $y\in P_F^{a(\chi)-m}$ and $z$ runs over the 
 system of representatives for $O_F^\times/(1+P_F^{a(\chi)-m})$. Now computing $W(\chi)$ we have to consider the sum 
 \begin{equation}\label{eqn 5.4.6}
  \sum_{x\in O_F^\times/(1+P_F^{a(\chi)})}\chi^{-1}(x)\psi_F(c^{-1}x)=\sum_{z\in O_F^\times/(1+P_F^{a(\chi)-m})}
  \sum_{y\in P_F^{a(\chi)-m}/P_{F}^{a(\chi)}}\chi^{-1}(z(1+y))\psi_F(c^{-1}z(1+y)).
 \end{equation}
Now using (\ref{eqn 5.4.5}) we obtain
$$\chi^{-1}(z(1+y))=\chi^{-1}(z)\chi^{-1}(1+y)=\chi^{-1}(z)\chi(1-y)=\chi^{-1}(z)\psi_F(-c^{-1}y)$$
 and therefore our double sum (\ref{eqn 5.4.6}) rewrites as 
 \begin{equation*}
  \sum_{z\in O_F^\times/(1+P_F^{a(\chi)-m})}\chi^{-1}(z)\psi_F(c^{-1}z)\cdot
  \left(\sum_{y\in P_F^{a(\chi)-m}/P_F^{a(\chi)}}\psi_F(c^{-1}y(z-1))\right).
 \end{equation*}
But the inner sum is the sum on the additive group $P_F^{a(\chi)-m}/P_F^{a(\chi)}$ and 
$(c^{-1}(z-1))\psi_F$ is a character of that group. Hence this sum is equal to $[P_F^{a(\chi)-m}:P_F^{a(\chi)}]=q_{F}^{m}$
if the character is $\equiv1$ and otherwise the sum will be zero. But:
$$n(c^{-1}(z-1)\psi_F)=\nu_F(c^{-1}(z-1))+n(\psi_F)=-a(\chi)+\nu_F(z-1).$$
So the character $(c^{-1}(z-1))\psi_F$ is trivial on $P_F^{a(\chi)-\nu_F(z-1)}$, and therefore it will be $\equiv1$
on $P_{F}^{a(\chi)-m}$ if and only if $\nu_F(z-1)\ge m$, i.e., $z=1+y'\in 1+P_F^m$. Therefore our sum (\ref{eqn 5.4.6})
rewrites as 
  \begin{equation}\label{eqn 5.4.7}
  \sum_{x\in O_F^\times/(1+P_F^{a(\chi)})}\chi^{-1}(x)\psi_F(c^{-1}x)=
  q_F^m \sum_{z\in (1+P_F^m)/(1+P_F^{a(\chi)-m})}\chi^{-1}(z)\psi_F(c^{-1}z).
 \end{equation}
 And substituting this result into our original formula (\ref{eqn 2.2}) we get 
 \begin{align*}
  W(\chi)
  &=\chi(c)q_F^{-\frac{a(\chi)}{2}}\sum_{x\in O_F^\times/(1+P_F^{a(\chi)})}\chi^{-1}(x)\psi_F(c^{-1}x)\\
  &=\chi(c)\cdot q_{F}^{-\frac{(a(\chi)-2m)}{2}}\sum_{x\in (1+P_F^m)/(1+P_F^{a(\chi)-m})}\chi^{-1}(x)\psi_F(c^{-1}x).
 \end{align*}

\end{proof}

\begin{cor}\label{Corollary 6.1.2}
 Let $\chi$ be a character of $F^\times$. Let $\psi$ be a nontrivial additive character of $F$.
\begin{enumerate}
 \item When $a(\chi)=2d\, (d\ge 1)$, we have 
$$W(\chi)=\chi(c)\psi(c^{-1}).$$
\item When $a(\chi_\rho)=2d+1\, (d\ge1)$, we have 
$$W(\chi)=\chi(c)\psi(c^{-1})\cdot q_F^{-\frac{1}{2}}\sum_{x\in P_F^d/P_F^{d+1}}\chi^{-1}(1+x)\psi(c^{-1}x).$$
\item {\bf Deligne's twisting formula (\ref{eqn 2.3.17}):}
If $\alpha, \beta\in \widehat{F^\times}$ with $a(\alpha)\ge 2\cdot a(\beta)$, then 
$$W(\alpha\beta,\psi)=\beta(c)\cdot W(\alpha,\psi).$$
\end{enumerate}
Here $c\in F^\times$ with $F$-valuation $\nu_F(c)=a(\chi)+n(\psi_F)$, and in Case (1) and Case (2), $c$ also satisfies 
\begin{center}
 $\chi(1+x)=\psi(\frac{x}{c})$ for all $x\in F^\times$ with $2\cdot\nu_F(x)\ge a(\chi)$.
\end{center}
And in case (3), we have $\alpha(1+x)=\psi(x/c)$ for all $\nu_F(x)\ge \frac{a(\alpha)}{2}$
\end{cor}

\begin{proof}

From the above formula (\ref{eqn 6.0.9}), the assertions are followed.\\
{\bf (1).} When $a(\chi)=2d$, where $d\ge 1$. In this case, we take $m=d$, and from equation (\ref{eqn 6.0.9}) we obtain
 \begin{equation}
  W(\chi,\psi)=\chi(c)\cdot \sum_{x\in (1+P_F^d)/(1+P_F^d)}\chi^{-1}(x)\psi(c^{-1}x)=\chi(c)\cdot\psi(c^{-1}).
 \end{equation}
{\bf (2).} When $a(\chi)=2d+1$, where $d\ge 1$. In this case, we also take $m=d$, and then from equation (\ref{eqn 6.0.9}) we obtain
\begin{align*}
 W(\chi,\psi)
 &=\chi(c)\cdot q_F^{-\frac{1}{2}}\cdot\sum_{x\in(1+P_F^d)/(1+P_F^{d+1})}\chi^{-1}(x)\psi(c^{-1}x)\\
 &=\chi(c)\cdot\psi(c^{-1})\cdot q_F^{-\frac{1}{2}}\cdot\sum_{x\in P_F^d/P_F^{d+1}}\chi^{-1}(1+x)\psi(c^{-1}x).
\end{align*}
Here $c\in F^\times$ with $\nu_F(c)=a(\chi)+n(\psi)$, satisfies 
\begin{center}
 $\chi(1+x)=\psi(c^{-1}x)$ for all $x\in F^\times$ with $2\nu_F(x)\ge a(\chi)$.
\end{center}
 {\bf (3) Proof of Deligne's twisting formula (\ref{eqn 2.3.17}):}
By the given condition,  
$a(\alpha)\ge 2 a(\beta)$, hence we have $a(\alpha\beta)=a(\alpha)$. 
Now take $m=a(\beta)$, then from equation (\ref{eqn 6.0.9}) we can write:
\begin{align*}
 W(\alpha\beta,\psi)
 &=\alpha\beta(c)\cdot q_{F}^{-\frac{(a(\alpha)-2m)}{2}}\cdot\sum_{x\in (1+P_F^m)/(1+P_F^{a(\chi)-m})}(\alpha\beta)^{-1}(x)\psi(\frac{x}{c})\\
 &=\beta(c)\cdot \alpha(c)q_F^{-\frac{(a(\alpha)-2m)}{2}}\sum_{x\in(1+P_F^m)/(1+P_F^{a(\alpha)-m})}\alpha^{-1}(x)\psi(\frac{x}{c})\\
 &=\beta(c)\cdot W(\alpha,\psi),
\end{align*}
since $a(\beta)=m$, hence $\beta(x)=1$ for all $x\in (1+P_F^m)/(1+P_F^{a(\alpha)-m})$.

\end{proof}

\begin{rem}\label{Remark Appendix}

Let $\mu_{p^\infty}$ denote as the group of roots of unity of $p$-power order. 
Let $F$ be a non-archimedean local field. Let $\psi_F$ be an additive character of $F$. Since $\psi_F$ is additive, its image lies 
in $\mu_{p^{\infty}}$. We also know that $U_F^1$ is a pro-p-group, hence $\chi(U_F^1)\subset \mu_{p^{\infty}}$.

\begin{enumerate}
 \item When $a(\chi)$ is even, we have
 \begin{equation}\label{eqn 5.2.8}
  W(\chi,\psi_F)=\chi(c)\cdot\psi_F(c^{-1}).
 \end{equation}
\item When $a(\chi)=2d+1$ is odd, we have
\begin{equation}\label{eqn 5.2.9}
 W(\chi,\psi_F)=\chi(c)\cdot\psi_F(c^{-1})\cdot q_{F}^{-\frac{1}{2}}\sum_{x\in P_{F}^{d}/P_{F}^{d+1}}\chi^{-1}(1+x)\psi_F(x/c).
\end{equation}
\end{enumerate}
Now we give explicit formula for $W(\chi,\psi_F)$ modulo $\mu_{p^\infty}$.
For this it is sufficient
to know $c\in F^\times\rm{mod}\, 1+P_{F}$.

As for the correcting term it is $\equiv 1\mod{\mu_{p^\infty}}$ in case $p=2$. If $p\ne 2$ we compare the function 
$Q(x):=\chi_{\rho}^{-1}(1+x)\cdot\psi_F(x/c)$ and $H(x):=\psi_F(\frac{x^2}{2c})$ on $P_F^d$. 
It is easy to see (see the review of Henniart's artile \cite{GH} by E.-W. Zink) 
that $(1+x)(1+y)=(1+x+y)(1+\frac{xy}{1+x+y})$. Then we have 
\begin{equation}
 \frac{Q(x+y)}{Q(x)\cdot Q(y)}=\chi_\rho(1+\frac{xy}{1+x+y})=\psi_F(\frac{xy}{c}).
\end{equation}
Similarly,
\begin{equation}
 \frac{H(x+y)}{H(x)\cdot H(y)}=\psi_F(\frac{xy}{c}).
\end{equation}
Then 
\begin{equation}\label{eqn 5.2.11}
\frac{Q(x+y)}{Q(x)\cdot Q(y)}=\frac{H(x+y)}{H(x)\cdot H(y)}=\psi_F(\frac{xy}{c}).
\end{equation}
Therefore $Q$ and $H$ differ only by {\bf an additive} character on $P_{F}^{d}$ and then we can write:
\begin{enumerate}
 \item $W(\chi,\psi_F)\equiv\chi(c)\mod{\mu_{p^\infty}}$ if $a(\chi)$ is even,
 \item $W(\chi,\psi_F)\equiv\chi(c)G(c)\mod{\mu_{p^\infty}}$ if $a(\chi)=2d+1$, where 
 $$G(c):=q_{F}^{-\frac{1}{2}}\cdot\sum_{x\in P_F^d/P_{F}^{d+1}}\psi_F(\frac{x^2}{2c})$$
 depends only on $c\in F^\times\mod{1+P_F}$.
 
\end{enumerate}

\end{rem}

\begin{rem}[{\bf On $G(c)$}]
Let $\bbF_q$ be a finite field with $q$ element and $q$ is odd. Let $\psi_0$ be a nontrivial additive character of $\bbF_q$.
Let $\eta$ be the quadratic character of $\bbF_q^\times$. Then from the generalization (cf. \cite{BRK}, p. 47, Problem 24) of Theorem
1.1.5 of \cite{BRK} on p. 12 we have 
\begin{equation}\label{eqn 6.1.11}
 \sum_{x\in \bbF_q}\psi_0(x^2)=\sum_{x\in\bbF_q^\times}\eta(x)\psi_0(x)=G(\eta,\psi_0).
\end{equation}
Now from Theorem \ref{Theorem 3.5} we can observe that $q^{-\frac{1}{2}}\sum_{x\in\bbF_q}\psi_0(x^2)$ is a {\bf fourth} root of unity.

Since $\nu_F(c)=a(\chi)+n(\psi_F)$, $P_F^d/P_F^{d+1}\cong k_{F}$, we can deduce $G(c)$ as 
$$G(c)=q_F^{-\frac{1}{2}}\sum_{x\in k_F}\psi'(x^2)=q_F^{-\frac{1}{2}}G(\eta,\psi'),$$
where $\psi'$ is a nontrivial additive character of $k_F$ and $\eta$ is the quadratic character of $k_F^\times$.
Therefore we can conclude that $G(c)$ is a fourth root of unity.

Furthermore, since $G(c)\cdot q_F^{\frac{1}{2}}=G(\eta,\psi')$ is a quadratic classical Gauss sum, from the computation of 
$\lambda_{K/F}$, where $K/F$ is a tamely ramified quadratic extension (cf. Subsection 3.4.1 ), we can conclude that $G(c)$ is actually
a $\lambda$-function for some quadratic tamely ramified local extension.
 
\end{rem}


\end{document}